\documentclass[envcountchap,sectrefs]{svmono}
\usepackage{mathptmx} % Incompatible with footnotes and mathfrak!!! KH 15.2. 
\usepackage{amsmath}
\usepackage{amssymb}

\usepackage{amsthm}
\usepackage{stmaryrd}

\newcommand\bv{{\bf{v}}}

\newcommand\lbr{\llbracket}
\newcommand\rbr{\rrbracket}

\renewcommand{\div}{\mathop{{\rm div}}\nolimits}

\renewcommand{\dot}{\overset .}
\newcommand{\pmat}[1]{\begin{pmatrix} #1 \end{pmatrix}}

\newcommand{\fB}{{\mathfrak B}}

\newcommand{\fS}{{\mathfrak S}}

\newcommand{\g}{{\mathfrak g}}

\newcommand{\fb}{{\mathfrak b}}

\newcommand{\fe}{{\mathfrak e}}

\newcommand{\fg}{{\mathfrak g}}
\newcommand{\fh}{{\mathfrak h}}

\newcommand{\fq}{{\mathfrak q}}
\newcommand{\fp}{{\mathfrak p}}

\renewcommand\sp{\mathfrak {sp}}

\renewcommand{\:}{\colon}
\newcommand{\1}{\mathbf{1}}

\newcommand{\cA}{\mathcal{A}}
\newcommand{\cB}{\mathcal{B}}

\newcommand{\cD}{\mathcal{D}}
\newcommand{\cE}{\mathcal{E}}
\newcommand{\cF}{\mathcal{F}}
\newcommand{\cG}{\mathcal{G}}
\newcommand{\cH}{\mathcal{H}}
\newcommand{\cK}{\mathcal{K}}
\newcommand{\cL}{\mathcal{L}}
\newcommand{\cM}{\mathcal{M}}
\newcommand{\cN}{\mathcal{N}}
\newcommand{\cO}{\mathcal{O}}
\newcommand{\cP}{\mathcal{P}}

\newcommand{\cR}{\mathcal{R}}
\newcommand{\cS}{\mathcal{S}}

\newcommand{\cV}{\mathcal{V}}

\newcommand\bt{{\bf{t}}}

\newcommand\bm{{\bf{m}}}
\newcommand\bx{{\bf{x}}}
\newcommand\by{{\bf{y}}}

\newcommand\bp{{\bf{p}}}

\newcommand{\bg}{\mathbf{g}}

\newcommand{\eset}{\emptyset}

\newcommand{\dd}{{\tt d}}

\newcommand{\subeq}{\subseteq}
\newcommand{\supeq}{\supseteq}

\newcommand{\into}{\hookrightarrow}
\newcommand{\eps}{\varepsilon}

\newcommand{\N}{{\mathbb N}}
\newcommand{\Z}{{\mathbb Z}}
\newcommand{\R}{{\mathbb R}}
\newcommand{\C}{{\mathbb C}}

\renewcommand{\H}{{\mathbb H}}
\newcommand{\T}{{\mathbb T}}
\newcommand{\V}{{\mathbb V}}

\newcommand{\bS}{{\mathbb S}}

\newcommand{\bV}{{\mathbb V}}

\renewcommand{\hat}{\widehat}

\renewcommand{\tilde}{\widetilde}

\newcommand{\Aff}{\mathop{{\rm Aff}}\nolimits}

\newcommand{\GL}{\mathop{{\rm GL}}\nolimits}
\newcommand{\SL}{\mathop{{\rm SL}}\nolimits}

\newcommand{\SO}{\mathop{{\rm SO}}\nolimits}
\newcommand{\SU}{\mathop{{\rm SU}}\nolimits}
\newcommand{\OO}{\mathop{\rm O{}}\nolimits}

\newcommand{\U}{\mathop{\rm U{}}\nolimits}

\newcommand{\Sp}{\mathop{{\rm Sp}}\nolimits}
\newcommand{\Sym}{\mathop{{\rm Sym}}\nolimits}

\newcommand{\gl}  {\mathop{{\mathfrak{gl} }}\nolimits}

\newcommand{\fsl} {\mathop{{\mathfrak{sl} }}\nolimits}

\newcommand{\so}  {\mathop{{\mathfrak{so} }}\nolimits}

\newcommand{\Exp}{\mathop{{\rm Exp}}\nolimits}
\newcommand{\Fix}{\mathop{{\rm Fix}}\nolimits}

\newcommand{\Ad}{\mathop{{\rm Ad}}\nolimits}

\renewcommand{\Re}{\mathop{{\rm Re}}\nolimits}
\renewcommand{\Im}{\mathop{{\rm Im}}\nolimits}

\newcommand{\tr}{\mathop{{\rm tr}}\nolimits}

\newcommand{\Hom}{\mathop{{\rm Hom}}\nolimits}

\newcommand{\pr}{\mathop{{\rm pr}}\nolimits}

\newcommand{\Isom}{\mathop{{\rm Isom}}\nolimits}

\newcommand{\Herm}{\mathop{{\rm Herm}}\nolimits}

\newcommand{\Heis}{\mathop{{\rm Heis}}\nolimits}

\newcommand{\bL}{{\mathbb L}}

\newcommand{\Aut}{\mathop{{\rm Aut}}\nolimits}

\newcommand{\Diff}{\mathop{{\rm Diff}}\nolimits}

\newcommand{\diag}{\mathop{{\rm diag}}\nolimits}
\newcommand{\End}{\mathop{{\rm End}}\nolimits}
\newcommand{\id}{\mathop{{\rm id}}\nolimits}

\newcommand{\Res}{\mathop{{\rm Res}}\nolimits}

\renewcommand{\dim}{\mathop{{\rm dim}}\nolimits}

\newcommand{\im}{\mathop{{\rm im}}\nolimits}

\newcommand{\supp}{\mathop{{\rm supp}}\nolimits}

\newcommand{\Spann}{\mathop{{\rm span}}\nolimits}
\newcommand{\ev}{\mathop{{\rm ev}}\nolimits}

\newcommand{\Bil}{\mathop{{\rm Bil}}\nolimits}
\newcommand{\Sesq}{\mathop{{\rm Sesq}}\nolimits}

\newcommand{\dS}{\mathop{{\rm dS}}\nolimits}

\newcommand{\vphi}{\varphi}
\renewcommand{\phi}{\varphi}

\newcommand{\Rarrow}{\Rightarrow}
 
\newcommand{\oline}{\overline}

\newcommand{\la}{\langle}
\newcommand{\ra}{\rangle}

\newcommand{\up}{\mathop{\uparrow}}

\newcommand{\res}{\vert}

\newcommand{\Spec}{{\rm Spec}}
\newcommand{\Spin}{{\rm Spin}}

\newcommand{\ssssarr}{\hbox to 15pt{\rightarrowfill}}
\newcommand{\sssarr}{\hbox to 20pt{\rightarrowfill}}
\newcommand{\ssarr}{\hbox to 30pt{\rightarrowfill}}
\newcommand{\sarr}{\hbox to 40pt{\rightarrowfill}}
\newcommand{\arr}{\hbox to 60pt{\rightarrowfill}}
\newcommand{\larr}{\hbox to 60pt{\leftarrowfill}}
\newcommand{\Arr}{\hbox to 80pt{\rightarrowfill}}

\newcommand{\mapright}[1]{\smash{\mathop{\arr}\limits^{#1}}}

\newcommand{\hgf}{{}_2F_1}
\newcommand{\sx}{\sigma_{V}}
\newcommand{\sr}{\sigma_{\R}}
\newcommand{\kl}{K_\lambda}
\newcommand{\Lnp}{\mathbb{L}^{n}_+}
\newcommand{\Tu}{T_{V_+}}
\newcommand{\wH}{\H^n_V}

\newcommand{\rI}{\mathrm{I}}

  \newcommand{\lf}[2]{[ #1,#2 ]_V}
\newcommand\smm[1]{{\small\arraycolsep=0.3\arraycolsep\ensuremath{\begin{matrix}#1\end{matrix}}}}

\spnewtheorem{thm}{Theorem}[section]{\bf}{\it}
\spnewtheorem{conj}{Conjecture}[section]{\bf}{\it}
\spnewtheorem{lem}[thm]{Lemma}{\bf}{\it}
\spnewtheorem{cor}[thm]{Corollary}{\bf}{\it}
\spnewtheorem{prop}[thm]{Proposition}{\bf}{\it}
\spnewtheorem{clm}[thm]{Claim}{\bf}{\it}
\spnewtheorem{clms}[thm]{Claims}{\bf}{\it}
\spnewtheorem{defn}[thm]{Definition}{\bf}{\rm}
\spnewtheorem{convent}[thm]{Convention}{\bf}{\rm}
\spnewtheorem{rem}[thm]{Remark}{\bf}{\rm}
\spnewtheorem{rems}[thm]{Remarks}{\bf}{\rm}
\spnewtheorem{nte}[thm]{Note}{\bf}{\rm}
\spnewtheorem{ntes}[thm]{Notes}{\bf}{\rm}
\spnewtheorem{ex}[thm]{Example}{\bf}{\rm}
\spnewtheorem{exs}[thm]{Examples}{\bf}{\rm}
\spnewtheorem{probl}[thm]{Problem}{\bf}{\rm}
\spnewtheorem{probs}[thm]{Problems}{\bf}{\rm}
\spnewtheorem{numba}[thm]{\nocaption}{\bf}{\rm}
\newenvironment{prf}{\begin{proof}}{\end{proof}}

\usepackage{type1cm}         

\usepackage{makeidx}         % allows index generation
\usepackage{multicol}        % used for the two-column index
\usepackage[bottom]{footmisc}% places footnotes at page bottom

\usepackage{bibunits} % for local bibliographies 

\makeindex             % used for the subject index
\newcommand{\ip}[2]{\la #1,#2 \ra}
\newcommand{\du}[2]{\la #1,#2 \ra}

\newcommand{\hE}{\widehat{\cE}}
\newcommand{\hD}{\widehat{\cD}}
\newcommand{\hU}{\widehat{U}}
\newcommand{\he}{\hat{\eta}}

\newcommand{\rS}{\mathrm{S}}
\newcommand{\rU}{\mathrm{U}}

\begin{document}

\author{Karl-Hermann Neeb, Gestur \'Olafsson} 
\title{Reflection Positivity}
\subtitle{-- A Representation Theoretic Perspective --}
\maketitle
\frontmatter%%%%%%%%%%%%%%%%%%%%%%%%%%%%%%%%%%%%%%%%%%%%%%%%%%%%%%

\preface

\begin{bibunit} [abbnamed]

The concept of reflection positivity (RP) occurs as an important 
theme in various areas of mathematics and physics: 
\begin{itemize}
\item in the representation theory of Lie groups it establishes a passage of a 
unitary representation of a symmetric Lie group (such as the euclidean 
motion group) to a unitary representation of the Cartan dual group 
(such as the Poincar\'e group) (\cite{FOS83}, \cite{JOl00, JOl98}, 
\cite{NO14}, \cite{JPT15}). 
\item in constructive 
Quantum Field Theorem (QFT) it arises as the condition of 
Osterwalder--Schrader (OS) positivity for a 
euclidean field theory to correspond to a relativistic one 
(\cite{GJ81, Ja08, Ja18}, \cite{Os95a, Os95b}, \cite{OS73, OS75}). 
\item for stochastic processes it is weaker than the Markov property 
and specifies processes arising in lattice gauge theory. 
It plays a central  role in the mathematical study of 
phase transitions and symmetry breaking (\cite{FILS78}, \cite{JJ16, JJ17}, 
\cite{Nel73}).
\item in analysis it is a crucial condition 
that leads to inequalities such as the 
Hardy--Littlewood--Sobolev inequality (\cite{FL10}).
\end{itemize}

Only recently it became apparent that there are many hidden and still not 
sufficiently well understood structures underlying the duality 
between unitary representations of a symmetric Lie group and its dual. 
Establishing reflection positivity in this context  
requires new analytic methods and new geometric
insight into constructions and realizations of representations
in analytic contexts. New developments concern
analytic issues such as criteria for integrating Lie algebra 
representations to Lie group representations, 
reflection positive functions, distributions and kernels,
new dilation techniques for representations and 
unexpected connections between Kubo--Martin--Schwinger (KMS) 
states of $C^*$-algebras and 
reflection positive unitary representations. 

This was our motivation to write this ``light'' introduction to the 
representation theoretic aspects of reflection positivity 
to present this perspective on a level suitable for doctoral students. 

{\bf Acknowledgment:} We are most grateful to Jan Frahm for reading earlier versions of this manuscript 
and for helping us to eliminate typos and inaccuracies.

\end{bibunit} 

\tableofcontents

\mainmatter%%%%%%%%%%%%%%%%%%%%%%%%%%%%%%%%%%%%%%%%%%%%%%%%%%%%%%%

\chapter{Introduction} 
\label{ch:1}

\begin{bibunit}[abbnamed] 

In the context of quantum physics reflection positivity is often 
related to Wick rotation, which roughly means multiplying 
the time coordinate by $i = \sqrt{-1}$. This can be made precise 
and used in analytic constructions if this process is related to analytic 
continuation of the time variable to a domain in the complex plane which 
provides a means to go back and forth between real and imaginary time. 

A duality of a similar flavor also exists in the context of Lie groups, 
where it arose almost a century ago in the work of \'E.~Cartan on 
the classification of symmetric spaces. 
Here one considers a 
{\it symmetric Lie group} \index{Lie group!symmetric} 
$(G,\tau)$, 
i.e., a Lie group $G$, endowed with an involutive automorphism~$\tau$. 
Then the Lie algebra $\g$ of $G$ decomposes into $\tau$-eigenspaces 
\[ \fh = \{ x \in \g \: \tau x = x\} \quad \mbox{ and } \quad 
\fq = \{ x \in \g \: \tau x = -x\}.\] 
From the bracket relations 
$[\fh,\fh] \subeq \fh, [\fh,\fq] \subeq \fq$, and $[\fq,\fq] \subeq \fh$ 
it then follows that the {\it Cartan dual} 
\index{Lie algebra!symmetric dual}
\index{Lie algebra!symmetric} 
\index{dual symmetric Lie algebra} 
\[ \g^c := \fh + i \fq \] 
also is a Lie subalgebra of the complexified Lie algebra $\g_\C = \g \oplus i\g$. 
We thus obtain a duality relation between symmetric 
Lie groups $(G,\tau)$ and $(G^c, \tau^c)$, where $G^c$ denotes a 
Lie group with Lie algebra $G^c$ and $\tau^c$ an involutive automorphism 
acting by $x + i y \mapsto x - i y$ on the Lie algebra $\g^c = \fh + i \fq$. 
The classical examples from quantum 
physics are the euclidean motion group $G = E(d) \cong  \R^d \rtimes \OO_d(\R)$ 
and the automorphism $\tau$ of $E(d)$ induces by time reflection. 
This establishes a duality with the 
Poincar\'e group $G^c = P(d) \cong  \R^{1,d-1} \rtimes \OO_{1,d-1}(\R)$. 

In many cases both groups $G$ and $G^c$ are contained in one complex 
Lie group $G_\C$ and $H = G \cap G^c$ is a Lie subgroup with 
Lie algebra $\fh$, contained in both. Therefore any passage from 
$G$ to $G^c$ should be related to analytic continuation to domains in $G_\C$ 
whose closure intersects both groups $G$ and $G^c$. On the 
Lie algebra level the passage from $\g = \fh \oplus \fq$ to $\g^c = \fh \oplus 
i\fq$ very much resembles Wick rotation because the elements of $\fq$ are 
multiplied by~$i$ (cf.\ \cite{HH17}, where this context is discussed for 
pseudo-Riemannian manifolds). Simple examples are 
\begin{itemize}
\item the circle 
group $G = \T \subeq G_\C = \C^\times$ with the dual group 
$G^c = \R^\times$ and $\tau(z) = z^{-1}$. 
\item the additive group $G = \R \subeq G_\C = \C$ with $\tau(x) = -x$ and 
$G^c = i \R$. 
\item the group $G = \GL_n(\R) \subeq G_\C = \GL_n(\C)$ with $\tau(g) = g^{-\top}$ 
and $G^c = \U_n(\C)$. 
\item the group $G = \OO_n(\R) \subeq G_\C = \OO_n(\C)$ with 
$\tau(g) = r g r$, where $r$ is an orthogonal reflection in a hyperplane 
and $G^c = \OO_{1,n-1}(\R)$. 
\item the group $G = \OO_{1,n}(\R) \subeq G_\C = \OO_{n+1}(\C)$ with 
$\tau(g) = \theta  g \theta$, where $\theta$ is an orthogonal reflection in a Minkowski hyperplane 
and $G^c = \OO_{2,n-1}(\R)$. 
\end{itemize}

If $U \: G \to \U(\cE)$ is a unitary representation of $G$, then, 
for $x \in \g$, the infinitesimal generator $\partial U(x)$ of the 
unitary one-parameter $t \mapsto U(\exp tx)$ is a skew-adjoint operator,  
and multiplication by $i$ leads to a selfadjoint operator. Therefore 
we cannot expect unitary representations of $G$ and $G^c$ to live on 
the same Hilbert space. What we need instead is some extra structure 
on $\cE$ that permits us to construct another Hilbert space on which 
a unitary representation of $G^c$ may be implemented. This is 
where reflection positivity comes in as a framework establishing a 
bridge between unitary representations of $G$ and $G^c$. 
This perspective isolates many of the key features of reflection positivity and 
subsumes not only the representation theoretic aspects of classical applications 
along the lines of Osterwalder and Schrader \cite{OS73, OS75}, 
but also quite recent developments in Algebraic Quantum Field Theory (AQFT), 
where Haag--Kastler nets of operator algebras 
are constructed on space times by methods relying very much on the 
unitary or anti-unitary 
representations of the groups involved (\cite{Bo92, BJM16}, \cite{NO17, Ne17}). 
Another recent branch of applications of reflection positivity for the euclidean 
conformal group along these lines concerns 
Hardy--Littlewood--Sobolev inequalities in analysis 
(see \cite{FL10, FL11, NO14}). 

The extra structure required on the Hilbert space $\cE$  
can be specified axiomatically as follows. 
A {\it reflection positive Hilbert space} \index{reflection positive!Hilbert space}
is a triple $(\cE,\cE_+,\theta)$,  
consisting of a Hilbert space 
$\cE$ with a unitary involution $\theta$ and a closed subspace  
$\cE_+$ satisfying 
\[ \la \xi,\xi\ra_\theta := \la \xi, \theta \xi \ra \geq 0 
\quad \mbox{  for } \quad \xi \in \cE_+.\]
This structure immediately leads to a new Hilbert space 
$\hat\cE$ that we obtain from the positive semidefinite form \
$\la \cdot,\cdot\ra_\theta$ on $\cE_+$. 
We write $q \: \cE_+ \to \hat \cE, \xi \mapsto \hat\xi$ for the natural map. 
Bounded or unbounded operators $S$ on $\cE_+$ preserving the kernel of $q$ 
induce an operator $\hat S$ on $\hat\cE$ via $\hat S \hat\xi := \hat{S\xi}$. 
The passage $S \mapsto \hat S$ is called the 
{\it Osterwalder--Schrader (OS) transform}. 
\index{Osterwalder--Schrader transform} 

On the level of $\cE$ (the euclidean side), we consider 
a unitary representaton $U$ of a symmetric Lie group $(G,\tau)$ 
on a reflection positive Hilbert space $(\cE,\cE_+,\theta)$. 
There are several ways to express the compatibility of the representation 
$U$ with $\cE_+$ and~$\theta$. One is the compatibility relation 
\[ \theta U_g \theta = U_{\tau(g)} \quad \mbox{ for }\quad g \in G \] 
between $\tau$ and the unitary involution $\theta$ 
and another is the invariance of $\cE_+$ under the operators 
$U_h$, where $h$ belongs to the identity component $H := G^\tau_0$ 
of the group $G^\tau$ of $\tau$-fixed points in~$G$. \index{$G^\tau$}
These two already ensure that the OS transform yields a 
unitary representation $(\hat U_h)_{h \in H}$ on $\hat\cE$. 
We are aiming at a unitary representation of the Cartan dual group 
$G^c$ on $\hat\cE$ and this can only be achieved by additional requirements. 
On the algebraic level, if we only consider the representation $\dd U$ 
of the Lie algebra $\g$ on the subspace $\cE^\infty$ of smooth vectors for $(U,\cE)$, 
it suffices to have a subspace $\cD \subeq \cE_+ \cap \cE^\infty$ 
which is $\g$-invariant. 
Then OS transformation immediately leads to a representation 
of $\g^c = \fh + i \fq$ by skew-symmetric operators on 
the subspace $\hat\cD \subeq \hat\cE$ by 
\[ x + i y \mapsto \big(\dd U(x)\res_\cD + i \dd U(y)\res_\cD\big)\,\hat{}.\]
This simple passage already shows how the 
$\theta$-twisting of the scalar product on $\cE_+$ turns the symmetric operators  
$i\dd U(y), y \in \fq,$ on $\cD$ into skew-symmetric operators on $\hat\cD$, 
but it completely ignores all issues related to essential selfadjointness 
and, accordingly, integrability to group representations. We therefore 
need more global ways to express the reflection positivity requirements. 

One way to express such requirements refers 
to subsemigroups $S \subeq G$ which are {\it symmetric} 
\index{subsemigroup!symmetric} in the 
sense that they are invariant under the involution $s \mapsto s^\sharp := \tau(s)^{-1}$. 
Then we call $(U,\cE)$ 
\index{representation!reflection positive!w.r.t.\ subsemigroup $S$} 
{\it reflection positive with respect to $S$} if 
$U_s \cE_+ \subeq \cE_+$ holds for $s \in S$. 
Then OS transformation yields a $*$-representation 
$(\hat U_s)_{s \in S}$ of the involutive semigroup $(S,\sharp)$, and 
if $S$ has interior points one can expect this representation 
to ``extend analytically'' to a unitary representation of a dual group~$G^c$. 
A typical situation arises for $(G,S,\tau) = (\R,\R_+, -\id_\R)$ 
in euclidean field theory from the one-parameter group of time 
translations. Then $(\hat U_t)_{t \geq 0}$ is a one-parameter group of hermitian 
contractions on $\hat\cE$, hence of the form $\hat U_t = e^{-tH}$ for a 
positive selfadjoint operator $H$,  and $U^c_t := e^{itH}$ defines a unitary 
representation of the dual group $G^c = i \R$ with positive spectrum 
(in QFT $H$ corresponds to the Hamiltonian, the energy observable,
 which should be positive). 

There are, however, many situations where there are no natural 
symmetric subsemigroups $S \subeq G$, or some which do not have interior points, 
such as the subsemigroup $S$ of the euclidean motion group mapping a closed 
half space into itself. In this case the reflection positivity requirements 
on $(U,\cE,\cE_+,\theta)$ have to be formulated differently. 
Instead of a subsemigroup, we consider a domain $G_+ \subeq G$ 
(mostly open or with dense interior) and the reflection positivity condition 
is inspired by situations in QFT, where Hilbert spaces are generated by field operators: 
Instead of fixing $\cE_+$ a priori, we consider a real linear space $V$ and a linear map 
$j \: V \to \cH$ whose range generates $\cE$ under $U_G$ and $\theta$ 
and call $(U,\cE, j,V)$ 
\index{representation! reflection positive w.r.t.\ subset $G_+$} 
{\it reflection positive with respect to $G_+$} if 
the subspace $\cE_+ := \lbr U_{G_+}^{-1} j(V)\rbr$ is $\theta$-positive. 
The prototypical examples arise 
for circle groups $G = \R/\beta\Z$ with $\tau(g) = g^{-1}$, where 
$G_+ =[0,\pi] + \beta\Z$ is a half circle. In physics they occur 
in the context of quantum statistical mechanics, where 
$\beta$ plays the role of an inverse temperature (\cite{Fro11}). 
Both approaches, the one based on semigroups $S$ and on domains $G_+$ 
lead to situations in which we can use 
suitable integrability results (cf.~Chapter~\ref{ch:6}) to obtain 
unitary representations of the $1$-connected Lie group $G^c$ with 
Lie algebra $\g^c = \fh + i \fq$ on $\hat\cE$. \\

We now turn to the contents of this book. 
We shall not turn to finer aspects of 
unitary representations of higher-dimensional Lie groups 
before Chapter~\ref{ch:6}. In the first half, 
Chapters~\ref{ch:2} to \ref{ch:5}, we deal with 
rather concrete contexts and how they relate to reflection 
positivity. Various aspects concerning general Lie groups 
are postponed to Chapters~\ref{ch:6} to \ref{ch:9}. 

Chapter~\ref{ch:2} develops the notion of a reflection positive 
Hilbert space $(\cE,\cE_+,\theta)$ from various 
perspectives. For instance $\theta$-positive subspaces 
$\cE_+$ can be constructed as graphs of contractions from the $1$ to the $-1$-eigenspace 
of $\theta$ (Section~\ref{subsec:2.1.1}). 
In physics reflection positive Hilbert spaces often arise from 
distributions. Here $\cE$ is a Hilbert space arising by completing 
the space $C^\infty_c(M)$ of smooth test functions on a manifold with respect to 
a singular scalar product 
\begin{equation}
  \label{eq:D1}
\la \xi,\eta \ra = \int_{M \times M} \oline{\xi(x)} \eta(y)\, dD(x,y),
\end{equation}
where $D$ is a positive definite distribution on $M \times M$. 
Then $\theta$ is supposed to come from a diffeomorphism of $M$ 
and $\cE_+$ from an open subset $M_+ \subeq M$, which leads to 
the reflection positivity condition
\begin{equation}
  \label{eq:D2}
\int_{M_+ \times M_+} \oline{\xi(x)}\xi(y)\, dD(\theta(x),y) \geq 0
\quad \mbox{ for }\quad \xi \in C^\infty_c(M_+)
\end{equation}
(Section~\ref{subsec:2.1.3}). 
Typical concrete examples arise from reflections of complete Riemannian 
manifolds and resolvents $(\lambda \1 - \Delta)^{-1}$ of the Laplacian 
(Section~\ref{subsec:2.1.4}). Motivated by these examples, we briefly 
describe an abstract operator theoretic context for reflection positivity 
that we feel should be developed further (Section~\ref{sec:2.6}; 
\cite{JR07, An13, AFG86, Di04}). 
In probabilistic contexts, 
one encounters situations satisfying the Markov condition, i.e., 
there exists a subspace $\cE_0 \subeq \cE_+$ mapped 
isometrically onto $\hat\cE$ (Section~\ref{subsec:2.1.2}). 

The connection between reflection positive Hilbert spaces 
and representation theory is introduced in Chapter~\ref{ch:4}. 
After discussing some general properties of the OS transform, 
we introduce symmetric Lie groups $(G,\tau)$, 
symmetric subsemigroups $S \subeq G$ and various kinds of reflection 
positivity conditions for unitary representations. 
As in the representation theory of operator algebras, where 
cyclic representations are generated from states, 
it is an extremely fruitful approach 
to generate representations of groups and semigroups 
by positive definite functions via the Gelfand--Naimark--Segal (GNS) construction. 
Here reflection positivity requirements lead to the concept 
of a reflection positive function whose values may also 
comprise bounded operators or bilinear forms 
(Section~\ref{sec:3.4}). 

After these generalities, we turn in Chapter~\ref{ch:4} 
to the most elementary concrete symmetric Lie group 
$(G,\tau) = (\R,-\id_\R)$, where 
the RP condition is based on the subsemigroup $\R_+$. 
Although this Lie group is quite trivial, reflection positivity 
on the real line has many interesting facets and is therefore quite rich. 
As reflection positive functions play a crucial role,  
we start Chapter~\ref{ch:4} with 
functions on intervals $(-a,a)\subeq \R$ which are 
reflection positive in the sense that both kernels 
\[ (\phi(x-y))_{-a/2 < x,y < a/2} \quad \mbox{ and } \quad 
(\phi(x+y))_{0 < x,y < a/2} \] 
are positive definite. 
For $a = \infty$, this combines the positive definiteness 
of the group $(\R,+)$ with the positive definiteness on the 
involutive semigroup $(\R_+, \id_\R)$. Accordingly, 
these two conditions ask for techniques related to Fourier- and Laplace transforms. 

Reflection positive representations of $(\R,\R_+, - \id_\R)$ are 
unitary one-parameter groups $(U_t)_{t \in \R}$ on a reflection positive 
Hilbert space $(\cE,\cE_+,\theta)$ satisfying 
$U_t \cE_+ \subeq \cE_+$ for $t>0$ and 
$\theta U_t \theta = U_{-t}$ for $t\in \R$. 
On $\hat\cE$ this leads to a semigroup $(\hat U_t)_{t \geq 0}$ of hermitian 
contractions and we show in particular that, under  
the OS transform, 
fixed points of $U$ on $\cE$ correspond to fixed points of $\hat U$ in $\hat\cE$
(Proposition~\ref{prop:e.5b}). 

For $(\R,\R_+, - \id_\R)$, we obtain a 
complete classification of reflection positive representations in terms of 
integral formulas, resp., spectral theorems. 
From these results one obtains an interesting converse of the 
OS transform in this context. Any hermitian contraction semigroup 
$(C_t)_{t \geq 0}$ on a Hilbert space $\cH$ 
has a so-called minimal dilation represented by the 
reflection positive function $\psi(t) := C_{|t|}$ on~$\R$. 

We conclude Chapter~\ref{ch:4} by showing that, for any 
reflection positive one-parameter group for which $\cE_+$ is cyclic 
and fixed points are trivial, the space $\cE_+$ is outgoing in the sense of 
Lax--Phillips scattering theory (Proposition~\ref{prop:4.11}). 
This establishes a remarkable connection between reflection 
positivity and scattering theory that leads to a normal form of 
reflection positive one-parameter groups by translations on 
spaces of the form $\cE = L^2(\R,\cH)$ with $\cE_+ = L^2(\R_+,\cH)$. 
Applying the Fourier transform to our concrete dilation model leads precisely 
to this normal form.

In Chapter~\ref{ch:5} 
we still work with the same symmetric group $(\R,-\id_\R)$ or rather its 
quotient circle group $\R/2\beta\Z\cong \T$, but now 
reflection positivity is based on the interval $[0,\beta]$, where 
$\beta > 0$ is interpreted as an inverse temperature in physical models 
(\cite{NO15b}, \cite{KL81, KL81b}, \cite{Fro11}, \cite{NO16}). 
In this context reflection positivity is closely connected 
with the Kubo--Martin--Schwinger (KMS) condition for states of $C^*$-dynamical 
systems (\cite{Fro11, BR02}). This connection is established by a 
purely representation theoretic perspective on the KMS condition 
formulated as a property of form-valued positive definite functions on~$\R$: 
Let $V$ be a real vector space 
and $\Bil(V)$ be the space of real bilinear maps $V \times V \to \C$.
For $\beta > 0$, we consider the open strip 
\[ \cS_\beta := \{ z \in \C \: 0 < \Im z < \beta\}.\] 
We say that a positive definite 
function $\psi \: \R \to \Bil(V)$ (Definition~\ref{def:a.3}) 
satisfies the 
\index{KMS condition} 
{\it $\beta$-KMS condition} if $\psi$ 
extends to a pointwise continuous 
function $\psi$ on $\oline{\cS_\beta}$ which is 
pointwise holomorphic on the interior $\cS_\beta$ and satisfies 
\[ \psi(i \beta+t) = \oline{\psi(t)}\quad \mbox{ for } \quad t \in \R.\]  
The classification of such functions in terms of an integral representation 
is based on relating them to  standard (real) subspaces of a complex Hilbert space 
which occur naturally in the modular theory of operator algebras 
(\cite{Lo08}). These are closed real subspaces $V \subeq \cH$ for which 
$V \cap i V = \{0\}$ and $V + i V$ is dense. Any standard subspace determines a pair 
$(\Delta, J)$ of {\it modular objects}, \index{modular objects} 
where $\Delta$ is a positive selfadjoint 
operator and $J$ is an anti-linear involution 
(a {\it conjugation}) \index{conjugation!on Hilbert space} 
satisfying $J\Delta J = \Delta^{-1}$. 
The connection is established by 
\[  V = \Fix(J\Delta^{1/2}) = \{ \xi \in \cD(\Delta^{1/2}) \: J\Delta^{1/2} \xi = \xi\}.  \] 
This connects reflection positivity very naturally to the 
aforementioned recent developments in AQFT initiated by the work of 
Borchers \cite{Bo92} and now exploited systematically in the constructions 
of QFTs (see \cite{BJM16}, \cite{BLS11},  \cite{LL14}, \cite{LW11} for typical applications). 
Here standard subspaces can be considered as ``one-particle space analogs'' 
of the modular data $(\Delta, J)$ arising in Tomita--Takesaki theory 
in the context of von Neumann algebras \cite{Lo08, LL14}. 

After the discussion of the concrete examples, the reader 
should be prepared to appreciate the Lie theoretic aspects of the 
theory, which start  
in Chapter~\ref{ch:6} with the development of the integration techniques that 
are used to obtain a unitary representation of the simply connected 
Lie group $G^c$ on $\hat\cE$ from a reflection positive 
representation of $(G,\tau)$ on $(\cE,\cE_+,\theta)$. 
Our techniques are based on the fact that 
the Hilbert spaces are mostly constructed from $G$-invariant positive definite kernels 
or positive definite $G$-invariant distributions. 
We have already seen that any reflection positive representation of $(G,\tau)$  
immediately yields a unitary representation $U^c$ of 
$H = G^\tau_0$ on $\hat\cE$, so that it remains to find a unitary 
representation of the one-parameter groups $\exp_{G^c}(\R i x)$ for $x\in \fq$. 
By Stone's Theorem, the main point is to show that, 
for $y \in \fq$, the symmetric operator $\hat{\dd U}(y)$ defined 
on a dense subspace 
of $\hat\cE$ is essentially selfadjoint. In our geometric setting, 
this can be derived from 
Fr\"ohlich's Theorem \cite{Fro80} 
which provides a criterion for the essential selfadjointness of a symmetric operator 
in terms of the existence of enough local solutions of the corresponding 
linear ODE. The natural setting for the corresponding 
integrability results are pairs 
$(\beta,\sigma)$ of a homomorphism  $\beta \: \g \to \cV(M)$ 
to the Lie algebra $\cV(M)$ of smooth vector fields on a manifold $M$ 
which is compatible with a smooth $H$-action~$\sigma$. 
Then, for any smooth kernel $K$ on $M$ satisfying a suitable invariance condition 
with respect to $(\beta,\sigma)$, a unitary representation of 
$G^c$ on $\cH_K$ exists (Theorem~\ref{thm:4.8}). 
We also show that a similar result holds if we replace the kernel $K$ by a 
positive definite distribution $K\in C^{-\infty}(M\times M)$ 
compatible with $(\beta,\sigma)$ 
(Theorem \ref{thm:4.12}). From this we easily derive 
the existence of a unitary representation of the 
simply connected group $G^c$ on $\hat\cE$ for a 
reflection positive representation $(U,\cE)$ of~$(G,\tau)$.
Our exposition is based on new aspects developed in 
\cite{MNO15} which  complements the classical approach from \cite{FOS83}.
%, also used in \cite{BJM16}. 

The most effective tool to deal with 
reflection positive representations of 
symmetric Lie groups $(G,\tau)$ are reflection positive distributions on $G$  
%(cf.~\eqref{eq:D1} and \eqref{eq:D2}) 
and their relation with reflection positive distribution 
vectors of unitary representations. A key advantage of this 
method, outlined in Chapter~\ref{ch:7}, 
is that is leads naturally to reflection positive 
representations in Hilbert spaces of distributions on 
homogeneous spaces $G/H$, where $H$ may be non-compact. 
To illustrate this technique, we apply it to spherical 
representations of the Lorentz group $G = \OO_{1,n}(\R)$. 
These representations consist of two series, the principal series and 
the complementary series. Both have natural realizations in 
spaces of distributions on the $n$-sphere $\bS^n \cong G/P$ on which the Lorentz 
group acts by conformal maps; the principal series can even be realized 
in~$L^2(\bS^n)$.

In Chapter~\ref{ch:8} we take a closer look at the representations 
of the Poincar\'e group corresponding to scalar generalized free fields 
and their euclidean realizations by representations of the euclidean 
motion group $E(d)$. In particular, we 
discuss Lorentz invariant measures on the forward 
light cone $\oline{V_+}$ and 
the corresponding unitary representations of the Poincar\'e group. 
Applying the dilation 
construction to the time translation semigroup leads immediately 
to a Hilbert space $\cE$ on which we have a unitary 
representation of the euclidean motion group. 
We also characterize those representations 
which extend to the conformal group $\OO_{2,d}(\R)$ 
of Minkowski space~$\R^{1,d-1}$. Then the euclidean 
realization is a unitary representation of the Lorentz group 
$\OO_{1,d+1}(\R)$, acting as the conformal group on euclidean~$\R^d$. 

A particularly fascinating aspect 
of reflection positivity is its intimate connection 
with stochastic processes which is briefly scratched in Chapter~\ref{ch:9}. 
This is already interesting in the context of 
one-parameter groups, where it surfaces for example in the 
fact that the unitary one-parameter group $(U_t)_{t \in \R}$ 
leads by OS transform to the one-parameter group 
$U^c_t = e^{-it\Delta}$ on $L^2(\R^n)$, respectively to the heat 
semigroup $e^{t\Delta}$, is the translation action of $\R$ on a  suitable 
Lebesgue--Wiener space. This connection was observed by Nelson in \cite{Nel64} and led to 
a new approach to Feynman--Kac type integral formulas. 
In Chapter~\ref{ch:9} we describe some recent generalizations of 
classical results of Klein and Landau \cite{Kl78, KL75} 
concerning the interplay between reflection positivity and stochastic 
processes. Here the main step is the passage from the symmetric 
semigroup $(\R,\R_+,-\id_\R)$ to a more general context 
$(G,S,\tau)$. This leads to the concept of a 
$(G,S,\tau)$-measure space generalizing Klein's Osterwalder--Schrader 
path spaces for $(\R,\R_+,-\id_\R)$. 
A key result of this theory is the correspondence between 
$(G,S,\tau)$-measure spaces and the corresponding 
positive semigroup structures on the Hilbert space $\hat\cE$. 

%We conclude this booklet with a brief chapter on perspectives 
%and directions for further research. 

\subsection*{Notation}

We write $\R_{\geq 0} := [0,\infty)$ for the closed half line, 
$\R_+ = (0,\infty)$ for the open half line and 
$\N = \{ 1,2,3,\ldots\}$ for the set of natural numbers. 

As customary in physics, all scalar products on Hilbert spaces $\cH$ will be linear in the second argument. A subset $S$ of $\cH$ is called {\it total} if 
\index{total subsets of Hilbert space} 
it spans a dense subspace. We write 
$\lbr S \rbr := \oline{\Spann S}$. 
%, so that elements of $\cH$ can symbolically be written in Dirac's notation as 
%{\it kets} $|y \ra$ and continuous linear functionals als 
%{\it bras} $\la x |$. 

We write $\U(\cH)$ for the unitary group of a Hilbert space~$\cH$. 

For a measure space $(X,\fS,\mu)$, we accordingly write 
\[ \la f, g \ra = \int_X \oline{f(x)}g(x)\, d\mu(x) \quad \mbox{ for } \quad f, g \in L^2(X,\mu).\] 

We write elements of $\R^d$ as 
$x = (x_0, x_1, \ldots, x_{d-1}) = (x_0, \bx)$. The standard inner product on
$\R^d$ is denoted $\la x,y \ra = x \cdot y = xy = \sum_{j = 0}^{d-1} x_j y_j$, 
and the Lorentzian inner product by 
\[ [x,y] = x_0 y_0 - \bx \by.\] 

In $d$-dimensional Minkowski space $M^d$ we write 
\[ V_+ := \{ p = (p_0, \bp) \in \R^d \: p_0 > 0, p_0^2 > \bp^2 \}  \] 
for the open {\it forward light cone}. 
\index{forward light cone} 

We write $C^\infty_c(M)$ for the space of 
complex-valued test functions and $\cS(\R^d)$ is the space of 
complex-valued Schwartz functions on $\R^d$. 
\index{Fourier transform!of a measure} 
For the {\it Fourier transform of a measure $\mu$} on the
dual $V^*$ of a finite-dimensional real vector space $V$, we write
\[ \hat\mu(x) := \int_{V^*}e^{-i\alpha(x)}\, d\mu(\alpha).\]
\index{Fourier transform!of a function}
The {\it Fourier transform of an $L^1$-function $f$} on $\R^d$ is defined by
\begin{equation}
  \label{eq:ftrafo}
\hat f(p) := \int_{\R^d} e^{-i\la p,x\ra} f(x) \,  d\lambda_d(x) 
= \frac{1}{(2\pi)^{d/2}} \int_{\R^d} e^{-i\la p,x\ra} f(x) \,  dx 
\end{equation}
which corresponds to the Fourier transform of the measure 
$f\lambda_d$, where $d\lambda_d(x) := (2\pi)^{-d/2} \cdot dx$ 
is a suitably normalized Lebesgue measure. 
We likewise define convolution of $L^1$-functions with respect to $\lambda_d$. 

For tempered distributions $D \in \cS'(\R^d)$, which we define as 
continuous anti-linear functionals on the Schwartz space $\cS(\R^d)$, 
we define the Fourier transform by 
\begin{equation}
  \label{eq:2}
\hat D(\phi) := D(\tilde \phi), \quad \mbox{ where} 
\quad \tilde \phi(p) = \hat\phi(-p) 
= \int_{\R^d} e^{i\la p,x\ra} \phi(x) \,  d\lambda_d(x).
\end{equation}
For $D_f(\phi) = \int_{\R^d} \oline{\phi(x)} f(x)\, d\lambda_d(x)$, we then have 
$\hat{D_f} = D_{\hat f}$ and for a tempered measure 
$\mu$ the corresponding distribution $D_\mu(\phi) := \int \oline\phi\, d\mu$ 
satisfies $\hat{D_\mu} = D_{\hat\mu}$ if we consider $\hat\mu$ as a function. 
For the point measure $\delta_0$ 
we then have in particular the relation 
\[ 1= \hat{\delta_0} \] 
 which corresponds to the normalized 
Lebesgue measure $\lambda_d$.

\end{bibunit} 

\chapter{Reflection positive Hilbert spaces} 
%and representations} 
\label{ch:2} 
\begin{bibunit}[abbnamed]

In this chapter we discuss the basic framework of 
reflection positivity: reflection positive Hilbert spaces. 
These are triples $(\cE,\cE_+, \theta)$, consisting of a Hilbert space 
$\cE$, a unitary involution $\theta$ on $\cE$ and a 
closed subspace $\cE_+$ which is {\it $\theta$-positive} 
in the sense that  \index{$\theta$-positive subspace} 
$\la \xi,\theta\xi\ra \geq 0$ for $\xi \in \cE_+$. This structure immediately 
leads to a new Hilbert space $\hat\cE$ and a linear map $q \: \cE_+ \to \hat\cE$ 
with dense range. When the so-called Markov condition 
is satisfied, there even exists a closed subspace $\cE_0 \subeq \cE_+$ mapped 
isometrically onto $\hat\cE$ (Section~\ref{subsec:2.1.2}). 
Reflection positive Hilbert spaces arise naturally 
in many different contexts: as graphs of contractions 
(Section~\ref{subsec:2.1.1}), from reflection positive 
distribution kernels on manifolds (Section~\ref{subsec:2.1.3}) 
and in particular from dissecting reflections of complete Riemannian 
manifolds and resolvents of the Laplacian 
(Section~\ref{subsec:2.1.4}). This motivates the short discussion 
of an abstract operator theoretic context of reflection positivity 
in Section~\ref{sec:2.6}. 

\section{Reflection positive Hilbert spaces} 
\label{sec:2.1}

We start with the definition of a 
reflection positive Hilbert space: 
\begin{defn}[Reflection positive Hilbert space] 
\label{def:x.1} Let $\cE$ be a real or complex Hilbert space and 
$\theta \in \U(\cE)$ be a unitary involution. 
A closed subspace $\cE_+ \subeq \cE$ is called {\it $\theta$-positive} 
\index{$\theta$-positive subspace of Hilbert space} 
if 
%$\|\eta \|_\theta^2=\ip{\eta }{\eta }_\theta := 
$\la \eta , \theta \eta\ra \geq 0$ for $\eta \in \cE_+$. 
We then call the triple $(\cE,\cE_+,\theta)$ a {\it reflection positive 
Hilbert space}. \index{Hilbert space!reflection positive} 
\end{defn}

If $(\cE,\cE_+,\theta)$ is a reflection positive Hilbert space, then 
\[\cN 
:= \{ \eta  \in \cE_+ \: \la \eta, \theta \eta \ra = 0\} 
= \{ \eta  \in \cE_+ \: (\forall \zeta \in \cE_+)\ \la \zeta, \theta \eta\ra  = 0\}\] 
is the subspace of $\cE_+$ on which the new scalar product 
\[ \la \eta, \xi \ra_\theta := \la \eta, \theta \xi \ra \] 
degenerates. We thus consider the quotient map 
\begin{equation}
  \label{eq:quotmap}
q \: \cE_+ \to \cE_+/\cN, \quad \eta  \mapsto \hat \eta
\end{equation}
and write $\hat\cE$ for the Hilbert space completion of $\cE_+/\cN$ with respect to 
the norm $\|\he\|_{\hE}:=\|\he\| := \sqrt{\la \eta, \theta \eta \ra}$. 
%If $\cD\subset\cE_+$ is a liner subspace, then we set 
%\[\hD=\{\he \: \eta\in\cD\} . \]
%We note that $\hD$ is again a linear subspace of $\hE$.

\section{Reflection positive subspaces as graphs}
\label{subsec:2.1.1}

To get a better picture of how reflection positive 
Hilbert spaces arise, we now describe a construction of these 
structures in terms of contractions on a subspace of~$\cE^\theta$. 

We start with a unitary involution $\theta$ on the 
Hilbert space  $\cE$. Then $\theta$ is diagonalizable 
with the two eigenspaces 
$\cE_{\pm 1}:=\{\eta \in\cE\: \theta \eta =\pm \eta\}$ and 
$\cE =\cE_{+1}\oplus\cE_{-1}$. Then the 
twisted inner product $\ip{\cdot }{\cdot} _\theta$ is positive 
definite on $\cE_{+1}$ and
negative definite on $\cE_{-1}$: 
\begin{equation}
  \label{eq:basis-prod}
\la \xi_+ + \xi_-, \xi_+ + \xi_-\ra_\theta 
%= \la \xi_+ + \xi_-, \xi_+ - \xi_-\ra 
= \|\xi_+\|^2 - \|\xi_-\|^2 \quad \mbox{ for } \quad 
\xi_\pm \in \cE_{\pm 1}. 
\end{equation}
Denote by $p_{\pm}$ the projection onto $\cE_{\pm 1}$. 
%To avoid the trivial cases we assume that both
%$\cE_{+1}$ and $\cE_{-1}$ are non-zero. 
Let $\cE_+ \subeq \cE$ be a $\theta$-positive subspace 
and $\cF := p_+(\cE_+)$ be its projection onto $\cE_{+1}$. 
Then $\cE_+ \cap \cE_{-1} = \{0\}$ implies that 
there exists a linear map $C \: \cF \to \cE_{-1}$ such that 
\[ \cE_+ = \cG(C) = \{ u + C u \: u \in \cF \} \] 
is the graph of $C$. Now \eqref{eq:basis-prod} yields 
$\ip{u+Cu}{u+Cu}_\theta = \|u\|^2-\|Cu\|^2\ge 0$ for $u\in\cF$,  
so that $\|C \| \leq 1$, i.e., $C$ is a contraction. 
If, conversely, $\cF \subeq \cE_{+1}$ is a subspace and 
$C \: \cF \to \cE_{-1}$ is a contraction, then 
its graph $\cG(C) \subeq \cE_{+1} \oplus \cE_{-1} = \cE$ 
is $\theta$-positive. Since $\cG(C)$ is closed if and only if 
$\cF$ is closed, we obtain the following lemma which 
provides a description 
of all $\theta$-positive subspaces in terms of contractions 
(cf.\ \cite[Lemma 5.1]{Jo02}): 

\begin{lem} \label{lem:pos-graph}
A closed subspace $\cE_+ \subeq \cE$ is $\theta$-positive if and only 
if there exists a closed subspace $\cF \subeq \cE_{+1}$  
and a contraction $C \:  \cF \to \cE_{-1}$ such that $\cE_+ = \cG(C)$. 
\end{lem}

\begin{rem} \label{rem:decomp} Let $(\cE,\cE_+,\theta)$ be a reflection positive 
Hilbert space. 

(a) Put $\cE_- := \theta(\cE_+)$. Then $\cE_+ \cap \cE_-$ is the maximal 
$\theta$-invariant subspace of $\cE_+$, and $\theta$-positivity of 
$\cE_+$ implies that it coincides with 
$\cE_0 := \{ v \in \cE_+ \: \theta v = v\}.$ 
This is the maximal subspace of $\cE_+$ on which $q$ is isometric. 

(b) For $\cE_+=\cG (C)$ as in Lemma~\ref{lem:pos-graph}, we have 
$\cE_0 = \cE_+\cap \cE_{+1}= \ker(C).$ 
In particular, $\cE_0 =\{0\}$ if and only if $C$ is injective. 

(c) Writing $\cE$ as $\cE_0 \oplus \cE_1$ with $\cE_1 := \cE_0^\bot$, 
the reflection 
positive Hilbert space is a direct sum of the trivial reflection positive 
Hilbert space $(\cE_0,\cE_0,\id)$ and the reflection 
positive Hilbert space 
$(\cE_1, \cE_{1,+},\theta_1)$, where 
$\theta_1 := \theta\res_{\cE_1}$, $\cE_{1,+} = \cE_1 \cap \cE_+$, 
and $\cE_{1,0} = \{0\}$. 

(d) If $(\cE,\cE_+,\theta)$ is a reflection positive Hilbert space, 
then $(\cE,\cE_-,\theta)$ is reflection positive as well. If $\cE_+ 
= \cG(C)$ as above, then $\cE_- = \cG(-C)$. 
\end{rem}

\section{The Markov condition} 
\label{subsec:2.1.2} 

For a reflection positive Hilbert space $(\cE,\cE_+,\theta)$, 
\eqref{eq:basis-prod} 
shows that the subspace \index{$\cE_0$} 
\[ \cE_0 = \{\xi \in \cE_+ \: \theta \xi = \xi\}\] 
is maximal 
with respect to the property that 
$q\res_{\cE_0} \: \cE_0 \to \hat\cE$ is isometric. In particular,
$\cE_0 \cong q(\cE_0)$ is a closed subspace of $\hat\cE$. 

An interesting special case arises if $\hat\cE = q(\cE_0)$. 
Then $q$ restricts to a unitary operator $\cE_0 \to \hat\cE$, so that 
$q \:  \cE_+ \to \hat\cE$ is a partial isometry with kernel 
$\cN = \cE_+ \ominus \cE_0$. The following lemma characterizes 
this situation in terms of the {\it Markov condition} \index{Markov condition} 
that originally 
arose in the context of stochastic processes  
(cf.~Chapter~\ref{ch:9}).  

\begin{defn} \label{def:markov-cond}
Let $(\cE,\cE_+,\theta)$ be a reflection positive Hilbert space. 
If  $\cE_0' \subeq \cE_0$ is a closed subspace, 
$\cE_- := \theta(\cE_+)$,  and $P_0, P_\pm$ are the orthogonal projections 
onto $\cE_0'$ and $\cE_\pm$, then we say that $(\cE,\cE_0', \cE_+,\theta)$ is a 
{\it reflection positive Hilbert space of Markov type} if 
\index{reflection positive!Hilbert space!of Markov type} 
\begin{equation}\label{eq:Markov}
P_+ P_0 P_- = P_+ P_-.
\end{equation}
\end{defn}

\begin{lem}\label{le:Markov} 
The Markov condition \eqref{eq:Markov} is equivalent to 
$\cE_0' = \cE_0$ and $q(\cE_0) = \hat\cE$.  
If it is satisfied, 
then 
\begin{itemize}
\item[\rm(a)]\ \ $\Gamma := q\res_{\cE_0} \: \cE_0 \to \hat\cE$ is a unitary 
isomorphism and $q = \Gamma \circ P_0\res_{\cE_+}$. 
\item[\rm(b)]\ \ If $\cE_+ + \cE_-$ is dense in $\cE$, then $\cE_+$ 
is maximal $\theta$-positive. 
\end{itemize}
\end{lem}

\begin{prf} If $\cE_0' = \cE_0$ and 
$q(\cE_0) = \hat\cE$, then 
$\cN = \cE_+ \cap \cE_-^\bot = \ker q$ 
implies that $\cN = \cE_+ \ominus \cE_0$. 
This leads to the orthogonal decomposition 
$\cE_+ + \cE_- = \theta(\cN) \oplus \cE_0 \oplus \cN$ 
and to $\cE_0 = \cE_+ \cap \cE_-$. 
Therefore $P_+ P_0 P_- = P_0 = P_+ P_-$. 

Suppose, conversely, that the Markov condition holds. 
%From $\theta P_+=P_-\theta$, we derive $\cE_0 \subseteq \cE_+\cap \cE_-$. 
For $u\in\cE_0 \subeq \cE_-$, we have 
$P_+ P_- u = u$, but $P_+ P_0 P_- u = P_0 u$, so that 
$\cE_0 = \cE_0'$. As $\cE_0\subset \ker(\theta-\1)$, 
we have $P_0 \theta = \theta P_0 = P_0$. This implies 
\[ P_+ \theta P_+ = P_+ P_- \theta = P_+ P_0 P_- \theta 
= P_0 \theta P_+ = P_0 P_+ = P_0\, .\]
For $u\in\cE_+$, we thus obtain 
$\du{u}{\theta u}= \du{u}{P_+\theta P_+u}=\du{u}{P_0u}=\|P_0u\|^2.$
Therefore $\cN = \ker q = \cE_+ \ominus \cE_0$ and 
$q(\cE_0) = \hat\cE$. The remaining assertions are now clear. 
This implies the first assertion and (a). 

We now verify (b). From $\cN = \cE_+ \ominus \cE_0$ we derive that 
$\cE_- = \cE_0 \oplus \theta(\cN)$, so that 
$\cE_+ + \cE_- = \cE_+ \oplus \theta(\cN)$ is an orthogonal decomposition 
and our assumption implies that $\cE = \cE_+ \oplus \theta(\cN)$. 
It follows that any proper enlargement $\tilde\cE_+ \supeq \cE_+$ 
contains a non-zero $\theta(\xi)$, $\xi \in\cN$. Then 
$\C \xi + \C \theta(\xi) \subeq \tilde\cE_+$ is $\theta$-positive 
and $\theta$-invariant, hence contained in $\cE^\theta$, which contradicts 
the orthogonality of $\cN$ and $\theta(\cN)$. 
\end{prf}

%If $(\cE,\cE_0,\cE_+,\theta)$ satisfies the Markov 
%condition \eqref{eq:Markov}, then the preceding lemma 
%implies $\cE_0=\cE_+^\theta$, so that $\cE_0$ is redundant 
%and can be omitted in the notation. 

\begin{rem} (Relation to stochastic processes) 
Let $(X_t)_{t \in \R}$ be a full stochastic process on the probability 
space $(Q,\Sigma, \mu)$ (see Definition~\ref{def:g-process}) 
and $\Sigma_t \subeq \Sigma$ the smallest $\sigma$-subalgebra 
for which $X_t$ is measurable. Accordingly, we define 
$\Sigma_\pm$ as the $\sigma$-subalgebra generated by 
all $\Sigma_t$ for $\pm t\geq 0$. In 
$\cE := L^2(Q,\Sigma,\mu)$ we thus obtain closed subspaces 
$\cE_\pm := L^2(Q,\Sigma_\pm,\mu)$ and $\cE_0 := L^2(Q,\Sigma_0,\mu)$.
If $P_\pm$ and $P_0$ are the corresponding projections 
(corresponding to conditional expectations in this context), 
then the Markov condition \eqref{eq:Markov} holds for all translates 
of the process $(X_t)_{t \in \R}$ if and only if 
it is a Markov process (cf.~\cite[\S 7]{JT17}). 
\end{rem}

\section{Reflection positive kernels and distributions} 
\label{subsec:2.1.3}

There are many ways to specify Hilbert spaces concretely. 
Often they arise as $L^2$-spaces of measures, but here we shall 
mostly deal with spaces on which the inner product is specified differently, 
namely by a positive definite kernel. 
For detailed definitions and basic properties of 
positive definite kernels in various contexts, we refer 
to Appendix~\ref{sec:a}. 

\begin{defn} \label{def:2.1.3}
Suppose that $K \colon X \times X \to \C$ is a positive definite kernel 
on the set $X$ and that $\tau \colon X \to X$ is an involution 
leaving $K$ invariant: $K(\tau x, \tau y) = K(x,y)$ for $x,y \in X$. 
If $X_+ \subeq X$ is a subset with the property that the 
kernel 
\begin{equation}
  \label{eq:ktau}
K^\tau \: X_+ \times X_+ \to \C, \quad 
K^\tau(x,y) := K(x, \tau y) 
\end{equation}
is also positive definite, then we 
say that $K$ is {\it reflection positive} \index{reflection positive!kernel}
with respect to $(X,X_+, \tau)$. 
\end{defn}

\begin{lem} \label{lem:2.1.4}  
Let $K \: X \times X \to \C$ be a kernel which is 
reflection positive with respect to $(X,X_+,\tau)$ and 
let $\cE := \cH_K\subeq \C^X$ denote the corresponding reproducing kernel 
Hilbert space. Then the following assertions hold: 
\begin{enumerate}
\item[\rm(a)] $\theta f := f \circ \tau$ defines a unitary 
involution on $\cE$. 
\item[\rm(b)] $\cE_+ := \lbr K_x \: x \in X_+\rbr$ is a 
$\theta$-positive subspace, so that $(\cE, \cE_+, \theta)$ is reflection positive. 
\item[\rm(c)] The  map 
$\cE_+ \to \C^{X_+}, f \mapsto f\circ \tau \res_{X_+}$ 
induces a unitary isomorphism $\hat\cE \to  \cH_{K^\tau}$, so that we may 
identify $\hat\cE$ with the reproducing kernel space 
$\cH_{K^\tau}$ and write $q(f) = f\circ\tau\res_{X_+}$. 
\end{enumerate}
\end{lem}

\begin{prf} (a) The invariance of $K$ under $\tau$ implies the existence 
of a unitary involution $\theta$ on $\cH_K$ with 
$\theta(K_x) = K_{\tau x}$. Then 
$(\theta f)(x) = \la K_x, \theta f \ra = \la K_{\tau x}, f \ra = f(\tau x)$ 
shows that $\theta f = f \circ \tau$.

(b) For $x,y \in X_+$, we have 
$\la K_x, \theta K_y \ra = \la K_x, K_{\tau y} \ra = K(x, \tau y) 
= K^\tau(x,y),$ 
and this implies that the closed subspace $\cE_+ \subeq \cE$ generated 
by $(K_x)_{x \in X_+}$ is $\theta$-positive. 

(c) The Hilbert space $\hat\cE$ is generated by the elements 
$q(K_x)$, $x \in X_+$, and we have 
\[ \la q(K_x), q(K_y) \ra_{\hat\cE} = \la K_x, \theta  K_y \ra = K^\tau(x,y).\]
This implies that $\hat\cE \cong \cH_{K^\tau}$ and that the 
function on $X_+$ corresponding to 
$q(f) \in \hat\cE$ is given by 
$q(f)(x) = \la q(K_x), q(f) \ra_{\hat\cE} = \la K_x, \theta  f \ra = f(\tau x).$
\end{prf}

All reflection positive spaces can be construction from 
reflection positive kernels: 
If $(\cE, \cE_+,\theta)$ is a reflection positive Hilbert space, 
then the scalar product defines a reflection positive kernel 
$K(\eta ,\zeta) := \la \eta , \zeta \ra$ on $X = \cE$, 
and this kernel is reflection positive with respect to 
$(\cE,\cE_+, \theta)$. 

\begin{ex} \label{ex:2.1.9}
On $X = \R$, we consider the involution $\tau(x) = -x$. 

(a) We claim that, for every $\lambda \geq 0$, the kernel 
$K(x,y) = e^{-\lambda|x-y|}$
is reflection positive with respect to $(\R,\R_+,\theta)$. 

The positive definiteness of $K$ means that the function 
$\phi_\lambda(x) := e^{-\lambda|x|}$ 
(multiples of euclidean Green's functions \cite{DG13}) 
is a positive definite function on the group 
$(\R,+)$. In view of Bochner's Theorem (Theorem~\ref{thm:bochner}), 
this is equivalent to 
$\phi_\lambda$ being the Fourier transform of a positive measure. 
In fact, 
\begin{equation}\label{eq:7a}
\phi_\lambda(x) = e^{-\lambda|x|} 
= \int_\R e^{-ixp}\, d\mu_\lambda(p),\quad \mbox{ where } \quad
d\mu_\lambda(p) = \frac{\lambda}{\pi} \frac{dp}{\lambda^2 + p^2}
\end{equation}
is the Cauchy distribution. 
%\begin{footnote}{
%For $\lambda = 1$, this measure corresponds to the rotation invariant
%probability measure on $\bS^1$, with respect to the conformal embedding
%$\R \into \bS^1$, defined by the stereographic projection.
%}\end{footnote}
To verify reflection positivity, we observe that, for 
$x,y \geq 0$,  
\[ K^\tau(x,y) = e^{-\lambda|x+y|} = e^{-\lambda(x+y)} 
= e^{-\lambda x} e^{-\lambda y}.\] 
This factorization implies positive definiteness 
by Remark~\ref{rem:a.1.2}. 
 
(b) Here is a related example corresponding to a periodic function. 
Fix $\beta > 0$, $\lambda \geq 0$, and consider on 
$X = \R$ the $\beta$-periodic function 
given by 
\[ \phi_\lambda(x) := e^{-\lambda x} + e^{-\lambda(\beta - x)}\quad 
\mbox{ for } \quad 0 \leq x \leq  \beta\] 
(multiples of thermal euclidean Green's functions \cite{DG13}). 
We claim that the kernel $K(x,y) :=\phi_\lambda(x-y)$ is reflection 
positive for $X_+ := [0,\beta/2]$. 

A direct calculation shows that the 
Fourier series of $\phi_\lambda$ is given by 
\begin{equation}
  \label{eq:fourexp1}
\phi_\lambda(x) = \sum_{n \in\Z} c_n e^{2\pi i nx/\beta} \quad \mbox{ with }  \quad 
c_n = \frac{2\beta \lambda(1 - e^{-\beta \lambda})}{(\lambda\beta)^2 + (2\pi n)^2}.
\end{equation}
As $c_n \geq 0$ for every $n \in \Z$, the function $\phi_\lambda$ 
is positive definite, i.e., $K$ is positive definite.

Next we observe that, for $0 \leq x,y \leq \beta/2$, we have
\[ K^\tau(x,y) = \phi_\lambda(x+y) 
= e^{-\lambda (x+y)} + e^{-\lambda\beta} e^{\lambda(x+y)}
= e^{-\lambda x} e^{-\lambda y} + e^{-\lambda\beta} e^{\lambda x} e^{\lambda y}.\] 
Here both summands are positive definite kernels 
by Remark~\ref{rem:a.1.2}. 
\end{ex}

\begin{ex} \label{ex:1.4} 
Reflection positive kernels show up naturally in the 
context of distributions if $X$ is a manifold.
We write $C_c^\infty(X)$ for the space of complex-valued compactly supported smooth 
functions on $X$ and $C^{-\infty}(X)$ for the space of distributions, 
the space of continuous \index{$C^{-\infty}(X)$: space of distributions} 
{\bf anti-linear} functionals on $C^\infty_c(X)\to \C$ 
with respect to the natural LF topology on this space (\cite{Tr67}). 

The ``distribution analog'' of a positive definite kernel on 
$X$ is a distribution 
$D \in C^{-\infty}(X\times X)$ which is {\it positive definite} 
\index{distribution kernel!positive definite} 
\index{positive definite!distribution kernel} in the sense the hermitian form 
\[ K_D(\phi,  \psi) := D(\phi \otimes \oline\psi) 
=: \int_{X \times X} \oline{\phi(x)} \psi(y)\, dD(x,y) \]
on $C^\infty_c(X)$ is positive semidefinite (this form is linear in 
the second argument because $D$ is anti-linear). Then the corresponding 
reproducing kernel space $\cH_D := \cH_{K_D}$ consists of functions 
on $C_c(X)$ which are continuous and anti-linear, hence is a linear subspace  of 
the space $C^{-\infty}(X)$ of distributions on $X$. The natural 
map 
\[ \iota_D \: C^\infty_c(X) \to \cH_D \subeq C^{-\infty}(X), \quad 
\iota_D(\psi) = K_{D,\psi}, \quad 
\iota_D(\psi)(\phi) = D(\phi \otimes \oline{\psi}), \] 
then has dense range and 
\begin{equation}
  \label{eq:pd-distr}
\la \iota_D(\phi), \iota_D(\psi) \ra =  D(\phi \otimes \oline\psi).
\end{equation}
\end{ex}

\begin{defn} \label{def:twistdist} 
Let $X$ be a smooth manifold, $D$ a positive definite 
distribution on $X \times X$, let $\tau \: X \to X$ be an involutive diffeomorphism 
of $X$ and $X_+ \subeq X$ be an open subset. 
We say that $D$ is {\it reflection positive} with respect to $(X,X_+, \tau)$ 
\index{reflection positive!distribution kernel} 
\index{distribution kernel!reflection positive} 
if  the distribution 
$D^\tau$ on $X_+ \times X_+$ defined by 
\[ D^\tau(\phi) := D(\phi \circ (\id_X \times \tau)) 
= \int_{X \times X} \oline{\phi(x,\tau(y))}\, dD(x,y) \quad \mbox{ for }\quad 
\phi \in C^\infty_c(X_+ \times X_+)  \] 
is positive definite. 
\end{defn}

Specializing Lemma~\ref{lem:2.1.4} to the context of reflection positive 
distributions, where the set $X$ is replaced by the space 
$C^\infty_c(X)$, we obtain: 
\begin{lem}
  \label{lem:2.1.12} 
If the distribution $D$ on $X \times X$ 
is  reflection positive with respect to $(X,X_+, \tau)$, 
then $\cE := \cH_D$, $\theta(E)(\phi) := E(\phi \circ \tau)$ 
and $\cE_+ := \oline{\iota_D(C^\infty_c(X_+))}$ defines 
a reflection positive Hilbert space of distributions. 
Further, $\hat\cE \cong \cH_{D_+} \subeq C^{-\infty}(X_+)$, where the map 
$q$ is realized by 
\[ q \: \cE_+ \to \hat\cE \cong \cH_{D_+}, \qquad 
q(E)(\phi) := \la \iota_D(\phi), \theta E \ra 
= E(\phi \circ \tau).\] 
\end{lem}

\begin{ex} \label{ex:dm}
For $m > 0$, we consider the distribution 
$D_m := (m^2 - \Delta)^{-1} \delta_0$ on $\R^d$ 
which is the fundamental solution of the elliptic PDE 
$(m^2 - \Delta) D_m = \delta_0.$
As $\delta_0 = \hat 1$ is the Fourier transform 
of the normalized Lebesgue measure $d\lambda_d(x)= \frac{dx}{(2\pi)^{d/2}}$ and 
$\frac{\partial}{\partial x_j} \hat D = (-i x_j D)\,\hat{}\ $
for any tempered distribution $D$ on $\R^d$, it follows that 
$D_m = D_{\hat{\nu_m}}$ for the measure 
\begin{equation}
  \label{eq:num-def}
d\nu_m(p) 
= \frac{1}{m^2 + p^2} d\lambda_d(p) 
= \frac{1}{(2\pi)^{d/2}} \frac{dp}{m^2 + p^2}. 
\end{equation}
For $d = 1$, we obtain a multiple of the function from 
Example~\ref{ex:2.1.9}: 
\[ \hat{\nu_m}(x) 
= \frac{1}{\sqrt{2\pi}}\frac{\pi}{m} e^{-m|x|} 
= \frac{\sqrt{\pi}}{\sqrt 2 m}  e^{-m|x|}. \] 

It is easy to see that this distribution is reflection positive 
with respect to $(\R^d, \R^d_+, \tau)$, where 
\[ \R^d_+ = \{ x = (x_0, \bx) \: x_0 > 0 \} \quad \mbox{ and } \quad 
\tau(x_0, \bx) = (-x_0, \bx) \] 
  is the reflection in the hyperplane $x_0 = 0$.
First we observe that, for every test function $\psi$ on $\R^d_+$, we have  
\[ D_m(\psi \otimes \theta\oline \psi) 
= \frac{1}{(2\pi)^{d/2}} 
%\int_{\R^d}  \oline{\hat \psi}\cdot \theta \hat  \psi \, d\nu_m 
\int_{\R^d} \oline{\hat \psi(p_0, \bp)} \frac{\hat \psi(-p_0, \bp)}{m^2 + p^2}\, dp 
\quad \mbox{ for }\quad  p = (p_0, \bp)\in \R \times \R^{d-1}.\] 
For each $\bp \in \R^{d-1}$, the function 
$h_\bp(p_0) := \hat \psi(-p_0, \bp)$ is a Schwartz function with 
$\supp(\hat h_{\bp}) \subeq (0, \infty)$, and 
\[ \int_{\R^d} \oline{\hat \psi}\cdot \theta \hat  \psi\, d\nu_m 
= \int_{\R^{d-1}} \Big(\int_\R 
\frac{\oline{h_{\bp}(p_0)} h_{\bp}(-p_0)}{p_0^2 + m^2 + \bp^2}\, dp_0\Big) \, d\bp.\] 
The reflection positivity of $D_m$ now follows from 
$\int_\R \frac{\oline{h_{\bp}(p_0)} h_{\bp}(-p_0)}{p_0^2 + m^2 + \bp^2}\, dp_0\geq 0$ 
for $\bp \in \R^{d-1},$ 
which is a consequence of Example~\ref{ex:2.1.9}(a). 

For $m = 0$, the measure $d\nu_0(p) = p^{-2}d\lambda_d(p)$ is locally 
finite if and only if $d \geq 3$. In this case the above arguments 
even show that $D_0 := D_{\hat{\nu_0}}$ is a reflection positive distribution 
on~$(\R^d, \R^d_+,\tau)$. For $d = 2$ we still obtain a reflection 
positive functional (defined in the obvious fashion) 
on the subspace of all test functions $\phi \in C^\infty_c(\R^2)$ with 
$\int_{\R^2}\, \phi(x)\, dx = 0$, and for $d = 1$ we have to impose 
in addition that $\int_{\R}\, x\phi(x)\, dx =~0$. 
\end{ex}

In the following section we shall see a common geometric source 
of the preceding example and Example~\ref{ex:2.1.9}. 

\section{Reflection positivity in Riemannian geometry} 
\label{subsec:2.1.4}

In this section we describe a very natural class of reflection 
positive Hilbert spaces arising from isometric reflections 
of Riemannian manifolds. 

Let $M$ be a connected complete Riemannian manifold. 
An involution $\tau \in \Isom(M)$ is called a {\it reflection} 
\index{reflection!of Riemannian manifold}
\index{reflection!dissecting} 
if there exists a fixed point $p \in M^\tau$ such that 
$T_p(\tau)$ is a hyperplane reflection in the tangent space 
$T_p(M)$. Then $\Sigma := M^\tau$ is a submanifold of $M$ and the connected 
component containing $p$ is of codimension one. 
We say that a reflection is {\it dissecting} if 
$M \setminus \Sigma$ has exactly two connected components which 
are exchanged by $\tau$, i.e., 
\[ M = M_+ \dot\cup \Sigma \dot\cup M_- \quad \mbox{ with } \quad 
\tau(M_+) = M_-.\] 

\index{Laplace--Beltrami operator} 
We consider the {\it Laplace--Beltrami operator} $\Delta_M$ on $L^2(M)$ 
as a negative selfadjoint operator on $L^2(M)$ 
(\cite[Thm.~2.4]{Str83}). For each $m > 0$, we thus obtain a bounded 
positive operator 
$C := (m^2 - \Delta_M)^{-1}$ on $L^2(M)$ . 

\begin{thm} \label{thm:2.1.11} 
If $\tau$ is a dissecting reflection on the connected 
complete Riemannian manifold $M$ and $m > 0$. Then the involution 
$\theta$ on $L^2(M)$ defined by $\theta f := f \circ \tau$ satisfies 
\[ \la \phi, C\theta \phi \ra \geq 0 \quad \mbox{ for } \quad 
\phi \in C^\infty_c(M_+).\] 
\end{thm}

\begin{proof} (cf.~\cite[Thm.~8.3]{An13}) 
The starting point is the divergence formula on a Riemannian manifold 
$M$ with boundary
\begin{equation} \label{eq:div1}
\int_M \div X \, dV = \int_{\partial M} \la X, \vec n \ra \, dS,  
\end{equation}
where $X$ is a compactly supported vector field and 
$\vec n$ is the outward normal vector field of 
$\partial M$. 
\begin{footnote}
  {See \cite[\S 3.8, Satz~26]{AF01} and also \cite[Prop.~4.9]{GHL87}, which has 
different sign conventions.}
\end{footnote}
In index notation, this reads 
\begin{equation}\label{eq:div2}
\int_M \nabla_a X^a \, dV = \int_{\partial M} n^a X_a \, dS.  
\end{equation}

For $\phi \in C^\infty_c(M_+)$ and 
$u = C \phi$ the function $u$ is analytic in $M \setminus \supp(\phi)$ because 
it satisfies the elliptic equation $(m^2 - \Delta)u =0$ on this open subset. 
We now have 
\begin{align*}
\la C\phi, \theta \phi \ra_{L^2} 
&= \int_{M_-} \oline{C\phi} \theta\phi\, dV = \int_{M_-} \oline{u} C^{-1} \theta u\, dV \\
&= \int_{M_-} \oline{u} C^{-1} \theta u- \theta(u) C^{-1} \oline u\, dV\qquad 
(\phi = C^{-1}u\ \  \text{vanishes on}\ M_-) \\
&= \int_{M_-} \oline{u} (m^2 - \Delta) \theta(u) - \theta(u)(m^2 - \Delta) \oline u\, dV = \int_{M_-} \theta(u)\Delta \oline u - \oline{u} \Delta \theta(u)\, dV. 
\end{align*}
For $\Sigma = \partial M_-$, we also obtain 
\begin{align*}
& \int_{\Sigma} \theta(u)\nabla_{\vec n} \oline u - \oline{u} \nabla_{\vec n} 
\theta(u) \, dS 
=  \int_{\Sigma} \la \vec n, \theta(u)\nabla\oline u - \oline{u} \nabla \theta(u)\ra \, dS \\
&=  \int_{M_-} \div\big(\theta(u)\nabla\oline u - \oline{u} \nabla \theta(u)\big)\, dV \\
&=  \int_{M_-} \la \nabla \theta(u), \nabla\oline u \ra 
+ \theta(u) \Delta \oline u - \la \nabla \oline u, \nabla \theta(u) \ra 
- \oline u \Delta \theta(u) \, dV \\ 
&= \int_{M_-}  \theta(u) \Delta \oline u - \oline{u} \cdot \Delta \theta(u)\, dV. 
% = \la C\phi, \theta \phi\ra_{L^2}.
\end{align*}
This finally leads to 
\begin{align*}
\la C\phi, \theta \phi\ra_{L^2} 
&=  \int_{\Sigma} \theta(u)\nabla_{\vec n} \oline u - \oline{u} \nabla_{\vec n} \theta(u) \, dS
=  \int_{\Sigma} \theta(u)\nabla_{\vec n} \oline u + \oline{u} \theta(\nabla_{\vec n} u)\, dS\\
&=  \int_{\Sigma} u\nabla_{\vec n} \oline u + \oline{u} \nabla_{\vec n} u \, dS
=  2 \Re \int_{\Sigma} \oline{u} \nabla_{\vec n} u \, dS.   
\end{align*}
Now 
\begin{align*}
&\int_{\Sigma} \oline{u} \nabla_{\vec n} u \, dS   
= \int_{\Sigma} \la \vec n, \oline{u} \nabla u\ra \, dS  
= \int_{M_-} \div(\oline{u} \nabla u)\, dV   
= \int_{M_-} \la \nabla \oline{u}, \nabla u \ra + \oline u \Delta u\, dV  \\ 
&= \int_{M_-} \la \nabla \oline u, \nabla u \ra + \oline u m^2 u\, dV   
= \|\nabla u\|^2_{L^2(M_-)} + m^2 \|u\|^2_{L^2(M_-)}
\end{align*}
shows that $\la C\phi, \theta \phi \ra_{L^2} \geq 0$. 
\end{proof}

\begin{rem} For $M = \R^d$ and $\tau(x_0,\bx) = (-x_0,\bx)$, the 
reflection positivity of the distribution $D_m$ in 
  Example~\ref{ex:dm} is a very special case of 
Theorem~\ref{thm:2.1.11}. 
\end{rem}

Let $\cE := \cH^C$ denote the completion of $L^2(M)$ with respect to the scalar product 
$\la f, h \ra_C := \la C f, h \ra_{L^2(M)}.$ 
Then $\theta$ induces on $\cE$ a unitary involution $\theta_C$, and 
Theorem~\ref{thm:2.1.11} implies that the subspace 
$\cE_+$ generated by $C^\infty_c(M_+)$ is $\theta_C$-positive. 
We thus obtain a reflection positive 
Hilbert space $(\cE, \cE_+, \theta_C)$. 

Another interpretation of 
Theorem~\ref{thm:2.1.11} is that the distribution 
$D$ on $M \times M$ defined by 
$D(\phi \otimes \psi) := \la \phi, C\oline \psi \ra_{L^2(M)}$
is reflection positive with respect to $(M, M_+, \theta)$. 
From this perspective, we have $\cE = \cH_D$ as in Example~\ref{ex:1.4} 
and $\hat\cE$ can be identified with the Hilbert space 
$\cH_{D^\tau} \subeq C^{-\infty}(M_+)$ of distributions on~$M_+$. 

\section{Selfadjoint extensions and reflection positivity} 
\label{sec:2.6}

In this section we briefly indicate an operator theoretic approach to 
reflection positivity which makes it particularly clear 
how the space $\hat\cE$ depends on the choice of certain 
selfadjoint extensions of symmetric operators, resp., suitable 
boundary conditions. 

We consider a Hilbert space $\cH$ with 
a unitary involution $\theta$ and a closed subspace 
$\cH_+$ such that 
$\cH_- := \theta(\cH_+) = \cH_+^\bot$. 
Then we may identify $\cH$ with $\cH_+ \oplus \cH_+$ on which 
$\theta$ acts by $\theta(v_+, v_-) = (v_-, v_+)$. 

We consider a (densely defined) non-negative symmetric operator $A$ on 
$\cD_+ \subeq \cH_+$ and a selfadjoint extension $L$ of $A$ on $\cH$ 
which commutes with $\theta$ and which is bounded from below.\begin{footnote}
{See \cite{AS80, AG82} for a systematic discussion of the set of 
positive extensions of positive symmetric operators.}  
\end{footnote}
For $-\lambda < \inf \Spec(L)$ we thus obtain a positive operator 
$\lambda\1 + L$ with a bounded inverse 
\[ C := (\lambda\1 + L)^{-1}.\] 
Accordingly, we obtain on $\cH$ a new scalar product 
$\la v,w \ra_C := \la  v, C w \ra$ 
and a corresponding completion $\cE := \cH^C$. We identify 
$\cH$ with a linear subspace of $\cE$ and write 
$\cE_+$ for the closure of $\cH_+$ in $\cE$ and 
$\theta_C$ for the unitary involution on $\cE$ obtained by extending~$\theta$. 

\begin{defn}
  We say that $L$ is {\it reflection positive} \index{reflection positive!operator} 
if 
$(\cE,\cE_+, \theta_C)$ is a reflection positive Hilbert space, i.e., if 
  \begin{equation}
    \label{eq:repo}
\la \xi, \theta C  \xi \ra \geq 0 \quad \mbox{ for }\quad
\xi \in \cE_+.
\end{equation}
\end{defn}

The following proposition shows that non-trivial spaces 
$\hat\cE$ can only be derived from operators $L$ which are not simply 
the closure of $A \oplus \theta A \theta$ on $\cD_+ \oplus \theta(\cD_+)$. 

\begin{prop} If the symmetric operator $A$ is essentially selfadjoint on $\cH_+$, then
$L$ is reflection positive and $\hat\cE = \{0\}$. 
\end{prop}

\begin{prf} As $\lambda\1 + L$ is strictly positive, there exists an $\eps > 0$ with 
\[ \la (L + \lambda\1)\xi,\xi\ra \geq \eps \|\xi\|^2 \quad \mbox{ for } \quad 
\xi \in \cD(L).\] 
This implies in particular that 
$\la (A + \lambda\1)\xi,\xi\ra \geq \eps \|\xi\|^2$ for $\xi \in \cD_+.$
Since $A$ is essentially selfadjoint and non-negative, it follows that 
the selfadjoint operator $\lambda\1 + \oline A$ on $\cE_+$ 
satisfies $\lambda\1 + \oline A \geq \eps$. 
In particular, it is invertible on $\cH_+$. Therefore 
$\cR(\lambda\1 + A) = (\lambda\1 + A)\cD_+$ is dense in $\cH_+$.
We conclude that the continuous operator $C = (\lambda\1 + L)^{-1}$ 
maps the dense subspace $\cR(\lambda\1 + A)$ of $\cH_+$ into $\cH_+$, 
so that $C \cH_+ \subeq \cH_+$. Now 
$\theta C \theta = C$ further implies that $C\cH_- \subeq \cH_-$, so that 
$\la \theta \xi, \xi \ra_C = 0$ for $\xi \in \cH_+$. 
This shows that $L$ is reflection positive with $\hat\cE = \{0\}$. 
\end{prf}

\begin{cor} If $\cH$ is finite dimensional, then $\hat \cE = \{0\}$. 
\end{cor}

If $L$ is reflection positive, then the continuous linear map 
$q_\cH := q\res_{\cH_+} \: \cH_+ \to \hat\cE$ 
has dense range, so that its adjoint 
$q_\cH^* \: \hat\cE \to \cH_+$ 
is injective. We may therefore consider $\hat\cE$ as a linear subspace of $\cH_+$. 
The following observation shows that the image of $\hat\cE$ in 
$\cH_+$ consists of solutions of the eigenvalue equation 
\[ A^* \xi = - \lambda \xi, \quad \xi \in \cD(A^*) \subeq \cH_+.\] 

\begin{prop}
  \label{prop:2} Suppose that $L$ is reflection positive. 
Then $q_\cH$ maps the eigenspace 
$\ker(\lambda\1 + A^*)$ onto a dense subspace of 
$\hat\cE$, the image of $q_\cH^*$ is contained in $\ker(\lambda\1 + A^*)$, 
and 
\begin{equation}
  \label{eq:3} \ker(q_\cH) = (\lambda\1 + L)(\cD(L) \cap \cH_+) 
\cap \cH_+.
\end{equation}
\end{prop}

\begin{prf} For $\xi \in \cH_+$, 
the relation $q_\cH(\xi) = 0$ is equivalent to 
\begin{equation}
  \label{eq:1.1}
0 = \la \eta, C\theta(\xi)\ra = \la (\lambda\1 + L)^{-1} \eta, \theta(\xi) \ra 
\quad \mbox{ for all } \quad \eta \in \cH_+.
\end{equation}
If $\xi = (\lambda\1 + A)\zeta$ for $\zeta \in \cD_+$, then also 
$\xi= (\lambda\1 + L)\zeta$, so that  
\[ C\theta(\xi) = \theta C \xi = \theta \zeta \in \cH_- \] 
implies that 
$\cR(\lambda\1 + A) = (\lambda\1 + A)\cD_+ \subeq \ker(q_\cH).$
We now obtain 
\[ \im(q_\cH^*) \subeq \ker(q_\cH)^\bot \subeq \cR(\lambda\1 + A)^\bot 
= \ker(\lambda\1 + A^*).\] 
This in turn shows that the restriction of $q_\cH$ to 
$\cR(\lambda\1 + A)^\bot = \ker(\lambda\1 + A^*)$ 
has dense range. 
Finally, we note that \eqref{eq:1.1} is equivalent to 
$C \theta\xi= \theta C \xi \in \cH_+^\bot = \cH_-,$ 
which in turn is equivalent to $C\xi \in \cH_+$, i.e., 
to $\xi \in (\lambda\1 + L)\cH_+$. This proves~\eqref{eq:3}. 
\end{prf}

For general results on the existence of reflection positive extensions 
of semibounded symmetric operators, we refer to \cite{Ne18}. 

\begin{ex}
  \label{ex:1.1} 
The preceding discussion is an operator theoretic abstraction of 
the geometric example in Section~\ref{subsec:2.1.4}. 
To  match the abstract framework, we put 
$\cH := L^2(M)$, $\cH_\pm := L^2(M_\pm)$ and 
consider the positive selfadjoint operator $L := -\Delta$ 
as a $\theta$-invariant extension of the restriction 
$A := -\Delta\res_{C^\infty_c(M_+)}.$ 
In this case the eigenvalue equation 
\[ A^*f = - \lambda f \qquad \mbox{ for } \quad f \in \cD(A^*) \subeq \cH_+  \] 
is equivalent to 
$(\lambda \1- \Delta) f  \bot C^\infty_c(M_+),$ 
which means that $f \in \cH_+ = L^2(M_+)$ satisfies the PDE 
\begin{equation}
  \label{eq:pot}
\Delta f = \lambda f \quad \mbox{ on } \quad M_+ 
\end{equation}
in the distribution sense. Ellipticity 
of $\Delta$ implies that $f$ can be represented on $M_+$ 
by an analytic function (\cite[Thm.~8.12]{Ru73}). We thus obtain a 
realization of $\hat \cE$ in the space of $L^2$-solutions 
of \eqref{eq:pot} on the open subset $M_+$.
\end{ex}

\begin{ex} For the simple example $M = \R$ with 
$\tau(x) = -x$ and $M_+ = (0,\infty)$ with 
$Af= - f''$, we consider 
$\lambda = m^2$ for some $m > 0$. Then the solutions of \eqref{eq:pot} 
on $\R_+$ are for $\lambda = m^2$ of the form 
$f(x) = a e^{mx} + b e^{-mx}$, so that the $L^2$-condition leads to 
$f(x) = b e^{-mx}$. This already shows that $\dim \hat\cE \leq 1$ 
and that $q_\cH$ must be a multiple of the linear functional 
$h \mapsto \int_0^\infty h(x) e^{-mx}\, dx 
= \cL(h)(m)$ 
(cf.~Example~\ref{ex:2.1.9}). 
\end{ex}

\section*{Notes on Chapter~\ref{ch:2}} 

%\S~\ref{subsec:2.1.2}: 
The Markov condition \eqref{eq:Markov} in \S~\ref{subsec:2.1.2} 
is an abstraction of the Markov condition for Osterwalder--Schrader positive 
processes that one finds in \cite{Kl77, Kl78, Nel73}. 
For a detailed analysis of its operator theoretic aspects we refer to 
\cite{JT17}, where one also finds a discussion of reflection positive 
Hilbert spaces in terms of graphs of contractions.  

Example~\ref{ex:dm} corresponds to OS-positivity for free fields in 
$d$-space (\cite{GJ81}, \cite{Ja08}). 

Theorem~\ref{thm:2.1.11} 
and several variants 
can be found in \cite{An13} and the work of Jaffe and Ritter \cite{JR07}; 
see also \cite[Thm.~2]{AFG86} and \cite[Thm.~2]{Di04} for related results.

\end{bibunit}

\chapter{Reflection positive representations} 
\label{ch:3} 

\begin{bibunit}

In this chapter we turn to operators on reflection positive 
(real or complex) Hilbert spaces and introduce 
the Osterwalder--Schrader transform to pass from 
operators on $\cE_+$ to operators on $\hat\cE$ (Section~\ref{subsec:2.1.5}). 
The objects represented in reflection positive Hilbert spaces 
$(\cE,\cE_+,\theta)$ are symmetric Lie groups $(G,\tau)$, i.e., 
a Lie group $G$, endowed with an involutive automorphism~$\tau$.  
A typical example in physics arises from the euclidean motion group 
and time reversal. There are several ways to 
specify compatibility of a unitary representation $(U,\cE)$ of 
$(G,\tau)$ with $\cE_+$ and $\theta$ and thus to define 
reflection positive representations (Section~\ref{sec:2.3}). 
One is to specify a subset 
$G_+\subeq G$ and assume that $\cE_+$ is generated by applying $G_+^{-1}$ to a 
suitable subspace of $\cE_+$. 
The other simpler one applies if 
$S := G_+^{-1}$ is a subsemigroup of $G$ invariant under the involution 
$s \mapsto s^\sharp = \tau(s)^{-1}$. Then we simply require $\cE_+$ to be $S$-invariant. 
In both cases we can use the integrability results in Chapter~\ref{ch:7} to obtain 
unitary representations of the $1$-connected Lie group $G^c$ with 
Lie algebra $\g^c = \fh + i \fq$ on $\hat\cE$. 
As reflection positive unitary representations are mostly constructed 
by applying a suitable Gelfand--Naimark--Segal (GNS) construction to 
reflection positive functions, we discuss this correspondence in some detail 
in Section~\ref{sec:3.4}. 
In particular, we discuss the Markov condition in this context 
(Proposition~\ref{prop:3.9}).

\section{The OS transform of linear operators} 
\label{subsec:2.1.5}

We have already seen how to pass from a reflection positive Hilbert space 
$(\cE, \cE_+,\theta)$ to the new Hilbert space $\hat\cE$. 
We now follow this passage for linear operators on $\cE_+$. 

\begin{defn}{\rm(OS transform)} \label{def:e.1-os}
Suppose that $S\: \cE_+ \supeq \cD(S) \to \cE_+$ 
is a linear operator (not necessarily bounded) with 
$S(\cD(S) \cap \cN) \subeq \cN$. 
Then $S$ induces a linear operator 
\[ \hat S \: \cD(\hat S) := \widehat{\cD (S)} 
= \{ \hat v \: v \in \cD(S)\} \to  \hE, \quad 
\hat S \he:= \hat{S\eta}.\] 
The passage from $S$ to $\hat S$ is called the {\it Osterwalder--Schrader 
transform} (or OS transform for short). 
\index{Osterwalder--Schrader (OS) transform} 
\index{OS transform: of an operator} 
\end{defn}

%The following trivial lemma provides a convenient criterion for the domain 
%of an operator $\hat S$ to be dense in $\hat\cE$: 
%\begin{lem} \label{lem:dense-crit} 
%For a subspace $\cD \subeq \cE_+$, the following are equivalent: 
%  \begin{itemize}
%  \item[\rm(a)]\ The image $\hat\cD \subeq \hat\cE$ is dense. 
%  \item[\rm(b)]\ If $\xi \in \cE_+$ satisfies 
%$\la \xi, \eta \ra_\theta = 0$ for every $\eta \in \cD$, then $\xi \in \cN$. 
%  \item[\rm(c)]\ $(\theta\cD)^\bot \cap \cE_+ = \cN$. 
%  \end{itemize}
%\end{lem}
 
\begin{lem}
  \label{lem:d.1}  % \label{lem:2.5} 
Let $(\cE,\cE_+,\theta)$ be a real of complex reflection positive Hilbert space. 
Suppose that $\cD \subeq \cE_+$ is a linear subspace such that 
$\hD = \{\hat v \: v \in \cD\}$ is dense in 
$\hE$, and that 
$S, T \: \cD \to \cE_+$ 
are linear operators. Then 
the following assertions hold:
\begin{enumerate}
\item[\rm(a)] If $\la S\eta, \zeta \ra_\theta = \la \eta, T \zeta\ra_\theta$ 
for $\eta,\zeta \in \cD$, then $S(\cN) \subeq \cN$, so that 
$\hat S, \hat T \:\hat\cD \to \hat\cE$ are well-defined and 
\[ \la \hat S \hat\eta, \hat\zeta \ra = \la \hat\eta, \hat T \hat \zeta \ra 
\quad \mbox{ for } \quad 
\hat\eta, \hat\zeta \in \hat \cD.\] 
\item[\rm{(b)}] Let $\tilde S \in \U(\cE)$ be unitary with 
$\tilde S \cE_+ = \cE_+$ and $\theta \tilde S\theta = \tilde S$. 
For $\cD_+ = \cE_+$ and $S := \tilde S\res_{\cE_+}$, the operator  $\hat{S}$ 
extends to a unitary operator on $\hat\cE$. 
\item[\rm{(c)}] If   $\ip{S\eta }{ \zeta }_\theta  = \ip{\eta}{ S \zeta}_\theta  $ 
for all $\eta,\zeta \in \cD$, then 
$\hat S$ is a symmetric operator. If, in addition, 
$S$ is bounded and $\cD = \cE_+$, 
then so is $\hat S$, and $\|\hat S\| \leq \|S\|$. 
\item[\rm(d)] If $U \in \U(\cE)$ satisfies 
$U\cE_+ = \cE_+$ and $\theta U \theta = U^{-1}$, then $\hat U^2 = \id_{\hat\cE}$. 
Further, $\cE$ is a direct sum of reflection positive Hilbert subspaces 
$(\cF,\cF \cap \cE_+, \theta\res_{\cF})$ and $(\cG,\cG \cap \cE_+, \theta\res_{\cG})$, 
invariant under $U$ and $U^{-1}$, such that 
$\hat\cG = \{0\}$ and $(U\res_{\cF})^2 =~\1$. 
\end{enumerate}
\end{lem}

\begin{prf} (a) For $\eta, \zeta \in\cD$, we obtain from 
$\la S\eta, \zeta \ra_\theta = \la \eta, T \zeta\ra_\theta$ 
that $\eta \in \cN$ implies that $\hat{S\eta} = 0$, i.e., 
$S\eta \in \cN$. 
Therefore $\hat S\hat\eta := \hat{S\eta}$ is 
well-defined and the remainder of (a) follows. 

(b) In this case (a) holds with $T = S^{-1}$, so that 
$\hat S$ and $\hat T$ are well-defined 
and  mutually inverse on $\hat\cD$. In particular, we have $S\hat\cD = \hat\cD$. 
From 
\[\ip{S\eta}{S\zeta}_{\theta}=\ip{S\eta}{\theta  S\zeta }
= \ip{\tilde S \eta }{\tilde S\theta  \zeta}=\ip{\eta}{\zeta}_\theta
\quad \mbox{ for } \quad \xi,\eta \in \cE_+,\] 
it further follows that $\hat S \: \hat\cD \to \hat\cD$ is unitary. 
Therefore it extends uniquely to a unitary operator on $\hat\cE$. 

(c) From (a) it follows that $\hat S$ is well-defined and symmetric. 
Now we assume that $S$ is bounded and defined on all of $\cE_+$. Then 
\[ \|\hat S^k \he\|^2 =  \la \he, \hat S^{2k} \he \ra 
\leq \|\he\| \|\hat S^{2k} \he\|  \quad \mbox{ for } \quad \eta \in \cE_+\] 
and therefore 
\[ \|\hat S \he\|^{2^n} 
\leq \|\hat S^2 \he\|^{2^{n-1}} \|\he\|^{2^{n-1}} 
\leq \|\hat S^4 \he\|^{2^{n-2}} \|\he\|^{2^{n-1} + 2^{n-2}} 
\leq \cdots \leq \|\hat S^{2^n} \he\| \|\he\|^{2^n-1}.\] 
We also have 
$\|\hat S^m \he\|^2 = \ip{\theta S^m \eta}{S^m \eta} \leq \|S\|^{2m} \| \eta \|^2,$ 
which leads to 
\[ \|\hat S \he\|^{2^n} 
\leq \|S\|^{2^n} \|\eta \| \|\he\|^{2^n-1}.\] 
We conclude that 
\[ \|\hat S \he\| \leq \|S\| \cdot 
\lim_{n \to \infty}\big(\|\eta\|^{2^{-n}} \|\he\|^{1 - 2^{-n}} \big)
= \|S\| \|\hat \eta\|.\] 
Therefore $\hat S$ is bounded with 
$\|\hat S\| \leq \|S\|$. 

(d) From (c) it follows that $\hat U$ is a well-defined symmetric contraction. 
The same argument applies to $V := U^{-1}$ and 
leads to another symmetric contraction $\hat V$. 
Now $\hat U \hat V \hat v = \hat{UV}\hat v = \hat v$ for every $v \in \cE_+$ 
implies that $\hat U \hat V = \id_{\hat\cE}$. We likewise 
get $\hat V \hat U = \id_{\hat\cE}$, so that $\hat U^{-1} = \hat V$. 
This shows that $\hat U^{-1}$ also is a symmetric contraction. 
We conclude that  $\Spec(\hat U) \subeq \{-1,1\}$, which further 
leads to $\hat U^2 = \id_{\hat\cE}$.

Next we observe that $\cE_+$ is invariant under $U$ and $U^{-1}$, 
so that $\cE^0 := \oline{\cE_+ + \theta(\cE_+)}^\bot$ is also 
invariant under $U^{\pm 1}$ and $\theta$. 
Since the closed subspace $\cN\subeq \cE_+$ is 
invariant under $U$ and $V= U^{-1}$, the subspace 
$\cE^1 := \cN \oplus \theta(\cN) \subeq (\cE^0)^\bot$ 
is invariant under $U$, $U^{-1}$ and $\theta$, 
and this property is inherited by $\cE^2 := (\cE^0 \oplus \cE^1)^\bot$. 
With $\cE^j_+ := \cE^j \cap \cE_+$, we now obtain a direct 
sum decomposition of the reflection positive 
Hilbert space $(\cE,\cE_+,\theta)$ into the orthogonal sum of the 
three reflection positive spaces 
$(\cE^j, \cE^j_+, \theta\res_{\cE^j})$, $j = 0,1,2$. 
We put $\cG = \cE^0 \oplus \cE^1$ and $\cF:= \cE^2$. 
As $\cE^0_+ = \{0\}$ and $\cE^1_+ = \cN$, we have $\hat{\cG} = \{0\}$ 
and, accordingly, $\hat\cF = \hat\cE$. 
Further $\cN \cap \cF_+ = \{0\}$ implies that 
$q\res_{\cF_+}$ is injective. Hence 
$q \circ U\res_{\cE_+} = \hat U \circ q$ implies that $U_+ := U\res_{\cF_+}$ 
also satisfies $U_+^2 = \1$. Likewise $U\res_{\theta\cF_+} 
= \theta U_+ \theta$ is an involution. 
By construction, 
$\cF_+ + \theta(\cF_+)$ is dense in $\cF$, and this leads to 
$(U\res_{\cF})^2 =~\1$.
\end{prf}

\begin{rem} \label{rem:2.1.16} 
(a) Typical operators 
to which part (b) of the preceding lemma applies are 
unitary operators $S \in \U(\cE)$ with $S\cE_+ = \cE_+$ and 
$\theta S \theta = S$. 
% and hermitian operators $S$ leaving 
%$\cE_+$ invariant and commuting with $\theta$. 

(b) Suppose that $\cE$ is finite-dimensional and that $U \in \U(\cE)$ 
satisfies $U\cE_+ \subeq \cE_+$ and 
$\theta U \theta = U^{-1}$. Then the finite dimension 
implies that $U \cE_+ = \cE_+$, so that Lemma~\ref{lem:d.1}(d)  
shows that $\hat U^2 = \1$. 
\end{rem}

For symmetries of the whole structure encoded in $(\cE, \cE_+, \theta)$,  
the corresponding actions on $\cE$, 
resp., $\cE_+$ lead to unitary operators on $\hat\cE$: 

\begin{prop} \label{prop:osquant-group} 
Let $\cE$ be a real or complex 
Hilbert space, $\theta$ be a unitary involution on $\cE$, and 
$\cE_+ \subeq \cE$ be a $\theta$-positive subspace. 
Suppose that $(U,\cE)$ is a strongly continuous unitary representation 
of a topological group $G$ on $\cE$ such that 
\[ U_g \cE_+ \subeq \cE_+ \quad \mbox{ and } \quad U_g \theta = \theta U_g 
\quad \mbox{ for } \quad g \in G.\] 
Then the OS transform defines a continuous 
unitary representation $(\hat U, \hat{\cE})$ of~$G$. 
\end{prop}

As we shall see below, far more interesting situations 
arise from unitary representations not commuting with 
$\theta$ and not leaving $\cE_+$ invariant. The structure 
required in this context is introduced in the following section.

\section{Symmetric Lie groups and semigroups} 
\label{sec:2.2}

\begin{defn} (Symmetric Lie groups) \label{def:symlie}
Let $G$ be a Lie group with Lie algebra $\fg$ and let $\tau : G\to G$ be an 
involutive automorphism. We then call $(G,\tau)$ a 
{\it symmetric Lie group}.  
\index{Lie group!symmetric} \index{Lie algebra!symmetric}
Likewise, a {\it symmetric Lie algebra} $(\g,\tau)$ is a Lie algebra 
$\g$, endowed with an involutive automorphism of $\g$. 

We shall see below that it is often convenient to encode $\tau$ in 
the larger group 
\begin{equation} \index{$G_\tau =G\rtimes \{\id_G,\tau\}$} 
  \label{eq:gtau}
G_\tau :=G\rtimes \{\id_G,\tau\}.
\end{equation}
Then $\tau \in G_\tau$ and conjugation with $\tau$ on the normal subgroup $G$ 
satisfies $\tau g\tau = \tau(g)$ for $g \in G$. 

We put $H:=(G^\tau)_0$, where ${}_0$ stands for the connected
component containing the identity element $e$. 

The involution $\tau$ induces an
involution $\dd\tau : \fg\to \fg$.  We also write 
\[ \fh := \{x\in\fg\: \dd\tau (x)=x\} 
\quad \mbox{ and } \quad 
\fq:=\{x\in\fg\: \dd\tau (x)=-x\}.\]
 Then $\fg=\fh\oplus \fq$ and $\fh$ is the Lie algebra of $H$. Furthermore,
\[ [\fh,\fh] + [\fq,\fq]\subseteq \fh \quad \mbox{ and }\quad 
[\fh,\fq]\subseteq \fq.\] 
In particular $\fg^c:= \fh\oplus i\fq$ is a Lie subalgebra of 
the complexification $\g_\C = \g + i \g$, 
called the {\it Cartan dual} of $\fg$. \index{Lie algebra!dual symmetric}
We denote by $G^c$ a simply connected Lie group with
Lie algebra $\fg^c$. We observe that 
\begin{equation}
  \label{eq:gsharp}
g^\sharp : = \tau (g)^{-1} 
\quad \mbox{ satisfies } \quad 
(g^\sharp)^\sharp = g \quad \mbox{ and } \quad 
(gh)^\sharp = h^\sharp g^\sharp,
\end{equation}
so that $\sharp$ defines on $G$ the structure of an involutive (semi-) group. 
\end{defn}

\begin{ex} \label{ex:eucmot} 
Let $G=E(n) = \R^n\rtimes \OO_n(\R)$ be the euclidean motion group 
and $\g=\fe(n)$ be its Lie algebra. 
Its elements $(b,A)$ act on $\R^n$ by $(b,A).v=Av+b$. 
The product in $G$ is given by $(x,A)(y,B)=(x+Ay,AB)$.

Let $r_0 := \diag(-1,1,\ldots, 1)$ 
and define an involution on $G$ by
\[\tau (x,A) = (r_0x, r_0A r_0)\, .\]
As 
\[ r_0 \pmat{a & b \\ c & D} r_0
=  \pmat{a & -b \\ -c & D} \quad \mbox{ for } \quad 
\pmat{a & b \\ c & D} \in M_n(\R),\] 
$\fg^c\simeq (i\R\times \R^{n-1}) \rtimes \so_{1,n-1}(\R)
\simeq \R^{1,n-1}\rtimes \so_{1,n-1}(\R)=: \fp(n)$ 
is the Lie algebra of the Poincar\'e group $P(n)$. We then obtain 
the duality relation 
\[ \fe(n)^c \cong \fp(n),\] 
which is of fundamental importance in physics 
(cf.~Chapter~\ref{ch:8}). 
%Note that $S:=(\R_+ \times \R^{n-1})\rtimes \SO_{n-1}(\R)$ 
%is a ${}^\sharp$-invariant subsemigroup of $E(n)$ and 
%$S^c:=(i\R_+\times \R^{n-1})\rtimes \SO_{n-1}(\R)$ is a ${}^\sharp$-invariant 
%subsemigroup of $P(n)$. Both semigroups have empty interior.
\end{ex} 

\begin{ex} (a) The 
affine group $\Aff(V)\cong V \rtimes \GL(V)$ of a vector space $V$ carries a 
natural structure of a symmetric Lie group. 
We write its elements as pairs $(b,A)$, corresponding to the map 
$v \mapsto Av + b$. Then $\tau(b,A) := (-b,A)$ is an involutive 
automorphism of $\Aut(V)$. 

For $G = \Aff(V)$, we then have $G^\tau = \GL(V)$, 
$\g \cong V \rtimes \gl(V)$, $\fh = \gl(V)$ and $\fq \cong V$. 
Since $[\fq,\fq] = \{0\}$, we have $\g^c \cong \g$. 

(b) We obtain a particularly important example for $V = \R$, 
the ``$ax+b$-group''. Here 
\[ S := (\R_{\geq 0},+)  \rtimes (\R^\times_+, \cdot) 
= \{ (b,a) \: (b,a)\R_+ \subeq \R_+ \} 
= \{ (b,a) \:  b \geq 0, a > 0\} \] 
is a closed $\sharp$-invariant subsemigroup of $\Aff(\R)$. 
\end{ex}

\begin{ex} If $r \in \OO_n(\R)$ is any involution of determinant 
$-1$, then $\tau(g) := rgr$ defines an involutive automorphism of 
$\SO_n(\R)$ such that $\OO_n(\R) \cong \SO_n(\R)_\tau$ in the sense 
of \eqref{eq:gtau}. 
\end{ex} 

\begin{defn} A {\it symmetric semigroup} \index{symmetric semigroup}
is a triple 
$(G,S,\tau)$, where $(G,\tau)$ is a symmetric 
Lie group and $S \subeq G$ is a subsemigroup satisfying
\begin{description}
\item[\rm(S1)] $S$ is invariant under $s \mapsto s^\sharp $, so that 
$(S,\sharp)$ is an involutive semigroup.
\item[\rm(S2)] $HS = S$. 
\item[\rm(S3)] $\1 \in \oline S$. 
\end{description}

\index{subsemigroup!symmetric}
If (S1) holds for a subsemigroup $S \subeq G$ we simply call it 
a {\it symmetric subsemigroup of $(G,\tau)$}. 
We shall mostly use 
only (S1). Note that (S1/2) imply that also $SH = (HS)^\sharp = S$. 
\end{defn}

\begin{exs} \label{ex:semigroups}
(a) $(\R,\R_+, -\id_\R)$ and $(\Z,\N_0, -\id_\Z)$ are the most 
elementary examples of symmetric semigroups. 

(b) If $\R^{1,d-1} \cong \R^{d}$ is $d$-dimensional Minkowski space and 
$G= (\R^d,+)$ its translation group, then time reversal 
$\tau(x_0, \bx) = (-x_0,\bx)$ is an involutive automorphism 
and the open light cone $V_+ \subeq \R^{1,d-1}$ is a subsemigroup 
invariant under the map $x \mapsto x^\sharp = - \tau(x)= (x_0, - \bx)$. 

A closely related example is the euclidean space $G = (\R^d,+)$ 
with the same involution and the open half 
space  $S = \R^d_+ = \{ (x_0,\bx) \: x_0 > 0\}$. 

(c) Semigroups with polar decomposition: Let $(G,\tau)$ 
be a symmetric Lie group and $H$ be an open 
subgroup of $G^\tau =\{g\in G\: \tau (g)=g\}$.  
We denote the derived involution $\g \to \g $ by the same letter and define
$\fh = \{x\in \g \: \tau (x)=x\} =\g^\tau$ and $\fq =\{x\in \g\: \tau (x)=-x\}=\g^{-\tau}$. Then
$\g = \fh \oplus\fq$. We say that the open subsemigroup $S\subeq G$ 
has a polar decomposition if there exists an $H$-invariant open convex
cone $C\subset \fq$ 
such that $S=H\exp C$ and the map $H\times C\to S, (h,X)\mapsto h\exp X$ is
a diffeomorphism (cf.~\cite{La94,Ne00,HN93}). Typical 
examples are the complex 
Olshanski semigroups in
complex simple Lie groups such as 
$\SU_{p,q}(\C)_\C \cong \SL_{p+q}(\C)$. 
Complex Olshanski semigroups 
exist if and only if the non-compact Riemannian symmetric space 
associated to $G$ 
is a bounded symmetric domain. This is equivalent to the existence of a $G$-invariant convex 
cone $C\subset i\g$ such that $G\exp C$ 
is a subsemigroup of $G_\C$, if $G_\C\supeq G$ is an injective complexification of~$G$. 
More generally, we have the
causal symmetric spaces of non-compact type like {\it de Sitter space} 
\index{de Sitter space, $\dS^n$} 
$\dS^n \cong \SO_{1,n}(\R)^\uparrow/\SO_{1,n-1}(\R)^\uparrow$  (\cite{HO97}; 
see also Section~\ref{subsec:H-invar}. 
In this case $\fq\simeq \R^{1,n-1}$ is the $n$-dimensional Minkowski 
space and $C$ corresponds to the open light-cone in~$\fq$.

(d) The simply connected covering group $G := \tilde\SL_2(\R)$ of 
$\SL_2(\R)$ carries an involution $\tau$ acting on $\g= \fsl_2(\R)$ by 
\[ \tau\pmat{x & y \\ z & -x} = \pmat{x & -y \\ -z & -x},\] 
and there exists a closed subsemigroup $S \subeq G$ whose boundary is 
\[ \partial S = H(S) := S \cap S^{-1} = \exp(\fb) \quad \mbox{ with } 
\quad \fb := \Big\{ \pmat{ x & y \\ 0 & -x} \: x,y \in \R\Big\}.\]
This semigroup satisfies $S^\sharp = S$, and the subgroup 
$H(S)$ is $\tau$-invariant, but strictly larger than 
$G^\tau_0$.% (see also Section~\ref{sec:2.5} for more on this semigroup). 
\end{exs}

\section{Reflection positive representations} 
\label{sec:2.3}

Suppose that $(G,\tau)$ is a symmetric Lie group. 
For a unitary representation $(U,\cE)$ of $G$ on 
the reflection positive Hilbert space 
$(\cE,\cE_+,\theta)$, the condition 
$\theta U_g \theta = U_{\tau(g)}$ for $g \in G$ is equivalent to 
$U_\tau := \theta$ defining a unitary representation 
$U \: G_\tau \to \U(\cE)$ 
(cf.~\eqref{eq:gtau}).  Accordingly, we shall always work 
with representations of the enlarged group $G_\tau$ in the following 
and assume that $\theta = U_\tau$. 

Next we address the additional requirements that 
make a unitary representation $(U,\cH)$ of $G_\tau$ 
on a reflection positive Hilbert space 
compatible with the subspace $\cE_+$. An obvious natural 
assumption is that the operators $(U_h)_{h \in H}$ act by automorphisms 
of the full structure, i.e., $U_h \cE_+ = \cE_+$ for $h \in H$. 
Since $U_H$ commutes with $\theta$, it preserves both eigenspaces 
$\cE_{\pm 1} = \ker(\theta \mp\1)$. 
If $\cE_+ = \cG(C)$ is the graph of a contraction $C \: \cE_{1} \supeq \cF \to \cE_{-1}$ 
as in Subsection~\ref{subsec:2.1.1}, then the invariance of $\cE_+$ under $U_H$ 
is equivalent to $C$ being an intertwining operator for the representations 
of $H$ on $\cE_{+1}$ and $\cE_{-1}$. 

Eventually, 
one would like to impose conditions that can be used 
to derive a unitary representation of 
the simply connected Lie group $G^c$ with Lie algebra $\g^c$ on 
the space $\hat\cE$. The group $G^c$ always contains a subgroup 
with the Lie algebra $\fh$, so that the representation of this subgroup 
is provided directly by Proposition~\ref{prop:osquant-group},  
but for the operators generated by the subspace 
$i\fq \subeq \g^c$ it is less clear how they should be obtained 
(cf.\ Chapter~\ref{ch:7}). 

One way to express such requirements uses 
to symmetric subsemigroups of $G$, but in many relevant examples 
there are no such subsemigroups with interior points 
and one has to consider more general domains $G_+ \subeq G$.

\begin{defn}
  \label{def:2.2.3} Let $(G,\tau)$ be a symmetric Lie group 
and $(U, \cE)$ be a unitary representation of $G_\tau$ on the 
Hilbert space $\cE$. We put  $\theta := U_\tau$. 

(a) Let $G_+ \subeq G$ be a subset. 
We consider a real linear space $V$ and a linear map 
$j \: V \to \cH$ whose range is cyclic for the unitary representation $U$ of $G_\tau$, 
i.e., $\lbr U_{G_\tau} j(V)\rbr = \cH$. 
Then we say that $(U, \cE, j,V)$ is 
\index{representation!reflection positive!w.r.t.\ subset $G_+$} 
{\it reflection positive with respect to the subset $G_+ \subeq G$} if 
the subspace $\cE_+ := \lbr U_{G_+}^{-1} j(V)\rbr$ is 
$\theta$-positive. 

\index{representation!reflection positive!w.r.t.\ subsemigroup $S$} 
(b) If $S \subeq G$ is a $\sharp$-invariant subsemigroup and 
$(\cE,\cE_+,\theta)$ is a reflection positive Hilbert space, 
then $(U, \cE)$ is said to be {\it reflection positive with respect to $S$} if 
$U_s\cE_+ \subeq \cE_+$ for every $s \in S$. 
Then the conditions under (a) are satisfied for 
$V = \cE_+$, $j = \id_V$, and $G_+ := S^{-1} = \tau(S)$. 
\end{defn}

\begin{lem} \label{lem:bigsemi} 
Let $(\cE,\cE_+,\theta)$ be a reflection positive 
Hilbert, $\cE_- := \theta\cE_+$ and put $U^\sharp := \theta U^* \theta$ 
for $U \in \U(\cE)$. 
Then 
\[ S(\cE_+) := \{ U \in \U(\cE) \: 
U \cE_+ \subeq \cE_+\} \] 
is a subsemigroup of\ $\U(\cE)$, 
%\begin{equation}
%  \label{eq:inflate}
%S(\cE_+)^\sharp = \{ U \in \U(\cE) \: 
%U \cE_- \supeq \cE_- \}, 
%\end{equation}
and $S(\cE_+,\theta) := S(\cE_+) \cap S(\cE_+)^\sharp$ 
is $\sharp$-invariant. 
The OS transform defines a $*$-representation 
$(\Gamma, \hat\cE)$ of the involutive semigroup $(S(\cE_+,\theta),\sharp)$ by 
contractions on $\hat\cE$ 
which is continuous with respect to the strong operator topologies 
on $S(\cE_+,\theta)$ and $B(\hat\cE)$. 
\end{lem}

\begin{prf} 
%For $U \in S$ we first observe that 
%the relation $U\cE_- \supeq \cE_-$ is equivalent to 
%$\theta U \theta \cE_+ \supeq \cE_+$, i.e., to 
%$U^\sharp \cE_+ \subeq \cE_+$. This is \eqref{eq:inflate}. 
Clearly, $S := S(\cE_+,\theta)$ is $\sharp$-invariant 
and hence an involutive semigroup. 
For $\xi, \eta\in \cE_+$ we have 
  \begin{equation}
    \label{eq:sharp-inv}
\la U \xi,\eta \ra_\theta
= \la U \xi, \theta \eta\ra
= \la \xi, U^{-1}\theta \eta\ra
= \la \xi, \theta U^\sharp \eta\ra
= \la \xi, U^\sharp \eta\ra_\theta, 
\end{equation} 
and this implies 
\begin{equation}
  \label{eq:scalprod1}
 \la U\xi, U \xi \ra_\theta = \la \xi, U^\sharp U \xi \ra_\theta.
  \end{equation}

Lemma~\ref{lem:d.1} shows that any $U\in S$ induces a linear operator 
$\hat{U}$ on the dense subspace $q(\cE_+) \subeq \hat\cE$. 
Since $U^\sharp U$ is also contained in $S$, we obtain 
from Lemma~\ref{lem:d.1} that 
$\|\hat{U^\sharp U}\|  \leq \|U^\sharp U\| = 1$. 
With \eqref{eq:scalprod1} we thus get 
$\|\hat{U}\| \leq 1$ 
so that $\hat{U}$ extends to a contraction, also denoted $\hat U = \Gamma(U)$, 
on~$\hat\cE$. 
The relation $\hat{UV} = \hat U \hat V$ for 
$U, V \in S$ follows on the 
dense subspace $q(\cE_+)$ immediately from the definition,  
and $\hat U^* = \hat{U^\sharp}$ is a consequence of \eqref{eq:sharp-inv}. 

The continuity of $\Gamma$ with respect to the 
weak operator topology on $B(\hat\cE)$ 
follows from the fact that, for $\xi \in \cE_+$,  the function 
$\Gamma^{\hat\eta, \hat\xi}(U) := \la \hat\eta, \hat U \hat \xi \ra 
= \la \theta\eta, U\xi \ra$  
is continuous on $S$, endowed with the 
strong operator topology (which equals the weak operator topology). 
Now \cite[Cor.~IV.1.18]{Ne00} 
implies that $\Gamma$ is strongly continuous.
\end{prf}

The preceding lemma implies that, 
for symmetric subsemigroups, reflection positive 
representations lead by the OS transform to $*$-representations of~$S$ 
by contractions on $\hat\cE$: 

\begin{prop}
  \label{prop:1.7} {\rm(OS transform of a representation)} 
\index{OS transform!of a representation}
If  $(U, \cE)$ is a unitary representation of $G_\tau$ on the reflection positive 
Hilbert space $(\cE,\cE_+,\theta)$ which is 
reflection positive with respect to the symmetric subsemigroup 
$S \subeq G$, then the OS transform 
defines a strongly continuous $*$-representation 
$(\hat U, \hat\cE)$ of the involutive semigroup $(S,\sharp)$ by contractions.
\end{prop}

\begin{prf} The invariance of $S$ under $\sharp$ 
and the relation $U_{s^\sharp} = U_{\tau(s^{-1})} = \theta U_s^{-1} \theta = U_s^\sharp$ 
imply that $U_S \subeq S(\cE_+,\theta)$. 
The remaining assertions now follow immediately from 
Lemma~\ref{lem:bigsemi}. 
\end{prf}

\begin{defn} \label{def:dil} 
In the context of Proposition~\ref{prop:1.7}, we 
call $(U,\cE,\cE_+, \theta)$ a {\it euclidean realization} 
of the contractive $*$-representation $(\hat U,\hat\cE)$ 
of~$S$. 
\index{euclidean realization: of contraction representation} 
\end{defn}

In Chapter~\ref{ch:6} we shall encounter methods to 
derive by analytic continuation from a 
$*$-representation $\hat U$ of $S$ (if $S$ has interior points) 
a unitary representation of the simply connected 
$c$-dual group $G^c$ 
(cf.\ Example~\ref{ex:7.anaext}). In this context, we also speak 
of euclidean realizations of unitary representations of $G^c$. 
\index{euclidean realization: of unitary representation of $G^c$} 

\begin{ex} \label{ex:duality} 
For $(G,S,\tau) = (\R,\R_{\geq 0},-\id_\R)$, the situation 
is particularly simple. Then 
$U \: \R_\tau \to \U(\cE)$ is a unitary representation 
and $\hat U \: \R_{\geq 0} \to B(\cE)$ is a continuous one-parameter semigroup 
of hermitian contractions, hence of the form 
$\hat U_t = e^{-tH}$ for some selfadjoint positive operator 
$H = H^* \geq 0$. Then $U^c_t := e^{itH}$ defines a unitary 
representation $U^c$ of the $c$-dual group $G^c \cong \R$ on $\hat\cE$ 
related to $\hat U$ by analytic continuation. 
We shall analyze such examples more closely in Chapter~\ref{ch:4}.
\end{ex}

There is also the following rather weak notion of a reflection 
positive representation: 

\begin{defn} \label{def:repo-rep3}
Let $(G,H)$ be a symmetric Lie group and 
$(\cE,\cE_+,\theta)$ be a reflection positive
Hilbert space. A unitary representation 
$(U,\cE)$ of $G_\tau$ is called {\it infinitesimally reflection positive} if 
\begin{itemize} \index{representation!infinitesimally reflection positive}  
\item[\rm (a)] \ $U_h\cE_+ = \cE_+$ for every $h \in H$, and 
\item[\rm (b)] \ there exists a  %$U_H$-invariant 
subspace $\cD \subseteq \cE^\infty\cap \cE_+$ 
such that $\hat\cD$ is dense in $\hE$ and $\dd U(\fq)\cD\subset \cD$.
\end{itemize}
\end{defn}

\begin{rem} Condition (a) in Definition~\ref{def:repo-rep3} 
implies the existence of a 
unitary representation $\hU$ of $H$ on $\hE$ given by $\hU_h=\widehat{U_h}$ 
(Proposition~\ref{prop:osquant-group}). 
Condition (b) ensures that each operator  
$\dd U(x)$, $x \in \fq$, has an OS transform 
$\widehat{\dd U (x)} : \hD\to \hD$, and one easily verifies 
the relation 
$\widehat{\dd U} (\Ad (h)x) = \hU_h\widehat{\dd U (x)} \hU_{h^{-1}}$ 
for $h \in H$ and $x \in \fq$. 
\end{rem}

\begin{ex} Recall the setting of Theorem~\ref{thm:2.1.11}, 
where $M$ is a Riemannian manifold and we obtain the reflection 
positive Hilbert space $(\cE,\cE_+, \theta_C)$ with 
$\cE = \cH_D$. 
Since the operator $C$ commutes with the unitary 
representation of the Lie group $G := \Isom(M)$ on $L^2(M)$, 
we obtain a unitary representation $(U^C,\cE)$ of this group 
which also contains $\theta$. 
For the identity component $H := G^\theta_0$, we then have 
$U_h \cE_+ = \cE_+$ because elements in $H$ cannot map $M_+$ to $M_-$.
Now the image $\cD$ of $C^\infty_c(M_+)$ in $\cE_+$ is a $U_H$-invariant dense subspace 
invariant under the action of the Lie algebra $\g$ of $G$ which acts by 
Lie derivatives
\[ (\cL_X f)(m) = \frac{d}{dt}\Big|_{t = 0} f(\exp(tX).m),\] 
where we identify $\g$ with a Lie algebra of vector fields on $M$. 
We conclude that all requirements of Definition~\ref{def:repo-rep3} are satisfied.
\end{ex}

\begin{rem} \label{rem:3.3.8} (Infinitesimally unitary representations) 
On the infinitesimal level, the core idea 
of reflection positivity is easily seen. 
Starting with a symmetric Lie algebra $(\g,\tau)$, we obtain 
the corresponding decomposition $\g = \fh \oplus \fq$ and 
form the dual Lie algebra $\g^c := \fh \oplus i \fq \subeq \g_\C$. 

Let $(\cD, \cD_+, \theta)$ be a complex 
{\it reflection positive pre-Hilbert space} 
\index{reflection positive!pre-Hilbert space}
(defined as in Definition~\ref{def:x.1} but omitting the completeness of 
$\cE$ and the closedness of $\cE_+$) and $\pi$ be a representation 
of $\g$ on $\cD$ by skew-symmetric operators. We also assume that 
$\theta\pi(x) \theta = \pi(\tau x)$ for $x \in \g$ and that 
$\cD_+$ is $\g$-invariant. 
Then complex linear extension leads to a representation of $\g^c$ on $\cD_+$ by 
operators which are skew-symmetric with respect to 
the twisted scalar product $\la \cdot, \cdot \ra_\theta$. 
By the OS transform, we then obtain an infinitesimally unitary representation of 
$\g^c$ on the associated pre-Hilbert space $\hat\cD$ via 
\[ \pi^c(x + i y) := \hat{\pi(x)} + i \hat{\pi(y)}.\] 
This is the basic idea behind the reflection positivity correspondence between 
infinitesimally unitary representations of $\g$ on $\cE$ and $\g^c$ 
on $\hat\cE$. 

What this simple picture completely ignores are 
issues of integrability and essential selfadjointness of operators. There are various 
natural ways to address these problems. Important first steps in 
this direction have been 
undertaken by Klein and Landau in \cite{KL81, KL82}, and 
Fr\"ohlich, Osterwalder and Seiler introduced in \cite{FOS83}
 the concept of a virtual representation, 
which was developed in greater generality  by Jorgensen in \cite{Jo86, Jo87}. 
We shall return to these issues in Chapter~\ref{ch:7}. 
\end{rem}

\section{Reflection positive functions} 
\label{sec:3.4} 

\begin{defn}\label{def:1.2c} 
Let $V$ be a real vector space and $(G,\tau)$ be a symmetric Lie group. 
We recall the group $G_\tau = G \rtimes \{\id_G,\tau\}$ from Definition~\ref{def:symlie}. 

(a) A function 
$\phi: G_\tau  \to \Bil(V)$ (the space of bilinear forms on $V$) 
is called \index{reflection positive!function, w.r.t.\ subset $G_+$}
{\it reflection positive with respect to the subset $G_+ \subeq G$} if 
\begin{description}
\item[\rm(RP1)] $\phi$ is positive definite (cf.\ Section~\ref{sec:a}) 
and 
%and\begin{footnote}{
%We refer to Section~\ref{sec:a} for precise definitions concerning 
%positive definite functions of various kinds.}\end{footnote}
\item[\rm(RP2)] the kernel $(s,t) \mapsto 
\phi(s t^\sharp \tau) = \phi(s\tau t^{-1})$ is positive definite on~$G_+$. 
\end{description}

(b) A $\tau$-invariant function $\phi \:  G \to \Bil(V)$ is called 
{\it reflection positive with respect to $G_+$} if the 
extension $\hat\phi$ of $\phi$ to $G_\tau$ by 
$\hat\phi(g,\tau) := \phi(g)$ has this property, i.e., if 
the kernel $(\phi(st^\sharp))_{s,t \in G_+}$ is positive definite. 

\index{reflection positive!function, w.r.t.\ subsemigroup $S$} 
(c) A function 
$\phi: G  \to \Bil(V)$ is called 
{\it reflection positive with respect to the symmetric subsemigroup $S\subeq G$} if 
\begin{description}
\item[\rm(RP1)] $\phi$ is positive definite and $\tau$-invariant and  
\item[\rm(RP2)] the kernel $(\phi(s t^\sharp))_{s,t \in S}$ 
is positive definite on~$S$, i.e., the restriction 
$\phi\res_S$ is a positive definite function on the involutive semigroup  
$(S,\sharp)$.  
\end{description} 
This can also be phrased as the requirement that the kernel 
$K(g,h) := \phi(gh^{-1})$ on $G$ is 
reflection positive with respect to the symmetric subsemigroup $S 
\subeq (G,\tau)$ in the sense of Definition~\ref{def:2.1.3}. 
\end{defn} 

\begin{rem} \label{rem:2.4.2} 
Let $\phi \: G_\tau \to \Bil(V)$ be a positive definite 
function, so that the kernel 
$K((x,v), (y,w)) := \phi(xy^{-1})(v,w)$ on $G_\tau \times V$ is positive definite. 
The involution $\tau$ acts on $G_\tau \times V$ by 
$\tau.(g,v) := (g\tau,v)$ and the corresponding kernel 
$K^\tau((x,v),(y,w)) := K((x,v),(y\tau,w)) 
= \phi(x\tau y^{-1})(v,w)$ is positive definite on 
${G_+ \times V}$ if and only if $\phi$ is reflection positive 
in the sense of Definition~\ref{def:2.1.3}. 

From Lemma~\ref{lem:2.1.4}(c) it follows that the corresponding 
space $\hat\cE$ can be identified with $\cH_{K^\tau} \subeq (V^*)^{G_+}$ 
such that 
\[ q \: \cE_+ \to \cH_{K^\tau}, \quad q(f)(g) := f(\tau(g)), \qquad g \in G_+.\] 
\end{rem} 

The following lemma shows that 
positive definite functions  on $G$ extend canonically to 
$G_\tau$ if they are $\tau$-invariant: 

\begin{lem} \label{lem:biinvar} Let $V$ be a real vector space 
and let $(G,\tau)$ be a symmetric Lie group. Then the following assertions hold: 
\begin{enumerate}
\item[\rm(i)] 
If $\phi\: G \to \Bil(V)$ is a positive definite function 
which is $\tau$-invariant in the sense that $\phi \circ \tau = \phi$, then 
$\hat\phi(g,\tau)  := \phi(g)$ defines an extension to 
$G_\tau$ which is positive definite and $\tau$-biinvariant. 
\item[\rm(ii)] Let $(U,\cH)$ be a unitary 
representation of $G_\tau$, let $\theta := U_\tau$, let $j \: V \to \cH$ be a linear map, 
 and let $\phi(g)(v,w) = \la j(v), U_g  j(w)\ra$ be the corresponding $\Bil(V)$-valued 
positive definite function. Then the following are equivalent: 
\begin{enumerate}
\item[\rm(a)] $\theta j(v) = j(v)$ for every $v \in V$. 
\item[\rm(b)] $\phi$ is $\tau$-biinvariant.  
\item[\rm(c)] $\phi$ is left $\tau$-invariant.  
\end{enumerate}
\end{enumerate}
\end{lem}

\begin{prf} (i) From the GNS construction 
(Proposition~\ref{prop:gns}), we obtain a continuous unitary representation 
$(U,\cH)$ of $G$ and a linear map $j \: V \to \cH$ such that 
\[ \phi(g)(v,w) = \la j(v), U_g j(w) \ra \quad \mbox{ for } \quad g \in G, v,w \in V.\]
As $\phi(g)(v,w) = \phi(\tau(g))(v,w)$, the uniqueness in the GNS construction 
provides a unitary operator $\theta \: \cH \to \cH$ with 
\[ \theta U_g j(v) = U_{\tau(g)} j(v) \quad \mbox{ for } \quad g \in G, v \in V.\] 
Note that $\theta$ fixes each $j(v)$. 
Therefore $U_{\tau} := \theta$ defines an extension of $U$ to a unitary 
representation of $G_\tau$ on $\cH$. Hence 
$\psi(g)(v,w) = \la j(v), U_g j(w) \ra$ defines a positive definite 
$\tau$-biinvariant 
$\Bil(V)$-valued function on $G_\tau$ with $\psi\res_{G} = \phi$. 

(ii) Clearly, (a) $\Rarrow$ (b) $\Rarrow$ (c). 
It remains to show that (c) implies (a). So we assume that 
$\phi(\tau g) = \phi(g)$ for $g \in G_\tau$. This means that, for every 
$v,w \in V$, we have
\[ \la j(v), U_g j(w) \ra = \phi(g)(v,w) 
= \phi(\tau g)(v,w) = \la j(v), \theta U_g j(w) \ra 
= \la \theta j(v), U_g j(w) \ra.\] 
Since $U_{G_\tau}j(V)$ is total in $\cH$, this implies that 
$\theta j(v) = j(v)$ for every $v \in~V$. 
\end{prf}

\begin{rem}
If $S \subeq G$ is a symmetric subsemigroup, 
then a function \break $\phi \: G \to \Bil(V)$ 
is  reflection positive with respect to $S$ if and only 
if its $\tau$-biinvariant extension to $G_\tau$ 
(Lemma~\ref{lem:biinvar}) is reflection positive with respect to $G_+ = S$. 
\end{rem}

\index{Theorem!GNS construction for reflection positive functions}
\begin{thm}{\rm(GNS construction for reflection positive functions)} 
  \label{thm:1.x} 
Let $V$ be a real vector space, let $(U, \cE)$ be a unitary representation
 of $G_\tau$ and put $\theta := U_{\tau}$. Then the following assertions hold: 
\begin{enumerate}
\item[\rm(i)] If $(U, \cH,j,V)$ is  reflection positive with respect to $G_+$, 
then 
\[ \phi(g)(v,w) := \la j(v), U_g j(w)\ra, \qquad g \in G_\tau, v, w \in V, \]
is a reflection positive $\Bil(V)$-valued function. 
\item[\rm(ii)] If $\phi \: G_\tau \to \Bil(V)$ is a reflection positive function 
with respect to~$G_+$, 
then the corresponding GNS representation 
$(U^\phi, \cH_\phi,j,V)$ is a reflection positive 
representation, where $\cE := \cH_\phi \subeq \C^{G_\tau \times V}$ is the Hilbert subspace 
with reproducing kernel 
$K((x,v), (y,w)) := \phi(xy^{-1})(v,w)$ on which $G_\tau$ acts by 
\[ (U^\phi_g f)(x,v) := f(xg,v).\] 
\end{enumerate}
Further, $\cE_+ := \lbr U_{G_+}^{-1} j(V)\rbr$ and 
$\hat\cE \cong \cH_{K^\tau}$ for the kernel 
$K^\tau(s,t) := \phi(s\tau t^{-1})$ on $G_+$, where 
$q \: \cE_+ \to \cH_{K^\tau}, q(f)(g) := f(g\tau).$
\end{thm}

\begin{prf} (i) For $s, t \in G_+$,  we have 
  \begin{align*}
 \phi(s\tau t^{-1})(v,w) 
&= \la j(v), U_{s \tau t^{-1}} j(w)\ra 
= \la U_{s^{-1}} j(v), U_{\tau} U_{t^{-1}} j(w)\ra \\
&= \la \theta U_{s^{-1}} j(v), U_{t^{-1}} j(w)\ra, 
  \end{align*}
so that the kernel $(\phi(s\tau t^{-1}))_{s,t \in G_+}$ is positive definite 
by Proposition~\ref{prop:gns}. 

(ii) Recall the relation 
$\phi(g)(v,w) = \la j(v), U_g j(w) \ra$ for 
$g \in G, v,w \in V$ from Proposition~\ref{prop:gns}. 
Moreover, $(\theta f)(x,v) = f(x\tau,v)$, and 
\[  \la \theta  U^\phi_{s^{-1}} j(v), U^\phi_{t^{-1}} j(w)\ra 
=  \la j(v), U^\phi_{s \tau t^{-1}} j(w) \ra 
= \phi(s\tau t^{-1})(v,w),\] 
so that the positive definiteness of the kernel 
$(\phi(s\tau t^{-1}))_{s,t \in G_+}$ implies that we obtain with $\cE = \cH_\phi$ and 
$\cE_+ := \lbr (U^\phi_{G_+})^{-1}j(V)\rbr$ a reflection 
positive Hilbert space $(\cE,\cE_+,\theta)$.
The remaining assertions follow from Remark~\ref{rem:2.4.2}. 
\end{prf}

\begin{defn}
  \label{def:1.2b}  
Let $S \subeq (G,\tau)$ be a symmetric subsemigroup. 
A triple $(U,\cE,\cV)$, where $(U,\cE)$ is a 
unitary representation of $G_\tau$ and $\cV \subeq \cE$ 
is a $G$-cyclic subspace fixed pointwise by $\theta = U_\tau$ 
is said to be a \index{reflection positive!$V$-cyclic representation}
{\it a reflection positive $\cV$-cyclic representation} 
if the closed subspace $\cE_+ := \lbr U_S\cV\rbr$ is $\theta$-positive. 
If, in addition, $\cV = \C \xi$ is one-dimensional, 
then we call the triple $(\pi,\cE,\xi)$ 
{\it a reflection positive cyclic representation}. 
\index{reflection positive!cyclic representation}
\end{defn}

\begin{cor}
  \label{cor:1.xx} {\rm(Reflection positive GNS construction---operator-valued case)} 
Let $S$ be a symmetric subsemigroup of $(G,\tau)$. 
\begin{enumerate}
\item[\rm(i)] If $(U, \cE,\cV)$ is an $\cV$-cyclic reflection positive 
representation of $G_\tau$ and \break $P \: \cE \to \cV$ the orthogonal projection, 
then $\phi(g) := P U_g P^*$ is a reflection positive function 
on~$G$ with $\phi(e) = \1_\cV$. 
\item[\rm(ii)] Let $\phi \: G \to B(\cV)$ is a reflection positive function 
with respect to $S$ on $G$ 
with $\phi(e) = \1_\cV$ and let $\cH_\phi \subeq \cV^G$ be  the Hilbert subspace 
with reproducing kernel $K(x,y) := \phi(xy^{-1})$ on which $G$ acts by $(U^\phi(g)f)(x) := f(xg)$ 
and $\tau$ by $\theta f := f \circ \tau$. 
We identify $\cV$ with the subspace $\ev_e^*\cV \subeq \cH_\phi$. 
Then $(U^\phi, \cH_\phi,\cV)$ is a $\cV$-cyclic reflection positive 
representation and we have an $S$-equivariant unitary map 
\[ \Gamma \: \hat\cE \to \cH_{\phi\res_S}, \quad 
\Gamma(\hat f) = f\res_S \quad \mbox{ for } \quad 
f \in \cE_+ = \lbr U_S^\phi \cV\rbr.\] 
\end{enumerate} 
\end{cor}

\begin{prf} (i) To match this with Theorem~\ref{thm:1.x}(i), we put 
$V := \cV$ and consider the inclusion map $j \:  V \to \cH$. 
Then $\phi(g) \in \Bil(V)$ corresponds to the operator 
$j^* U_g j = P U_g P^* \in B(V)$. Therefore 
$\phi$ is positive definite with $\phi(e) = \1$. 
That $\phi$ is $\tau$-invariant follows from $\theta\res_\cV = \id_\cV$ 
(cf.~Lemma~\ref{lem:biinvar}). 

(ii) We use the second half of Example~\ref{ex:vv-gns}, i.e., 
the special case of Proposition~\ref{prop:gns} dealing with 
operator-valued positive definite functions, and 
identify $\cV$ with $\ev_e^*\cV \subeq \cH_\phi$. 
Lemma~\ref{lem:biinvar} implies that $\theta$ fixes $\cV$ pointwise. 

To see that $\cE_+ := \lbr U^\phi_S\cV\rbr$ 
is $\theta$-positive, we note that 
\[ \theta(\cE_+)
= \lbr U^\phi_{\tau(S)} \theta\cV\rbr 
= \lbr U^\phi_{S^{-1}} \cV\rbr,\] 
and this subspace is $\theta$-positive by  
Theorem~\ref{thm:1.x}(ii). Therefore $\cE_+$ is also $\theta$-positive 
(Remark~\ref{rem:decomp}). 
From Theorem~\ref{thm:1.x}(ii) we further derive that 
\[ \theta(\cE_+) \to \cH_{K^\tau} \subeq (V^*)^S, \quad 
f \mapsto (f \circ \tau)\res_S = \theta(f)\res_S \] 
induces a unitary isomorphism $\hat\cE \to \cH_{K^\tau}, 
\hat f \mapsto f\res_S$,
and this implies that $\Gamma$ is unitary.
\end{prf}

\begin{cor}
  \label{cor:1.x}
Let $S \subeq (G,\tau)$ be a symmetric subsemigroup. 
\begin{enumerate}
\item[\rm(i)] If $(U, \cE,\xi)$ is a cyclic reflection positive 
representation of $G_\tau$,
then $U^\xi(g) := \la \xi, U_g \xi\ra$ is a reflection positive function
on~$G$. 
\item[\rm(ii)] If $\phi$ is a reflection positive function on $G$,
then $(U^\phi, \cH_\phi,\phi)$ is a cyclic reflection positive
representation. 
\end{enumerate}
\end{cor}

The following result characterizes reflection positive representations 
for which $(\cE,\cE_+,\theta)$ is of Markov type. 

\begin{prop} \label{prop:3.9} 
Let $(U, \cE)$ be a reflection positive unitary representation 
on $(\cE, \cE_+, \theta)$ with respect to the unital 
symmetric subsemigroup $S \subeq (G,\tau)$. 
Let \break $P_0 \: \cE \to \cE_0$ be the orthogonal projection and consider the 
reflection positive definite function $\varphi (g)=P_0U_gP_0$. 
Then the following assertions hold: 
\begin{enumerate}
\item[\rm(a)] If $\phi\res_S$ is multiplicative and 
$\cE_+ = \lbr U_S \cE_0 \rbr$, then $(\cE,\cE_+,\theta)$ is of Markov type.
\item[\rm(b)] If $(\cE,\cE_+,\theta)$ is of Markov type and 
$\Gamma := q\res_{\cE_0} \: 
\cE_0 \to \hat\cE$ is the corresponding unitary isomorphism, 
then $\vphi\res_S$ is multiplicative and 
$\vphi(s) = \Gamma^* \hat U_{s} \Gamma$ for $s\in S$, i.e., 
$\Gamma$ intertwines $\vphi\res_{S}$ with the contraction 
representation $(\hat U, \hat\cE)$ of~$S$. 
\end{enumerate}
\end{prop}
 
\begin{prf} That $\phi$ is reflection positive follows from 
Corollary~\ref{cor:1.xx}(i). 

(a) By Corollary~\ref{cor:1.xx}(ii),  the restriction map 
$\Gamma \: \cE_+ \to \cH_{\phi\res_S}, \hat f \mapsto f\res_S$ 
is a unitary $S$-intertwining operator. 
From $\cE_+ = \lbr U_S \cE_0 \rbr$ it follows that 
$\hat\cE = \lbr \hat U_S q(\cE_0)\rbr$, so that 
the multiplicativity of $\phi\res_S$ implies that 
$\Gamma(\cE_0) = \hat\cE$ (Lemma~\ref{lem:mult}), 
i.e., $(\cE,\cE_+,\theta)$ is of Markov type. 

(b) Let $\cK \subeq \cH$ be the $U$-invariant closed subspace generated 
by $\cE_0$ and let $(\cE_0)^G$ denote the linear space of all maps $G \to \cE_0$. 
Then the map 
\[ \Phi \: \cK \to (\cE_0)^G, \quad 
\Phi(\xi)(g) := P_0 U_g \xi \] 
is an equivalence of the representation $U$ of $G$ on $\cK$ with the 
GNS representation defined by $\vphi$ (Proposition~\ref{prop:gns}). 
Further, the representation $\hat U$ of $S$ on $\hat\cE$ is equivalent to the 
GNS representation  defined by $\vphi\res_S$, where the map 
$q \: \cE_+ \to \hat\cE$ corresponds to the restriction 
$f \mapsto f\res_S$  (Corollary~\ref{cor:1.xx}(ii)). 
The inclusion $\iota \: \cE_0 \into \cH_\vphi$ is given by 
$\iota(\xi)(g) = P_0 U_g\xi = \vphi(g)\xi$ for $g \in G$, and likewise the inclusion 
$\hat\iota \: \cE_0 \into \cH_{\vphi\res_S}$ is given by 
$\hat\iota(\xi) = \vphi \cdot \xi$. 
Lemma \ref{le:Markov} implies the surjectivity of $\hat\iota$. 
In view of Lemma~\ref{lem:mult}, 
this is equivalent to the multiplicativity of $\vphi\res_S$. 

Recall $q = \Gamma \circ P_0\res_{\cE_+}$ from Lemma~\ref{le:Markov}. 
For $s\in S$, the relation 
$\hat U_s \circ q = q \circ U_s\res_{\cE_+}$
leads to 
$\hat U_s \Gamma P_0\res_{\cE_+} = \Gamma P_0 U_s\res_{\cE_+},$
so that $\Gamma^* \hat U_s \Gamma  = P_0 U_s P_0 = \vphi(s),$
i.e., $\Gamma$ intertwines $\vphi(s)$ with $\hat U_s$. 
\end{prf}

\begin{ex} Let $(G, \tau)$ be a symmetric Lie group and 
let $\sigma \: G_\tau  \to \Diff(M)$ be 
a smooth right action of $G_\tau$ on the manifold~$M$.
Then $\tau_M := \sigma_\tau$ is an involutive diffeomorphism of~$M$.
Further, let $K:M\times M\to B(V)$ be a $G$-invariant reflection positive 
kernel with respect to $(M, M_+, \tau_M)$ 
(Definition~\ref{def:2.1.3}), where $M_+ \subeq M$ is a $H$-invariant subset. 

Then $U_g f := f \circ \sigma_g$ defines a 
unitary representation of $G$ on $\cH_K$. It clearly satisfies 
\begin{equation}
  \label{eq:ker-trafo-M}
K_x \circ U_g = K_{x.g} \quad \mbox{ and thus } \quad 
U_g K_x^* = K_{x.g^{-1}}^*.
\end{equation}
Here the unitarity of $U$ follows from the
$G$-invariance of $K$, and the $H$-invariance of $\cE_+$ follows from the 
$H$-invariance of $M_+$ and \eqref{eq:ker-trafo-M}. 

A special case of this construction arises for 
$G=M$ and right-invariant kernels of the form 
$K(x,y)=\varphi (xy^{-1})$, 
where $\varphi : G \to B(V)$ is a reflection positive function. 
Here the $\tau$-invariance of $K$ is equivalent to the 
relation $\phi \circ \tau = \phi$. 
\end{ex}

\section*{Notes on Chapter~\ref{ch:3}} 

\S~\ref{subsec:2.1.5}: 
%With the terminology ``Osterwalder--Schrader quantization'' 
%for the passage from operators on $\cE$, resp., $\cE_+$ to operators 
%on $\hat\cE$ (Definition~\ref{def:e.1-os}), we follow 
%A.~Jaffe \cite[\S VII.7]{Ja08}. 
Variants of Lemma~\ref{lem:d.1} also appear 
in \cite{JOl00}.

\S~\ref{sec:2.3}: In the context of intervals in the real line 
which we discuss in Chapter~\ref{ch:4}, 
the notion of reflection positive functions 
already appears in \cite{KL81}, where such functions are called 
(OS)-positive. 

Proposition~\ref{prop:1.7} is already in \cite{JOl00}. 

A version of Proposition~\ref{prop:3.9} for the case 
$(\R,\R_+,-\id_\R)$ can already be found in \cite{Kl77} 
(see also \cite[\S 7]{JT17}). 

Reflection positivity for the lattice $G = \Z^d$ and 
$\tau(x_0,\bx) = (-x_0, \bx)$ is discussed in the context 
of correlation functions by Usui in \cite{Us12}.

\end{bibunit}

\chapter{Reflection positivity on the real line} 
\label{ch:4} 

\begin{bibunit}

After providing the conceptual framework for reflection positive representations 
in the preceding two chapters, 
we now turn to the fine points of reflection positivity on the additive 
group~$(\R,+)$. Although this Lie group is quite trivial, reflection positivity 
on the real line has many interesting facets and is therefore quite rich. 
We thus describe its main features in this and the subsequent chapter. 
As reflection positive functions play a crucial role,  
we start in Section~\ref{subsec:3.0} with reflection 
positive functions on intervals $(-a,a)\subeq \R$. 
Here we already encounter the main 
feature of reflection positivity dealing with two different notions 
of positivity, one related to the group structure on $\R$ and the other 
related to the $*$-semigroup structure on $\R_+$, resp., the convex 
structure of intervals. All this is linked to representation theory 
in Section~\ref{subsec:3.1}, where we start our investigation 
of reflection positive representations of the symmetric 
semigroup $(\R,\R_+,-\id_\R)$. 
These are unitary one-parameter groups $(U_t)_{t \in \R}$ on a reflection positive 
Hilbert space $(\cE,\cE_+,\theta)$ satisfying 
$U_t \cE_+ \subeq \cE_+$ for $t>0$ and 
$\theta U_t \theta = U_{-t}$ for $t\in \R$. 
On $\hat\cE$ this leads to a semigroup $(\hat U_t)_{t \geq 0}$ of hermitian 
contractions. The main result in  Section~\ref{subsec:3.1} 
is that the OS transform ``commutes with reduction'', where reduction  
refers to the passage to the fixed points of $U$ and $\hat U$ in $\cE$ 
and $\hat\cE$, respectively 
(Proposition~\ref{prop:e.5b}). 
Reflection positive functions for  $(\R,\R_+,-\id_\R)$ are 
classified in terms of integral representations in Section~\ref{sec:3.2}. 
We shall see in particular that any hermitian contraction semigroup 
$(C_t)_{t \geq 0}$ on a Hilbert space $\cH$ 
has a so-called minimal dilation represented by the 
reflection positive function $\psi(t) := C_{|t|}$. 
We also provide a concrete model for this dilation 
on the space $\cE = L^2(\R,\cH)$ with $(U_t f)(p) = e^{itp}f(p)$, 
where $\cE_+ = L^2_+(\R,\cH)$ 
is the positive spectral subspace for the translation group, which is, 
by the Laplace transform, isomorphic to the $\cH$-valued 
Hardy space $H^2(\C_+,\cH)$ on the right half plane $\C_+ = \R_+ + i \R$. 
%As general reflection positive functions are superpositions 
%of the functions $e^{-\lambda|x|}$, $\lambda \geq 0$, 
%it is of particular interest to understand these functions and 
%related reflection positive representations from various perspectives. 
%This motivates our detailed discussion of the Hardy space of the complex half plane 
%as a reflection positive representation of $\R$ in Section~\ref{ex:3.13}. 
We conclude this chapter by showing that, for any 
reflection positive one-parameter group for which $\cE_+$ is cyclic 
and fixed points are trivial, the space $\cE_+$ is outgoing in the sense of 
Lax--Phillips scattering theory (Proposition~\ref{prop:4.11}). 
This establishes a remarkable connection between reflection 
positivity and scattering theory that leads to a normal form of 
reflection positive one-parameter groups by translations on 
spaces of the form $\cE = L^2(\R,\cH)$ with $\cE_+ = L^2(\R_+,\cH)$. 
Applying the Fourier transform to our concrete dilation model leads precisely 
to this normal form. 
%Section~\ref{se:LaxPhil}. 

\section{Reflection positive functions on intervals} 
\label{subsec:3.0}

Before we turn to representation theoretic issues, 
we briefly discuss reflection positive functions on open intervals 
in~$\R$. 
There are two natural types of positive definiteness conditions 
for functions on real intervals. 
The first one comes from the additive group $(\R,+)$, for which a 
function $\phi \: \R \to \C$ is positive definite if and only if the kernel 
$(\phi(x-y))_{x,y \in \R}$ is positive definite. This condition makes also sense 
on symmetric intervals of the form $(-a/2,a/2)$ if $f$ is defined on 
$(-a,a)$. Bochner's Theorem (Theorem~\ref{thm:bochner}) 
 asserts that a continuous function 
on the additive group $\R$ is positive definite if and only if 
it is the Fourier transform 
$\phi(x) = \int_\R e^{-ix\lambda}\, d\mu(\lambda)$ of a bounded positive 
Borel measure $\mu$ on~$\R$. 

The second type makes sense for functions 
$\phi \: (a,b) \to\C$ on any real interval and requires that the kernel 
$\big(\phi\big(\frac{x+y}{2}\big)\big)_{a < x,y < b}$ is positive definite.  
Widder's Theorem below asserts that 
this is equivalent to $f$ being a Laplace transform 
of a positive Borel measure $\mu$ on~$\R$.
For $(a,b) = (0,\infty),$ this is precisely the condition of 
positive definiteness on the $*$-semigroup $(0,\infty)$ with the trivial 
involution $t^* =t$ for $t > 0$. 

\index{Theorem!Widder}
\begin{thm} {\rm(Widder; \cite{Wi34}, \cite[Thm.~VI.21]{Wi46})}  \label{thm:widder}
Let $-\infty \leq a < b \leq \infty$. 
A~function $\vphi\: (a,b) \to \R$ is positive definite 
in the sense that the kernel $\phi\big(\frac{x+y}{2}\big)$ 
is positive definite if and only 
if there exists a positive Borel measure $\mu$ on $\R$ such that 
\[ \vphi(t) = \cL(\mu)(t) := \int_\R e^{-\lambda t}\, d\mu(\lambda) 
\quad \mbox{ for }  \quad t \in (a,b).\] 
This implies in particular that $\vphi$ is analytic. 
\end{thm}

The following theorem provides a characterization of functions $\psi \:  (0,\infty) 
\to \R$ which are {\it completely monotone}, i.e.,  
\index{function!completely monotone}
$(-1)^k \psi^{(k)} \geq 0$ for $k = 1,2,3, \ldots$ 
(see \cite[Thm.~3.6]{JNO18}, \cite[Thms.~1.4]{SSV10}, \cite[Thm.~IV.12b]{Wi46}). 
Its most remarkable point is that it characterizes the global property of positive 
definiteness on the additive semigroup $(0,\infty)$ in terms of the infinitesimal 
condition of being completely monotone. 

\index{Theorem!Hausdorff--Bernstein--Widder}
\begin{thm} {\rm(Hausdorff--Bernstein--Widder)} \label{thm:bern1}
For a function $\vphi \: (0,\infty) \to [0,\infty)$, the following are equivalent: 
\begin{enumerate}
\item[\rm(i)] $\vphi$  is completely monotone.  
\item[\rm(ii)] $\vphi$ is a Laplace transform of a positive Borel measure on 
$[0,\infty)$. 
\item[\rm(iii)] $\vphi$ is decreasing and 
positive definite on the $*$-semigroup $((0,\infty),\id)$. 
\end{enumerate}
\end{thm}

After these preparations, we now turn to reflection positive functions 
on intervals. 

\begin{defn} \label{def:4.1.3} 
Let $a \in (0,\infty]$ and consider the interval $(-a,a)$, 
endowed with the reflection $\tau(t) = -t$ about the midpoint. 
We call a function 
$\vphi \: (-a,a) \to \R$ {\it reflection positive} 
\index{reflection positive!function} 
\index{function!reflection positive} 
if both kernels 
\begin{equation}
  \label{eq:kernels}
\vphi\Big(\frac{t-s}{2}\Big)_{-a < s,t < a} \quad \mbox{ and } \quad 
\vphi\Big(\frac{t+s}{2}\Big)_{0< s,t < a} 
\end{equation}
are positive definite. 
For the kernel $K(s,t) := \vphi\big(\frac{t-s}{2}\big)$, this corresponds to the 
situation of Definition~\ref{def:2.1.3} with 
$X = (-a,a)$, $X_+ = (0,a)$ and $\tau(x) = -x$. 
\end{defn}

Since the kernels in \eqref{eq:kernels} are both hermitian, 
reflection positive functions on $(-a,a)$ satisfy 
$\vphi(-t) = \oline{\vphi(t)} = \vphi(t)$. Therefore 
Widder's Theorem~\ref{thm:widder} provides a 
positive Borel measure $\mu$ on $\R$ with 
\begin{equation}
  \label{eq:lapl-abs}
\vphi(t) = \cL(\mu)(|t|) = \int_\R e^{-\lambda |t|}\, d\mu(\lambda) \quad \mbox{ 
for } \quad |t| < a.
\end{equation}
Conversely, for such functions the kernel 
$\vphi\big(\frac{t+s}{2}\big)$ is positive definite on $(0,a)$. Therefore 
$\vphi$ is reflection positive if and only if the kernel 
$\vphi\big(\frac{t-s}{2}\big)$ is positive definite on $(-a,a)$. 
So the main point is to relate this condition to properties 
of the measure~$\mu$.

\begin{ex} \label{ex:4.1.4} (a) 
For $\lambda \geq 0$, the functions 
$\vphi_\lambda(t) := e^{-\lambda |t|}$ are positive definite 
(Example~\ref{ex:2.1.9}(a)). 
Therefore $\cL(\mu)(|t|)$ is reflection positive if $\mu$ is 
supported by $[0,\infty)$.   

(b)  Basic examples of  positive definite 
$\beta$-periodic functions on $\R$ are given by  
%the functions $f_\lambda$ satisfying 
\[ f_\lambda(t) = e^{-t\lambda} + e^{-(\beta - t)\lambda} 
= 2 e^{-\beta \lambda/2} \cosh((\textstyle{\frac{\beta}{2}}-t)\lambda)
\quad \mbox{ for } \quad 0 \leq t \leq \beta, \lambda \geq 0\] 
(Example~\ref{ex:2.1.9}(b)). 
For $|t| < \beta$, we then have
\begin{equation}
  \label{eq:f-lambda}
f_\lambda(t) = f_\lambda(|t|) 
= e^{-|t|\lambda} + e^{-(\beta - |t|)\lambda} 
= e^{-|t|\lambda} + e^{-\beta\lambda} e^{|t|\lambda}.
\end{equation}
Hence, for reflection positivity on a finite interval 
$(-\beta,\beta)$, it is not necessary that the measure 
$\mu$ in \eqref{eq:lapl-abs} 
is supported by the positive half line, as in (a). 

From the positive definiteness of the functions $f_\lambda$ for every 
$\beta > 0$, we conclude that, for a fixed $a > 0$ and a positive Borel 
measure $\mu$ on 
$[0,\infty) \times [a,\infty)$, the function 
\begin{equation}
  \label{eq:doubleint}
f(t) := \int_{[0, \infty) \times [a, \infty)} 
e^{-\lambda|t|} + e^{-\beta\lambda} e^{\lambda|t|}\, d\mu(\lambda,\beta) 
\end{equation}
is reflection positive on $(-a,a)$ whenever the integrals are finite. 
\end{ex} 

The proof of the following theorem 
(\cite[Thm.~5.8]{JNO18}) uses 
P\'olya's classical result relating positive definiteness of real-valued 
functions on $\R_+$ with convexity (\cite[Thm.~4.3.1]{Luk70}) 
and provides a sufficient 
conditions for positive definiteness. 

\index{Theorem!Characterization of reflection positive functions on interval}
\begin{thm} \label{thm:5.9} 
{\rm(Characterization of reflection positive functions on $[-a,a]$)} 
Fix $a > 0$  and let $\mu$ be a finite positive Borel measure on $\R$ for which 
$\vphi(t) := \cL(\mu)(|t|)$ exists for $|t| \leq a$. 
\begin{enumerate}
\item[\rm(i)] If the left-sided derivative in $a$ satisfies 
$\cL(\mu)'(a-) \leq 0$, then $\vphi$ is reflection positive on $[-a,a]$ and extends to a symmetric positive definite function on $\R$. 
\item[\rm(ii)] If $\vphi$ is reflection positive on $[-a,a]$ and non-constant, 
then there exists a number  $b \in (0,a]$ with $\cL(\mu)'(b-) < 0$. 
\end{enumerate}
\end{thm}

\begin{rem} If $b$ is as in (ii), then it follows from part (a) that the positive definite function $\vphi|_{[-b,b]}$ extends to a positive definite function
on $\R$. But this extension does not have to coincide 
 with $\vphi$ if $b<a$.
\end{rem}

\begin{rem} (The $\beta$-periodic case) \label{rem:4.9}
We consider the group $G = (\R,+)$ and the open interval $G_+ := (0,\beta/2)$. 
Then a symmetric continuous function $f \: \R \to \C$ is reflection positive 
with respect to $G_+$ if it is positive definite and the kernel 
$\big(f(t+s)\big)_{0 < s,t < \beta/2}$ is positive definite 
(Definition~\ref{def:1.2c}(b)), which implies that 
$f$ is reflection positive on the interval $(-\beta,\beta)$ 
(Definition~\ref{def:4.1.3}). 
As $f$ is symmetric and $\beta$-periodic, it is also symmetric with respect to 
$\beta/2$, i.e., $f(\beta-t) = f(t)$. 
The latter relation implies that the corresponding measure $\mu$ on 
$\R$ satisfies 
$d\mu(-\lambda) = e^{-\beta\lambda} d\mu(\lambda)$, hence 
has the form 
\begin{equation}
  \label{eq:mu+}
d\mu(\lambda) = d\mu_+(\lambda) + e^{\beta\lambda} d\mu_+(-\lambda)
\end{equation}
for a measure $\mu_+$ on $\R_{\geq 0}$. We thus obtain the 
integral representation 
\[ f(t) = \int_0^\infty e^{-t\lambda} + e^{-(\beta - t)\lambda}\, d\mu_+(\lambda)
\quad \mbox{ for }\quad 0 \leq t \leq \beta, \] 
which determines $f$ by $\beta$-periodicity. That, conversely, 
all such functions are reflection positive follows from 
Example~\ref{ex:2.1.9}(b). 
Note that $f\res_{[0,\beta]}$ is convex and symmetric with respect to 
$\beta/2$, where it has a global minimum. In particular 
Theorem~\ref{thm:5.9} only applies to the restriction 
of $f$ to the interval $[-\beta/2,\beta/2]$ which also determines 
$f$ by $\beta$-periodicity. 
\end{rem}

\section{Reflection positive one-parameter groups} 
%The associated contraction semigroup} 
\label{subsec:3.1}

We now turn to reflection positivity on the whole real line~$X = \R$ 
with respect to the right half line $X_+ = \R_{\geq 0} = [0,\infty)$ which 
is an additive $*$-semigroup with $s^*=s$. Therefore reflection positive functions 
provide  close relations between unitary representations 
of $\R$ and one-parameter semigroups of hermitian contractions. 
Specializing Definition~\ref{def:2.2.3} to 
the symmetric semigroup $(\R,\R_+,-\id_\R)$, we obtain: 
\begin{defn} A {\it reflection positive} 
\index{reflection positive!unitary one-parameter group}
unitary one-parameter group on 
the reflection positive Hilbert space 
$(\cE,\cE_+,\theta)$ is a strongly continuous unitary one-parameter group 
$(U_t)_{t \in \R}$ on $\cE$ for which $\cE_+$ is invariant under $U_t$ for $t > 0$ and 
$\theta U_t \theta  = U_{-t}$ for $t \in \R$. 
\end{defn}

From Proposition~\ref{prop:1.7} we immediately obtain: 

\begin{prop} \label{prop:os-onepar} If $(U_t)_{t \in \R}$ is a reflection positive 
unitary one-parameter group on $(\cE,\cE_+, \theta)$, 
then $(\hat U_t)_{t \geq 0}$ is a 
strongly continuous one-parameter semigroup 
of symmetric contractions on~$\hat\cE$. 
\end{prop}

\begin{defn} \label{def:3.3} In the context of Proposition~\ref{prop:os-onepar}, we call 
the quadruple \break 
$(\cE,\cE_+, \theta, U)$ a {\it euclidean realization} 
\index{euclidean realization}
of the contraction semigroup $(\hat U, \hat\cE)$.  
Writing $\hat U_t = e^{-tH}$ with a $H = H^* \geq 0$, we obtain 
by analytic continuation the unitary one-parameter group 
$U^c_t := e^{itH}$ (Example~\ref{ex:duality}). Accordingly, we also speak of a  
euclidean realization of $(U^c_t)_{t \in \R}$. 
\end{defn}

We shall see in Proposition~\ref{prop:exist-real} below 
that every strongly continuous contraction semigroup on a Hilbert space has a 
euclidean realization, but there are many non-equivalent ones 
with different sizes and different specific properties  
(Example~\ref{ex:4.2.5}). 

%The following lemma characterizes the situations where 
%the subspace $\cE_+$ is cyclic for a reflection positive one-parameter group. 
%
%\begin{lem}\label{lem:eplus-cyc}
%Let $(U_t)_{t \in \R}$ be a unitary one-parameter group on 
%$\cE$ and $\cE_+$ be a subspace invariant under $(U_t)_{t > 0}$. 
%Then the following conditions
% are equivalent: 
%\begin{enumerate}
%  \item[\rm(a)] The subspaces $(U_t \cE_+)_{t < 0}$ span a dense subspace of $\cE$. 
%  \item[\rm(b)] There exists a dense subspace $\cD \subeq \cE$ such that 
%$U_t v \in \cE_+$ for $v \in \cD$ and $t$ sufficiently large. 
%\item[\rm(c)] $\cE_+$ is $U$-cyclic in $\cE$. 
%\end{enumerate}
%\end{lem}
%
%\begin{prf} Since $U_t \cE_+ \subeq U_s \cE_+$ for $t > s$, the subset 
%$\cE_{-\infty} := \bigcup_{t \in \R} U_t \cE_+= \bigcup_{t < 0} U_t \cE_+$ 
%is a linear subspace. Assertions (a) and (c) mean 
%that $\cE_{-\infty}$ is dense and (b) means that $\cE_{-\infty}$ contains a dense subspac%e. 
%Therefore (a)-(c) are equivalent.
%\end{prf}

The following lemma provides a criterion for the 
density of a subspace of $\hat\cE$. We shall use it 
to verify that certain operators on $\hat \cE$ are densely defined. 
\begin{lem} \label{lem:e.6} 
Let $(U_t)_{t \in \R}$ be a reflection positive unitary one-parameter group 
on $(\cE,\cE_+, \theta)$. 
If $\cD \subeq \cE_+$ is a subspace invariant under the operators 
$(U_t)_{t > 0}$, for which 
\[ \cE_+^0 := \{ v \in \cE_+ \: (\exists T > 0)\ U_T v \in \cD \} \] 
is dense in $\cE_+$, then $\hat\cD \subeq \hat\cE$ is dense. 
\end{lem} 

\begin{prf} For $w \in \cE_+^0$ there exists a $T > 0$ with 
$U_T w \in \cD$, and this implies that 
$\hat U_t \hat w \in \hat\cD$ for $t \geq T$. Since the curve 
$\R_+ \to \hat\cE, t \mapsto \hat U_t w$, is analytic, 
$\hat U_t w \in \oline{\hat\cD}$ for every $t > 0$, and therefore 
$w \in \oline{\hat\cD}$ follows from the strong continuity of the semigroup 
$(\hat U_t)_{t \geq 0}$ (Proposition~\ref{prop:os-onepar}). As $\cE_+^0$ is dense 
in $\cE_+$,  it follows that $\widehat{\cD}$ is dense in $\widehat{\cE}$.
\end{prf}

\begin{rem} (Reduction to the $\cE_0$-cyclic case if 
$\hat{\cE_0}$ is cyclic in $\hat\cE$) 
Assume that $(U_t)_{t \in \R}$ is 
reflection positive on $(\cE,\cE_+,\theta)$ 
and that $q(\cE_0)$ is $\hat U$-cyclic in $\hat\cE$. 

Let $\tilde\cE \subeq \cE$ denote the 
closed $U$-invariant subspace generated by $\cE_0$ and 
$\tilde\cE_+ := \tilde\cE \cap \cE_+$. Then 
$\theta U_t \cE_0 = U_{-t} \theta \cE_0 = U_{-t} \cE_0$ implies that 
$\tilde\cE$ is $\theta$-invariant. Therefore $U_t' := U_t\res_{\tilde\cE}$ 
is a reflection positive unitary one-parameter group 
on $(\tilde\cE,\tilde\cE_+, \theta\res_{\tilde\cE})$. Since 
$q\res_{\tilde\cE_+'}$ has dense range, all the relevant data 
is contained in $\tilde\cE$. It is therefore natural to assume that 
$\cE_0$ is $U$-cyclic in $\cE$ whenever 
$q(\cE_0) = \hat{\cE_0}$ is cyclic in $\hat\cE$. 
\end{rem} 

The following proposition 
shows that the OS transform is compatible 
with the passage to the space of fixed points. 

\begin{prop} \label{prop:e.5b} 
{\rm(OS transform commutes with reduction)}
Let $(U_t)_{t \in \R}$ be a reflection positive unitary one-parameter group 
on $(\cE,\cE_+, \theta)$. 
Suppose that $\cE_+$ is $U$-cyclic and that $(\hat U_t)_{t \geq 0}$ 
is the corresponding one-parameter semigroup of contractions on $\hat\cE$. 
Let $\cE_{\rm fix}$ denote the subspace of elements fixed under 
all $U_t$ and $\hat\cE_{\rm fix}$ the subspace of fixed points 
for the semigroup $(\hat U_t)_{t > 0}$. Then the following assertions hold: 
\begin{enumerate}
  \item[\rm(a)] $\cE_{\rm fix} \subeq \cE_0$, the space of $\theta$-fixed points 
in $\cE_+$. 
  \item[\rm(b)] The map $q\res_{\cE_{\rm fix}} 
\: \cE_{\rm fix} \to \hat\cE_{\rm fix}, v \mapsto \hat v$ 
is a unitary isomorphism. 
  \item[\rm(c)] $\cE_{\rm fix} = \cE_\infty := \bigcap_{t > 0} U_t\cE_+$.
\end{enumerate}
\end{prop}

\begin{prf} (a) We write $P \: \cE \to \cE_{\rm fix}$ for the 
orthogonal projection onto the subspace of $U$-fixed points in $\cE$. 
Then 
\[ \lim_{N \to \infty} \frac{1}{N} \int_0^N U_t\, dt = P  \] 
holds in the strong operator topology (\cite[Cor.~V.4.6]{EN00}). 
%Let $\cD := \bigcup_{t \in \R} U_t \cE_+$. 
%be as in \break Lemma~\ref{lem:eplus-cyc}(b). 
For any $v \in U_s \cE_+$ and $s \in \R$, there exists a $T > 0$ with 
$U_t v \in \cE_+$ for $t > T$. 
Since 
\[ \lim_{N \to \infty} \frac{1}{N} \int_0^T U_t\, dt = 0, \] 
we obtain 
$P v\in \cE_+$ for every $v \in U_s\cE_+$. As $\cE_+$ is $U$-cyclic, 
we thus obtain $\cE_{\rm fix} = P\cE \subeq \cE_+.$ 
Since $\theta U_t \theta = U_{-t}$ for $t \in \R$, the subspace 
$\cE_{\rm fix}$ is $\theta$-invariant. 
Now the $\theta$-positivity of $\cE_+$ implies that 
$\theta\res_{\cE_{\rm fix}} \geq 0$, and thus $\cE_{\rm fix} \subeq \cE^\theta$. 

(b) Since $P$ commutes with $\theta$, 
Lemma~\ref{lem:d.1}(c) shows that $P$ defines a hermitian contraction 
$\hat P \: \hat\cE \to \hat\cE$ with $\hat P \hat v = \hat{Pv}$ for $v \in \cE_+$. 
%As $P^2 = P$ entails $\hat P^2 = \hat P$, $\hat P$ is a projection. 
For $v,w \in \cE_+$, we obtain 
\[ \lim_{N \to \infty} \frac{1}{N} \int_0^N \la \hat v, \hat U_t \hat w \ra\, dt 
=  \lim_{N \to \infty} \frac{1}{N} \int_0^N \la \theta v, U_t w \ra\, dt 
= \la \theta v, P w \ra = \la \hat v, \hat P \hat w \ra.\] 
Hence \cite[Cor.~V.4.6]{EN00} implies 
that $\hat P$ is the orthogonal projection onto~$\hat\cE_{\rm fix}$.

Let $q \: \cE_+ \to \hat\cE, v \mapsto \hat v$, denote the canonical projection onto $\widehat\cE$. Then 
$q \circ P = \hat P \circ q$ implies that  
$q(\cE_{\rm fix}) = q(P\cE_+) = \hat P q(\cE_+)$, and hence that 
$q(\cE_{\rm fix}) \subeq \hat\cE_{\rm fix}$ is a dense subspace. 
On the other hand, $\cE_{\rm fix} \subeq \cE_0$ implies that 
$q\res_{\cE_{\rm fix}}$ is isometric, hence a unitary isomorphism onto 
$\hat\cE_{\rm fix}$. 

(c) The subspace $\cE_\infty$ is closed and it is easily seen to be invariant 
under $U$.  Therefore 
$\cF := \oline{\cE_\infty + \theta\cE_\infty}$ is invariant under 
$U$ and $\theta$, so that we obtain a reflection 
positive unitary one-parameter group 
$V_t := U_t\res_{\cF}$ on  $(\cF,\cF_+, \theta\res_{\cF})$ with 
$\cF_+ := \cE_\infty$, 
satisfying $V_t \cF_+ = \cF_+$ for every $t > 0$. 
Now Lemma~\ref{lem:d.1}(d) leads to $\hat V_t = \hat V_{t/2} \hat V_{t/2} = \1$ 
for every $t > 0$. Therefore $\hat\cF \subeq \hat \cE_{\rm fix}$, 
and (b) implies that $\hat\cF\subeq q(\cE_{\rm fix})$, so that 
$\cE_\infty = \cF_+ \subeq \cE_{\rm fix} + \cN$. 

Since the elements of $\cE_{\rm fix}$ are $\theta$-fixed 
and $\cN = \cE_+ \cap \theta(\cE_+)^\bot$, we have $\cN \bot \cE_{\rm fix}$. 
{}From $\cE_{\rm fix} \subeq \cE_\infty$ it thus follows that 
$\cE_\infty = \cE_{\rm fix} \oplus (\cN \cap \cE_\infty)$ 
is a $U$-invariant orthogonal decomposition. 
As $\cN \cap \cE_\infty$ is orthogonal to the $U$-cyclic subspace $\theta(\cE_+)$, 
it must be zero, and this shows that $\cE_\infty = \cE_{\rm fix}$. 
\end{prf}

\begin{rem} \label{rem:3.8}
Let $\cE^1 := \cE_{\rm fix}^\bot$ in the context of Proposition~\ref{prop:e.5b}. 
Then the reflection positive one-parameter group 
is adapted to the orthogonal decomposition \break $\cE = \cE_{\rm fix} \oplus \cE^1$: 
\[ \cE_+ = \cE_{\rm fix} \oplus \cE^1_+, \quad 
\theta = \1 \oplus \theta_1, \quad 
U_t = \1 \oplus U^1_t \] 
with respect to the obvious notation. The data corresponding to $\cE_{\rm fix}$ is 
trivial and the one-parameter group $(U^1_t)_{t \in \R}$ on $(\cE^1, \cE^1_+,\theta)$ 
has the additional property that $\cE^1_{\rm fix} = \{0\}$. We also have that 
$\hat\cE \cong \hat\cE_{\rm fix} \oplus \hat\cE_1$. 
\end{rem}

\section{Reflection positive operator-valued functions} 
\label{sec:3.2}

We start with a characterization of continuous reflection positive 
functions for the symmetric semigroup $(G,\tau,S) = (\R, -\id_\R, \R_+)$. 
This is motivated by the GNS construction in Theorem~\ref{thm:1.x}. 

\begin{prop}[Integral representation of reflection positive functions]
  \label{prop:2.1b} 
 Let $\cF$ be a Hilbert space and
$\phi \: \R \to B(\cF)$ be strongly continuous. % with $\phi(\1) = \1$.
Then $\phi$ is reflection positive if and only if
there exists a finite $\Herm(\cF)_+$-valued Borel measure
$Q$ on $[0,\infty)$ such that
\begin{equation}
  \label{eq:2.1b}
  \phi(x) = \int_0^\infty e^{-\lambda |x|}\, dQ(\lambda).
\end{equation}
\end{prop}

\begin{prf} Suppose first that $\phi$ is
reflection positive for $(\R, \R_+, -\id)$ and consider the additive unital 
semigroup  $S :=([0,\infty),+)$. Then 
$\phi_S := \phi\res_S$ is positive definite with respect to the
trivial involution and corresponds to a contraction
representation of $S$ because
$|\la \xi,\phi(s)\xi\ra| \leq \la \xi,\phi(e)\xi\ra$
holds for the positive definite functions
$\phi^{\xi,\xi}(x) := \la \xi,\phi(x)\xi\ra$, $\xi \in \cF$ 
(\cite[Cor.~III.1.20(ii)]{Ne00}).
Using \eqref{eq:factori} in Example~\ref{ex:vv-gns} to write
$\phi(s) = \ev_0 \circ U^\phi_s \circ \ev_0^*$ for the GNS 
representation $(U^\phi, \cH_\phi)$ of~$S$  
and representing $U^\phi$ by a spectral measure $P$ on $[0,\infty)$ as
\[ U^\phi_s = \int_0^\infty e^{-\lambda s}\, dP(\lambda)\] 
(here we use that the operators $U^\phi_s$ are contractions), 
we obtain the desired integral representation of $\phi$ with
$Q := \ev_0 \circ P(\cdot) \circ \ev_0^*$. 
Now \eqref{eq:2.1b} follows from the fact that $\phi(-x)= \phi(x)^*
= \phi(x)$ holds for $x \geq 0$.

For the converse, we assume that $\phi$ has an integral representation as
in \eqref{eq:2.1b}. This immediately
implies that $\phi\res_S$ is positive definite on $S$
for the involution $s^\sharp = s$ and that
$\phi$ is continuous (\cite[Prop.~II.11]{Ne98}).
To show that $\phi$ is positive definite, we first 
recall from ~Example~\ref{ex:2.1.9} 
that 
\[ e^{-\lambda|x|} 
= \int_\R e^{i xy}\, \frac{1}{\pi} \frac{\lambda}{\lambda^2 + y^2}\ dy.\]
This implies that 
\[ \phi(x) 
= \int_\R e^{i xy}\Big( \int_0^\infty \, \frac{1}{\pi} \frac{\lambda}{\lambda^2 + y^2}
 \, dQ(\lambda)\Big)\ dy,\] 
and since 
$\tilde Q(y):=  \int_0^\infty \, \frac{\lambda}{\lambda^2 + y^2}\, dQ(\lambda)$
is an integrable function with values in positive operators, 
the positive definiteness of $\phi$ follows from Theorem~\ref{thm:bochner}. 
\end{prf}

Specializing Proposition~\ref{prop:2.1b} 
to $\cF = \C$, we obtain the following integral 
representation (cf.\ Example~\ref{ex:4.1.4}(a)):  
\begin{cor}
  \label{cor:2.1}  A continuous function $\phi \: \R \to \C$
is reflection positive if and only if it
has an integral representation of the form
\begin{equation}
  \label{eq:2.1}
  \phi(x) = \int_0^\infty e^{-\lambda|x|}\, d\nu(\lambda),
\end{equation}
where $\nu$ is a finite positive Borel measure on $[0,\infty)$.
\end{cor}

We note the following corollary to the first part of the proof 
of Proposition~\ref{prop:2.1b}: 

\begin{cor} \label{cor:repo-ext} 
Let $\cF$ be a Hilbert space and
$\phi \: [0,\infty)  \to B(\cF)$ be a bounded strongly continuous 
function which is positive definite on the $*$-semigroup $([0,\infty), \id)$. 
Then $\psi(t) := \phi(|t|)$ is reflection positive 
for $(\R,\R_+, -\id_\R)$. 
\end{cor}

\begin{defn} (Minimal unitary dilation) 
If $(C_t)_{t \geq 0}$ is a one-parameter semigroup of hermitian 
contractions on the Hilbert space $\cF$ and $\psi(t) := C_{|t|}$ 
is the corresponding reflection positive function from Corollary~\ref{cor:repo-ext}, 
then the unitary representation 
$U^{\psi}$ of $\R$ on the reproducing kernel Hilbert space 
$\cH_{\psi}$ (Theorem~\ref{prop:gns}) 
is called the {\it minimal unitary dilation} of~$C$. \index{minimal unitary dilation}

As $\psi(0)= \1$, the space $\cF$ may be considered as a subspace of $\cH_\psi$ 
and the orthogonal projection $P \:  \cH_\psi \to \cF$ satisfies
\begin{equation}
  \label{eq:dil-rel}
\phi(s) = P U^\psi_s P^* \quad \mbox{ for } \quad s \geq 0.
\end{equation}
For a detailed account on unitary dilations of semigroups, we refer to 
\cite{SzN10}; see in particular Proposition~\ref{prop:exist-real} 
below. 
\end{defn}

From Proposition~\ref{prop:3.9} we now derive 
that $(C_t)_{t \geq 0}$ has a canonical euclidean realization of Markov type 
in the sense of Definition~\ref{def:dil}, 
but this euclidean realization is rather large as we shall see in Example~\ref{ex:4.2.5}. 

\begin{prop} \label{prop:exist-real} 
For every strongly continuous one-parameter 
semigroup $(C_t)_{t \geq 0}$ of hermitian contractions on a Hilbert space $\cH$, 
there exists a euclidean realization $(U_t)_{t \in \R}$ of Markov type 
on $(\cE,\cE_+, \theta)$ with 
$\cE_0$ cyclic in $\cE$ and $\cE_+ = \lbr U_{\R_+} \cE_0\rbr$. 
Any realization with these two properties 
is equivalent to the minimal unitary dilation obtained by 
the $B(\cH)$-valued positive definite function $\psi(t) := C_{|t|}$ on $\R$. 
\end{prop}

We now develop a concrete picture of the minimal unitary dilation 
of a contraction semigroup.

\begin{ex} \label{ex:4.2.5a} Let $(C_t)_{t \geq 0}$ be a  
continuous semigroup of hermitian contractions on $\cH$. 
We write $C_t = e^{-tH}$ for a selfadjoint non-negative operator 
$H \geq 0$. To exclude trivialities, we assume that $\ker H =0$, 
so that the spectral measure $E$ of $H$ is supported by $(0,\infty)$ 
and we have $H = \int_0^\infty x\, dE(x)$. 

(a) On the Hilbert space $\cE := L^2(\R,\cH)$ we consider the 
unitary one-parameter group given by $(U_t f)(p) = e^{itp} f(p)$. 
We claim that $(U,\cE)$ is equivalent to the minimal unitary 
dilation of $(C_t)_{t \geq 0}$. 
To verify this claim, we consider the map 
\[ j \: \cH \to \cE, \qquad j(\xi)(p) := \frac{1}{\sqrt \pi} 
H^{1/2} (H + i p\1)^{-1} \xi \quad \mbox{ for } \quad p \not=0.\] 
For $p\not=0$, the operators $H^{1/2} (H \pm i p\1)^{-1}$ are bounded. 
For $\xi,\eta \in \cH$, we have 
\begin{align*}
\la j(\xi), U_t j(\eta)\ra 
&= \frac{1}{\pi} \int_\R 
\la H^{1/2} (H + i p\1)^{-1} \xi, e^{itp} H^{1/2} (H + i p\1)^{-1} \eta\ra\, dp \\
&= \frac{1}{\pi} \int_\R \int_0^\infty \frac{x}{x^2 + p^2} e^{itp}\, 
dE^{\xi,\eta}(x)\, dp\quad \mbox{for} \quad E^{\xi,\eta} = \la \xi,E(\cdot)\eta \ra \\ 
&= \int_0^\infty \Big(\int_\R \frac{1}{\pi} \frac{x}{x^2 + p^2} e^{itp}\, dp\Big)
dE^{\xi,\eta}(x)\\ 
&= \int_0^\infty e^{-x|t|}\, dE^{\xi,\eta}(x) = \la \xi, e^{-|t|H} \eta \ra 
= \la \xi, C_{|t|} \eta\ra.  
\end{align*}
For $t = 0$, this calculation implies that $j$ is isometric 
onto the subspace $\cE_0 := j(\cH)$ of $\cE$ and that the representation 
on the subspace $\tilde\cE := \lbr U_{\R}\cE_0\rbr$ is equivalent 
to the GNS representation $(U^\psi,\cH_\psi)$ for $\psi(t) := C_{|t|}$, 
hence the minimal unitary dilation. 

To verify our claim, it remains to show that $\cE = \tilde\cE$, i.e., that 
$\cE_0$ is $U$-cyclic.
%, so that 
%the uniqueness assertion in Proposition~\ref{prop:exist-real} applies. 
Since $\cH$ is generated by the $C$-invariant spectral subspaces 
$\cH_{a,b} := E([a,b])\cH$, $0 < a < b$, of~$H$, it suffices 
to argue that $\tilde\cE$ contains all subspaces $L^2(\R,\cH_{a,b})$. 
%We may thus assume that $H$ and $H^{-1}$ are bounded operators. 
Multiplication with $m(p) := H^{1/2}(H + ip\1)^{-1}$ defines 
for $c < d$ a bounded invertible operator $(Af)(p) = m(p)f(p)$ 
on each subspace $L^2([c,d],\cH_{a,b})$ which commutes 
with~$U$. Hence 
\[ L^2([c,d],\cH_{a,b})
  = A L^2([c,d],\cH_{a,b})
  = \lbr L^\infty([c,d]) j(\cH_{a,b})\rbr 
  \subeq \lbr U_\R  j(\cH_{a,b})\rbr \] 
follows from $U_\R'' = L^\infty(\R)$ and therefore $\cE_0$ is $U$-cyclic 
and $\tilde\cE = \cE$.  

(b) Now we determine the subspace $\cE_+ = \lbr U_{\R_+} \cE_0\rbr$ and the 
involution $\theta$. From 
$\cE_0 \subeq \cE^\theta$, $\theta U_t \theta= U_{-t}$ and the cyclicity 
of $\cE_0$, we immediately obtain 
\[ (\theta f)(p) 
= \frac{H-ip\1}{H+ip\1} f(-p) := (H-ip\1)(H+ip\1)^{-1} f(-p),\] 
so that $U$ is reflection positive on $(\cE,\cE_+,\theta)$. 
The Markov property follows from the multiplicativity of 
$\psi(t) = C_{|t|}$ for $t \geq 0$ (Proposition~\ref{prop:3.9}). 

Next we show that $\cE_+ \subeq \cE$ coincides with $L^2_+(\R,\cH) 
:= \cF L^2(\R_+, \cH)$, where $\cF$ is the Fourier transform. 
Writing $\C_+ := \R_+ + i \R$ for the right half plane, 
the map 
\begin{equation}
  \label{eq:hardy-laplace}
\frac{1}{\sqrt{2\pi}} \cL \: L^2(\R_+,\cH) \to \cO(\C_+, \cH), 
\quad \cL(f)(z) = \int_0^\infty e^{-zx} f(x)\, dx 
\end{equation}
is isometric onto the $\cH$-valued \index{Hardy space}
{\it Hardy space} $H^2(\C_+, \cH)$  with the norm 
\[ \|f\|^2 = \lim_{z \to 0_+} \int_\R \|f(z + ip)\|^2\, dp.\] 
Its reproducing kernel $Q(z,w) \in B(\cH)$ is given by 
$Q(z,w) = \frac{1}{2\pi}\frac{1}{z + \oline w}\1$ because 
the functions $Q_{z,\xi} := (2\pi)^{-1/2} e_{-\oline z}\chi_{\R_+}\xi 
\in L^2(\R_+,\cH)$ with $e_z(x) = e^{zx}$ and $\xi \in \cH$ 
satisfy 
\begin{equation}
  \label{eq:hardyker}
\la Q_{z,\xi}, Q_{w,\eta} \ra 
= \frac{1}{2\pi} \int_0^\infty e^{-x(z + \oline w)}\la \xi,\eta\ra\, dx 
= \frac{1}{2\pi} \frac{\la \xi,\eta\ra}{z + \oline w}. 
\end{equation}

That $\cE_0$ is contained in 
$L^2_+(\R, \cH)$ follows from the following 
calculation for $\Re z \geq 0$ (and evaluating in $z = ip$), 
where we put $E^\xi := E(\cdot)\xi$: 
\begin{align*}
&H^{1/2} (H + z\1)^{-1} \xi
=   \int_0^\infty \frac{x^{1/2}}{x + z}\, dE^\xi(x)
=   \int_0^\infty \int_0^\infty e^{-\lambda z} e^{-\lambda x} x^{1/2}\, 
d\lambda \, dE^\xi(x)\\
&=   \int_0^\infty \int_0^\infty e^{-\lambda z} e^{-\lambda x} x^{1/2}\, 
dE^\xi(x)\, d\lambda =   \int_0^\infty e^{-\lambda z} 
\Big(e^{-\lambda H} H^{1/2}\xi\Big)\ d\lambda. 
\end{align*}
Now $\cE_+ = \lbr U_{\R_+} \cE_0 \rbr\subeq L^2_+(\R, \cH)$ 
follows from the invariance of $L^2_+(\R,\cH)$ 
under the operators $U_t = \cF \circ V_t \circ \cF^{-1}$ for 
$t \geq 0$, where %$\cF$ is the Fourier transform and 
$(V_t f)(x) = f(x-t)$. 
In view of the maximal $\theta$-positivity of $\cE_+$ 
(Lemma~\ref{le:Markov}), equality will follow if the Hardy space 
is $\theta$-positive. This is verified as follows. The functions 
$f_{z,\xi}(p) := \frac{1}{\oline z + ip} \xi 
= \int_0^\infty e^{-x(\oline z + i p)}\xi\, dx$, $\Re z > 0$, $\xi\in \cH$, 
generate $L^2_+(\R,\cH)$. We have 
\begin{equation}
  \label{eq:H-theta-pos}
  \la f_{z,\xi}, \theta f_{w,\eta} \ra 
= \int_\R \frac{\la \xi, (H- i p\1)(H + i p\1)^{-1} \eta \ra}
{(z - i p)(\oline w - ip)} 
\, dp \quad \mbox{ for } \quad \Re z, \Re w > 0.  
\end{equation}
Since the function 
%\[ G(\zeta) :=\frac{\la \xi, (\zeta\1 + i H)(\zeta\1-iH\1)^{-1} \eta \ra}
%{(\zeta + i z)(\zeta + i \oline w)} 
\[ G(\zeta) 
:=\frac{\la \xi, (H - i \zeta\1)(H + i \zeta \1)^{-1} \eta \ra}
{(z - i \zeta)(\oline w - i \zeta)} 
=- \frac{\la \xi, (H - i \zeta\1)(H + i \zeta \1)^{-1} \eta \ra}
{(\zeta + iz)(\zeta + i\, \oline w)} \] 
is meromorphic in the lower half plane $\{\Im \zeta < 0\}$ with 
poles in $-iz$ and $-i\,\oline w$ and $\lim_{\zeta \to \infty} |\zeta G(\zeta)| = 0$, 
the Residue Theorem, applied to negatively oriented paths in the lower 
half plane which lead to winding number $-1$, yields for $z \not=\oline w$:  
\[   \la f_{z,\xi}, \theta f_{w,\eta} \ra 
= 2\pi i (-\Res_{-iz}(G) - \Res_{-iw}(G)),\] 
where 
\[ \Res_{-iz}(G) 
= -\frac{\la \xi, (H - z\1)(H + z \1)^{-1}\eta\ra}{-iz+ i\, \oline w} 
= -i\frac{\la \xi, (H - z\1)(H + z \1)^{-1}\eta\ra}{z-\oline w}\] 
and 
\[ \Res_{-i\,\oline w}(G) 
= -\frac{\la \xi, (H - \oline w\1)(H + \oline w\1)^{-1}\eta\ra}{-i\,\oline w 
+ i z} 
= i \frac{\la \xi, (H - \oline w\1)(H + \oline w\1)^{-1}\eta\ra}{z-\oline w}.\]
We thus arrive at 
\begin{align*}
  \la f_{z,\xi}, \theta f_{w,\eta} \ra 
&=  \frac{2\pi}{z-\oline w} 
\la \xi, 
(z\1-H)(H + z\1)^{-1} - (\oline w \1-H)(H + \oline w \1)^{-1} \eta\ra \\ 
&=  4\pi \la \xi, (H + z\1)^{-1} H (H + \oline w \1)^{-1} \eta\ra \\ 
&=  4\pi \la H^{1/2}(H + \oline z\1)^{-1} \xi, H^{1/2} (H + \oline w \1)^{-1} \eta\ra, 
\end{align*}
which obviously is a positive definite kernel on $\C_+ \times \cH$, 
and therefore $L^2_+(\R, \cH)$ is $\theta$-positive. 

(c) Finally we note that the map 
\[ (Tf)(p) := \sqrt{\pi} H^{-1/2}(H + i p\1) f(p) \] 
maps $L^2(\R,\cH)$ unitarily onto the space 
$\tilde \cE$ of all $\cH$-valued $L^2$-functions with respect to the norm 
given by 
\[ \|f\|_H^2 := \frac{1}{\pi} \int_\R \la f(p), H(H^2 +p^2)^{-1} f(p)\ra\, dp.\] 
The operator $T$ intertwines $U$ with the representation 
$(\tilde U_t f)(p) = e^{itp} f(p)$ and the involution $\theta$ with 
$(\tilde\theta f)(p) := f(-p)$. 
\end{ex}

\begin{ex} We take a closer look at the Hardy space $H^2(\C_+,\cH) 
= \cH_Q \subeq \cO(\C_+,\cH)$ 
with the reproducing kernel $Q(z,w) = \frac{1}{z + \oline w}\1$ 
introduced in \break Example~\ref{ex:4.2.5a}(b). For simplicity we omit 
the factor $\frac{1}{2\pi}$, so that the Laplace transform 
$\cL \: L^2(\R_+, \cH) \to H^2(\C_+, \cH)$ is unitary. 

(a) For the 
translation action $(V_t f)(x) = f(x-t)$ on $L^2(\R,\cH)$ we have for $t \geq 0$ 
and $f \in L^2(\R_+,\cH)$: 
\begin{equation}
  \label{eq:covar1}
\cL(V_t f)(z) 
%= \int_0^\infty e^{-xz}f(x-t)\, dt 
= \int_t^\infty e^{-xz}f(x-t)\, dt 
= \int_0^\infty e^{-(x+t)z}f(x)\, dt 
= e^{-t z}\cL(f)(z), 
\end{equation}
so that the subsemigroup $\R_+ \subeq \R$ acts on $H^2(\C_+,\cH)$ by 
multiplication $(U_t f)(z) = e^{-tz}f(z)$, $t \geq 0$. 
This action is isometric because the boundary values of $e_{-t}(z) = e^{-tz}$ 
on $i\R$ have absolute value~$1$. 

(b) We now specialize to the scalar case where 
$\cH = \C$ and $C_t = e^{-t\lambda}$ for some $\lambda > 0$. 
Then Example~\ref{ex:4.2.5a}(a) shows that the 
subspace $\cE_0$ of $\cE_+ \cong H^2(\C_+,\C)$ is generated by the 
function $Q_\lambda(z) = \frac{1}{\lambda + z}$ 
with $\|Q_\lambda\|^2 = Q(\lambda,\lambda) = \frac{1}{2\lambda}$. 
Therefore the quotient map $q \: \cE_+ \cong H^2(\C_+,\C) \to \hat\cE \cong \C$ 
is given by evaluation: 
\[ q(f) 
= \sqrt{2\lambda} \la Q_\lambda, f \ra 
= \sqrt{2\lambda} f(\lambda).\] 
This formula also shows immediately that 
\[ q(U_t f) = e^{-t\lambda} q(f)\quad \mbox{for } \quad 
f \in H^2(\C_+,\C), t \geq 0.\] 

(c) (Anti-unitary involutions) 
On the space $\cE = L^2(\R)$, the conjugation
\[ (Jf)(p) := \oline{f(-p)} \] 
commutes with $(U_t)_{t \in \R}$ and satisfies 
$J\cE_+ = \cE_+$ and $J Q_z = Q_{\oline z}$ for $\Re z \geq 0$. 
Therefore it induces on the Hardy space $H^2(\C_+,\C)$ the conjugation given by 
\[ (Jf)(z) = \oline{f(\oline z)}\quad \mbox{ for } \quad f \in \cH_Q.\] 
We also observe that $J \theta = \theta J$ on $L^2(\R)$, 
and since $J Q_\lambda = Q_\lambda,$ 
it induces on $\hat\cE \cong \C$ the involution given by complex conjugation. 
\end{ex} 

\begin{ex} \label{ex:4.2.5} (Cyclic contraction semigroups) 
%(see Section ~\ref{ex:4.5} and example \ref{ex:4.11} for details). 
For a $\sigma$-finite measure space \break $(X,\fS,\rho)$ and a 
measurable function $h \: X \to [0,\infty)$, consider 
on $\cH := L^2(X,\rho)$ the hermitian contraction semigroup 
$C_t f= e^{-th} f$. 
%Then $1 \in \cH$ is a cyclic vector and, 
By the Spectral Theorem, all cyclic contraction semigroups 
can be represented this way  with a finite measure 
$\rho$ on $X = [0,\infty)$ and $h(\lambda) = \lambda$, so that $1$ is a cyclic vector. 

(a) With Example~\ref{ex:4.2.5a}(c), we obtain a euclidean realization of 
Markov type (and, in general, infinite multiplicity) by 
\begin{align*}
 \cE &= L^2(\R^\times \times X, \zeta), \quad 
d\zeta(x,\lambda ) = \Big(\frac{1}{\pi} \frac{h(\lambda)}{h(\lambda)^2 + x^2}
\, dx\Big)\, d\rho(\lambda ), \\ 
& (U_t f)(x,\lambda ) = e^{itx} f(x,\lambda), \qquad 
(\theta f)(x,\lambda) := f(-x,\lambda).  
\end{align*} 
Here $\cE_0\cong L^2(X,\rho)$ is the subspace of functions 
$f(x,\lambda) = f(\lambda)$ not depending on $x$ and this subspace is 
$U$-cyclic in $\cE$. 

(b) For a finite measure 
$\rho$ on $X = [0,\infty) = \R_{\geq 0}$ and $h(\lambda) = \lambda$, 
a multiplicity free euclidean realizations can be obtained as follows. 
The projection \break $\pr \: \R \times \R_{\geq 0} \to \R, \pr(x,\lambda) := x$ 
maps the measure $\zeta$ to the measure 
$\nu := \pr_*\zeta$ given by 
\begin{equation}
  \label{eq:nu-form2}
d\nu(x) = \frac{1}{\pi}\Big(\int_0^\infty \frac{\lambda }{\lambda^2 + x^2}
\, d\rho(\lambda )\Big)\, dx.
\end{equation}
Then $\cF := L^2(\R, \nu)$ can be identified with the 
$U$-invariant subspace of $\cE$ consisting 
of functions not depending on 
$\lambda$. For $\cF_+ := \cF \cap \cE_+$, we obtain on 
$(\cF,\cF_+,\theta)$ a reflection positive one-parameter group $(V_t)_{t \in \R}$ 
by restriction. Now $\cF_0 = \cF_+^\theta = \cE_0 \cap \cF = \C 1$ is the space of 
constant functions. 
Here $1 \in \cF_0$ is a cyclic vector corresponding to the reflection positive function 
\[ \phi(t) :=  \la 1, V_t 1 \ra = \int_\R e^{-itx}\, d\nu(x)
= \int_0^\infty e^{-\lambda |t|}\, d\rho(\lambda) \] 
(Example~\ref{ex:2.1.9}(a)). Its restriction to $\R_+$ leads to a 
GNS representation equivalent to the multiplication representation 
of $\R_+$ on $\cH = L^2(\R_{\geq 0},\rho)$ given by $C$, 
so that $(V,\cF,\cF_+,\theta)$ also 
is a euclidean realization of $(C,\cH)$, but not of Markov type 
if $\rho$ is not a point measure. 

In both cases, the subspaces  $\cE_0$ and $\cF_0$, respectively, are 
cyclic, but in the first case the Markov condition 
$q(\cE_0) = \hat\cE$ holds, 
whereas in the second case $\cF_0 = \C 1$ is one-dimensional.  

(c) For $\dim \cE_0 = 1$ and $C_t = e^{-t\lambda}$, $\lambda \geq 0$, the minimal 
dilation $\phi(t) = e^{-\lambda|t|}$ from (a) leads to the Hilbert space 
$\cE \cong L^2(\R, \frac{\lambda}{\pi}\frac{dx}{\lambda^2 + x^2})$ 
because $\rho = \delta_\lambda$ is a point measure. 
In this case the realizations in (a) and (b) coincide. 
\end{ex}

\section{A connection to Lax--Phillips scattering theory}
\label{se:LaxPhil}

One parameter groups and reflection positivity are closely  related
to the Sinai/Lax-Phillips scattering theory and translation invariant 
subspaces (\cite{LP64,LP67,LP81,Sin61}). %  as was already noticed in \cite{JOl98}. 
In short, this theory says that every  
unitary representation
of $\R$ on a Hilbert space $\cE$ satisfying some 
simple conditions stated below can be realized by translations  
in $L^2(\R,\cM)$ for some multiplicity Hilbert space $\cM$.  

Let $(U,\cE)$ be a unitary representation of $\R$. \index{outgoing!subspace}
A closed subspace $\cE_+\subset \cE$ is called \textit{outgoing} if
\begin{description}
\item[\rm(LP1)] $\cE_+$ is invariant under $U_t, t > 0$,
\item[\rm(LP2)] $\cE_\infty := \displaystyle \bigcap_{t>0} U_t\cE_+ =\{0\}$,
\item[\rm(LP3)] $\displaystyle \bigcup_{t<0} U_t\cE_+$ is dense in $\cE$.
\end{description}

\index{Theorem!Lax--Phillips Representation} 
The following theorem is classical (\cite[Thm.~1]{LP64}): 
\begin{thm} {\rm(Lax--Phillips Representation Theorem)} \label{thm:LP}
If $\cE_+$ is outgoing for $(U,\cE )$, then there exists a Hilbert
space $\cM$ such that $\cE\simeq L^2 (\R ,\cM)$,  $\cE_+\simeq L^2( [0,\infty),\cM)$, and $U$ is represented by translations $(U_tf)(x)=f(x-t)$. This representation
is unique up to isomorphism of $\cM$.
\end{thm}

This realization of $(U, \cE)$ is called the 
\textit{outgoing realization of} $U$. 
\index{outgoing!realization}
\begin{prf} \begin{footnote}{We thank Bent \O{}rsted for communicating 
this short representation theoretic proof.}\end{footnote}
For $t \in \R$, we put $\cE_t := U_t \cE_+$ 
and write $E_t$ for the corresponding projection. Then our assumptions imply that 
\[ \lim_{t \to \infty} E_t = \{0\} \quad \mbox{ and } \quad 
\lim_{t \to -\infty} E_t = \1\]
in the strong operator topology. We further have 
$\lim_{s \to t-} E_s = E_t.$ 
In fact, $\cE_t \subeq \cE_s$ for $s < t$ implies that 
$E_t^- := \lim_{s \to t-} E_s \geq E_t$ exists. 
If, conversely, $v \in E_t^- \cE = \bigcap_{s < t} \cE_s$,  we have 
for every $h > 0$ that $U_h v \in \cE_t$ and thus 
$v \in \cE_t$ by closedness. Thus $E_t^- \leq E_t$, and therefore $E_t^- = E_t$. 
We conclude that the family 
$(E_t)_{t \in \R}$ defines a Stieltjes spectral measure $P$ on 
$\fB(\R)$ such that $P([t,\infty)) = E_t$ for $t \in \R$. 
\index{$\fB(X)$, Borel subsets of top. space $X$}

Consider the unitary one-parameter group defined by 
$V_s := \int_\R e^{isx}\, dP(x)$. Then 
$U_t E_x = E_{x+t}$ shows that 
$U_t P([x,\infty)) = P([x+t,\infty))$ for $x,t \in \R$, and therefore 
$U_tP([a,b]) = P([t+a, t+b])$ for $a < b$. This implies 
\[ U_t V_s U_{-t} 
= \int_\R e^{isx} dP(x+t) 
= \int_\R e^{is(x-t)} dP(x) 
= e^{-ist} V_s.\] 
Therefore we obtain a unitary representation of the Heisenberg group 
$\Heis(\R^2) = \T \times \R^2$ with the product 
\[ (z,s,t)(z',s', t') = (zz'e^{-its'},  s + s', t + t')
\quad \mbox{ by } \quad 
\pi(z,s,t) := z V_s U_t.\] 
Now the assertion follows from the Stone--von Neumann Theorem \break 
(\cite[Thm.~X.3.1]{Ne00}). 
 \end{prf}

We now connect the Lax--Phillips construction to the dilation process.   
The following proposition is an obvious consequence of 
the Lax--Phillips Theorem~\ref{thm:LP} and
Proposition~\ref{prop:e.5b}. 

\begin{prop} \label{prop:4.11} Let $(U_t)_{t\in\R}$ be a 
reflection positive unitary one-parameter group on $(\cE,\cE_+,\theta)$ 
for which $\cE_+$ is cyclic and $\cE_{\rm fix} = \{0\}$. Then 
$\cE_+$ is outgoing, so that 
$(U,\cE)$ is unitarily equivalent to the translation representation 
on $L^2(\R,\cM)$ for some Hilbert space $\cM$. 
This realization is unique up to   isomorphism of $\cM$.
\end{prop}

\begin{ex} As we have seen in Example~\ref{ex:4.2.5a}, 
the Fourier transform immediately yields an outgoing realization 
of the minimal dilation representation of a contraction semigroup 
$(C_t)_{t \geq 0}$ with trivial fixed points on $\cH$ on the space $\cE = L^2(\R,\cH)$. 
\end{ex}

\begin{rem}
Proposition~\ref{prop:4.11} shows in particular that, up to a 
direct summand consisting of fixed points, 
the spectrum of any euclidean realization is all of $\R$ 
and the representation is a multiple of the translation representation of~$\R$ 
on~$L^2(\R)$. 
\end{rem}

Proposition~\ref{prop:4.11} suggests to attempt a classification of reflection 
positive one-parameter groups in an outgoing realization 
on $\cE = L^2(\R,\cM)$ with $\cE_+ = L^2(\R_+,\cM)$ by classifying 
the unitary equivalence classes of corresponding unitary involutions. 

The Fourier transform of the subspace $\cE_+ = L^2(\R_+, \cM)$ 
is the $\cM$-valued Hardy  space $\tilde\cE_+ = H^2(\C_+,\cH)
 \subeq \cO(\C_+,\cM)$ 
which can be considered 
as a closed subspace of $\tilde\cE := L^2(\R,\cM)$ by the natural boundary-value map. 
The translation action of $\R$ on $\cE$ leads to the multiplication action 
\[ (\tilde U_t f)(z) = e^{-tz} f(z)\quad \mbox{ for } \quad f \in \tilde\cE_+.\]  
To classify reflection positive one-parameter groups 
(without fixed points) now corresponds to the problem to determine 
the involutions $\theta$ on $\tilde\cE$ for which $\tilde\cE_+$ is 
$\theta$-positive and $\theta \tilde U_t \theta = \tilde U_{-t}$. 
Since the commutant of the multiplication action on $\tilde \cE$ 
is the von Neumann algebra $L^\infty(\R,B(\cM))$, the involution $\theta$ must be of 
the form 
\[  (\theta f)(p) = m(p) f(-p), \quad \mbox{ where } \quad 
m \: \R \to \U(\cM) \] 
is a unitary operator in $L^\infty(\R,B(\cM))$,  
which basically is a measurable map with values in $\U(\cM)$) satisfying 
$m(-p) = m(p)^* = m(p)^{-1}$ almost everywhere. That the Hardy space 
is $\theta$-positive is by \eqref{eq:H-theta-pos} 
equivalent to the positive definiteness of the 
$B(\cM)$-valued kernel 
\begin{equation}
  \label{eq:R-ker}
   R(z,w) := 
\int_\R \frac{m(p)}{(z - i p)(\oline w - i p)}\, dp.
\end{equation}

In Example~\ref{ex:4.2.5a}, we have seen that this is the 
case if $m(p) = \frac{H-ip\1}{H+ip\1}$ for a strictly positive 
operator $H$ on $\cM$. This corresponds to the case of reflection positive 
one-parameter groups of Markov type. 

\begin{probl} Characterize unitary-valued functions $m \in L^\infty(\R,B(\cM))$ 
with \break $m(-p) = m(p)^*$ for which the kernel 
\eqref{eq:R-ker} on $\C_+$ is positive definite. 
\end{probl}

\section*{Notes on Chapter~\ref{ch:4}} 

\S~\ref{subsec:3.0}: 
The material discussed briefly in Section~\ref{subsec:3.0} 
is contained in \cite{JNO18}. 
For recent progress in the local theory of positive definite 
functions on groups we refer to \cite{JN15, JPT15}. 

\S~\ref{subsec:3.1}: 
A version of Corollary~\ref{cor:2.1} for reflection positivity on the group $\Z$ 
can be found in \cite[Prop.~3.2]{FILS78}. In this context reflection positivity 
is also analyzed in \cite{JT17}. 

For the special 
case where $\phi(0) = \1$, strongly continuous reflection positive 
functions $\phi \: \R \to B(V)$ 
are called (OS)-positive covariance functions in \cite{Kl77}, and 
Proposition~\ref{prop:2.1b} specializes to \cite[Rem.~2.7]{Kl77}. 

Corollary~\ref{cor:repo-ext} 
can also be found in \cite[Thm.~I.8.1]{SzN10}. 
For a detailed account on unitary dilations of semigroups, we refer to 
\cite{SzN10}; see in particular Proposition~\ref{prop:exist-real}.

\end{bibunit}

\chapter{Reflection positivity on the circle}
\label{ch:5} 

\begin{bibunit}
In this chapter we turn to 
the close relation between reflection positivity 
on the circle group~$\T$ and 
the Kubo--Martin--Schwinger (KMS) condition for states of $C^*$-dynamical 
systems. Here a crucial point 
is a pure representation theoretic perspective on the KMS condition 
formulated as a property of form-valued positive definite functions on~$\R$: 
For $\beta > 0$, we consider the open strip 
\[ \cS_\beta := \{ z \in \C \: 0 < \Im z < \beta\}.\] 
For a real vector space $V$, we say that a positive definite 
function $\psi \: \R \to \Bil(V)$ (Definition~\ref{def:a.3}) 
satisfies the {\it $\beta$-KMS condition} if $\psi$ 
\index{KMS condition!positive definite function}
extends to a pointwise continuous 
function $\psi$ on $\oline{\cS_\beta}$ which is 
pointwise holomorphic on $\cS_\beta$ and satisfies 
$\psi(i \beta+t) = \oline{\psi(t)}$ for $t \in \R.$ 

The key idea in the classification of positive definite functions 
satisfying a KMS condition is 
to relate them to  standard (real) subspaces of a (complex) Hilbert space 
which occur naturally in the modular theory of operator algebras 
(\cite{Lo08}). These are closed real subspaces $V \subeq \cH$ for which 
$V \cap i V = \{0\}$ and $V + i V$ is dense. 
Any standard subspace determines a pair 
$(\Delta, J)$ of {\it modular objects}, \index{modular objects}
where $\Delta$ is a positive selfadjoint 
operator and $J$ an anti-linear involution (a {\it conjugation}) 
\index{conjugation!on Hilbert space}
satisfying $J\Delta J = \Delta^{-1}$. 
The connection is established by 
\begin{equation}
  \label{eq:stand-rel}
 V = \Fix(J\Delta^{1/2}) = \{ \xi \in \cD(\Delta^{1/2}) \: J\Delta^{1/2} \xi = \xi \}. 
\end{equation}
A key result is the  characterization of the 
KMS condition in terms of standard subspaces (Theorem~\ref{thm:kms}) 
which also contains a classification in terms of an integral representation. 

For a function $\psi$ satisfying the $\beta$-KMS condition, analytic 
continuation to $\oline{\cS_\beta}$ leads to an operator-valued function  
\[ \tilde\phi \: [0,\beta] \to B(V_\C)\quad \mbox{ by } \quad 
\la \xi, \tilde\phi(t)\eta \ra = \psi(it)(\xi,\eta)\quad \mbox { for } \quad 
\xi,\eta \in V.\] 
This function satisfies $\tilde\phi(\beta) = \oline{\tilde\phi(0)}$, hence 
extends uniquely to a weak operator continuous 
function $\tilde\phi \: \R \to B(V_\C)$ satisfying 
\begin{equation}
  \label{eq:phi-per}
\tilde\phi(t + \beta) = \oline{\tilde\phi(t)} \quad \mbox{ for } \quad t \in \R.
\end{equation}
Here we write complex linear operators on $V_\C$ as $A + i B$ with 
$A,B \in B(V)$ and put $\oline{A + i B} = A - i B$. 

Recall the group $\R_\tau = \R \rtimes \{e,\tau\}$ with 
$\tau(t) = -t$. In Theorem~\ref{thm:5.2.3} we 
show that there exists a natural positive definite function 
\[ f \: \R_\tau \to \Bil(V) 
\quad \mbox{ satisfying } \quad 
f(t,\tau) = \tilde\phi(t). \] 
The function $f$ is $2\beta$-periodic,  
hence factors through a function on the group 
\[ \T_{2\beta,\tau} := \R_\tau/\Z 2\beta \cong \OO_2(\R)\] 
and it is reflection positive for $G = \T_{2\beta}$ and 
$G_+ = [0,\beta] + 2\beta\Z$. 
This leads to a natural euclidean realization of the 
unitary one-parameter group $U_t = \Delta^{-it/\beta}$ 
associated to~$\psi$. 
We conclude this section with a description of the GNS representation 
of $\T_{2\beta,\tau}$ in a natural space of sections of a vector bundle 
over the circle $\R/\Z\beta$ with two-dimensional fiber on which 
the scalar product is given by a resolvent of the Laplacian 
as in Section~\ref{subsec:2.1.4}; see also Section~\ref{sec:7.3.5}.

\section{Positive definite functions satisfying 
KMS conditions} 

In this section we present a characterization 
(Theorem~\ref{thm:kms}) of form-valued positive definite functions 
on $\R$ satisfying a KMS condition. 
We also explain how the corresponding 
representation of $\R$ can be realized in a Hilbert space 
of holomorphic functions on the strip $\cS_{\beta/2}$ with continuous boundary 
values (Proposition~\ref{prop:2.4}). 

\index{function!form-valued!pointwise continuous}
\index{function!form-valued!holomorphic} 
We call a function $\psi \: \oline{\cS_\beta} \to \Bil(V)$ 
{\it pointwise continuous} if, for all $v,w \in V$, the function 
$\psi^{v,w}(z) := \psi(z)(v,w)$ is continuous. 
Moreover, we say that $\psi$ is {\it pointwise holomorphic 
in $\cS_\beta$}, if, for all $v,w \in V$, the 
function $\psi^{v,w}\res_{\cS_\beta}$ is holomorphic. 
By the Schwarz reflection principle, any pointwise continuous 
pointwise holomorphic function $\psi$ is uniquely determined by 
its restriction to~$\R$. 

\begin{defn} \label{def:2.1} 
%For $\beta > 0$, let $\cS_\beta := \{ z \in \C \: 0 < \Im z < \beta\}$. 
For a real vector space $V$, we say that a positive definite 
function $\psi \: \R \to \Bil(V)$ %(cf.\ Definition~\ref{def:a.3})
satisfies the {\it KMS condition} \index{KMS condition!positive definite function}
for $\beta > 0$ if $\psi$ 
extends to a function $\psi \: \oline{\cS_\beta} \to \Bil(V)$ which is pointwise 
continuous and pointwise holomorphic on $\cS_\beta$, and satisfies 
\begin{equation}
  \label{eq:kms-gen} 
\psi(i \beta+t) = \oline{\psi(t)} 
\quad \mbox{ for } \quad t \in \R. 
\end{equation}
\end{defn} 

\begin{lem}
  \label{lem:2.2} 
Suppose that $\psi \: \R \to \Bil(V)$ satisfies the KMS 
condition for $\beta > 0$. 
Then 
\begin{equation}
  \label{eq:herm-cond}
\psi(i\beta + \oline z) = \oline{\psi(z)} =  \psi(-\oline z)^\top 
\quad \mbox{ for }\quad z \in \oline{\cS_\beta}.
\end{equation}
%  \label{eq:kms-gen-cplx} 
The function $\phi \: [0,\beta] \to \Bil(V), 
\phi(t) := \psi(it)$ has hermitian values and satisfies 
\begin{equation}
  \label{eq:kms-imag} 
\phi(\beta - t) = \oline{\phi(t)}
\quad \mbox{ for } \quad 0 \leq t \leq \beta.
\end{equation}
It extends to a unique strongly 
continuous symmetric $2\beta$-periodic function \break $\phi :\R \to \Herm(V)$ 
satisfying 
\[ \phi(\beta + t) = \oline{\phi(t)} \quad \mbox{ and } 
\quad \phi(-t) = \phi(t) \quad \mbox{  for } \quad t \in \R.\] 
\end{lem}

\begin{prf} Note that $\psi(-t) = \oline{\psi(t)}^\top = \psi(t)^*$ 
holds for every positive definite function \break $\psi \:  \R \to \Bil(V)$.
By analytic continuation (resp., the Schwarz Reflection Principle), this leads to 
the second equality in \eqref{eq:herm-cond}. Likewise, 
condition \eqref{eq:kms-gen} leads  to the first equality in 
\eqref{eq:herm-cond}. 
This in turn implies \eqref{eq:kms-imag}, and the remainder is clear.
\end{prf}

To obtain a natural representation of $\psi$, we 
now introduce standard subspaces $V \subeq \cH$ 
and the associated modular objects~$(\Delta, J)$.  

\begin{defn} \label{def:1.5} 
A closed real subspace $V$ of a complex Hilbert space $\cH$ 
is said to be {\it standard} if \index{standard subspace} 
\[ V \cap i V = \{0\} \quad \mbox{ and } \quad \oline{V + i V} = \cH.\] 

 For every standard real subspace $V \subeq \cH$,  
we define an unbounded anti-linear operator 
\[ S \: \cD(S) = V + i V \to \cH, \quad 
S(\xi + i \eta) := \xi - i \eta\quad \mbox{ for } \quad \xi, \eta \in V.\] 
Then $S$ is closed and has a polar decomposition 
$S = J \Delta^{1/2},$ 
where $J$ is an anti-unitary involution and $\Delta$ a strictly 
positive selfadjoint 
operator (cf.\ \cite[Lemma~4.2]{NO15b}; see also \cite[Prop.~2.5.11]{BR02}, 
\cite[Prop.~3.3]{Lo08}). 
We call  $(\Delta, J)$ the pair of {\it modular objects} of $V$.
\end{defn} \index{modular objects}

\begin{rem} \label{rem:2.2b}
(a) From $S^2 = \id$, it follows that 
the modular objects $(\Delta, J)$ of a standard subspace 
satisfies the {\it modular relation} \index{modular relation}
\begin{equation}
  \label{eq:modrel}
J\Delta J = \Delta^{-1}.   
\end{equation}
If, conversely, $(\Delta, J)$ is a pair 
of a strictly 
positive selfadjoint operator $\Delta$ and an anti-unitary involution $J$ 
satisfying \eqref{eq:modrel}, then $S := J \Delta^{1/2}$ is an 
anti-linear involution with $\cD(S) = \cD(\Delta^{1/2})$ whose fixed point space 
$\Fix(S)$ is a standard subspace. Thus standard subspaces  
are in one-to-one correspondence with pairs $(\Delta, J)$ satisfying
\eqref{eq:modrel} (cf.\ \cite[Prop.~3.2]{Lo08} and \cite[Lemma~4.4]{NO15b}). 

(b) As the unitary one-parameter group $(\Delta^{it})_{t \in \R}$ 
commutes with $J$ and $\Delta$, it leaves the real subspace $V = \Fix(S)$ invariant. 
\end{rem}

The following proposition (\cite[Prop.~3.1]{NO15b}) 
 provides various characterizations of unitary one-parameter groups 
with reflection symmetry. As we shall see below, these are precisely those 
for which a euclidean realization on $\T_{2\beta,\tau}$ can be obtained 
by a positive definite function satisfying the $\beta$-KMS condition. 

\begin{prop} \label{prop:2.9} 
For a unitary one-parameter group 
$(U_t)_{t \in \R}$ on $\cH$ with spectral measure $E \: \fB(\R) \to B(\cH)$, 
the following are equivalent:
\begin{enumerate}
\item[\rm(i)] There exists an anti-unitary involution $J$ on $\cH$ with 
$J U_t J = U_t$ for $t \in \R$. 
\item[\rm(ii)] For $\cH_\pm := E(\R^\times_\pm)\cH$, the unitary one-parameter groups 
$U^+_t := U_t\res_{\cH_+}$ and $U^-_t := U_{-t}\res_{\cH_-}$ are unitarily equivalent. 
\item[\rm(iii)] The unitary one-parameter group 
$(U,\cH)$ is equivalent to a GNS representation 
$(U^\psi, \cH_\psi)$, where $\psi \: \R \to B(V)$ 
is a symmetric positive definite function. \
\item[\rm(iv)] There exists a unitary involution $\theta$ on $\cH$ with 
$\theta U_t \theta = U_{-t}$ for $t \in \R$. 
\end{enumerate}
\end{prop}

\begin{rem} It is easy to see that conditions (i)-(iv) even imply the existence 
of an extension of $U$ to a representation of the group 
$\R_\tau \times \{\pm 1\} \cong (\R^\times)_\tau \cong 
\OO_{1,1}(\R)$ by unitary and anti-unitary operators, where 
$\tau$ is represented by a unitary involution and $(0,-1)$ by 
a conjugation~$J$. Since any unitary representation is a direct sum of cyclic 
ones, it suffices to verify our claim in the cyclic case. 
Under the assumption of Proposition~\ref{prop:2.9}, $(U,\cH)$ is 
equivalent to the representation in $\cH = L^2(\R,\nu)$ for a finite symmetric 
measure $\nu$ given by $(U_t f)(p) = e^{itp} f(p)$. Then 
$(\theta f)(p) := f(-p)$ is a unitary involution with $\theta U_t \theta^{-1} = U_{-t}$ and 
$(Jf)(p) := \oline{f(-p)}$ is an anti-unitary involution with $J U_t J^{-1} = U_{t}$ 
for $t\in \R$. Clearly, $\theta$ and $J$ commute. 

For a systematic discussion of anti-unitary representations 
we refer to \cite{NO17}. 
\end{rem}

\index{Theorem!KMS Characterization}
\begin{thm} {\rm(KMS Characterization Theorem; \cite[Thm.~2.6]{NO16})}  \label{thm:kms} 
Let $V$ be a real vector space, let $\beta > 0$, and 
let $\psi \:  \R \to \Bil(V)$ be a pointwise continuous positive definite function. 
Then the following are equivalent: 
\begin{enumerate}
\item[\rm(i)] $\psi$  satisfies the 
$\beta$-KMS condition. 
\item[\rm(ii)] There exists a standard subspace 
$V_1$ in a Hilbert space $\cH$ and a linear map 
$j \: V \to V_1$ such that 
\begin{equation}
  \label{eq:pdform2}
 \psi(t)(\xi,\eta) = \la j(\xi), \Delta_V^{-it/\beta}  j(\eta) \ra 
\quad \mbox{ for } \quad 
t \in \R, \xi,\eta \in V.
\end{equation}
\item[\rm(iii)] There exists a (uniquely determined) 
regular Borel measure $\mu$ on $\R$ 
with values in the cone $\Bil^+(V)\subeq \Bil(V)$,  consisting of forms 
with a positive semidefinite extension to $V_\C$, which satisfies 
$d\mu(-\lambda) = e^{-\beta\lambda} d\oline\mu(\lambda)$ and 
\[ \psi(t) = \int_\R e^{it\lambda}\, d\mu(\lambda)\quad \mbox{ for } \quad t \in \R.\] 
%r_* \mu = e_{-\beta} \oline \mu \quad \mbox{  for } \quad 
%r(\lambda) = - \lambda,\ e_\beta(\lambda) = e^{\beta\lambda}.\] 
\end{enumerate}
If these conditions are satisfied, then the function 
$\psi \: \oline{\cS_\beta} \to \Bil(V)$ is pointwise bounded. 
\end{thm}

The equivalence of (i) and (ii) in this theorem describes the tight connection 
between the KMS condition and the modular objects associated to a standard 
subspace. Part (iii) provides an integral representation that can be 
viewed as a classification result in the sense that it characterizes 
those  measures whose Fourier transforms satisfy the KMS condition 
from the perspective of Bochner's Theorem (Theorem~\ref{thm:bochner}). 

\begin{ex} \label{ex:5.1.8} 
If $V = \R$ and $\Bil(V) \cong \C$, $\Bil^+(V)= \R_{\geq 0}$, 
the integral representation 
in Theorem~\ref{thm:kms}(iii) specializes to the integral representation 
obtained in Remark~\ref{rem:4.9} for $\beta$-periodic reflection positive 
functions on $\R$: 
\[ \phi(t) := \psi(it) = \cL(\mu)(t) \quad \mbox{ for } \quad 
0 \leq t \leq \beta,\] 
for a finite measure $\mu$ on $\R$ that can be written as 
$d\mu(\lambda) = d\mu_+(\lambda) + e^{\beta\lambda} d\mu_+(-\lambda)$ 
for a measure $\mu_+$ on $\R_{\geq 0}$. 
This shows already that, in this case, the $\beta$-periodic extension of the 
function $\phi$ to $\R$ is reflection positive. Below we shall see how this 
observation can be extended to the general case. 

The corresponding Hilbert space can be identified 
with $\cH = L^2(\R,\mu)$, where 
$(U_t f)(\lambda) = e^{it\lambda} f(\lambda),$ 
so that $U_t = \Delta^{-it/\beta}$ leads to the modular operator 
\[ (\Delta f)(\lambda) = e^{-\beta\lambda} f(\lambda). \] 
As $\mu$ is finite, $1 \in \cH$ and 
we have $\psi(t) = \la 1, U_t 1\ra$ for $t \in \R$. 
To determine a suitable standard subspace $V_1$, respectively, a conjugation 
$J$ commuting with $U$, we note that the requirement $1 \in V_1$ and 
the requirement that $J$ commutes with $U$ lead to 
\[ (Jf)(\lambda) = e^{-\lambda\beta/2}\oline{f(-\lambda)},\] 
so that the corresponding operator  $S := J \Delta^{1/2}$ is given by 
$(Sf)(\lambda) = \oline{f(-\lambda)},$ 
and this leads to 
\[ V_1 = \{ f \in L^2(\R,\mu) \: 
f(-\lambda) = \oline{f(\lambda)}\ \mu\text{-almost everywhere}\}.\] 
We shall continue the discussion of this example in Remark~\ref{rem:5.2.6} below. 
\end{ex}

\begin{rem} The KMS condition is well known in 
Quantum Statistical Mechanics as a condition characterizing quantum 
versions of Gibbs states, resp.,  equilibrium states. 
The monograph \cite{BR96} and the lecture notes \cite{Fro11} are 
excellent sources for more information on 
KMS states and their applications. 

We now explain how the classical context of KMS states of operator algebras 
relates to our setup. \index{$C^*$-dynamical system}
Consider a {\it $C^*$-dynamical system} $(\cA,\R,\alpha)$, 
i.e., a homomorphism $\alpha \: \R \to \Aut(\cA)$, where 
$\cA$ is a $C^*$-algebra. Here we deal with the real linear space 
\[ V := \cA_h := \{A  \in \cA \:  A^* = A\} \] 
of hermitian elements in $\cA$, 
so that any state $\omega \in \cA^*$ defines an element of $\Bil(V)$ by 
$(A,B) \mapsto \omega(AB)$. 
An $\alpha$-invariant state $\omega$ on $\cA$ 
is called a {\it $\beta$-KMS state} if and only if \index{KMS state}
\[ \psi \: \R \to \Bil(\cA_h), \quad 
\psi(t)(A,B) :=\omega(A\alpha_t(B)) \] 
satisfies the $\beta$-KMS condition  
(cf.\  \cite[Prop.~5.2]{NO15b}, \cite[Thm.~4.10]{RvD77}). 
If $(\pi_\omega, U^\omega, \cH_\omega, \Omega)$ 
is the corresponding covariant GNS representation of $(\cA,\R, \alpha)$ 
(cf.~\cite{BGN17}, \cite{BR02}), 
then 
\[ \omega(A) = \la \Omega, \pi_\omega(A) \Omega \ra \quad \mbox{ for }\quad 
A \in \cA \quad \mbox{ and }\quad 
U^\omega_t \Omega = \Omega \quad \mbox{ for }\quad t \in \R.\] 
Therefore 
\begin{align*}
\psi(t)(A,B) &= \omega(A\alpha_t(B)) 
= \la \Omega, \pi_\omega(A\alpha_t(B)) \Omega \ra \\
&= \la \Omega, \pi_\omega(A) U^\omega_t \pi_\omega(B) U^\omega_{-t} \Omega \ra 
= \la \pi_\omega(A)\Omega,  U^\omega_t \pi_\omega(B) \Omega \ra 
\end{align*}
for $A, B \in \cA_h$. We conclude that the 
corresponding standard subspace of 
$\cH_\omega$ is $V_1 := \oline{\pi_\omega(\cA_h)\Omega}$. 
\end{rem}

\begin{cor}
  \label{cor:5.2} If $\psi \:  \R \to \Bil(V)$ satisfies
the $\beta$-KMS condition, then the kernel 
\begin{equation}
  \label{eq:kernel}
 K \: \oline{\cS_{\beta/2}} \times \oline{\cS_{\beta/2}} \to \Bil(V),\qquad 
K(z,w)(\xi,\eta) := \psi(z-\oline w)(\xi,\eta)
\end{equation}
is positive definite. 
\end{cor}

\begin{prf} From \eqref{eq:pdform2} 
in the KMS Characterization Theorem~\ref{thm:kms}, we obtain by 
uniqueness of analytic continuation 
\begin{equation}
  \label{eq:anacont1}
\psi(z-\oline w)(\xi,\eta)
= \la \Delta^{i\, \oline z/\beta} j(\xi), \Delta^{i\,\oline w/\beta} j(\eta) \ra,  
\qquad %\mbox{ for } \quad 
\xi,\eta\in V,\ z,w \in \oline{\cS_{\beta/2}}.
\end{equation}
Now Remark~\ref{rem:a.1.2} shows that $K$ is positive definite. 
\end{prf}

Now that we know from Corollary~\ref{cor:5.2} that the kernel 
$K$ in \eqref{eq:kernel} 
is positive definite, we obtain a corresponding reproducing kernel 
Hilbert space consisting of functions on $\oline{\cS_{\beta/2}} \times V$ 
which are linear in the second argument and holomorphic on 
$\cS_{\beta/2}$ in the first. We may therefore think of these 
functions as having values in the algebraic dual space 
$V^* := \Hom(V,\R)$ of~$V$. 
We write 
$\cO(\oline{\cS_{\beta/2}},V^*)$ for the space of functions 
$f \: \oline{\cS_{\beta/2}}\to V^*$ for which 
all function $f^\eta(z) := f(z)(\eta)$, $\eta \in V$, are continuous on 
$\oline{\cS_{\beta/2}}$ and holomorphic on the open strip~$\cS_{\beta/2}$. 
For a proof of the following proposition, see \cite[Prop.~2.9]{NO16}.

\begin{prop} \label{prop:2.4} 
{\rm(Holomorphic realization of $\cH_\psi$)}  
Assume that $\psi \: \R \to \Bil(V)$ satisfies the $\beta$-KMS 
condition, let $\psi \: \oline{\cS_\beta} \to \Bil(V)$ 
denote the corresponding extension and 
$\cH_\psi \subeq \cO(\oline{\cS_{\beta/2}},V^*)$ denote the Hilbert space with reproducing 
kernel 
\[ K(z,w)(\xi,\eta)  := \psi(z- \oline w)(\xi,\eta) 
\quad \mbox{ for } \quad \xi,\eta \in V,\] 
i.e., 
\[ f(z)(\xi) = \la K_{z,\xi}, f \ra \quad \mbox{ for } \quad f \in \cH_\psi, 
\quad \mbox{ where } \quad 
K_{z, \xi}(w)(\eta) = \psi(w- \oline z)(\eta, \xi).\] 
Then 
\[ (U^\psi_t f)(z) := f(z + t), \qquad t \in \R, z \in \oline{\cS_{\beta/2}}\]
defines a unitary one-parameter group on $\cH_\psi$, 
\[ j \: V \to \cH_\psi, \quad j(\eta)(z) := \psi(z)(\cdot, \eta) = K_{0,\eta}(z) 
\] 
is a linear map with $U^\psi$-cyclic range, and 
\[ \psi(t)(\xi,\eta) = \la j(\xi), U^\psi_t j(\eta) \ra \quad \mbox{ for } \quad 
t\in \R, \xi,\eta \in V.\] 
The anti-unitary involution $J_1$ on $\cH_\psi$ corresponding to 
the standard subspace ${V_1 \subeq \cH_\psi}$ from 
{\rm Theorem~\ref{thm:kms}} is given by 
$(J_1 f)(z) := \oline{f\Big(\oline z + \frac{i\beta}{2}\Big)}.$ 
\end{prop}

\section{Reflection positive functions and KMS conditions} 

In this section we build the bridge from positive definite functions 
$\psi \: \R\to \Bil(V)$ 
satisfying the $\beta$-KMS condition to reflection positive functions on 
the group $\T_{2\beta,\tau}\cong \OO_2(\R)$. 

We have already seen in Lemma~\ref{lem:2.2} that analytic continuation leads 
to a symmetric $2\beta$-periodic function 
$\phi \: \R\to \Bil(V)$ satisfying 
$\phi(t + \beta) =\oline{\phi(t)}$ for $t \in \R$ and 
$\phi(t) = \psi(it)$ for $0 \leq t \leq \beta$. We shall 
construct a positive definite extension $f \: \R_\tau \to \Bil(V)$ with 
$f(t,\tau) = \phi(t)$ for $t \in \R$; actually the values 
of $f$ will be represented by bounded operators on $V_\C$, 
so that we also consider it as a $B(V_\C)$-valued function. 
By construction, $f$ is then reflection 
positive with respect to the interval $[0,\beta/2] =: G_+ \subeq G := \R$ 
in the sense of Definition~\ref{def:1.2c}. 

Building on Theorem~\ref{thm:kms}, our first goal is to express,  
for a standard subspace $V \subeq \cH$, the $\Bil(V)$-valued function 
\begin{equation}
  \label{eq:phi1}
\phi_V \: [0,\beta]  \to \Bil(V), \qquad 
\phi_V(t)(\xi,\eta) := \psi(it)(\xi,\eta) = 
\la \Delta^{t/2\beta} \xi, \Delta^{t/2\beta} \eta \ra \end{equation}
from \eqref{eq:anacont1} in terms of a $B(V_\C)$-valued function. 
To this end, we shall need the description 
of a standard subspace $V_1$ in terms of a skew-symmetric 
{\it strict contraction} $C$ on~$V_1$ 
($\|Cv\| < \|v\|$ for $0 \not=v$), \index{strict contraction} 
and this leads to a quite explicit description of~$\phi$ 
that is used to obtain the main theorem asserting that, 
for every positive definite function \break 
$\psi \: \R \to \Bil(V)$ satisfying the $\beta$-KMS condition, 
there exists a reflection positive function 
$f \: \R_\tau %\T_{2\beta,\tau} 
\to B(V_\C)$ satisfying 
\[ \psi(it)(\xi,\eta) = \la \xi, f(it,\tau)\eta \ra 
\quad \mbox{ for } \quad \xi,\eta \in V, 0 \leq t \leq \beta. \]
Then the corresponding GNS representation 
$(U^f, \cH_f)$ of the group $\T_{2\beta,\tau} \cong \OO_2(\R)$ 
is a euclidean realization of the unitary one-parameter group 
$(\Delta^{-it/\beta})_{t \in \R}$ corresponding to~$\psi$ 
via \eqref{eq:pdform2} because 
$\cE = \cH_f$ leads to $\hat\cE \cong \cH_{\phi\res_{(0,\beta)}} \cong \cH_\psi$ 
(cf.~Theorem~\ref{thm:1.x}). 

%\subsection{Existence of reflection positive extensions} 

The following lemma describes the complex-valued scalar product on a standard 
real subspace in terms of the corresponding modular objects $(\Delta, J)$. 
For $v,w \in V$, we write $\la v,w \ra_V := \Re \la v,w \ra_\cH$.

\begin{lem} \label{lem:5.2.1} Let $V \subeq \cH$ be a standard subspace. 
Then there exists a skew-symmetric strict contraction $C$ on $V$ with 
\begin{equation}
  \label{eq:standandC}
\Im \la \xi,\eta \ra_\cH = \la \xi,C \eta \ra_V \quad \mbox{ for } \quad \xi,\eta
\in V.
\end{equation}
\end{lem}

\begin{prf} Since $\omega(v,w) := \Im \la v,w \ra_\cH$ defines a continuous 
skew-symmetric bilinear form on $V$, there exists a uniquely determined
skew-symmetric operator $C \in B(V)$ with $\omega(v,w) = \la v, C w \ra_V$ 
for $v,w \in V$. As 
$|\Im \la v, w \ra_\cH| \leq \|v\|\cdot\|w\|$ for $v,w \in V$, we have 
$\|C\| \leq 1$, i.e., $C$ is a contraction. 

To see that $C$ is a strict contraction, 
assume $\|Cv\| = \|v\|$, i.e., $v \in \ker(C^2 + \1)$. 
For $w := Cv$ we then have 
$C(v + iw) = w - i v = (-i)(v + iw)$. This leads to the relation 
$\la v - i w, v - i w \ra_{\cH} = 0$ and thus $v - i w = 0$ implies 
$v \in V \cap i V = \{0\}$. 
\end{prf}

With the preceding lemma, we can express the function 
$\phi_V$ from \eqref{eq:phi1} in terms of $C$ by bounded operators on $V_\C$. 

\begin{lem} \label{lem:3.3} {\rm(\cite[Lemma~4.2]{NO16})} 
Let $V \subeq \cH$ be a standard subspace with modular objects $(\Delta, J)$ 
and $C$ be the skew-symmetric strict contraction from {\rm Lemma~\ref{lem:5.2.1}}. 
%We assume that $\ker C = \{0\}$, so that 
%the polar decomposition $C = I |C|$ defines a complex structure $I$ on~$V$. 
Then the function $\phi_V(t)(\xi,\eta) 
= \la \Delta^{t/2\beta} \xi, \Delta^{t/2\beta} \eta \ra_\cH$ 
from \eqref{eq:phi1} can be written as 
\begin{equation}
  \label{eq:phi-rel}
 \phi_V(t)(\xi,\eta) = 
\la \xi, \tilde\phi(t) \eta \ra_{V_\C} \quad \mbox{ for }\quad t \in [0,\beta], 
\xi,\eta \in V_\C  
\end{equation}
with 
\[  \tilde\phi(t) = (\1 + i C)^{1-t/\beta} (\1- i C)^{t/\beta} 
\in B(V_\C).\] 
\end{lem}

Note that $\tilde\phi(0) = \1 + i C$ is not real if $C \not=0$ 
and that both operators $\1 \pm i C$ are bounded 
positive hermitian with a possibly unbounded inverse. Therefore 
\[  \tilde\psi(z) = (\1 + i C)^{1+ iz/\beta} (\1- i C)^{-iz/\beta} \in B(V_\C)\] 
is well-defined for $0 \leq \Im z \leq \beta$, strongly continuous and 
holomorphic for ${0 < \Im z < \beta}$. One also verifies immediately the 
$\beta$-KMS relation 
\[ \oline{\tilde\psi(z)} =  \tilde\psi(i\beta + \oline z)\quad \mbox{ for }\quad 
0 \leq  \Re z \leq \beta.\] 

\index{Theorem!Reflection positive extension}
\begin{thm} \label{thm:5.2.3} {\rm(Reflection positive extension)}  
Let $V \subeq \cH$ be a standard subspace and let 
$C = I|C|\in B(V)$ be the skew-symmetric strict contraction 
satisfying \eqref{eq:standandC}. 
We assume that $\ker C = \{0\}$, so that $I$ defines a complex structure on~$V$.
We define a weakly continuous function 
$\tilde \phi \:\R \to B(V_\C)$ by
\[ \tilde\phi(t) = (\1 + i C)^{1-t/\beta} (\1 - i C)^{t/\beta} \quad \mbox{ for } \quad 
0 \leq t \leq \beta \quad \mbox{ and } \quad 
\tilde\phi(t + \beta) = \oline{\tilde\phi(t)}\] 
for $t \in \R.$ Write 
$\tilde\phi(t) = u^+(t) + i I u^-(t)$ with 
$u^\pm(t) \in B(V)$ and $u^\pm(t+ \beta) = \pm u^\pm(t).$
Then 
\[ f \: \R_\tau \to B(V_\C), \qquad 
f(t,\tau^\eps) := u^+(t) + (iI)^\eps u^-(t), \qquad t \in \R, \eps \in \{0,1\}, \] 
is a weak-operator continuous positive definite function with 
$f(t,\tau) = \tilde\phi(t)$ for ${t \in\R}$. 
It is reflection positive with respect to the subset $[0,\beta/2] \subeq \R$ 
in the sense that the kernel 
$f\big((t,\tau)(-s,e)\big) = f(t + s,\tau), 0 \leq s,t \leq \beta/2$,  
is positive definite. 
%More specifically, we have 
%\begin{itemize}
%\item[\rm(a)] If $C = 0$, then $f(t,\tau^\eps) = \1$ is constant. 
%\item[\rm(b)] If $C$ is injective, then 
%$f(t,\1) = \Re \tilde\phi(t) - I \Im \tilde\phi(t)\in B(V)$ for $t \in \R$. 
%\end{itemize}
\end{thm}

Combining the preceding theorem with Lemma~\ref{lem:3.3}, we obtain 
in particular: 

\begin{cor} Let $V$ be a real vector space and let 
$\psi \: \R \to \Bil(V)$ be a continuous positive definite function 
satisfying the $\beta$-KMS condition. Then there exists a 
pointwise continuous function $f \: \R_\tau \to \Bil(V)$ 
which is reflection positive with respect to the subset 
$[0,\beta/2] \subeq \R$ and satisfies 
\[ f(t,\tau) = \psi(it) \quad \mbox{ for } \quad 0 \leq t \leq \beta \qquad 
\mbox{ and } \quad 
f(t+\beta,\tau) = \oline{f(t,\tau)} \quad \mbox{ for } \quad t\in \R.\] 
\end{cor}

%\subsection{Characterizing reflection positive extensions} 
%\label{subsec:4.4}

In Theorem~\ref{thm:5.2.3} we obtained 
for certain functions $\phi$ 
on the coset $\R \rtimes \{\tau\}\subeq \R_\tau$ 
reflection positive extensions $f$ to all of $\R_\tau$. 
The following lemma shows that, conversely, every reflection positive 
function on $\T_{2\beta,\tau}$ leads by analytic extension 
to a positive definite function on $\R$ satisfying the 
$\beta$-KMS condition. 

\begin{lem} \label{lem:5.2.5} 
Let $f \: \R_\tau \to \Bil(V)$ be a 
pointwise continuous function which is reflection positive 
with respect to $[0,\beta/2] \subeq \R$ such that the function  
\break $\phi \: \R \to \Bil(V), \phi(t) := f(t,\tau)$ satisfies 
\begin{equation}
  \label{eq:transrel}
 \phi(t) = \phi(-t) = \oline{\phi(\beta + t)}\quad \mbox{ for } \quad t \in \R.
\end{equation}
Then there exists a unique $\beta$-KMS positive definite function 
$\psi \: \R \to \Bil(V)$ with 
\[ \phi(t) = \psi(it) \quad \mbox{ for } \quad 0 \leq t \leq \beta.\] 
\end{lem}

\begin{prf} Reflection positivity implies that the kernel 
$\phi\big(\frac{t+s}{2}\big)$ for $0 \leq t,s \leq \beta$ 
is positive definite. By  Theorem~\ref{thm:I.7} 
there exists a $\Bil^+(V)$-valued Borel measure $\mu$ on $\R$ such that 
\begin{equation}
  \label{eq:lapl-form}
\phi(t) = \int_\R e^{-\lambda t}\, d\mu(\lambda) 
\quad \mbox{ for }\quad 
0 < t < \beta.
\end{equation}
The continuity of $\phi$ on $[0,\beta]$ actually implies that the 
integral representation also holds on the closed interval $[0,\beta]$ 
by the Monotone 
Convergence Theorem. In particular, the measure $\mu$ is finite. 
Therefore its Fourier transform 
$\psi(t) := \int_\R e^{it\lambda}\, d\mu(\lambda)$ is a pointwise continuous 
$\Bil(V)$-valued positive definite function on $\R$. 
Further, \eqref{eq:transrel} implies 
\begin{equation}
  \label{eq:measrel}
e^{\beta\lambda}\, d\mu(-\lambda) = d\oline{\mu}(\lambda)
\end{equation}
and Theorem~\ref{thm:5.2.3} shows that $\phi(t) = \psi(it)$ holds for the 
$\beta$-KMS function \break $\psi \: \R \to \Bil(V)$.
\end{prf} 

\begin{rem} \label{rem:5.2.6} 
From \eqref{eq:lapl-form} it follows that the function 
$\phi$ is real-valued if and only if the measure $\mu$ takes values 
in the subspace of real-valued forms in $\Bil^+(V)$. 

For the case where $V\subeq \cH$ is a standard subspace and 
$\phi = \phi_V$ as in \eqref{eq:phi1}, we have 
$\tilde\phi_V(0) = \1 + i C$, so that $C = 0$ 
if $\phi_V$ is real-valued, and this in turn implies that $\phi_V$ is constant.

Therefore the only way to obtain non-constant real-valued functions 
is to ensure that the map $j \: V \to V_1$ 
in Theorem~\ref{thm:kms} takes values in a subspace $j(V)$ which 
is isotropic for the skew-symmetric form 
$\omega(\xi,\eta) := \la \xi, C \eta\ra_V = \Im \la \xi, \eta \ra_\cH.$
This condition corresponds to $\phi(0)$ being real, but is still weaker 
than $\phi(t)$ being real for every $t \in [0,\beta]$. 

If $\phi$ is real-valued, then 
$f(t,\tau^\eps) := \tilde\phi(t)$ for $t \in \R, \eps \in \{0,1\}$ 
is $\tau$-biinvariant, $\beta$-periodic and 
reflection positive on $\R_\tau$ (Lemma~\ref{lem:biinvar}(ii)). 

It is instructive to take another look at Example~\ref{ex:5.1.8}, 
where $\cH = L^2(\R,\mu)$ for a finite measure satisfying 
$d\mu(\lambda) = d\mu_+(\lambda) + e^{\beta\lambda} d\mu_+(-\lambda)$ 
for a measure $\mu_+$ on $\R_{\geq 0}$. Here the standard subspace $V_1$ 
consists of all functions satisfying 
$f(-\lambda) = \oline{f(\lambda)}$ almost everywhere on $\R$. 
For simplicity we assume that $\mu(\{0\}) = 0$ (which excludes constant 
summands). For $\xi \in V_1$, the restriction $\xi_+ := \xi\res_{\R_+}$  
determines $\xi$ completely, so that we may consider $V_1$ as a space of 
functions on $\R_+$. The scalar product on this space is given by 
\begin{align*}
\la \xi,\eta \ra_\cH 
&= \int_\R \oline{\xi(\lambda)}\eta(\lambda)\, d\mu(\lambda) 
%&= \int_0^\infty \oline{\xi_+(\lambda)}\eta_+(\lambda)\, d\mu_+(\lambda) 
%+ \int_{-\infty}^0\oline{\xi(\lambda)}\eta(\lambda)e^{\beta\lambda}\, d\mu_+(-\lambda) \\
%&= \int_0^\infty \oline{\xi_+(\lambda)}\eta_+(\lambda)\, d\mu_+(\lambda) 
%+ \int_0^\infty \xi_+(\lambda)\oline{\eta_+(\lambda)}e^{-\beta\lambda}\, d\mu_+(\lambda)\\
= \int_0^\infty (\oline{\xi_+(\lambda)}\eta_+(\lambda) + 
\xi_+(\lambda)\oline{\eta_+(\lambda)}e^{-\beta\lambda}\big)\, d\mu_+(\lambda). 
\end{align*}
For the real part we obtain 
\[ \la \xi,\eta \ra_{V_1} = 
\Re \la \xi,\eta \ra_{\cH} 
= \int_0^\infty \Re\big(\oline{\xi_+(\lambda)}\eta_+(\lambda)\big) 
(1 + e^{-\beta\lambda})\, d\mu_+(\lambda),\] 
and 
\[ \omega(\xi,\eta) = \Im \la \xi,\eta \ra_{\cH} 
= \int_0^\infty \Im\big(\oline{\xi_+(\lambda)}\eta_+(\lambda)\big)
(1 - e^{-\beta\lambda})\, d\mu_+(\lambda)\] 
for the imaginary part. We conclude that 
\[ V_1 \cong L^2(\R_+, (1 + e^{-\beta\lambda})\, d\mu_+(\lambda); \C) \] 
and that the skew-symmetric operator $C$ representing $\omega$ 
is given by 
\[ (Cf)(\lambda) = C(\lambda) f(\lambda), \quad \mbox{ where } \quad 
C(\lambda) = -i \frac{1 - e^{-\beta\lambda}}{1 + e^{-\beta\lambda}} \] 
(cf.\ \cite[Lemma~B.9]{NO16}). Hence the corresponding complex structure 
is given by $(If)(\lambda) = -i f(\lambda)$ and 
$|C| = i C$ corresponds to multiplication with the positive function $i C(\lambda) 
= \frac{1 - e^{-\beta\lambda}}{1 + e^{-\beta\lambda}}$ on $\R_+$. 

The subspace 
\[ V := L^2(\R_+, (1 + e^{-\beta\lambda})\, d\mu_+(\lambda); \R) \] 
of real-valued functions is $\omega$-isotropic. 
As it is invariant under the operators 
\[ \tilde\phi(t) 
= (\1 + i C)^{1 - t/\beta}(\1 - i C)^{t/\beta} 
= (\1 + |C|)^{1 - t/\beta}(\1 - |C|)^{t/\beta},\] 
the corresponding function 
\[ \phi \: [0,\beta] \to \Bil(V), \quad 
\phi(t)(\xi,\eta) = \la \xi, \tilde\phi(t) \eta\ra, \quad 
\xi,\eta \in V \] 
is real-valued. 
\end{rem}

From the scalar case $(V = \R)$ in Remark~\ref{rem:4.9} 
one easily obtains the following characterization of $\beta$-periodic 
operator-valued reflection positive functions on~$\R$. It is concerned with 
the case where $\phi$ is real-valued, so that $f$ is $\tau$-biinvariant 
(Lemma~\ref{lem:biinvar}), corresponding to function on the circle group $\T_\beta$ 
(see also \cite[Thm.~3.3]{KL81}).  

\begin{thm}  \label{thm:2.4}
A $\beta$-periodic pointwise continuous function 
$\phi \: \R \to \Bil(V)$ 
is reflection positive with respect to $[0,\beta/2]$ 
if and only if there exists a $\Bil^+(V)$-valued Borel measure 
$\mu_+$ on $[0,\infty)$ such that 
\begin{equation}
  \label{eq:phi-intrepb}
\phi(t) = \int_0^\infty e^{-t\lambda} + e^{-(\beta - t)\lambda}\, d\mu_+(\lambda) 
\quad \mbox{ for } \quad 0 \leq t \leq \beta.  
\end{equation}
Then the measure $\mu_+$ is uniquely determined by $\phi$. 
\end{thm}

\begin{defn} (Euclidean realization in the periodic case) 
For any reflection positive function $f$ as in 
Lemma~\ref{lem:5.2.5}, the general discussion in 
Theorem~\ref{thm:1.x} shows that, for the corresponding 
reflection positive representation on 
$\cE = \cH_f$, we obtain $\hat\cE \cong \cH_{\phi\res_{(0,\beta)}}$. 

As $\psi$ is pointwise holomorphic on the strip 
$\cS_\beta$, it further follows 
by restriction that $\cH_{\phi\res_{(0,\beta)}}  \cong \cH_\psi$ 
(cf.~Proposition~\ref{prop:2.4}). 
Therefore the unitary one-parameter group $(U^c_t f)(z) := f(z + t)$ 
on $\hat\cE$ whose infinitesimal generator is given by $\frac{d}{dz}$, 
is obtained from the unitary representation 
$U^f$ on $\cE$ by the OS transform as in 
Example~\ref{ex:duality}, even if it is not positive. 
We thus call $(U^f,\cH_f)$ a {\it euclidean realization of $U^c$} 
(cf.~Definition~\ref{def:dil}). 
\end{defn} \index{euclidean realization}

At this point it is a natural question which unitary one-parameter groups 
$(U^c, \cH)$ have a euclidean realization. 
This can now be stated 
in terms of the conditions discussed in  Proposition~\ref{prop:2.9} 
(\cite[Thm.~3.4]{NO15b}): 

\index{Theorem!Realization Theorem for unitary one-parameter groups}
\begin{thm} \label{thm:realize} {\rm(Realization Theorem)} 
A unitary one-parameter group $(U_t^c)_{t \in \R}$ on a Hilbert space 
$\cH$ has a euclidean realization  in terms of a reflection positive 
representation of $(\T_{2\beta}, \T_{2\beta,+},\theta)$ 
if and only if there exists an anti-unitary 
involution $J$ on $\cH$ commuting with $U^c$. 
\end{thm}

In the setting of Theorem~\ref{thm:realize}, a particular euclidean 
realization can be obtained as follows. 
Let $U^c_t = e^{itH}$  be a unitary one-parameter group on $\cH$ and  
$J$ be a unitary involution on $\cH$ with $J H J = -H$. 
Then $\Delta := e^{-\beta H}$ satisfies $J\Delta J = \Delta^{-1}$, so that 
$V := \Fix(J \Delta^{1/2})$ is a standard subspace 
and Theorem~\ref{thm:kms} leads to a positive definite 
function $\psi \: \R \to \Bil(V)$ satisfying the $\beta$-KMS condition. 
Now Theorem~\ref{thm:5.2.3} yields  a reflection positive function 
on $\R_\tau$, resp., $\T_{2\beta,\tau}$, which provides a euclidean 
realization of $U^c$. 

\section{Realization by resolvents of the Laplacian}

Before we describe a realization of the GNS 
representation  $(U^f, \cH_f)$ in spaces of sections of a 
vector bundle, let us recall the general background for this. 

\begin{rem}
 For a $B(V)$-valued positive definite function 
$f \: G \to B(V)$, the reproducing kernel Hilbert space $\cH_f = \cH_K$ 
with kernel $K(g,h) = \phi(gh^{-1}) = K_g K_h^*$ is generated by the functions 
\[ K_{h,w} := K_h^* w \quad \mbox{ with } \quad 
K_{h,w}(g) = K_g K_h^* w = K(g,h) w = \phi(gh^{-1})w.\] 
The group  $G$ acts on this space by right translations 
\[  (U_g s)(h) := s(hg).\] 
If $P \subeq G$ is a subgroup and $(\rho,V)$ is a unitary representation 
of $P$ such that 
\[ f(hg) = \rho(h) f(g) \quad \mbox{ for all } \quad g\in G, h \in P,\] 
then 
\[ \cH_f \subeq \cF(G,V)_\rho := \{ s \: G \to V \:  (\forall g \in G)(\forall h \in P)\, 
s(hg) = \rho(h) s(g)\}.\] 
Therefore $\cH_f$ can be identified with a space of sections of the 
associated vector bundle 
\[ \V := (V \times_P G) = (V \times G)/P,\] where
$P$ acts on the trivial vector bundle $V \times  G$ over $G$ by 
$h.(v,g) = (\rho(h)v, hg)$. 
\end{rem} 

To derive a suitable characterization of the functions 
$f$ arising in Theorem~\ref{thm:5.2.3}, we identify 
$2\beta$-periodic functions $s$ on $\R$ via $s = s_+ + s_-$ 
with pairs of functions $(s_+, s_-)$ satisfying $s_\pm(\beta + t) = \pm s_\pm(t)$. 
Accordingly, any $2\beta$-periodic function 
$s \: \R \to V_\C$ defines a function
\[  \tilde s \: \R \to V_\C^2, \quad \tilde s = (s_+, s_-) 
\quad \mbox{ with }\quad 
\tilde s(\beta + t) = \pmat{\1 & 0 \\ 0 & -\1} \tilde s(t).\] 
In this sense $\tilde s$ is a section of the vector bundle 
over $\T_\beta$ with fiber $V_\C^2$ defined by the representation of 
$\beta\Z$, specified by 
$\rho(\beta) = \pmat{\1 & 0 \\ 0 & -\1}.$
Splitting the $B(V_\C)$-valued positive definite function 
\[f \: \R_\tau \to B(V_\C), \qquad 
f(t,\tau^\eps) =u^+(t) + u^-(t) (iI)^\eps 
\quad \mbox{ for }\quad t \in \R, \eps \in \{0,1\}\] 
as in Theorem~\ref{thm:5.2.3} 
into even and odd part with respect to the $\beta$-translation, we obtain 
the following lemma which shows in particular that
we may identify the Hilbert space $\cH_f \cong \cH_{f^\sharp}$ 
 as a space of section of a Hilbert bundle $V_\C^2 \times_\rho \R_\tau$ 
over the circle $\T_\beta$ with fiber~$V^2$. 

\begin{lem} \label{lem:indrep} For the 
subgroup $P := (\Z\beta)_\tau \cong \Z\beta \rtimes \{e,\tau\}$ of 
$G := \R_\tau$, we consider the unitary representation $\rho \: P \to \U(V_\C^2)$ defined by 
\[ \rho(\beta,e) := \pmat{ \1 & 0 \\ 0 & -\1} \quad \mbox{ and } \quad 
\rho(0,\tau) := \pmat{ \1 & 0 \\ 0 & i I},\] 
where $I$ is the complex structure from the polar decomposition $C = I |C|$ 
on the real Hilbert space $V$. 
Then 
\[ f^\sharp \: \R_\tau \to B(V_\C^2) \cong M_2(B(V_\C)), \qquad 
f^\sharp(t, \tau^\eps) := \pmat{ u^+(t) & 0 \\ 0 & u^-(t) (iI)^\eps} \] 
is a positive definite function satisfying 
\begin{equation}
  \label{eq:covar-f}
f^\sharp(hg) = \rho(h) f^\sharp(g) \quad \mbox{ for } \quad 
h \in P, g \in G.
\end{equation}
The corresponding GNS representation $(U^{f^\sharp}, \cH_{f^\sharp})$ 
is equivalent  to the GNS representation $(U^f, \cH_f)$. 
\end{lem}

\begin{prf} The first assertion follows from 
  \begin{align*}
f^\sharp((0,\tau)(t,\tau^\eps)) %= f^\sharp(-t, \tau^{\eps +1}) 
&= \pmat{u^+ (-t) & 0 \\ 0 & u^- (-t) (iI)^{\eps + 1}} 
= \pmat{u^+ (t) & 0 \\ 0 & u^- (t) (iI)^{\eps + 1}} 
  \end{align*}
and 
\[ f^\sharp(\beta + t, \tau^\eps) 
= \pmat{u^+ (t) & 0 \\ 0 & - u^- (t) (iI)^{\eps}}.\] 
As the GNS representation $(U^f, \cH_f)$ decomposes under the unitary involution 
$U^f_\beta$ into the $\pm 1$-eigenspaces, it is equivalent to 
the GNS representation $(U^{f^\sharp}, \cH_{f^\sharp})$ corresponding to~$f^\sharp$. 
\end{prf}

\begin{rem} (a) In view of \eqref{eq:measrel}, there exists a 
$\Bil^+(V)$-valued measure $\nu$ on $[0,\infty)$ 
for which we can write 
$d\mu(\lambda) = d\nu(\lambda) + e^{\beta\lambda} d\oline{\nu}(-\lambda).$ 
For $\nu = \nu_1 + i \nu_2$, this leads for $0 \leq t \leq \beta$ to  
\begin{equation}
  \label{eq:phi-intrep}
\phi(t) 
= \int_0^\infty e^{-t\lambda} + e^{-(\beta - t)\lambda}\, d\nu_1(\lambda) 
 + i \int_0^\infty e^{-t\lambda} - e^{-(\beta - t)\lambda}\, d\nu_2(\lambda). 
\end{equation}

In particular, the most basic examples correspond to 
Dirac measures of the form $\nu 
= \delta_\lambda \cdot (\gamma+ i \omega)$, where $\delta_\lambda$ is the Dirac measure 
in $\lambda > 0$: 
\[ \phi(t) = (e^{-t\lambda} + e^{-(\beta - t)\lambda})\gamma 
+ i (e^{-t\lambda} - e^{-(\beta - t)\lambda})\omega
= e^{-t\lambda} h  + e^{-(\beta - t)\lambda} \oline h,\] 
where $h := \gamma + i\omega \in \Bil^+(V).$ 

Writing $\omega(\xi,\eta) = \gamma(\xi,C\eta)$ (Lemma~\ref{lem:5.2.1}) 
and replacing $V$ by the real Hilbert space 
defined by the positive semidefinite form $\gamma$ on $V$, we obtain 
the $B(V_\C)$-valued function 
\[ \tilde\phi(t) 
= (e^{-t\lambda} + e^{-(\beta - t)\lambda})\1
+ (e^{-t\lambda} - e^{-(\beta - t)\lambda})i C
= e^{-t\lambda}(\1 + i C) + e^{-(\beta - t)\lambda}(\1 - i C) \] 
for $0 \leq t \leq \beta.$ This leads to 
\[f(t,\tau^\eps) 
=(1 + e^{-\beta\lambda}) (u^+_\lambda(t) \1 + u^-_\lambda(t) |C| (iI)^\eps)
\quad \mbox{ for }\quad t \in \R, \eps \in \{0,1\},\] 
where 
\[ u^\pm_\lambda(t) = \frac{e^{-t\lambda} \pm e^{-(\beta -t)\lambda}}{1 + e^{-\beta \lambda}}
\quad \mbox{ for } \quad 
0 \leq t \leq \beta, \qquad 
u^\pm_\lambda(t + \beta) = \pm u^\pm_\lambda(t).\] 

(b) This can also be formulated in terms of forms. With 
$\gamma(\xi,\eta) = \la \xi,\eta\ra_V$ and 
\[ h(\xi,\eta) = \gamma(\xi,\eta) + i \omega(\xi,\eta) 
= \la \xi,(\1 + i C)\eta\ra_{V_\C} = \la \xi,(\1 + i I |C|)\eta\ra_{V_\C},\] 
we get 
$f(t,\tau^\eps)(\xi,\eta) = (1 + e^{-\beta\lambda}) 
\la \xi, \big(u^+_\lambda(t) \1 + u^-_\lambda(t) 
|C| (iI)^\eps\big)\eta\ra.$ 
\end{rem}

We have seen above how to obtain a realization of the Hilbert space $\cH_{f}$  
as a space $\cH_{f^\sharp}$ 
of sections of a Hilbert bundle $\bV$ with fiber $V_\C^2$ over 
the circle $\T_\beta = \R/\beta \Z$. In this section we provide an 
analytic description of the scalar product on this space 
if $|C| = \mu \1$, $0 < \mu < 1$, so that 
$\frac{\1 + |C|}{\1 - |C|} = e^\lambda \1$ 
for $\lambda := \log\big(\frac{1 + \mu}{1-\mu}\big) > 0$. 
We shall see that it has a natural description 
in terms of the resolvent $(\lambda^2 - \Delta)^{-1}$ of the 
Laplacian $\Delta$ of $\T_\beta$ acting on section of the bundle~$\bV$. 

As in Lemma~\ref{lem:indrep}, we write 
\[ f^\sharp(t,\tau^\eps) = \pmat{ u^+_\lambda(t) \1  & 0 \\ 0 & 
u^-_\lambda(t) (iI)^\eps} \in B(V_\C^2) \cong M_2(B(V_\C)),\] 
For $\chi_n(t) = e^{\pi i n t/\beta}$ we then have 
$u_\lambda^+ = \sum_{n \in \Z} c_{2n}^\lambda \chi_{2n}$ and 
$u_\lambda^- = \sum_{n \in \Z} c_{2n+1}^\lambda \chi_{2n+1},$ 
where 
\[ c_{n}^\lambda = c_{-n}^\lambda 
=  \frac{1 - (-1)^n e^{-\beta \lambda}}{1 + e^{-\beta\lambda}}
\cdot \frac{2\beta\lambda}{(\beta \lambda)^2 + (n\pi)^2} 
=  \frac{1 - (-1)^n e^{-\beta \lambda}}{1 + e^{-\beta\lambda}}
\cdot \frac{2\lambda}{\beta} \cdot \frac{1} 
{\lambda^2 + (n\pi/\beta)^2} \] 
for $n \in \Z$ (the rightmost factors are called bosonic 
Matsubara coefficients if $n$ is even and fermionic if $n$ is odd 
\cite[\S 18]{DG13}). 
With 
\begin{equation}
  \label{eq:cn-form2}
c^\lambda_+ := \frac{1 - e^{-\beta \lambda}}{1 + e^{-\beta\lambda}} 
\frac{2\lambda}{\beta} 
= \tanh\Big(\frac{\beta\lambda}{2}\Big) \frac{2\lambda}{\beta} 
\quad \mbox{ and }\quad 
c^\lambda_- := \frac{2\lambda}{\beta}, 
\end{equation}
we thus obtain 
\begin{equation}
  \label{eq:cn-form1}
c_{2n}^\lambda = \frac{c^\lambda_+}{\lambda^2 + (2n\pi/\beta)^2}, \qquad 
c_{2n+1}^\lambda = \frac{c^\lambda_-}{\lambda^2 + ((2n+1)\pi/\beta)^2}. 
\end{equation}

The following proposition shows that the 
positive  operator $(\lambda^2 \1- \Delta_\R)^{-1}$  
on the Hilbert space of $L^2$-section of $\bV$ 
defines a unitary representation 
of $\R_\tau$ which is unitarily equivalent to the 
representation on $\cH_f$ (cf.~Lemma~\ref{lem:indrep}). 

\begin{prop} \label{prop:5.3.5} 
For $\lambda > 0$, let $\cH_\lambda$ be the Hilbert space obtained by completing the space 
\[ \Gamma_\rho := \{ s \in C^\infty(\R_\tau,V_\C^2) \: (\forall g \in 
\R_\tau, h \in (\Z \beta)_\tau)\ s(hg) =\rho(h)s(g)\} \] 
with respect to
\[ \la s_1, s_2 \ra := \frac{1}{2\beta}\int_0^{2\beta} 
\la s_1(t,e), ((\lambda^2\1 - \Delta_\R)^{-1} s_2)(t,e) \ra \, dt,  
\quad\mbox{ where } \quad 
\Delta_\R = \frac{d^2}{dt^2}. \] 
On $\cH_\lambda$ we have a natural unitary representation $U^\lambda$ 
of $\R_\tau$ by right translation which is unitarily equivalent to 
the GNS representation $(U^{f^\sharp}, \cH_{f^\sharp})$. Here the corresponding 
inclusion map  is given by 
\begin{equation}
  \label{eq:jmap}
 j \: V \to \cH_\lambda, \quad 
j\pmat{v_1 \\ v_2} = 
\sqrt{c^\lambda_+} \sum_{n \in\Z} \chi_{2n} \pmat{v_1 \\ 0} 
+ \sqrt{c^\lambda_-} \sum_{n \in\Z} \chi_{2n+1} \pmat{0 \\ v_2}. 
\end{equation}
\end{prop}

This result provides a natural euclidean realization of our representation 
on the Riemannian manifold $\T_\beta \cong \bS^1$ in the spirit 
of Theorem~\ref{thm:2.1.11}. For more recent work in this direction 
see \cite{NO17} and \cite{FNO18}. 

\begin{rem} In the context of anti-unitary representations, it is interesting 
to observe that the reflection positive representation of 
$\R_\tau$, resp., $\T_{2\beta,\tau}$ described in Proposition~\ref{prop:5.3.5} 
carries a natural anti-unitary involution given by 
\[(Js)(t,\tau^\eps) := \oline{s\big({\textstyle\frac{\beta}{2}}-t,\tau^\eps\big)} 
\quad \mbox{ for } \quad 
t\in \R, \eps \in \{0,1\}.\] 
In fact, one readily verifies that $J$ defines an anti-unitary 
involution on $\cH_\lambda$. We further have 
$JU_\tau J = U_\tau$ and $J U_t J = U_{-t}$ for $t \in \R.$ 
\end{rem}

\section*{Notes on Chapter~\ref{ch:5}} 

The material in this chapter mainly draws from \cite{NO16} 
which continued the investigations from \cite{NO15b} 
only dealing with real-valued functions~$\phi$.
This  was motivated by the work of Klein and Landau in \cite{KL81}.  
A long term goal is to combine our representation theoretic 
approach to reflection positivity 
with KMS states of operator algebras and Borchers 
triples corresponding to modular inclusions 
(\cite{NO17}, \cite{BLS11}, \cite{Bo92}, \cite{Lo08}, \cite{Sch99}). 

We have seen that the unitary one-parameter groups $(U^c,\hat\cE)$ 
arising from reflection positivity on $\T_{2\beta}$ always commute with an 
anti-unitary involution. It would be nice to incorporate anti-unitary operators 
such as conjugations and anti-conjugations more systematically into the whole 
setup of the OS transform on the level of representations. 
This requires a better understanding of the role of anti-unitary operators 
on the euclidean side. Some first steps to a more systematic understanding 
of anti-unitary representations have been undertaken in \cite{NO17} and 
\cite{Ne18}, but this 
has not yet been connected to reflection positivity.

For KMS states of the CCR (canonical commutation relations), resp.\ 
the corresponding Weyl algebra, 
we refer to the two papers of B.\ S.\ Kay \cite{Ka85, Ka85b}, 
dealing with uniqueness of KMS states for a given one-parameter 
group of symmetries and the embedding of KMS representations 
into irreducible ones by a doubling procedure 
(see also \cite{BR96} for a more direct but less conceptual 
approach to the uniqueness of KMS states). 

Interesting references for the relation of the KMS condition 
with (quantum) statistical mechanics are \cite{Fro11} and \cite{BR96}.

\end{bibunit}

\chapter{Integration of Lie algebra representations} 
\label{ch:6} 

\begin{bibunit}
A central problem in the context of reflection positive 
representations of a symmetric Lie group $(G,\tau)$ on a 
reflection positive Hilbert space $(\cE,\cE_+,\theta)$ 
is to construct on the associated Hilbert space $\hat\cE$ 
a unitary representations of the $1$-connected Lie group $G^c$ with 
Lie algebra $\g^c = \fh + i \fq$. As we have seen in 
Remark~\ref{rem:3.3.8}, the main point is to ``integrate'' a unitary 
representation of the Lie algebra $\g^c$ on a pre-Hilbert space. 
In general this problem need not have a solution, but we shall 
see below that in the reflection positive contexts, where the 
Hilbert spaces are mostly constructed from $G$-invariant positive definite kernels 
or positive definite $G$-invariant distributions, there are natural 
assumptions that apply in all cases that we consider. 

For any reflection positive representation of $(G,\tau)$, 
we immediately obtain a unitary representation of the subgroup 
$H = G^\tau_0$ on $\hat\cE$, so that we have to find a unitary 
representation on the one-parameter group $\exp_{G^c}(\R i y)$ for $y\in \fq$. 
Since we have already a symmetric operator $\hat{\dd U}(x)$ on a dense subspace 
of $\hat\cE$, the essential point is to show that it is essentially 
selfadjoint. 

In Section~\ref{sec:6.1} we introduce 
Fr\"ohlich's Theorem which provides a criterion for the essential selfadjointness 
of a symmetric operator. 
In Section~\ref{sec:6.2} we connect this tool with 
the geometric context, where we consider a pair 
$(\beta,\sigma)$ of a homomorphism  $\beta \: \g \to \cV(M)$ 
to the Lie algebra of smooth vector fields on a manifold $M$ 
which is compatible with a smooth $H$-action~$\sigma$. 
For any smooth kernel $K$ on $M$ satisfying a suitable invariance condition 
with respect to $(\beta,\sigma)$, a unitary representation of 
$G^c$ on $\cH_K$ exists (Theorem~\ref{thm:4.8}). 
In Section~\ref{sec:6.3} we show that this result remains valid 
if we replace the kernel $K$ by a 
positive definite distribution $K\in C^{-\infty}(M\times M)$ 
compatible with $(\beta,\sigma)$ 
(Theorem \ref{thm:4.12}). We finally explain in Section~\ref{sec:6.4} 
how these results apply to reflection positive representations. \\

Throughout this section $M$ denotes a smooth manifold modeled on a Banach 
space, if not stated otherwise, 
and $\cV(M)$ denotes the Lie algebra of smooth vector 
fields on $M$.

\section{A geometric version of Fr\"ohlich's Selfadjointness Theorem} 
\label{sec:6.1} 

We start with Fr\"ohlich's Theorem on unbounded symmetric semigroups
as it is stated in \cite[Cor.~1.2]{Fro80} (see also \cite{MN12}).
Actually Fr\"ohlich assumes that the Hilbert space $\cH$ is separable, 
but this is not necessary. Replacing the assumption 
of weak measurability by weak continuity,  
all arguments in \cite{Fro80} work for non-separable spaces as well. 

\index{Theorem!Fr\"ohlich's Selfadjointness}
\begin{thm}{\rm(Fr\"ohlich's Selfadjointness Theorem)} 
\label{thm:2.4b} 
Let $H$ be a symmetric operator defined on the dense subspace $\cD$ of 
the Hilbert space $\cH$. Suppose that, for every $\xi\in\cD$, there
exists an $\eps_\xi>0$ and a differentiable curve 
$\phi \: (0,\eps_\xi) \to \cD$ satisfying 
\[ \phi'(t)=H\phi(t) \quad \mbox{ and }\quad 
\lim_{t\to 0}\phi(t)=\xi.\] 
Then the operator $H$ is essentially 
selfadjoint and $\phi(t) = e^{t\oline H}\xi$ in the sense of 
spectral calculus of selfadjoint operators.
\end{thm}

For later applications, we explain how Fr\"ohlich's Theorem 
applies to linear vector fields on locally convex spaces. 
Let $V$ be a locally convex space
and the kernel ${K \: V \times V \to \C}$ be 
a continuous positive semidefinite hermitian form. 
Then the corresponding reproducing kernel space $\cH_K$ can be identified 
with a linear subspace of the space $V^\sharp$ of 
anti-linear continuous functionals on~$V$ 
(cf.~Section~\ref{sec:a}). It is generated by the functionals 
$K_w(v) := K(v,w), w \in V,$ satisfying 
\[ \la K_v, K_w \ra = K_w(v) = K(v,w).\] 
So it can also be interpreted as the completion of $V$ with 
respect to the hermitian form~$K$. 

The continuity of the kernel $K$ implies that the linear map 
$V\rightarrow \cH_K,\ v\mapsto K_v$ is continuous. 
For any continuous linear operator $L \: V \to V$, the formula
\[ L^K \: \cD_L \to \cH_K, \quad L^K\lambda:=-\lambda\circ L, \quad 
\cD_L:=\{\lambda\in \cH_K \subeq V^\sharp\:  \lambda\circ L \in\cH_K\} \] 
defines an unbounded closed operator on $\cH_K$. 
If there exists an operator $L^*:V\rightarrow V$ with
\[  K(v,Lw)= K(L^*v,w) \quad\text{for}\quad v,w\in V, \] 
then we have 
\begin{equation}\label{E:8}
L^K K_v=K_{-L^*v}\quad\text{for}\quad v\in V.
\end{equation}

We can now obtain from Theorem~\ref{thm:2.4b} 
(\cite[Cor.~4.9]{MNO15}): 

\begin{cor} \label{T:3.7}
Let $L \: V \to V$ be a continuous linear operator 
on the locally convex space $V$ which is $K$-symmetric in the sense that 
$K(Lv,w) = K(v,Lw)$ for $v,w \in V$. 
Suppose that, for every $v\in V$, there exists a  
curve $\gamma_v \: [0,\eps_v] \to V$ starting in $v$ 
and satisfying the differential equation 
\[ \gamma_v\!'(t) = L \gamma_v(t).\] 
Then the restriction $L^K|_{\cH_K^0}$ to the dense subspace 
$\cH_K^0 = \{ K_v \: v \in V \} \subeq \cH_K$ is essentially selfadjoint 
with closure~$L^K$. For $0 \leq t \leq \eps_v$, we have 
$e^{-t L^K} K_v = K_{\gamma_v(t)}.$ 
\end{cor}

Now we turn to the nonlinear setting 
of smooth positive definite kernels on manifolds. 
Here symmetric operators are obtained from smooth vector fields. 

\begin{defn} Let $K \in C^\infty(M \times M,\C)$ be a smooth positive definite 
kernel and $\cH_K \subeq C^\infty(M)$ be the corresponding 
reproducing kernel Hilbert space. 

(a) For a smooth vector field $X \in \cV(M)$, we write 
\[ \cL_X \: C^\infty(M) \to C^\infty(M), \quad 
(\cL_X f)(m) := \dd f(m) X(m) \] 
for the {\it Lie derivative on smooth functions}. \index{Lie derivative}
We thus obtain on the reproducing kernel space $\cH_K$ the unbounded operator 
\begin{equation}
  \label{eq:lxk}
\cL_X^K := \cL_X\res_{\cD_X} \: \cD_X \to \cH_K,\quad \mbox{ where } \qquad 
\cD_X := \{ \phi \in \cH_K \: \cL_X \phi \in \cH_K\}.
\end{equation}

(b) A vector field $X \in \cV(M)$ is said to be 
{\it $K$-symmetric ($K$-skew-symmetric)} if 
\index{vector field!$K$-symmetric}
\index{vector field!$K$-skew-symmetric}
\[ \cL_X^1 K = \eps\cL_X^2 K \quad \mbox{ for } \quad 
\eps = 1, \quad \mbox{ resp., } \quad \eps = -1.\] 
Here the superscripts indicate whether the Lie derivative acts on the 
first or the second argument of~$K$. 
\end{defn}

The following theorem can be obtained quite directly from 
Fr\"ohlich's Theorem if the Hilbert space under consideration 
has a smooth positive definite kernel 
(\cite[Thm.~4.6]{MNO15}): 

\index{Theorem!Geometric Fr\"ohlich}
\begin{thm} {\rm(Geometric Fr\"ohlich Theorem)} 
\label{T:Froelichgeo}
Let $M$ be a smooth manifold and $K$ be a smooth positive definite 
kernel. If $X$ is a $K$-symmetric vector 
field on $M$, 
then $\cL_X\res_{\cH_K^0}$ is an essentially selfadjoint 
operator on $\cH_K$ whose closure coincides with the operator 
$\cL_X^K$. 
For $m \in M$ and an integral curve $\gamma_m \: [0,\eps_m] \to M$ 
of $X$ with $\gamma_m(0)=m$, we have 
$e^{t\cL_X^K} K_m = K_{\gamma_m(t)}$ 
for $0 \leq t \leq \eps_m$. 
\end{thm}

\section{Integrability for reproducing kernel spaces} 
\label{sec:6.2} 

We now turn to actions of a symmetric Lie algebra $(\g,\tau)$ 
on a smooth manifold $M$ that are compatible with a smooth 
positive definite kernel $K$. 
Our first main result is Theorem~\ref{thm:4.8} 
which provides a sufficient condition for the Lie algebra 
representation of the dual Lie algebra 
$\g^c$ coming from an action of $\g$ on $\cH_K$ by Lie derivatives 
to integrate to a unitary representation of the corresponding simply connected Lie group~$G^c$.
Applying this result to open subsemigroups of Lie groups, we further obtain 
an interesting generalization of the L\"uscher--Mack Theorem \cite{LM75, HN93} 
for semigroups which no longer requires the existence of a 
polar decomposition.

\begin{defn} Let $(\g,\tau)$ be a symmetric Lie algebra, and let
$\beta \: \g \to \cV(M)$ be a homomorphism. 
A smooth positive definite kernel  
$K \in C^\infty(M \times M,\C)$ is said to be 
\index{distribution kernel!$\beta$-compatible} {\it $\beta$-compatible} if the vector fields 
in $\beta(\fh)$ are $K$-skew-symmetric and the vector fields 
in $\beta(\fq)$ are $K$-symmetric. 
\end{defn}

\begin{defn}
  \label{def:3.1} 
Let $H$ be a connected Lie group with Lie algebra $\fh$. \index{smooth right action}
A {\it smooth right action} of the pair $(\fg,H)$ on 
$M$ is a pair $(\beta,\sigma)$, where 
\begin{enumerate}
\item[(a)] $\sigma \: M \times H \to M, (m,h) \mapsto \sigma_h(m) = m.h$ is a 
smooth right action, 
\item[(b)] $\beta \: \g \to \cV(M)$ is a homomorphism 
of Lie algebras, and 
\item[(c)] $\dd\sigma(x) = \beta(x)$ for $x \in \fh$. 
\end{enumerate}
\end{defn} 

In the following $K$ denotes a smooth $\beta$-compatible positive definite 
kernel on $M \times M$. 
For $x \in \g$, we abbreviate $\cL_x := \cL_{\beta(x)}^K$ for 
the maximal restriction 
of the Lie derivatives to $\cD_x:=\cD_{\beta(x)}$ 
from \eqref{eq:lxk} and we extend this definition in a complex linear fashion 
to $\fg_\C$. We also consider the subspace
\[ \cD := \{ \phi \in \cH_K \: (\forall n \in \N)
(\forall x_1,\ldots, x_n \in \fg)\, \cL_{\beta(x_1)}\cdots 
\cL_{\beta(x_n)}\phi \in \cH_K\}\]
on which 
$$\alpha:\fg_\C \rightarrow \End(\cD),\ x\mapsto \cL_x|_\cD$$ 
defines a Lie algebra representation such that  $\fg^c$ acts by skew-symmetric operators.
The following theorem (\cite[Thm.~5.12]{MNO15}) 
asserts the integrability of $\alpha|_{\g^c}$. 
Besides the usual technicalities, a key point in its proof is to 
apply the Geometric Fr\"ohlich Theorem~\ref{T:Froelichgeo} to the 
vector fields $\beta(y), y \in \fq$.   

\begin{thm} \label{thm:4.8} Let $K$ be a smooth positive definite kernel 
on the manifold $M$ compatible with the smooth right action $(\beta,\sigma)$ of $(\g,H)$, 
where $\g = \fh \oplus \fq$ is a symmetric 
Lie algebra and $H$ is a connected Lie group with Lie algebra~$\fh$. 
Let $G^c$ be a simply connected Lie group with Lie algebra 
$\g^c = \fh + i\fq$. 
Then there exists a unique continuous unitary representation $(U^c,\cH_K)$ 
such that 
\begin{itemize}
\item[\rm(i)]\  $\oline{\dd U^c(x)} = \cL_x^K$ for $x \in \fh$. 
\item[\rm(ii)]\ $\oline{\dd U^c(iy)} = i\cL_y^K$ for $y \in \fq$. 
\end{itemize}
\end{thm}

\begin{rem} Note that (i) implies that the restriction of 
$U^c$ to the integral subgroup $\la \exp \fh \ra \subeq G^c$ 
induces the same representation 
on the universal covering group $\tilde H$ of~$H$ 
as the unitary representation $(U^H_h f)(m) := f(m.h)$ of $H$ on $\cH_K$ 
because their derived representations coincide (cf.~Chapter~\ref{ch:7}). 
\end{rem}

\begin{exs} \label{ex:3.15b} Let 
$(G,\tau)$ be a symmetric Lie group with Lie algebra 
$\g = \fh + \fq$ and let $H = G^\tau_0$ be the integral subgroup corresponding to the 
Lie subalgebra $\fh = \g^\tau$. 
Further, let $U = UH \subeq G$ be an open subset. 
A smooth function $\phi \: U\tau(U)^{-1} \to \C$ is called 
{\it $\tau$-positive definite} if the kernel \index{function!$\tau$-positive definite}
\[ K(x,y) := \phi(x\tau(y)^{-1}) \] 
is positive definite. 

Then $\sigma_h(g) := gh$ and $\beta(x)(g) := g.x$ define a smooth 
right action of $(\g,H)$ on the manifold $U$ and the kernel $K$ is 
$\beta$-compatible. We therefore obtain for a $1$-connected 
Lie group $G^c$ a corresponding unitary representation $U^c$
on $\cH_K \subeq C^\infty(U,\C)$ with 
\[  (U^c(h)\psi)(g) = \psi(gh) \quad \mbox{ for } \quad 
g\in U, h \in H\]  
and 
\[  \dd U^c(iy)\psi = i\cL_y \psi \quad \mbox{ for } \quad 
\psi\in\cH_K^\infty, y \in \fq.\] 
\end{exs}

So far we worked with scalar-valued kernels, but the corresponding 
results easily extend to the operator-valued setting as follows: 

\begin{ex} (Passage to operator-valued kernels) 
Let $(G,\tau)$ be a symmetric Lie group 
% and $S \subeq G$ be an open subsemigroup 
%invariant under multiplication with $H$ and the involution 
and $g^\sharp = \tau(g)^{-1}$. 
We consider a smooth right action of $G$ on the 
manifold $X$, a complex Hilbert space $V$, and 
a hermitian kernel $Q \: X \times X \to B(V)$. Further, suppose that 
$J \: G \times X \to \GL(V), (g,x) \mapsto J_g(x)$ 
satisfies the cocycle condition 
\[ J_{g_1 g_2}(x) = J_{g_1}(x) J_{g_2}(x.g_1) \quad \mbox{ for } \quad g_1, g_2 \in G, x \in X, \] 
so that $(g.f)(x) := J_g(x) f(x.g)$ defines a representation of 
$G$ on $V^X$. We also assume that 
the kernel satisfies the corresponding invariance relation 
\[ J_g(x) Q(x.g, y) = Q(x, y.g^\sharp) J_{g^\sharp}(y)^* 
\quad \mbox{ for } \quad x,y \in X, g \in G\]  
(cf.~\cite[Prop.~II.4.3]{Ne00}). 
On the set $M := X \times V$, we then obtain a $G$-right action by 
\[ (x,v).g := (x.g, J_g(x)^*v).\] 
We also obtain a positive definite kernel 
\[ K \: M \times M \to \C, \quad 
K((x,v), (y,w)) := \la v, Q(x,y)w  \ra\]  
which satisfies the natural covariance condition 
\begin{align*}
K((x,v).g, (y,w)) 
&= K((x.g, J_g(x)^*v), (y,w)) 
= \la J_g(x)^* v, Q(x.g, y) w  \ra\\ 
&= \la v, J_g(x) Q(x.g, y) w \ra 
= \la v, Q(x, y.g^\sharp) J_{g^\sharp}(y)^* w  \ra \\
& = K((x,v), (y.g^\sharp, J_{g^\sharp}(y)^*w))
= K((x,v), (y,w).g^\sharp).
\end{align*}

Let $X_+ \subeq X$ be an open $H$-invariant subset on which the kernel 
$Q$ is positive definite, so that $K$ is positive definite on 
$M_+ := X_+ \times V$. 
The corresponding reproducing kernel Hilbert space 
$\cH_K \subeq \C^{M_+}$ consists of functions that are continuous anti-linear 
in the second argument, and it is easy to see that the map 
\[  \Gamma \: \cH_Q \to \cH_K,\quad 
\Gamma(f)(x,v) := \la v, f(x) \ra \] 
is unitary (Example~\ref{ex:a.2}). In view of 
\[ \Gamma(g.f)(x,v) = \la v, J_g(x) f(x.g)  \ra 
= \la J_g(x)^* v, f(x.g)  \ra 
= \Gamma(f)((x,v).g), \] 
it intertwines the representation of $G$ on $V^X$ with the action on $\C^M$ by 
\[ (g.F)(x,v) := F((x,v).g).\]

Assume that the $G$-action on the Banach manifold $M$ is smooth, i.e., that 
the map 
$G \times X \times V \to V, (g,x,v) \mapsto J_g(x)^*v$ is smooth. 
Then we obtain a smooth right action of $(\g,H)$ on $M_+$ compatible 
with the kernel $K$, and thus 
Theorem~\ref{thm:4.8} yields a unitary representation 
of $G^c$ on the Hilbert $\cH_K \cong \cH_Q$. 
\end{ex}

\section{Representations on spaces of distributions} 
\label{sec:6.3} 

Now we slightly change our context. 
To extend the theory from smooth kernels to distribution 
kernels, we assume that 
 $M$ is a finite dimensional smooth manifold and that 
$K \in C^{-\infty}(M \times M)$ is a positive definite 
distribution so that $\cH_K \subeq C^{-\infty}(M)$ holds for the corresponding reproducing 
kernel Hilbert space (Section~\ref{subsec:2.1.3}). The canonical map 
\[ \iota \: C^\infty_c(M) \to \cH_K, \quad \phi \mapsto K_\phi \] 
is continuous (\cite[\S 7.1]{MNO15}) 
%, \begin{footnote}
%  {To see that $\iota$ is continuous, it suffices to verify this claim 
%on the subspace $C^\infty_C(M)$ of test functions supported in a 
%compact subset $C \subeq M$. 
%On this Fr\'echet space, the map 
%$C^\infty_C(M) \to C^\infty_{C \times C}(M \times M), \phi \mapsto 
%\phi \otimes \phi$ is easily seen to be continuous. 
%Therefore the continuity of $K$ implies that 
%$\phi  \mapsto \la \phi, \phi\ra_K$ is bounded on some $0$-neighborhood, 
%and this means that $\iota\res_{C^\infty_C(M)}$ is continuous. }
%\end{footnote}
and therefore the kernel $K$ defines a continuous hermitian 
form on $C^\infty_c(M)$. Hence 
Corollary~\ref{T:3.7} applies 
in particular to $K$. 

\begin{defn} The Lie derivative defines on  $C^\infty_c(M)$ the structure of a 
$\cV(M)$-module, and we consider on $C^{-\infty}(M)$ the adjoint representation: 
\[ (\cL_X D)(\phi) := - D(\cL_X \phi) \quad \mbox{ for } \quad X \in \cV(M), D \in C^{-\infty}(M), 
\phi \in C^\infty_c(M).\] 

For a distribution $D \in C^{-\infty}(M \times M)$
and $X \in \cV(M)$, we write 
\[ (\cL_X^1 D)(\phi \otimes \psi) := 
- D(\cL_X \phi \otimes \psi) \quad \mbox{ and } \quad 
(\cL_X^2 D)(\phi \otimes \psi) := 
- D( \phi \otimes \cL_X\psi).\] 
We say that $X$ is {\it $D$-symmetric, resp., $D$-skew-symmetric} 
\index{vector field!$D$-symmetric}
\index{vector field!$D$-skew-symmetric}
if $\cL^1_XD=\eps\cL^2_XD$ for $\eps =1$, resp., $-1$. 
\end{defn}

\begin{rem}\label{R:sym} Let $K$ be a positive definite distribution on $M$.
If $X$ is $K$-symmetric (resp. $K$-skew-symmetric), then $\cL_X$ defines 
a symmetric (resp. skew-symmetric) operator on $C^\infty_c(M)$ 
with respect to $\la\cdot,\cdot\ra_K$.
\end{rem}

The next observation allows us to use Corollary~\ref{T:3.7} and to adapt the methods 
used in Section~\ref{sec:6.2}. 
\begin{rem} Let $X\in\cV(M)$ and 
$\varphi\in C_c^\infty(M)$. We write $M_t \subeq M$ for the open subset 
of all points $m \in M$ for which the integral curve of $X$ through $m$ 
is defined in $t \in \R$. The corresponding time $t$ flow map is denoted 
$\Phi^X_t \: M_t \to M$. 
If $\supp\varphi\subeq M_{-t}$,
then $\varphi\circ \Phi_t^X$ has compact support $\Phi_{-t}^X(\supp\varphi)\subeq M_{t}$
and therefore defines an element of $C_c^\infty(M)$. 
\end{rem}

\index{Theorem!Geometric Fr\"ohlich for distributions}
\begin{thm} {\rm(Geometric Fr\"ohlich Theorem for distributions)} 
\label{T:Froelich-dist}
Let $M$ be a smooth manifold and $K \in C^{-\infty}(M \times M)$ be a positive definite 
distribution. If 
$X \in \cV(M)$ is $K$-symmetric, 
then the Lie derivative $\cL_X$ defines an essentially selfadjoint 
operator $\cH_K^0 \to \cH_K$ whose closure 
$\cL_X^K$ coincides with $\cL_X|_{\cD_X}$, where 
\[ \cD_X := \{D \in \cH_K \: \cL_XD \in \cH_K\}. \] 
If the local flow $\Phi^X$ is defined on 
$[0,\varepsilon] \times \supp(\phi)$ for some $\phi \in C^\infty_c(M)$, then 
\begin{equation}
  \label{eq:6.3.1}
 e^{t\cL_X^K} K_\phi = K_{\phi \circ \Phi^X_{-t}}\quad\text{for}\quad 0\leq t\leq \varepsilon.
\end{equation}
\end{thm}

\begin{prf} For every $\varphi \in C^\infty_c(M)$, 
there exists an $\eps > 0$ such that the flow $\Phi^X$ of $X$ is defined on 
the compact subset $[0,\eps] \times \supp(\varphi)$ of 
$\R \times M$. Then the curve 
\[ \gamma \: [0,\eps] \to C^\infty_c(M), \quad \gamma(t) 
:= \varphi \circ \Phi_{-t}^X \] 
satisfies $\gamma'(t)=-\cL_X\varphi$ in the natural topology 
on~$C^\infty_c(M)$. 
Therefore the assumptions of Corollary~\ref{T:3.7} are satisfied with $V=C_c^\infty(M)$ and $L=-\cL_X$. 
We conclude that $\cL_X|_{\cH_K^0}$ is essentially selfadjoint with closure 
equal to $\cL_X^K$ and that \eqref{eq:6.3.1} holds.
\end{prf}

\begin{defn} 
Let $\g = \fh + \fq$ be a symmetric Lie algebra with involution 
$\tau$ and $\beta \: \g \to \cV(M)$ be a homomorphism of Lie algebras. 
A positive definite distribution 
$K \in C^{-\infty}(M \times M,\C)$ is said to be 
{\it $\beta$-compatible} if  \index{distribution kernel!$\beta$-compatible}
\[ \cL^1_{\beta(x)} K = -\cL^2_{\beta(\tau(x))} K \quad \mbox{ for } \quad 
x \in \g.\] 
\end{defn}

In the following we assume that $K$ is a positive definite 
distribution on $M$ compatible with the smooth right 
action $(\beta,\sigma)$ of $(\g,H)$ (cf.\ Definition~\ref{def:3.1}).
For $z = x + i y \in\g_\C$, we put 
\[ \cL_{\beta(z)}:=\cL_{\beta(x)}+i\cL_{\beta(y)} \] 
and we write $\cL_z$ for the restriction of $\cL_{\beta(z)}$ to
its maximal domain 
$$\cD_z=\{D\in\cH_K\mid \cL_{\beta(z)}D\in\cH_K\}.$$
As in Section~\ref{sec:6.2}, we consider the subspace 
\[ \cD := \{ D \in \cH_K \: (\forall n \in \N)
(\forall x_1,\ldots, x_n \in \g)\, \cL_{\beta(x_1)}\cdots 
\cL_{\beta(x_n)}D \in \cH_K\} \]  
which carries the Lie algebra representation 
$\alpha \: \g_\C \to \End(\cD)$ 
for which the operator 
$\alpha(x)$ is skew-hermitian for 
$x \in \g^c = \fh + i \fq$. 
{}From $\eqref{E:8}$ and Remark~\ref{R:sym} we deduce that 
\begin{equation}\label{E:infaction}
\cL_x K_\varphi=K_{\cL_{\tau(x)}\varphi} 
\quad \mbox{ for } \quad \phi \in C^\infty_c(M),
\end{equation}
hence
$\cH_K^0 \subeq \cD$. In particular, $\cD$ is dense in $\cH_K$.

\begin{thm} \label{thm:4.12} 
Let $K \in C^{-\infty}(M \times M)$ be a positive definite distribution 
compatible with the smooth right action 
$(\beta,\sigma)$ of the pair $(\g,H)$ on $M$, where 
$\g = \fh \oplus \fq$ is a symmetric 
Lie algebra and $H$ is a connected Lie group with Lie algebra $\fh$. 
Let $G^c$ be a simply connected Lie group with Lie algebra 
$\g^c = \fh + i\fq$. 
Then there exists a unique smooth unitary representation $(U^c,\cH_K)$ of $G^c$ 
such that 
\[ \oline{\dd U^c(x)} = \cL_x \quad \mbox{ and } \quad 
\oline{\dd U^c(iy)} = i\cL_y \qquad \mbox{ for } \quad 
x \in \fh, y \in \fq.\]
\end{thm}

\section{Reflection positive distributions and representations}
\label{sec:6.4} 

In this subsection we connect the previously described integrability results  
to reflection positive representations. Let $D \in C^{-\infty}(M \times M,\C)$ be a 
positive definite distribution which is reflection positive 
with respect to the involution $\theta \: M\to M$ on the open subset $M_+ \subeq M$ 
(cf.\ Definition~\ref{def:twistdist}).  
Our main result is Theorem~\ref{thm:4.12b} which asserts that, 
under the natural compatibility requirements for an action of a 
symmetric Lie group $(G,H,\tau)$ on $(M,\theta)$, the representation 
of the pair $(\g^c,H)$ on the Hilbert space $\cH_{D^\theta}$ corresponding to the 
positive definite distribution kernel $D^\theta$ on $M_+$ integrates to a 
unitary representation of the simply connected group $G^c$ with Lie algebra~$\g^c$.\\

Let $(G,H,\tau )$ be a symmetric Lie group acting on $M$ such that
$\theta (g.m)=\tau(g).\theta (m)$ and $H.M_+=M_+$. We 
assume   that $D$ is invariant under $G$ and $\tau$. 
Then we have a natural unitary representation $(U_\cE, \cE)$ of $G$ on 
the Hilbert subspace $\cE := \cH_K \subeq C^{-\infty}(M)$.  
As $M_+$, and therefore $\cE_+$, is $H$-invariant, this representation is 
infinitesimally reflection positive in the sense of 
Definition~\ref{def:repo-rep3}. 

{}From the invariance condition 
\begin{equation}
  \label{eq:invcon2b}
\cL_{\beta(x)}^1 D = - \cL_{\beta(x)}^2 D \quad \mbox{ for} \quad 
x \in \g 
\end{equation}
we derive 
\begin{eqnarray} 
  \label{eq:invcon3}
\cL_{\beta(x)}^1 D^\theta = - \cL_{\beta(\tau(x))}^2 D^\theta \quad \mbox{ for} \quad 
x \in \g.\end{eqnarray} 

This implies that the assumptions of Theorem~\ref{thm:4.12} 
are satisfied, so that we obtain: 

\begin{thm}\label{thm:4.12b} Let $M$ be a smooth finite dimensional manifold and 
\break $D \in C^{-\infty}(M \times M)$ be a positive definite distribution which is 
reflection positive with respect to $(M,M_+, \theta)$. 
Let $(G,H,\tau )$ be a symmetric Lie group acting on $M$ such that
$\theta (g.m)=\tau(g).\theta (m)$ and $H.M_+=M_+$. We 
assume   that $D$ is invariant under $G$ and~$\tau$. 
Let $G^c$ be a simply connected Lie group with Lie algebra 
$\g^c = \fh + i\fq$ and define 
$(\cL_x)_{x \in \g}$ on its maximal domain in the Hilbert subspace 
$\cH_{D^\theta} \subeq C^{-\infty}(M_+)$. 
Then there exists a unique unitary representation $(U^c,\cH_{D^\theta})$ 
of $G^c$ 
such that 
\[ \oline{\dd U^c(x)}= \cL_x \quad \mbox{ and } \quad 
\oline{\dd U^c(iy)} = i\cL_y \quad \mbox{ for } \quad 
x \in \fh, y \in \fq.\]
\end{thm}

\begin{ex} \label{ex:7.anaext} 
The preceding theorem applies in particular 
to the situation where $M = G_\tau$, $\tau(g) = g\tau$ and 
$M_+ = G_+$ is an open subset of $G$ with $G_+H = G_+$ 
(Remark~\ref{rem:2.4.2}). Here we start with a reflection 
positive distribution $D \in C^{-\infty}(G_\tau)$ 
(Definition~\ref{def:repo-distr}). It 
defines a $G_\tau$-invariant distribution kernel 
$\tilde D$ on $G_\tau \times G_\tau$ 
which is reflection positive with respect to $G_+$. 
We thus obtain a positive definite distribution 
$\tilde D^\tau$ on $G_+ \times G_+$. 
\end{ex}

\begin{ex} Reflection positive representations of the euclidean motion 
group $E(d)$ (cf.~Example~\ref{ex:eucmot})  
lead to unitary representations of the simply 
connected covering $G^c  = \R^d \rtimes \Spin_{1,d-1}(\R)$ 
of the identity component 
$P(d)_0$ of the Poincar\'e group. 
More concretely, we consider $M = \R^d$, $M_+ = \R^d_+$, 
$\tau(x_0, \bx) = (-x_0, \bx)$ and 
$G = E(d) = \R^d \rtimes \OO_d(\R)$. 
Then the  $G$-invariance of a distribution $\tilde D$ on $\R^d \times \R^d$ means that 
it is determined by an $\OO_d(\R)$-invariant distribution 
$D \in C^{-\infty}(\R^d)$ by 
\[ \tilde D(\phi \otimes \psi) := D(\phi^\vee * \psi), \qquad 
\phi^\vee(x) := \phi(-x).\] 
For any reflection positive rotation invariant distribution $D \in C^{-\infty}(\R^d)$, 
we thus obtain a reflection positive representation 
$(U_\cE, \cE)$ of $G$ and a representation 
of the group $G^c$ on $\hat\cE \cong \cH_{D^\theta}$. 

For $d \geq 3$, the natural 
inclusion $\SO_{d-1}(\R) \to \OO_{1,d-1}(\R), g \mapsto \id_\R \times g$ 
induces a surjective homomorphism $\pi_1(\SO_{d-1}(\R)) \to \pi_1(\OO_{1,d-1}(\R))$, 
and since $U^c$ is compatible with the unitary representation 
$\hat U^H$ of $H$ on $\hat\cE$, the representation $U^c$ factors through a 
representation of the connected Poincar\'e group $P(d)_0 = \R^d \rtimes \SO_{1,d-1}(\R)$.

The concrete case of generalized 
free fields discussed in Chapter~\ref{ch:8} 
is of basic interest. 
\end{ex}

\section*{Notes on Chapter~\ref{ch:6}} 

Main reference for this section is \cite{MNO15}, where the results 
on smooth kernels are developed in the more general context 
of Banach--Lie groups acting on locally convex manifolds. 
Here we choose the simpler context of Banach manifolds 
because in this context every smooth vector field has a local flow. 

Fr\"ohlich's results from \cite{Fro80} have later been 
refined in several ways, in particular by 
Klein and Landau in \cite{KL81, KL82}. 
Fr\"ohlich, Osterwalder and Seiler introduced in \cite{FOS83}
the concept of a virtual representation, 
which was developed in greater generality  by Jorgensen in \cite{Jo86, Jo87}. 

In the context of involutive representations of a subsemigroup $S\subeq G$ 
with polar decomposition $S=H\exp C$, where $C\not=\eset$ is 
an $\Ad(H)$-invariant open convex cone in $\fq$, 
the L\"uscher--Mack Theorem \cite{LM75, HN93, MN12} provides a corresponding 
unitary representation of the dual group $G^c$.

\end{bibunit}

\chapter{Reflection positive distribution vectors} 
\label{ch:7} 

\begin{bibunit}[abbnamed]

In this chapter we first introduce the 
concept of a distribution vector of a unitary representation 
(Section~\ref{sec:7.1}). 
It turns out that certain distribution vectors semi-invariant 
under a subgroup $H$ correspond naturally to realizations of 
the representation in a Hilbert space of distributions on 
the homogeneous space $G/H$. 
In this context reflection positive representations 
can be constructed from reflection positive distributions 
on $G/H$ (Section~\ref{sec:7.2}). 
Such distributions can often be found and even classified 
in terms of the geometry of the homogeneous space. 

To illustrate this technique, we apply it in Section~\ref{sec:7.3} to spherical 
representations of the Lorentz group $G = \OO_{1,n}(\R)$. 
These representations consist of two series, the principal series and 
the complementary series. Both have natural realizations in 
spaces of distributions on the $n$-sphere $\bS^n \cong G/P$ on which the Lorentz 
group $G$ acts by conformal maps; the principal series can even be realized 
in~$L^2(\bS^n)$. That some of the representation of the complementary 
series exhibit reflection positivity with respect to the 
subsemigroup of conformal compressions of a half-sphere 
is shown in Subsection~\ref{subsec:7.3.4}. In Subsection~\ref{sec:7.3.5} 
we build a bridge from the natural reflection positivity 
on the sphere $\bS^n$ as a Riemannian manifold obtained from 
resolvents $(m^2-\Delta)^{-1}$ of the Laplacian 
(cf.~Section~\ref{subsec:2.1.4}) and unitary representations. 
Here the Lorentz group occurs as the dual group $G^c$ of the isometry 
group $G = \OO_{n+1}(\R)$ of $\bS^n$ and we identify the unitary representations 
of $G^c$ on the corresponding Hilbert spaces $\hat\cE$ as 
spherical representations of~$G^c$ realized in a space of holomorphic 
functions in the crown domain of hyperbolic space.

\section{Distribution vectors} 
\label{sec:7.1}

In this subsection we introduce the notion of distribution vectors 
in the following section we connect it with
reflection positivity. We start with the basic structures 
related to distributions on Lie groups and homogeneous spaces. 

\subsection{Distributions on Lie groups and homogeneous spaces} 

\begin{defn} Let $G$ be a Lie group. 
We fix a left invariant Haar measure $\mu_G$ on $G$. 
This measure defines on $L^1(G) := L^1(G,\mu_G)$ the structure of a 
Banach-$*$-algebra by the {\it convolution product} and  \index{convolution product} 
\begin{equation}
  \label{eq:convol}
(\phi*\psi)(u) =\int_G \phi(g)\psi(g^{-1}u)\, d\mu_G (g),  
\quad \mbox{ and } \quad \phi^*(g) = \overline{\phi (g^{-1})}\Delta_G (g)^{-1}
\end{equation}
is the involution, where 
$\Delta_G : G\to \R_+$ is the {\it modular function} determined by \index{modular 
function}
\begin{align*}
  \label{eq:modfunc}
 \int_G \phi(y)\, d\mu_G(y) 
&=\int_G \phi(y^{-1})\Delta_G(y)^{-1}\, d\mu_G(y) \quad \mbox{ and } \\ 
\Delta_G(x)\int_G \phi(yx)\, d\mu_G(y)&=\int_G \phi(y)\, d\mu_G(y) \quad 
\mbox{ for } \quad \phi\in C_c (G).
\end{align*}
%We put 
%\begin{equation}
%  \label{eq:check}
%\varphi^\vee(g)=\varphi (g^{-1})\quad \mbox{ and } \quad \varphi^* 
%=\overline{\varphi^\vee} \cdot\Delta_G^{-1}.
%\end{equation}

The formulas above show that we have two isometric actions of $G$ on 
$L^1(G)$, given by 
\begin{equation}
  \label{eq:left-right-action}
(\lambda_g f)(x) = f(g^{-1}x) \quad \mbox{ and }\quad 
(\rho_g f)(x) = f(xg) \Delta_G(g).
\end{equation}
Note that $(\lambda_g f)^* = \rho_g f^*$. 
\end{defn}

Let $H \subeq G$ be a closed subgroup and $X := G/H = \{gH \: g \in G\}$ 
be the space of $H$-left cosets, endowed with its canonical manifold 
structure. 
Let $\mu_H$ denote a left Haar measure on $H$. 
Then the map  
\[ \alpha \: C_c^\infty (G)\to C_c^\infty (X), \ \varphi \mapsto \varphi^\flat, \qquad 
\varphi^\flat(gH)
:=\int_H \varphi(gh)\, d \mu_H(h)\] 
is a topological quotient map, i.e., surjective, continuous and open 
(cf.\ \cite[p.~475]{Wa72} and \cite[p.~136]{vD09}). Its adjoint thus provides an 
injection 
\[ \alpha^* \: C^{-\infty}(X) \into C^{-\infty}(G),\ 
D \mapsto D^\sharp,  \qquad 
D^\sharp(\phi) := D(\phi^\flat)\] 
of the space of distributions on $X = G/H$ into the space of 
distributions on~$G$. 

On $C^\infty_c(X)$ the group $G$ acts naturally 
by left translations $(\lambda_g \phi)(x) := \phi(g^{-1}x)$ and, accordingly, 
by $(\lambda_g D)(\phi) := D(\lambda_g^{-1} \phi)$ on distributions.  
We also recall the two $G$-actions 
\eqref{eq:left-right-action} 
on $C_c^\infty(G) \subeq L^1(G)$ by left and right translations 
and note that they induce actions on the dual space $C^{-\infty}(G)$.
As 
\begin{equation}
  \label{eq:alpha-equiv}
\alpha \circ \lambda_g = \lambda_g \circ \alpha \quad \mbox{ and } \quad 
\alpha \circ \rho_h = \Delta_G(h)\Delta_H(h)^{-1} \alpha 
\quad \mbox{ for } \quad 
h \in H, g \in G,
\end{equation}
the map $\alpha$ and its adjoint intertwine the left translation 
actions of $G$ on $C^{-\infty}(X)$ and $C^{-\infty}(G)$. It also follows that 
$\rho_h\alpha^*(D)(\phi)
= D(\alpha(\rho_h^{-1}\phi)) 
= \frac{\Delta_H(h)}{\Delta_G(h)}\alpha^*(D)(\phi)$, 
and we even have: 

\begin{lem}
The range of $\alpha^*$ is the space 
$C^{-\infty}(G)_H$ of all distributions $D \in C^{-\infty}(G)$ satisfying 
\begin{equation}
  \label{eq:H-covar-dist}
\rho_h D = \Delta_{G/H}(h)^{-1} D 
\quad \mbox{ with } \quad 
\Delta_{G/H}(h) := \frac{\Delta_G(h)}{\Delta_H(h)} 
\quad \mbox{ for }\quad h \in H. 
\end{equation}
\end{lem}

\begin{prf} 
Let $D \in C^{-\infty}(G)_H$. As $\alpha$ is a quotient map, we have to show that 
$\ker \alpha \subeq \ker D$ because $\im(\alpha)^* = (\ker \alpha)^\bot$. 
First we note that, for $\psi \in C^\infty_c(G)$, the distribution defined by 
$(\psi * D)(\phi) := D(\psi^* * \phi)$ has a smooth density $\Psi$ with respect to 
$\mu_G$ satisfying $\rho_h\Psi = \Delta_{G/H}(h)^{-1}\Psi$ for $h \in H$. 

Let $\rho \: G \to (0, \infty)$ denote a smooth function on
$G$ with 
\[ \rho(e) = 1 \quad \mbox{ and } \quad 
\rho(gh) = \rho(g) \Delta_{G/H}(h)^{-1} \quad \mbox{ for } \quad 
g \in G, h \in H, \] 
and write $\mu_{G/H}$ for the corresponding quasi-invariant measure on $G/H$ 
defined by 
\[ \int_{G/H} \phi^\flat(gH)\, d\mu_{G/H}(gH) = \int_G \phi(g)\rho(g)\, d\mu_G(g) \] 
(\cite[p.~475]{Wa72}). That $\mu_{G/H}$ is well defined requires to verify that 
$\phi^\flat = 0$ implies that the right hand side vanishes. Now 
\begin{align*}
(\psi * D)(\phi) 
&= D(\psi^* * \phi) = \int_G \oline{\phi(g)} \Psi(g)\, d\mu_G(g) 
= \int_{G/H} (\oline{\phi} \Psi \rho^{-1})^\flat(gH)\, d\mu_{G/H}(gH) \\ 
&= \int_{G/H} \oline{\phi}^\flat(gH) (\Psi \rho^{-1})(g)\, d\mu_{G/H}(gH) = 0. 
\end{align*}
Replacing $\psi$ by a $\delta$-sequence in $C^\infty_c(G)$, we 
obtain for $n \to \infty$ that $D(\phi) = 0$. 
\end{prf}

The distribution $D_{\mu_G}(\phi) := \int_G \oline{\phi(g)}\, d\mu_G(g)$ 
is left and right invariant, hence contained in 
$C^{-\infty}(G)_H$ if and only if $\Delta_{G/H} =1$. If this is the case,
then $D_{\mu_G} = D_{\mu_X}^\sharp$ for a $G$-invariant measure 
$\mu_X$ on $X$. One can even show that, conversely, 
the existence of such a measure implies the vanishing of $\Delta_{G/H}$ 
%(cf.~\cite[Sec. 2.6]{Fo95}, \cite[p.146ff]{Fa00}). 
(cf.~\cite{Wa72} or \cite[\S 10.4]{HN12}).

\subsection{Smooth vectors and distribution vectors} 

Now let $(U,\cH)$ be a unitary representation of the Lie group~$G$, 
i.e., a homomorphism $U : G\to \rU(\cH), g \mapsto U_g$ 
such that for each $\eta \in\cH$ the  orbit map
$U^\eta (g)=U_g\eta$ is continuous. 
We say that $\eta\in\cH$ is {\it smooth} if  \index{smooth vector}
$U^\eta : G\to \cH$ is smooth. The space of smooth vectors is denoted by
$\cH^{\infty}$. This space carries a representation $\dd U $ of the Lie algebra 
$\fg$ on $\cH^\infty$ given by
\[\dd U(x)\eta =\lim_{t\to 0}\frac{U_{\exp t x}\eta -\eta}{t}\, .\] 
For a basis $x_1,\ldots ,x_k$ of $\fg$ and 
$\bm= (m_1,\ldots ,m_k)\in \N_0^k$ the family of semi-norms 
\[q_{\bm}(\eta )=\|\dd U(x_1)^{m_1}\cdots \dd U (x_k)^{m_k}\eta\|\]
defines a Fr\'echet space topology on $\cH^\infty$ such that the inclusion $\cH^\infty\hookrightarrow \cH$ is continuous. 
The space $\cH^\infty$ is $G$-invariant and 
$\dd U(\Ad (g)x)= U_g \dd U(x) U_g^{-1}$.% for $g \in G, x \in \g.$ 
%\begin{equation}
%  \label{eq:ad-der-inter}
%\dd U(\Ad (g)x)= U_g \dd U(x) U_g^{-1} \quad \mbox{ for } \quad 
%g \in G, x \in \g. 
%\end{equation}

For $\phi\in L^1(G)$ the operator-valued integral 
$U_\phi := \int_G \phi(g) U_g\, d\mu_G(g)$ exists and is uniquely determined by 
\begin{equation}
  \label{eq:l1-est}
\la \eta, U_\phi \zeta \ra =\int_G \phi(g)\ip{\eta}{U_g\zeta}\, d\mu_G(g) \quad 
\mbox{ for } \quad \eta,\zeta \in \cH.
\end{equation}
Then $\|U_\phi\|\le \|\phi\|_1$ and the 
so-obtained continuous linear map $L^1(G) \to B(\cH)$ 
is a representation of the Banach-$*$-algebra $L^1(G)$, i.e., 
$U_{\phi*\psi}=U_\phi U_\psi$ and $U_{\phi^*}=U_\phi^*.$ 
Note that $U_g U_\phi = U_{ \lambda_g\phi}$ and $U_\phi U_g = U_{\rho_g \phi}$.

The space $\cH^\infty$ of smooth vectors is $G$-invariant 
and we denote the corresponding representation by~$U^\infty$.
If $\varphi \in C_c^\infty (G)$ and $\xi \in\cH$, then 
$U_\varphi \xi \in \cH^\infty$ and 
\[ \dd U(x) U_\varphi \xi :=U_{\dd\lambda_x \varphi } \xi,  
\quad \mbox{ where } \quad 
\dd\lambda_x\varphi =\frac{d}{dt}\Big|_{t=0} \lambda_{\exp tx}\phi.\] 
This follows directly by differentiation under the integral sign. 
If $(\varphi_n)_{n \in \N}$ 
is a \break {$\delta$-sequence}, then $U_{\varphi_n}\xi \to \xi$, so that $\cH^\infty$
is dense in~$\cH$.
%Let $\varphi_j\ge 0$, $\int \varphi_j = 1$ be a delta sequence in $C^\infty_c(G)$ with $\supp \varphi_j \subset U_j$,
%$U_{j+1}\subset U_{j}$ and 
%$\bigcap U_j =\{e\}$. For   $\xi \in \cH$ we have
%\[ \| U_{\varphi_j}\xi - \xi\|\le \int_{U_j } \varphi_j(x)\| U_x\xi - \xi\| dx \le \sup_{x\in U_j}\|U_x\xi - \xi\| \]
%As the right hand side can be made arbitrary small, we see that   $\cH^\infty$ is dense in $\cH$.

The space of continuous anti-linear functionals on $\cH^\infty$ is 
denoted by $\cH^{-\infty}$. Its elements are called 
\index{distribution vector} \textit{distribution vectors}. 
The group $G$ and its Lie algebra $\g$ act on $\cH^{-\infty}$ by
\[(U^{-\infty}_g\eta ) (\xi ):= \eta  (U_{g^{-1}}^\infty\xi ), 
 \quad \text{resp., } \quad
(\dd U^{-\infty}(x) \eta ) (\xi ):= -\eta(\dd U(x) \xi), \, g\in G, x\in \fg .\]
We then have $U^{-\infty}_\varphi \eta :=\eta \circ U^\infty_{\varphi^*}$ for 
$\varphi \in C_c^\infty (G)$. 
%For an involutive automorphism $\tau$ of $G$, we define  $(U_\tau^{-\infty} \eta)(\xi )
%= \theta \eta (\xi) :=\eta (\theta \xi)$. If $\eta$ is $\tau $-invariant, then
%$\theta (U^{-\infty}_\varphi \eta )= U^{-\infty}_{\varphi\circ \tau}\eta$.
We obtain natural $G$-equivariant linear embeddings %with dense images
\[\cH^\infty \mapright{\iota_{\infty}(\xi )=\xi} \cH
\mapright{\iota_{-\infty}(\xi)= \ip{\cdot}{\xi}} \cH^{-\infty} \] 
and note that $U^{-\infty}_\phi \cH^{-\infty} \subeq \cH^\infty$ for 
$\phi \in C^\infty_c(G)$. 

\begin{ex}
Let $H \subeq G$ be a closed subgroup and $X = G/H$. 
Then there exists a quasi-invariant measure $\mu_X$ on $X$ with a 
smooth density with respect to Lebesgue measure in any chart; 
for details see \cite[Sec. 2.6]{Fo95},  \cite[p.146ff]{Fa00} and 
\cite[\S 10.4]{HN12}. Thus
there exists a smooth strictly positive function $j: G\times X\to \R_+$ such that 
for all $\varphi\in C_c^\infty(X)$ and $g\in G$ we have
$\int_X \varphi (g.x)\, d\mu_X(x)=\int_X\varphi(x) j(g,x)\, d\mu_X (x)$ 
or, equivalently,
\begin{equation}\label{eq:QuasiInt}
\int_X \varphi(g^{-1}.x)j(g,x)\, d\mu_X(x)
=\int_X \varphi(x)\, d\mu_X(x)\quad \mbox{ for } \quad \varphi \in C_c^\infty (X).
\end{equation}
We will also write $j_g(x)=j(g,x)$. 
As a consequence of \eqref{eq:QuasiInt}, we obtain a family of 
unitary representation, the {\it quasi-regular representation} of $G$ on $L^2(X)$ by 
\begin{equation}\label{eq:defUl} \index{quasi-regular representation, on $L^2(G/H)$}
U^\lambda_g\varphi(x):= j(g,x)^{\lambda + 1/2} \phi(g^{-1}.x)\quad \mbox{ for } 
\quad \lambda \in i\R, g\in G\text{ and } \varphi \in L^2(X).
\end{equation}
\end{ex}

If $(U,\cH)$ is a unitary representation of $G$ and 
$\eta\in \cH^{-\infty}$, then 
the adjoint of the linear map $C^\infty_c(G)\to \cH^\infty, \varphi 
\mapsto U_{\varphi}^{-\infty} \eta$, defines a linear map
\[ j_\eta \: \cH^{-\infty} \to C^{-\infty}(G), \quad 
j_\eta(\alpha)(\phi) := \alpha(U^{-\infty}_{\varphi} \eta).\] 
For $g \in G$, we have 
$j_\eta  \circ U^{-\infty}_g = \lambda_g \circ j_\eta$, 
%\begin{eqnarray}\label{eq:jInter}
%j_\eta (U^{-\infty}_g \alpha)(\varphi ) 
%&=&\alpha (U^{\infty}_{g^{-1}}U_{\varphi}^{-\infty}\eta)
%=\alpha({U^{-\infty}_{\lambda_g^{-1}\varphi}}\eta)
%= \lambda_g(j_\eta(\alpha))(\phi), 
%\end{eqnarray}
i.e., $j_\eta$ intertwines the action of $G$ on $\cH^{-\infty}$ with 
the left translation action on $C^{-\infty}(G)$. 
%In particular, if $H\subset G$ is a closed subgroup and $\alpha$ is 
%$H$-invariant, then the distribution $j_\eta(\alpha)$ is right $H$-invariant.

We now introduce the concepts used in the proposition below. 

\begin{defn} We call the distribution vector \index{distribution vector!cyclic}
$\eta \in\cH^{-\infty}$ {\it cyclic} if $U^{-\infty}_{C_c^{\infty} (G)}\eta$ 
is dense in $\cH$.
%, which is equivalent to $j_\eta\: \cH \to \cH_\eta$ 
%being injective (resp.~bijective).
\end{defn}

\begin{defn} We call a distribution $D \in C^{-\infty}(G)$ {\it positive definite} 
if the sesquilinear kernel \index{distribution!positive definite}
\begin{equation}
  \label{eq:posdefdist}
K_D(\phi, \psi) := D(\psi^* * \phi) \quad \mbox{ on } \quad 
C^\infty_c(G)
\end{equation}
is positive semidefinite. This is equivalent to the positive definiteness 
of the distribution $\tilde D$ on $G \times G$ determined by 
\[ \tilde D(\oline\psi \otimes \phi) = D(\psi^* * \phi) 
\quad \mbox{ for }\quad \phi, \psi \in C^\infty_c(G).\] 
We also note that \eqref{eq:posdefdist} is equivalent to $D$ defining 
a positive functional on the $*$-algbra $C^\infty_c(G)$, endowed with the 
convolution product. The corresponding reproducing kernel Hilbert space 
$\cH_D := \cH_{K_D}$ is a linear subspace of $C^{-\infty}(G)$ in which the 
distributions defined by 
$\psi * D = \lambda_\psi(D)$, i.e., 
$(\psi * D)(\phi) := D(\psi^* * \phi)$,  
form a dense subspace with 
\begin{equation}
  \label{eq:pos-def-hilb}
\la \phi * D, \psi * D\ra = D(\psi^* * \phi)\quad \mbox{ for } \quad 
\phi, \psi \in C^\infty_c(G).
\end{equation}
In particular, $D \in \cH_D^{-\infty}$, $j_D(D) = D$ and 
$j_D\res_{\cH_D} \: \cH_D \into C^{-\infty}(G)$ is the inclusion map. 
\end{defn}

%Suppose, conversely, that $D \in C^{-\infty}(X)$ is a distribution on $X 
%= G/H$. For 
%$\varphi \in C^\infty_c(G)$, $\psi \in C_c^\infty (X)$ and   $D\in C^{-\infty}(X)$ define
%Assume that $D$ is $H$-invariant and that $\Delta_G|_H=1$, then
%$ \varphi*D:= \varphi_0*D$ is well-defined 
%for $\varphi \in C_c^\infty (X)$ resulting in  a distribution 
%kernel on $X \times X$, given by
%\[K_D (\varphi,\psi):=D(\psi_0^* * \varphi).\]
%We say that $D$ is {\it positive definite} 
%if $K_D$ defines a positive semidefinite 
%form on $C_c^\infty(X)$. Then the Hilbert space 
%$\cH_D:= \cH_{K_D}\subeq C^{-\infty}(X)$ (Example~\ref{ex:1.4}) 
%can be obtained by completing the space $C^\infty_c(X)*D=C^\infty_c(G)*D$ with
%respect to the inner product
%\[\ip{\varphi*D}{\psi * D}:= K_D(\varphi , \psi).\]
%We then have (see \cite[Prop.~2.8]{NO14} for the case $H = \{e\}$):

\begin{prop}\label{prop:rea1}{\rm (Realization in spaces of distributions)} 
Let $(U,\cH)$ be a unitary representation of
$G$ and $\eta \in \cH^{-\infty}$.
Then the following assertions hold: 
\begin{itemize}
\item[\rm (a)]\  The map $j_\eta \: \cH^{-\infty} \to C^{-\infty}(G)$ 
 is injective if and only if $\eta$ is cyclic.
\item[\rm (b)]\  The distribution $D_\eta :=j_\eta (\eta)$ is positive definite. 
\item[\rm (c)]\  If $\eta$ is cyclic, then 
$j_\eta :\cH \to \cH_\eta:= j_\eta (\cH)\subset C^{-\infty} (G)$ is a 
$G$-invariant unitary operator onto the reproducing kernel Hilbert space
of distribution on $G$ for which the
inner product and the reproducing kernel are determined by
\[ \ip{j_\eta (U_\varphi^{-\infty}\eta)}{j_\eta (U_\psi^{-\infty}\eta)}_{\cH_\eta}
=\ip{U_\varphi^{-\infty}\eta}{U_\psi^{-\infty}\eta}_{\cH}=D_\eta (\psi^* *\varphi). \] 
\item[\rm (d)]\ If $H \subset G$ is a closed subgroup, 
then $j_\eta(\cH^{-\infty}) \subeq C^{-\infty}(G)_H$ 
if and only if $U_h^{-\infty} \eta = \Delta_{G/H}(h)\eta$ holds 
for all $h \in H$. 
\end{itemize}
\end{prop}

\begin{prf} For (a), we first observe that the injectivity of $j_\eta$ on $\cH$ is trivially 
equivalent to $\eta$ being cyclic. To see that this even implies that $j_\eta$ is injective 
on $\cH^{-\infty}$, assume that $j_\eta(\alpha) = 0$. Then equivariance implies 
$j_\eta(U^{-\infty}_\phi \alpha) = 0$ for every $\phi \in C^\infty_c(G)$. 
For any $\delta$-sequence $(\delta_n)_{n \in \N}$ in $C^\infty_c(G)$, we have 
$U^{-\infty}_{\delta_n} \alpha \to \alpha$ in the weak-$*$-topology on $\cH^{-\infty}$, 
and since $j_\eta$ is obviously weak-$*$-continuous, it follows that 
$j_\eta(\alpha) = 0$. 

For (b) we derive from 
$D_\eta(\phi) = \eta(U_{\phi}^{-\infty}\eta)$ the relation 
\begin{align*}
D_\eta(\psi^* * \phi) 
&= \eta(U_{\psi^* * \phi}^{-\infty}\eta) 
= \eta(U_{\psi^*}^{\infty} U_{\phi}^{-\infty}\eta) 
= (U_{\psi}^{-\infty}\eta)(U_{\phi}^{-\infty}\eta)
= \la U_{\phi}^{-\infty}\eta, U_{\psi}^{-\infty}\eta \ra.
 \end{align*}
By Remark~\ref{rem:a.1.2}, this implies that $D_\eta$ is positive definite. 

(c) follows from the fact that, for $\phi, \psi \in C^\infty_c(G)$, 
\[ j_\eta(U^{-\infty}_\psi\eta) (\phi) 
= \la U_\phi^{-\infty}\eta, U^{-\infty}_\psi\eta \ra 
=D_\eta(\psi^* * \phi) = (\psi * D_\eta)(\phi).\] 

To obtain (d), we first observe the relation 
$j_\eta(\alpha)(\phi) = \oline{j_\alpha(\eta)(\phi^*)}$ 
for $\alpha, \eta \in \cH^{-\infty}$,  
which easily follows from the existence of a factorization 
$\phi = \phi_1 * \phi_2$ with $\phi_j \in C^\infty_c(G)$ 
(Dixmier--Malliavin Theorem \cite[Thm.~3.1]{DM78}). 
For $h \in H$, this leads to 
\[ (\rho_h D_\eta)(\phi) 
=  D_\eta(\rho_h^{-1}\phi) 
= \eta(U_\phi^{-\infty} U_{h^{-1}}^{-\infty} \eta) 
= j_{U_{h^{-1}}^{-\infty} \eta}(\eta)(\phi) 
= \oline{j_\eta(U_{h^{-1}}^{-\infty}\eta)(\phi^*)},\] 
so that the assertion follows from \eqref{eq:H-covar-dist}. 
\end{prf}

From the preceding proposition we derive: 
\index{Theorem!Realization in spaces of distributions} 
\begin{thm}\label{thm:rea3}  
A unitary representation $(U,\cH)$ can be realized on a Hilbert
subspace of\  $C^{-\infty}(G/H)$ if and only
if there exists a cyclic distribution vector 
$\eta \in \cH^{-\infty}$ satisfying $U_h^{-\infty}\eta = \Delta_{G/H}(h)\eta$ for $h \in H$. 
%The map $\eta \mapsto j_\eta$ defines a  bijection
%between such cyclic distribution vectors and weak-$*$-continuous 
%$G$-equivariant injections $j : \cH^{-\infty}\to C^{-\infty}(G/H)$. 
\end{thm}

\begin{exs} (a) Let $G$ be a Lie group and $H$ a closed subgroup such that
$X=G/H$ carries a $G$-invariant measure~$\mu_X$. Then
$G$ acts unitarily on $L^2(X) = L^2(X, \mu_X)$ by $\lambda_g\varphi (x)=\varphi (g^{-1}.x)$.
The space $L^2(X)^\infty$ of smooth vectors is the space
of smooth functions $\varphi \in C^\infty(X)$ 
such that $\lambda_u f\in L^2(X)$ for all $u\in U(\fg)$ (\cite[Thm.~5.1]{Po72}). 
If $X$ is compact, then $L^2(X)^{\infty } =C^\infty (X)$ and
$L^2 (X)^{-\infty}=C^{-\infty} (X)$ is the space of distributions on $X$.
%We have $\lambda_\varphi \psi (x)=\int_G \varphi (g)\psi (g^{-1}.x)\, d g 
%:= \varphi * \psi (x)$.

(b) Let $G = \Heis(\R^{2n}) = \T \times \R^n \times \R^n$ 
be the Heisenberg group acting on $L^2(\R^n)$ via the Schr\"odinger
representation 
\[(\pi(z,x,y)f)(u)= z e^{i \la x, u \ra} f(u-y)\]
(cf.~the proof of Theorem~\ref{thm:LP}). 
Then the space $L^2(\R^n)^\infty$ of smooth vectors 
is $\cS (\R^n)$, the Schwartz space of rapidly decreasing functions. 
\end{exs}

For later reference, we record the following lemma 
(\cite[Lemma~D.7]{NO15a}) which identifies the distribution vectors 
in representations of $G = \R^d$ by multiplication operators. 

\begin{lem} \label{lem:dist-vec} Let $(X,\fS,\mu)$ be a measure space. 
We write $M(X,\C)$ for the vector space of measurable 
functions $X \to \C$. For $(H_j)_{j =1,\ldots, d}$ in $M(X,\R)$ 
and $R := \sqrt{\sum_{j = 1}^d H_j^2}$, we consider 
the continuous unitary representation of $\R^d$ on $L^2(X,\mu)$, given by 
\[ U_{\bt}(f) := e^{i \sum_{j = 1}^d t_j H_j} f\quad \mbox{ for } \quad 
\bt = (t_1,\ldots, t_d).\] 
Then 
\[ \cH^{-\infty} \cong \Big\{ h \in M(X,\C) \: (\exists n \in \N) \ \|(1+ R^{2})^{-n} f \|_2 < \infty\Big\},\] 
where the pairing $\cH^{\infty} \times \cH^{-\infty} \to \C$ is given by 
$(f,h) \mapsto \int_X \oline f h\, d\mu$. Moreover, 
the following assertions are equivalent: 
\begin{itemize}
\item[\rm(i)]\ \  The constant function $1$ is a distribution vector. 
\item[\rm(ii)]\ \  For the measurable map 
$\eta := (H_1, \ldots, H_d) \: X \to \R^d$, 
the measure $\eta_*\mu$ on $\R^d$ is tempered. 
\item[\rm(iii)]\ \  $\hat\phi \circ \eta \in L^2(X,\mu)$ for every $\phi \in C^\infty_c(\R^d)$. 
\end{itemize}
If these conditions are satisfied, then the corresponding distribution 
on $\R^d$ is given by the Fourier transform of 
$\eta_*\mu$. 
\end{lem}

\section{Reflection positive distribution vectors}\label{sec:7.2}

In this chapter $(U,\cE)$ will always denote a unitary representation of 
a Lie group $G$ and $H\subeq G$ will be a closed subgroup. 

\begin{defn} \label{def:repo-distr} 
Let $(G,\tau)$ be a symmetric Lie group and  
$G_\tau = G \rtimes \{\id_G,\tau\}$. 
\begin{itemize}
\item[\rm (a)]\ A positive definite $\tau$-invariant distribution 
$D\in C^{-\infty}(G)$ is called {\it reflection positive with respect to $(G,G_+,\tau)$} 
if \index{distribution!reflection positive}\index{reflection positive!distribution} 
\begin{equation}
  \label{eq:repodist}
D(\phi^\sharp * \phi) \geq 0 \quad \mbox{ for } \quad 
\phi \in C^\infty_c(G_+), \phi^\sharp(g) := \phi^*(\tau(g)).
\end{equation}
This is equivalent to the  corresponding distribution 
$\tilde D(\oline\psi \otimes \phi) = D(\psi^* * \phi)$ 
on $G \times G$ being reflection positive with respect to 
$(G,G_+, \tau)$ (cf.~Definition~\ref{def:twistdist}). 
\item[\rm (b)]\ Let $(U,\cE)$ be a unitary representation of $G_\tau$ and  
$\theta := U_\tau$. Then a $\tau$-invariant distribution vector $\eta$ 
is said to be \textit{reflection positive} 
\index{distribution vector!reflection positive}
with respect to $(G,G_+,\tau)$ if 
the subspace $\cE_+:=\lbr U^{-\infty}_{C^\infty_c(G_+)}\eta\rbr$ 
is $\theta$-positive 
(cf.\ Definition \ref{def:2.2.3}). 
 
\item[\rm (c)]\ If $\eta \in \cE^{-\infty}$ is cyclic and reflection positive, 
then we say that $(U,\cE,\eta)$ is a {\it distribution
cyclic reflection positive representation} of $G_\tau$. 
\end{itemize}
\index{representation!distribution cyclic}
\end{defn}

For the special case, where $G_+=S\subeq G$ holds for an 
open $\#$-invariant subsemigroup $S \subeq G$, 
a positive definite distribution $D\in C^{-\infty}(G)$ is 
\textit{reflection positive}
\index{distribution!reflection positive}
if $\tau D=D$ and $D|_S$ is positive definite as a distribution on the involutive semigroup $(S,\#)$, i.e.,
$D(\varphi^\#*\varphi )\ge 0$ for $\phi \in C^\infty_c(S)$ 
and $\phi^\sharp(s) := \phi^*(\tau(s))$. 
For a unitary representation of $G$, a $\tau$-invariant 
distribution vector  $\eta \in \cE^{-\infty}$ is
reflection positive with respect to $S$ if the subspace $\cE_+:=\lbr
U^{-\infty}_{C^\infty_c(S)}\eta\rbr$ is $\theta$-positive 
(cf.\ Definition~\ref{def:1.2b}). 

We now obtain easily: 
\begin{thm}\label{thm:PosdSemi}
For $(G,G_+,\tau)$ as above, the following assertions holds:
\begin{itemize}
\item[\rm (a)]\ If $(U,\cE,\eta)$ is a distribution cyclic reflection 
positive representation of $G_\tau$ with respect to $(G,G_+,\tau)$, then 
$D_\eta$ is reflection positive  with respect to $(G,G_+,\tau)$. 
\item[\rm (b)]\ If $D\in C^{-\infty}(G)$ is reflection positive with respect to 
$(G,G_+,\tau)$, then 
$(U_D,\cH_D, D)$ is a distribution cyclic 
reflection positive representation. If $G_+ = S$ is a \break 
{$\sharp$-invariant} open subsemigroup, 
then we have an $S$-equivariant 
unitary map
\[\Gamma :\hat{\cE} \to \cH_{D|_S}\subset C^{-\infty}(S), \quad \Gamma (\hat{\varphi * D}) 
=\varphi|_{S}. \] 
\end{itemize}
\end{thm}

\begin{proof} For $\phi \in C^\infty_c(G_+)$, we have
\[
D_\eta (\varphi^\sharp *\varphi)=\eta (U^{-\infty}_{\varphi^\sharp *\varphi}\eta)=
\ip{U^{-\infty}_{\tau_*\varphi}\eta}{U^{-\infty}_\varphi \eta}=\ip{\varphi}{\varphi}_\theta \ge 0. \]
The other parts of (a), as well as (b), now follow from Lemma~\ref{lem:2.1.12}. 
\end{proof}

\section{Spherical representation of the Lorentz group}
\label{sec:7.3}

In this, and the following section, 
we discuss reflection positivity related to the
conformal geometry of $\R^n$, resp., of its conformal completion $\bS^n$. 
We first discuss the  complementary series of the conformal group 
$G:=\OO_{1,n+1}(\R)^\uparrow$ and then we turn to 
the reflection positivity arising in 
Riemannian geometry from resolvents of the Laplacian  
on $\bS^n$ as described in Section \ref{subsec:2.1.4}. 

\subsection{The principal series}\label{ssec:Principal}
We
write elements of $\R^{n+2}$ as $(x_{-1},x_0,\bx)$. Correspondingly, elements of
$\R^{n+1}$ are written as $(x_0,\bx)$, and $e_{-1},e_0,e_1,\ldots ,e_n$ denotes 
the standard basis of $\R^{n+2}$. We then identify $e_0,\ldots ,e_n$ with the
standard basis for $\R^{n+1}$.
Elements of $G $ are written as
$g=\begin{pmatrix} a & b^\top\\ c & d\end{pmatrix}$, 
where $a\in \R$, $b,c\in \R^{n+1}$ and $d\in M_{n+1}(\R)$. 
%For $v\not= 0$, write $r_v(x)= x-\frac{2\ip{x}{v}}{\ip{v}{v}}v$ for
%the reflection in  $v^\bot$ and $r_j=r_{e_j}$.
Recall also the notation $[x,y]=x_{-1}y_{-1}-\ip{x}{y}$ 
and consider the set 
\[ \bL^{n+1}_+:=\{x\in \R^{n+2}\: [x,x]=0, x_{-1}>0\} \] 
of positive lightlike vectors. 
%\text{ and }
%\bL^{n+1}_1 :=\{x\in \bL^{n+1}_+
%\: x_{-1}=1\}\simeq \bL^{n+1}_+/\R_+^\times.\] 
The embedding $\xi : \bS^n \to \bL^{n+1}_+, x \mapsto (1,x)$,
yields a diffeomorphism  $\bS^n \to \bL^{n+1}_+/\R_+^\times$. 
As the standard linear action of $G$ on
$\R^{n+2}$ leaves $\bL^{n+1}_+$-invariant, we thus obtain a 
smooth action on the quotient space $\bL^{n+1}_+/\R_+^\times$ and 
hence on the sphere $\bS^n$ via 
\begin{equation}\label{eq:7.1}
g.x:=J(g,x)^{-1}(c + dx) = \xi^{-1}(J(g,x)^{-1}g(\xi (x))) 
\end{equation} 
with
\begin{equation}\label{eq:defJ}
J(g,x) :=a+\ip{b}{x} = (g.\xi(x))_0.
\end{equation}

Let
\[K := \left\{\begin{pmatrix} 1 & 0 \\ 0 & d\end{pmatrix}
\: d\in \OO_{n+1}(\R) \right\}  \cong \OO_{n+1}(\R) \, .\]
Then $K$ is a maximal compact subgroup of $G$ acting transitively on the sphere 
via the standard action
on $\R^{n+1}$, and $\bS^n \cong K/M$ for $M := K_{e_0}\simeq \OO_n(\R)$. Note that $K$ is the stabiliser of $e_{-1}$
in $G$ with respect to the standard linear action.

As a homogeneous space of $G$, the sphere is $G/P$, 
where $P = G_{e_0}$ is the stabilizer of $e_0$.
We have  $P = M AN$, where 
\[A=\left\{a_t=\begin{pmatrix} \cosh (t) & \sinh (t) & 0\\ \sinh (t) & \cosh (t)& 0 \\
0 & 0 & \rI_{n} \end{pmatrix}\: t\in \R\right\}\cong \R\]
and
\[
N:= \left\{ n_v= \begin{pmatrix} 1 +\frac{\|v\|^2}{2} & - \frac{\|v\|^2}{2}& v^\top  \\ \frac{\|v\|^2}{2} & 1-\frac{\|v\|^2}{2}& v^\top\\
v & -v & \rI_{n}\end{pmatrix}\:v\in\R^n\right\}\simeq \R^n.\] 
We define  
\[ J_\lambda (g,v): =J(g,v)^{-\lambda -\frac{n}{2}}, 
\qquad Q(u,v):=1-\ip{u}{v}
\quad \mbox{ and } \quad 
Q_\lambda (u,v):=Q(u,v)^{\lambda -\frac{n}{2}}.\] 

Part (a) of the following lemma is  \cite[Prop. 7.5.8]{vD09}. It also follows from 
\cite[Rem.~5.2]{NO14} by the transformation formula for integrals,  
and the remainder is obtained by direct calculation.

\begin{lem}\label{lem:7.2} Let $g, g_1,g_2\in G$ and
$u,v\in\bS^n$. For the $K$-invariant probability measure $\mu_{\bS^n}$ on $\bS^n$, 
we have 
\begin{itemize}
\item[\rm (a)] % \ \ $j_g(x)=J(g^{-1},x)^{-n}$. In particular
$\int_{\bS^n} \varphi (g. v)J_{\frac{n}{2} }(g,v) \, d \mu_{\bS^n}(v)=
 \int_{\bS^n} \varphi (v)\, d \mu_{\bS^n}(v)$\ \ for\ \  $\phi \in L^1(\bS^n).$
\item[\rm (b)]\ \ $J_\lambda (g_1g_2,v)=J_\lambda (g_1,g_2.v)J_\lambda (g_2,v)$.
\item[\rm (c)]\ \ $Q_\lambda (u,v)=J_{-\lambda }(g,u)Q_\lambda(g.u,g.v)J_{-\lambda}(g,v)$.
\end{itemize}
\end{lem}

\begin{defn}
For every $\lambda \in \C$, we obtain a 
representation of $G$ on $C^\infty (\bS^n)$ by
\begin{equation}
  \label{eq:lambda-rep}
(U^\lambda_g\varphi)(v)=J_\lambda (g^{-1},v) \varphi (g^{-1}.v). 
\end{equation}
We denote by $C^\infty_\lambda$ the space $C^\infty (\bS^n)$ with
the $G$-action given by $U^\lambda$. Similarly, $C^{-\infty}_\lambda$ will
denote the space of distributions with the contragradient action. 
\end{defn} \index{$C^{-\infty}_\lambda$}\index{$C^{\infty}_\lambda$}

From Lemma \ref{lem:7.2} %and  (\ref{def:Ul}) 
we get: 
\begin{lem}\label{lem:7.3} For $\varphi, \psi \in C^\infty (\bS^n)$ and $g\in G$, 
we have 
\begin{equation}\label{eq:7.2}
\ip{U^{-\oline \lambda}_g\varphi}{U^\lambda_g\psi}_{L^2}=\ip{\varphi}{\psi}_{L^2}.
\end{equation}
\begin{itemize}
\item[\rm (a)]\ The representation $U^\lambda$ extends to a unitary representation of $G$ on $L^2(\bS^n)$ if and only
if $\lambda \in i\R$.
\item[\rm (b)]\ The linear map $\psi \mapsto \ip{\cdot }{\psi }_{L^2}$ defines a linear and $G$-equivariant map
from $C_{-\bar \lambda}^\infty$ into $C^{-\infty}_\lambda$.
\end{itemize} 
\end{lem}
 
The following theorem follows  from \cite[Cor. 7.5.12]{vD09}, which is stated 
for the space $C(\bS^n)$ of continuous functions, but the same argument works 
for smooth functions. 
  
\begin{thm}\label{thm:Irred}The representation
$(U^\lambda, C^\infty(\bS^n))$ is irreducible
if $\pm\lambda \not\in \frac{n}{2} + \N_0$. 
In particular, the unitary representation $(U^\lambda,L^2(\bS^n))$ is irreducible
for $\lambda\in i\R$.
\end{thm}

\subsection{The complementary series}
\label{subsec:compser}

In this section we explain how $U^\lambda$ can be made unitary for 
$\lambda  \in (-\frac{n}{2},\frac{n}{2})$. 
As Lemma~\ref{lem:7.3} easily implies that 
$U^\lambda \simeq U^{-\lambda}$ holds for the corresponding 
unitary representations, we shall assume that $\lambda \in (0,\frac{n}{2})$.

Recall that the tangent space at $u\in \bS^n$ is given by
$T_u(\bS^n)\cong u^\bot$ and that the stabilizer of $u$ in $K$ acts by the natural 
linear action on $T_u(\bS^n)$. We also write 
\[ S_u(\bS^n):=\{w\in T_u(\bS^n)\: \|w\|=1\}.\]
The Riemannian exponential map $\Exp_u :T_u(\bS^n)\to \bS^n$ is given by
\begin{equation}\label{eq:exp}
\Exp_u(v)=\cos (\|v\|)u +\frac{\sin (\|v\| )}{\|v\|} v, 
\end{equation}
i.e., for $\|v\|=1$, the geodesic starting in $u$ in the direction of $v$ is given 
\begin{equation}\label{eq:geod}
x_u(t, v) :=\gamma_v(t)=\cos(t)u + \sin(t)v .
\end{equation}
The map $(t, v)\mapsto \gamma_v(t) $, $t\in (0,\pi), v\in S_u(\bS^n)$ 
defines the \textit{polar coordinates} on $\bS^n$. 

For further references we recall the following facts about the Beta and Gamma function. 
The {\it Beta function} is defined by 
\[B(z,w):=\int_0^1 r^{z-1}(1-r)^{w-1}dr
 =\frac{\Gamma (z)\Gamma (w)}{\Gamma (z+w)},\quad \Re z,\Re w>0.\]
 \begin{lem}\label{lem:Beta} For $\Re z, \Re w > 0$, the following 
assertions hold:
 \begin{itemize}
 \item[\rm (a)] \ \ $B(z,w)=\int_0^\infty\frac{t^{z-1}}{(1+t)^{z+w}}dt$.
 \item[\rm (b)]\ \ $ \sqrt{\pi} \Gamma (2z)=2^{2z-1} \Gamma (z)\Gamma (z+1/2)$.
 \item[\rm (c)] \ \  $\int_{-1}^1 (1-r)^{z-1}(1-r^2)^{w-1}dr =2^{2w+z -2}B(w,w+z-1)$. 
%For $z = 1$, we have in particular 
% \[\int_{-1}^1(1-r^2)^{w -1}dr= 
% \sqrt{\pi}
%  \frac{\Gamma\left(w\right)}{\Gamma \left(w +\frac{1}{2}\right)}.\]
  \item[\rm (d)]\ \ The euclidean surface measur of the sphere is 
$\mathrm{Vol}(\bS^{n-1})=2\frac{\pi^{n/2}}{\Gamma (n/2 )}$.
 \item[\rm (e)]\ \ For $\Re \sigma > - n$ and 
 $\Re \mu >\frac{1}{2}\Re (\mu + n)$ we have
 \[\int_{\R^n} (1+\|y\|^2)^{-\mu}\|y\|^{\sigma} dy
 = \pi ^{n/2} \frac{\Gamma ((\sigma +n)/2)\Gamma (\mu -(\sigma+n)/2)}{\Gamma(n/2)\Gamma (\mu)}.\]
 \end{itemize}
 \end{lem}

\begin{proof}  (a) follows  with $r = \frac{t}{1+t}$ 
and (b) can be found in \cite[\S 12.15]{WW63}. 
Formula (c) follows from (b) by the substitution $u=(1+r)/2$, 
(d) is \cite[\S 9.1]{Fa08}, and 
(e) follows from (a) and (d) by using polar  coordinates and substituting  $u=r^2$. 
 \end{proof}

\begin{lem} \label{lem:7.2.4} 
For $c_n:=\frac{\Gamma (\frac{n+1}{2})}{\sqrt{\pi}\Gamma (\frac{n}{2} )}$, we have 
\[\int_{\bS^n} \varphi (u)d\mu_{\bS^n}(u)
=c_n \int_0^{\pi}\int_{S_u} \varphi(x_u(t, v) )\sin^{n-1}(t)\, d\mu_{\bS^{n-1}}(v)\, d t
\quad \mbox{ for } \quad\varphi \in L^1(\bS^n).\]
If  $\varphi$ is $K_u$-invariant, 
then $\tilde\varphi (\cos t):= \varphi (x_u(t, v) )$ is independent
of $v\in S_u$ and
\begin{equation}
  \label{eq:inv-intform}
\int_{\bS^n} \varphi\,  d\mu_{\bS^n}
=   c_n \int_{-1}^1 \tilde\varphi(r)(1-r^2)^{\frac{n}{2} -1}\, dr.
\end{equation}
\end{lem}
\begin{proof} See \cite[Prop. 9.1.2]{Fa08}. The value of the constant follows
from Lemma~\ref{lem:Beta}(c) by taking $\lambda =n/2$ and $\varphi = 1$. 
\end{proof}
 
\begin{lem}  For $\lambda > 0$,
%$d_{\lambda,n}:= 2^{\lambda +\frac{n}{2} -1 }\frac{\Gamma (\frac{n+1}{2})\Gamma (\lambda )}{\sqrt{\pi}
%\Gamma (\lambda +\frac{n}{2} )}.$ 
the kenel $Q_\lambda  $ is integrable as a function
of one or two variables if and only if  $\lambda >0$. In that case we have for all $z\in \bS^n$:
\[
\int_{\bS^n}Q_\lambda (z,y)d\mu_{\bS^n}(y)
%=\int_{\bS^n\times \bS^n}Q_\lambda (x,y)d\mu_{\bS^n}(y)d\mu_{\bS^n}(x)
= \frac{2^{\lambda +\frac{n}{2} -1 }\Gamma (\frac{n+1}{2})\Gamma (\lambda )}{\sqrt{\pi}
\Gamma (\lambda +\frac{n}{2} )} 
=:d_{\lambda,n}
\]
\end{lem} 

\begin{proof} As $Q_\lambda$ and the function $\int_{\bS^n}Q_\lambda (\cdot,y) d\sigma(y)$ are $K$-invariant, we have
\[\int_{\bS^n}Q_\lambda (z,y)\, d\mu_{\bS^n}(y)=\int_{\bS^n} Q(e_0,y)\, d\mu_{\bS^n}(y)=
\int_{\bS^n} \!\! \int_{\bS^n} Q_\lambda (x,y)\, d\mu_{\bS^n}(y)\, d\mu_{\bS^n}(x).\] 
The function $Q_\lambda (e_0,\cdot)$ 
is invariant under $K$. Lemma~\ref{lem:7.2.4} and Lemma \ref{lem:Beta}(c)  imply that
\[
 \int_{\bS^n} Q_\lambda (e_0,y)\, d\mu_{\bS^n}(y) 
= c_n\int_{-1}^1 (1-r)^{\lambda-\frac{n}{2}}(1-r^2)^{\frac{n}{2} -1}\, d r
= \frac{2^{\lambda +\frac{n}{2}-1} \Gamma (\frac{n+1}{2})\Gamma(\lambda)}{\sqrt{\pi}
\Gamma(\lambda + \frac{n}{2})} .
\]
%(\cite[\S 12.41]{WW63}),  
Clearly the integral is finite if and only if $\Re \lambda > 0$. 
%Now the change of variables 
%$u=\frac{1}{2}(1+r)$ implies the assertion.
 \end{proof}
 
%As $n$ is fixed for the moment we write $d_{\lambda}$ for $d_{\lambda,n}$. 
For $ \Re \lambda >0$, define  
 \[ (A_\lambda \varphi)(x):=\frac{1}{d_{\lambda,n}}\int_{\bS^n} Q_\lambda (x,y)  
\varphi (y)\, d\mu_{\bS^n} (y) \quad \mbox{ for } \quad \varphi \in C^\infty (\bS^n).\]
 Then $\ip{\varphi}{A_\lambda\psi}=\ip{A_{\oline \lambda }\varphi}{\psi}$
 and $A_\lambda 1=1$. In particular, if $\lambda$ is real
then  $\ip{\varphi}{A_\lambda \psi}=\ip{A_\lambda \varphi}{\psi}$. Furthermore $A_\lambda : L^2(\bS^n)\to
 L^2(\bS^n) $ is bounded if $\Re \lambda > \frac{n}{2}$ 
because in this case the kernel $Q_\lambda$ is continuous and hence in $L^2$. 
 
\begin{thm}\label{thm:1.5} Let $\varphi\in C^\infty (\bS^n)$. Then the following assertions hold: 
\begin{itemize}
\item[\rm (a)]\ $ A_\lambda U^\lambda_g\varphi=U^{-\lambda }_gA_\lambda \varphi$ 
for $g \in G$. In particular, the form
\[ \ip{\varphi}{\psi}_\lambda := \ip{\varphi}{A_\lambda \psi}_{L^2} \] 
is $U^\lambda_G$-invariant if $\lambda >0$.
\item[\rm (b)]\ The map $\{\lambda\in\C\: \Re \lambda > 0\}\to L^2(\bS^n)$, $\lambda\mapsto A_\lambda \varphi $,
is holomorphic and has a meromorphic extension to all of $\C$. Furthermore, the
intertwining relation in {\rm(a)} holds then for almost all~$\lambda$. 
\end{itemize} 
\end{thm}

\begin{proof} The first part of (a) follows from Lemma~\ref{lem:7.2}(a,c) 
and the second part is a consequence
of Lemma \ref{lem:7.3}. For (b), we refer to \cite[Thm. 1.5]{VW90} 
or \cite[Thm.~9.2.12]{vD09}. A more direct argument can be based 
on $\Delta r^{-\lambda}=\lambda(\lambda +1)r^{-\lambda -2}$ on~$\R^n$.
\end{proof}

We will now determine those $\lambda > 0$ for which the 
form $\ip{\cdot }{\cdot }_\lambda$ is positive semidefinite.  For
that it is easier to work with the realization in a 
space of function on $\R^n$. For that we recall the stereographic projection
\begin{equation}\label{def:7eta}
s : \R^n \to \bS^n\setminus\{-e_0\},\quad x\mapsto \left(\frac{1-\|x\|^2}{1+\|x\|^2}, 
\frac{2x}{1+\|x\|^2} \right)=n_x^\top.e_{0}
\end{equation}
with inverse $s^{-1}(y_0,\mathbf{y})=\frac{1}{1+y_0}\mathbf{y}$.
\begin{lem} \label{lem:7.2.8}
For $\phi \in C^\infty(\bS^n)$ 
put $\varphi_\lambda (x): = \varphi (s(x))(1+\|x\|^2)^{-\lambda - \frac{n}{2}}$. 
For the  positive constants 
\[ a_n :=  \frac{2^{n-1} \Gamma(\frac{n+1}{2})}{\pi^{\frac{n+1}{2}}} 
\quad \mbox{ and } \quad 
b_\lambda 
= \frac{\Gamma(\lambda +n/2)}{\pi^{\frac{n}{2}}\Gamma (\lambda )},\]
%%b_\lambda := \frac{\Gamma (n/2)^22^{\lambda -\frac{n}{2}}}{d_{\lambda,n} \Gamma (n)}
% \mbox{ and } 
%e_\lambda 
%= \frac{2^{\lambda +\frac{n}{2}-1}\Gamma \left(\lambda +\frac{n}{2}\right) \Gamma\big(\frac{%n+1}{2}\big)}
%{\pi^{n+ \frac{1}{2}} \Gamma(\lambda)}\] 
we then have %such that the following assertions hold: 
\begin{itemize}
\item[\rm (a)]\  $\displaystyle \int_{\bS^n} \varphi  (v)\, d \mu_{\bS^n}(v)=   a_n\int_{\R^n}
\varphi (s(x))(1+\|x\|^2)^{-n}\, dx$\ \ for\ \ $\varphi \in L^1(\bS^n)$.
\item[\rm (b)]\ $\displaystyle Q_\lambda (s(x),s(y))=2^{\lambda- \frac{n}{2} }
(1+\|x\|^2)^{-\lambda +\frac{n}{2}}(1+\|y\|^2)^{-\lambda +\frac{n}{2}}\|x-y\|^{2\lambda -n}$.
\item[\rm (c)]\  $\displaystyle (A_\lambda \varphi ) (s(x)) = (1+\|x\|^2)^{-\lambda +n/2} b_{\lambda}  \int_{\R^n} 
\varphi_\lambda (y)\|x-y\|^{2\lambda -n}  d y$
\item[\rm (d)]\ $\displaystyle \ip{\varphi}{\psi}_\lambda = a_n b_{\lambda}  \int_{\R^n\times \R^n} \overline{\varphi_\lambda (x)}
\psi_\lambda (y)   \|x-y\|^{2\lambda -n}\, d x\, d y $ 
for $\phi, \psi \in C^\infty(\bS^n).$ %\setminus \{-e_1\})$ .
\end{itemize}
\end{lem}

\begin{proof} Up to a constant (a) follows from %\cite[Prop.~7.5.8]{vD09} 
\cite[Ex.~9.1]{Fa08}. The exact value of the constant can then be evaluated using Lemma \ref{lem:Beta}.
Parts (b) and (c) follow from (a) and Lemma \ref{lem:7.2}(b,c). Finally 
(d) follows from (c).
\end{proof}

\begin{prop} \label{prop:posdef-powerfct}
The function   $x\mapsto \|x\|^{-s}$ is
locally integrable on $\R^n$ if and only if $s< n$. The corresponding
distribution is positive definite if and only if $0\le s<n$. 
\end{prop}

\begin{proof}  This follows by using polar coordinates and the fact that 
\[\cF (r^{-s})= \pi^{s-n/2} \frac{\Gamma ((n-s)/2)}{\Gamma (s/2)}r^{s-n}\]
(see \cite[Ex.~5, VII.7.13]{Sch73a}). 
The right hand side is positive for $0<s<n$. The case $s=0$ is obvious.
\end{proof}

As a consequence we get the following theorem, up to the non-degeneracy 
of the form: 
\begin{thm}\label{thm:CompSrep} For $\lambda \ge 0$ the 
form $\ip{\cdot}{\cdot }_\lambda$ is positive semi-definite 
on $C^\infty(\bS^n)$ if and only if $0 < \lambda \leq \frac{n}{2}$. 
Let $\cE_\lambda$ denote the corresponding Hilbert space. 
For $\lambda = \frac{n}{2}$ this space is one-dimensional and 
for $0 < \lambda < \frac{n}{2}$ the form is non-degenerate. 
We thus obtain irreducible unitary representations 
$(U^\lambda ,\cE_\lambda)_{0 <\lambda \leq \frac{n}{2}}$, where 
$(U^{n/2}, \cE_{n/2})$ is trivial. 
\end{thm}

\begin{proof} For $0 < \lambda < \frac{n}{2}$ 
the non-degeneracy of the kernel on $C^\infty(\bS^n)$ 
follows from Theorem~\ref{thm:Irred} which asserts that 
the representation $U^\lambda$ on $C^\infty(\bS^n)$ is irreducible. 
As the space of null-vectors is invariant and proper, it is zero. 
\end{proof}

\begin{defn} The representations $(U^\lambda, \cE_\lambda)$, $0 < \lambda < \frac{n}{2}$, are
called the \textit{complementary series representations} of $G$.
\end{defn}
\index{representation!complementary series}

To unify notation we put $\cE_\lambda =L^2(\bS^n)$ for $\lambda\in i\R$ 
(cf.~Lemma~\ref{lem:7.3}). 
The proof of the following can be found in \cite[p.~119]{vD09}. 
We shall encounter this theorem again in the next section.

\begin{thm}\label{thm:spherical} The 
irreducible unitary representations $(U,\cE)$
of $G$ which are spherical in the sense that $\cE^K\not=\{0\}$ are exactly the 
representations $(U^\lambda,\cE_\lambda)$ with
$\lambda\in i\R\cup \big(0,\frac{n}{2}\big]$.  In these cases $\cE^K=\C \1$ is one dimensional. 
\end{thm}

The
function $\varphi_\lambda (g)=\ip{1}{U^\lambda_g 1}_\lambda$ 
is $K$-biinvariant. It is called the
{\it spherical function with spectral parameter $\lambda$}.
\index{function!spherical}

\subsection{$H$-invariant distribution vectors}
\label{subsec:H-invar} 

On $G = \OO_{1,n+1}(\R)^\uparrow$ we define an involution  $\tau: G\to G$ by $\tau (g):= 
r_{0}  g   r_{0}$, where $r_j$ is the orthogonal reflection in $e_j^\bot$. Then, with respect to the
linear action on $\R^{n+2}$, 
\[ H := G_{e_0} = \left\{\begin{pmatrix} a & 0 & b^\top\\
0 & 1 & 0\\
c & 0 & d\end{pmatrix}\in \OO_{1,n+1}(\R)^\uparrow \: a\in \R, b,c\in \R^{n-1},
 d\in M_{n-1}(\R)\right\}
\subsetneq  G^\tau\]
is a subgroup isomorphic to $\OO_{1,n}(\R)^\uparrow$. 
The relation 
$ r_0 \xi (v)=\xi (r_0x)$ implies that $r_{0}(x.v)=\tau(x).r_{0}(v)$. 
Here we have also viewed $r_0$ as a reflection in $\R^{n+1}$ via restriction. 

In $\bS^n$, the subgroup $H$ has two open orbits
\[H.(\pm e_0)=\bS_\pm ^n=\{(x_0,\bx)\: \pm x_0>0\}\]
and the closed orbit $H.e_n=\{(0,\bx)\: \bx\in\bS^{n-1}\} \cong \bS^{n-1}$.

%\begin{remark}The above does not lead to
%the standard realization of $\OO_{1,n}(\R)$ inside $G$. For that we would have to use 
%the reflection  $\tilde r_{n}$.
%But all reflections are conjugate, so that it is easy to move from one realization to anothe%r. This will be used later
%in this chapter.
%\end{remark}

Considering the standard linear action of $G$ on $\R^{n+2}$, we note that
$G.e_0 \cong G/H$ is the  $(n+1)$-dimensional de Sitter space
\[\dS^{n+1}:=\{(x_{-1},x_0,\bx)\: x_{-1}^2- x_0^2-\|\bx\|^2=-1\}.\]
Both $G$ and $H$ are unimodular. Hence, there exists a $G$-invariant measure on
$\dS^{n+1}$.

Define
\[p^\pm_\lambda (x)
:=\frac{[\xi (x),\mp e_0 ]^{\oline\lambda - \frac{n}{2} }}{\Gamma ((\oline \lambda - \frac{n}{2} +1)/2)}
\chi_{\bS^n_\pm} (x)
=\frac{(\pm x_0)^{\oline\lambda - \frac{n}{2} }}{\Gamma ((\oline \lambda - \frac{n}{2} +1)/2)}
\chi_{\bS^n_\pm} (x)
,\quad x\in\bS^n\]
and let
\begin{equation}  \label{eq:plambda} p_\lambda := p^+_\lambda + p^-_\lambda .
\end{equation}
For $\Re \lambda > n/2$ the
 functions $p^\pm_\lambda$ and $p_\lambda$ are continuous
 and hence  integrable on $\bS^n$.  
We define a distribution 
$\iota_{-\infty} (p_\lambda):=\widetilde\eta_\lambda \in C^{-\infty}_\lambda$ by 
\[\widetilde\eta_\lambda (\varphi ):=  \int_{\bS^n} 
\overline{\varphi (x)}p_\lambda (x)\, d\mu_{\bS^n} (x)
=\ip{\varphi} {p_\lambda}_{L^2}\quad \mbox{ for } \quad \Re \lambda >\frac{n}{2} .\]
Then $\lambda \mapsto \tilde\eta_\lambda (\varphi )$ is antiholomorphic on $\{\lambda \in  \C\: \Re (\lambda) >\frac{n}{2} \}$.
We define $\widetilde{\eta}_\lambda^\pm$ in the same way.

\begin{thm}\label{thm:7.9} The families of distributions 
$\widetilde\eta_\lambda$, $\widetilde\eta^\pm_\lambda$ are antiholomorphic
for $\Re \lambda > \frac{n}{2}$ and have an antiholomorphic
extension to $\C$. The distributions $\widetilde\eta_\lambda$, $\widetilde{\eta}^\pm_\lambda$ are $H$-invariant
and for almost all $\lambda$ we have
$(C^{-\infty}_\lambda)^H=\C p_\lambda^+ + \C p_\lambda^-$. 
\end{thm}

\begin{proof} For the analytic continuation we refer to \cite[Prop. 9.2.9]{vD09}. It is clear that $p_\lambda$ is 
$r_0$-invariant. A  simple calculation shows that
$J_{\oline -\lambda}(h , x )^{-1}p_\lambda (h.x)=p_{\lambda} (x)$ 
which implies that $U^{\oline \lambda}_h p_{\lambda }=p_{\lambda}$ for $h \in H$. 
Hence $\tilde{\eta_\lambda}$ is $H$-invariant for
$\Re \lambda < -\frac{n}{2}$. The uniqueness of analytic extension then 
implies the assertion for all $\lambda$. The last statement can be found in \cite[Thm.~5.10]{vB88}.
\end{proof}

For Proposition~\ref{prop:rea1} we 
adjust the definition so that $\eta_\lambda := \widetilde{\eta}_\lambda$ for
$\lambda\in i\R$ and $\eta_\lambda (\varphi )
=\widetilde \eta_{-\lambda} (A_\lambda\varphi) $ for
$\lambda \in (0,\frac{n}{2})$. Then $\eta_\lambda$ is still invariant under 
$H$ and~$\tau$. Furthermore, as $U^\lambda $ is irreducible
and $\eta_\lambda\not= 0$, it follows
that $\eta_\lambda$ is cyclic. Hence Theorem~\ref{thm:rea3} gives:
 
\begin{thm}\label{thm:7.10} The unitary representation 
$(U^\lambda)_{\lambda\in i\R\cup (0,\frac{n}{2})}$ can
be realized in a Hilbert space of distributions on de Sitter 
space $\dS^{n+1} \cong G/H$.
\end{thm}

\section{Reflection positivity} 

We now turn to reflection positivity, as it manifests itself for spherical 
representations of the Lorentz group.

\subsection{Reflection positivity for the conformal group}
\label{subsec:7.3.4}

In this section we discuss the reflection positivity of the representation 
$(U^\lambda ,\cE_\lambda)$ of $G = \OO_{1,n+1}(\R)^\uparrow$ 
for $\lambda\in (0,\frac{n}{2})$.
We consider again the involutions $\tau$ and $r_0$. Define $\theta : \cE_\lambda \to \cE_\lambda$
by $\theta (\varphi):= \varphi\circ r_0$. Then $\theta p_\lambda =p_\lambda$,
$\theta (U^\lambda_g\varphi)= U^\lambda_{\tau (g)}\theta \varphi$
and $A_\lambda \theta = \theta A_\lambda$. In particular
\begin{equation}\label{eq:Alt}
\ip{\theta \varphi}{\psi}_\lambda 
=\int_{\bS^n\times \bS^n}\overline{\varphi (x)}\psi (y)Q_\lambda (r_0x,y)\, d\mu_{\bS^n}(x)\, d\mu_{\bS^n}(y).
\end{equation}
We let $\cE_{\lambda +}$ be the space generated by the functions supported  
by the half sphere $\bS^n_+$.  For
the positivity of the twisted inner product on $\cE_{\lambda +}$ we switch 
to the realization of $(U^\lambda,\cE_\lambda)$ as acting 
on functions on $\R^n$ via the stereographic projection $s$ from (\ref{def:7eta}).

\begin{lem} Let $R_\lambda (x,y):= (1-\ip{x}{y}+\|x\|^2\|y\|^2)^{\lambda -\frac{n}{2}}$ and
define $ \sigma : \R^n \to \R^n$ by
$\sigma (x)=\frac{x}{\|x\|^2}$. Then the following holds:
\begin{itemize}
\item[\rm (a)]\ The stereographic projection $s : B_1(0)\to \bS^n_+$ from 
\eqref{def:7eta} is a diffeomorphism from the open unit ball 
$B_1(0)\subeq \R^n$ onto $\bS^n_+$. In particular, 
$\supp \varphi \subset \bS^n_+$
if and only if $\supp(\varphi\circ s) \subset B_1(0)$.
\item[\rm (b)]\ $  \sigma = s^{-1} \circ r_{0}\circ  s$.
\item[\rm (c)]\ $\|\sigma (x)-y\|^2=\|x\|^{-2}( 1-2\ip{x}{y}+\|x\|^2\|y\|^2)$.
\item[\rm (d)]\ $\ip{\theta \varphi}{\psi}_\lambda = e_\lambda \int_{B_1(0)\times B_1(0)} \overline{\varphi_\lambda (x)}
\psi_\lambda (y)R_\lambda (x,y)\, d x\, d y$ 
for $\varphi,\psi\in C_c^\infty (\bS^n_+)$. 
\end{itemize}
\end{lem}

\begin{proof} (a)  follows from $1-\|x\|^2>0$ if and only if $\|x\|<1$ and (b) and (c) are simple calculations
and then (d) follows from (c), Lemma~\ref{lem:7.2.8} and (\ref{eq:Alt}).
\end{proof}
 
\begin{thm}\label{thm:RefPos}{\rm(\cite[Prop. 6.2]{NO14})} Let $n\ge 1$.
The kernel $R_\lambda$ is positive definite on $B_1(0)$ if and only if   $\lambda =\frac{n}{2}$ or $\lambda\le \min\{\frac{n}{2} , 1\}$.
\end{thm}

The group $H=\OO_{1,n}(\R)^\uparrow$ maps $\bS^n_+$ into itself, so that 
$U^\lambda_H\cE_{\lambda,+}=\cE_{\lambda, +}$. Furthermore 
$\dd U^\lambda(\fg) C^{\infty} (\bS^n_+)\subseteq C^\infty (\bS^n_+)$. 
The subsemigroup $S:= \{s\in G\: s.\overline{\bS^n_+}\subset \bS^n_+\}$ 
is open and $\sharp$-invariant with $e\in \overline{S}$.  
Combining Theorem \ref{thm:RefPos} with Theorem \ref{thm:CompSrep}, 
we obtain: 

\begin{thm} For $n\ge 2$, the following assertions hold: 
\begin{itemize}
\item[\rm (a)]\ The subspace 
space $\cE_{\lambda, +}$ is $U^\lambda_S$-invariant for all $\lambda \in i\R\cup (0,n/2)$.
\item[\rm (b)]\ For $\lambda \in i\R$ we have $\cE_{\lambda,+}\perp \theta \cE_{\lambda, +}$, so 
that $(\cE_\lambda, \cE_{\lambda,+}, \theta)$ is reflection positive with 
$\widehat{\cE_\lambda}=\{0\}$.
\item[\rm (c)]\ For $0 \leq \lambda \leq \frac{n}{2}$, the triple  
$(\cE_\lambda,\cE_{\lambda, +},\theta)$  
and  the distribution vector $\eta_\lambda$ are 
reflection positive with respect to $(G,S,\tau)$ if and only if 
$\lambda \leq 1$. In this case $ \widehat\cE_\lambda$ is infinite dimensional except for 
$n=2$ and $\lambda =1$, where $\widehat\cE_\lambda $ is one dimensional.
\end{itemize}
\end{thm}

\begin{rem} 
(a) The domain where $R_\lambda$ is positive definite includes the half-line, $\lambda \le \min\{1,\frac{n}{2}\}$. On this
half line we always have an $G$-invariant hermitian form on $C^\infty_\lambda$ which is positive definite only
for $\lambda\geq 0$. 
This leads to the situation where we have a Fr\'echet  space with 
a $G$-invariant hermitian form which is not positive definite, 
but the induced form on $\cE_{\lambda,+}$
is positive leading to a OS-quantization for a non-unitary representation of $G$. For detailed
discussion see \cite{FOO18, JOl98, JOl00, Ol00}.
\smallskip

(b)  The group $G^c$ is
the universal covering of the group $\SO_{2,n}(\R)_0$. It acts 
transitively on the Lie ball
\[\cD := \{z= \xi +i\eta \in \C^n\: 
\xi^2 + \eta^2 +2\sqrt{\xi^2\eta^2-(\xi\eta)^2}<1\}\] 
The stabilizer of $0\in \cD$ is the universal cover $K^c$ of 
$\rS(\OO_{2}(\R)\times \OO_n(\R))_0$ and
$\cD\cong G^c/K^c$.

The real ball $B_1(0)$ is a totally real submanifold in $\cD$. Furthermore 
\begin{equation}\label{eq:RlambdaD}
R_\lambda (z,w):= (1-z\oline w + z^2\oline w^2)^{\lambda -\frac{n}{2}},
\end{equation}
where $s t=\sum s_j t_j$ and $s^2=s s$, is well-defined, holomorphic in the first variable,
antiholomorphic in the second variable and $R_\lambda (z,w)=\overline{R_\lambda (w,z)}$. Thus
$R_\lambda$ is a hermitian kernel on $\cD$ and positive definite if $\lambda =\frac{n}{2}$ or
$\lambda \le \min \{1, \frac{n}{2}\}$. The representation $U^{\lambda c}$, which exists by Theorem \ref{thm:4.8},
is a {\it negative energy representation} of $G^c$ 
\index{representation!negative energy}
\index{representation!highest weight} 
(i.e., a {\it highest weight representation}). 
We refer to \cite[Chap. XIII]{FK94} and \cite{Ne00} for detailed discussion
of such representations.
\end{rem}

\subsection{Resolvents of the Laplacian on the sphere} 
\label{sec:7.3.5} 

Now we continue the example of the sphere by specializing the 
construction from Section~\ref{subsec:2.1.4} based on resolvents 
$(m^2-\Delta)^{-1}$, $m > 0$, of the Laplacian on $\bS^n$ with the involution $r_0$. 
In this section the starting group will be $\OO_{n+1}(\R)$ 
acting on the sphere $\bS^{n}$, whereas
%$\OO_{1,n+1}(\R)^\uparrow$ will only show up at all. But 
$\OO_{1,n}(\R)^\uparrow$ will play the role of the dual group. 
%will be conjugate to the   $c$-dual of $\OO_{n+1}(\R)$. 
We therefore change our notation a little and let 
$G=\OO_{n+1}(\R)$ and $K = G_{e_0} \cong \OO_n(\R)$. 
Accordingly, reflection positivity leads to unitary representations 
of $G^c$ depending on the parameter~$m$. Accordingly, in the discussion about
the representations $(U^\lambda, \cE^\lambda)$ the $n$ in Section~\ref{sec:7.3} 
will change  to $n-1$.
 
We start with some  general simple facts.
\begin{lem} Let $\tau$ be a dissecting reflection on the connected complete Riemannian manifold $M$ and
$m>0$. Let $C = (m^2 - \Delta_M)^{-1}$ 
and $\theta$ be as in {\rm Theorem \ref{thm:2.1.11}}. Let $\Theta  : M\to M$ be an isometric
diffeomorphism. Then
the following assertions hold:
\begin{itemize}
\item[\rm (a)]\  \ $C\circ \Theta_* =\Theta_* \circ C$.
 \item[\rm (b)]\ \ Let $D $ be the reflection positive distribution defined by $D(\varphi \otimes \psi)=
 \ip{\varphi}{C\oline\psi}_{L^2(M)}$. Then $D((\Theta_*\varphi)\otimes \psi)=D(\varphi \otimes (\Theta^{-1}_*\psi))$.
 \end{itemize}
 \end{lem}

Note that $(m^2-\Delta_M)D (x,y) =\delta_M(x,y)$, where the distribution 
$\delta_M$ on $M \times M$ is given by 
$\delta_M(\phi) = \int_M \oline{\phi(x,x)}\, dV_M(x)$, where 
$V_M$ denotes the volume measure on~$M$. 
This implies in particular 
\[ (m^2-\Delta_M)_x D =(m^2-\Delta_M)_y D
= (m^2-\Delta_M)_x(m^2- \Delta_M)_yD=0 \] 
 off the diagonal in
$M\times M$ and that $(\Delta_M-m)_x(\Delta_M-m)_y$ is an elliptic differential operator on $M$ we
have: 
  
\begin{lem}\label{le:Phim}  On the open submanifold  $(M\times M)\setminus \diag (M)$, the distribution 
$D$ is represented by an analytic function $\Phi$, which is 
invariant under   the isometry group $\Isom(M)$. 
\end{lem}
  
Define $C^\tau := C\circ \tau_*$. 
Then, by the above lemma, there exists an analytic function $\Psi$ 
on $M_+\times M_+$
such that
\[ (C^\tau \varphi)(x)=\int_{M_+}\Psi (x,y)\varphi (y)dV_M(y) 
\quad \mbox{ for } \quad x \in M_+. \]
As $(m^2 - \Delta_M)C^\tau \varphi =0$ for $\varphi\in C_c^\infty (M_+)$ it  follows that
$C^\tau \varphi$ is analytic on $M_+$ and  $\Delta_M (C^\tau \varphi |_{M_+}) = m^2 C^\tau \varphi|_{M_+}$. 

Since cyclic one-parameter semigroups of contractions are given by multiplication 
with functions on $L^2$-spaces, we take a closer look at this special situation. 
We give here one example, a second example is discussed in the following subsection.

After these general remarks, we now specialize to 
\[ M=\bS^n\supset M_+=\bS^n_+ = \{x \in \bS^n\: x_0 > 0\}.\]
To work with the exponential function
we introduce the analytic functions
\begin{equation}
  \label{eq:CandS}
 C(z) := \sum_{k = 0}^\infty \frac{(-1)^k}{(2k)!} z^{k} \quad \mbox{ and } \quad 
 S(z) := \sum_{k = 0}^\infty \frac{(-1)^k}{(2k+1)!} z^{k}
\end{equation}
which satisfy 
$\cos z = C(z^2)$ and $\frac{\sin z}{z} = S(z^2)$  
for $z \in \C^\times.$
We thus obtain as in \eqref{eq:exp} that 
 \begin{equation}\label{eq:exp-rel}
\Exp_{u}(v) = C(v^2)  \cdot u + S(v^2)\cdot v, \quad u\in\bS^n, v\in S_u.
\end{equation}
 
The complex sphere 
\[ \bS^n_\C=\{u\in \C^{n+1}\: u^2=1\} 
= \OO_{n+1}(\C).e_0 \cong \OO_{n+1}(\C)/\OO_{n}(\C)\] also is a symmetric space 
(in the category of complex manifolds) with respect to the reflections 
$ s_x(y) := y- 2(x y) x$,  for $ x,y \in \bS^n_\C$  
and the corresponding exponential map is 
\begin{equation}
 \label{eq:exp-relc}
\Exp_{u}(v) = C(v^2)  \cdot u + S(v^2)\cdot v 
\quad \mbox{ for }\quad u \in \bS^n_\C, v \in T_u(\bS^n_\C).
\end{equation} 

\begin{defn}
Let $\iota(x) = (x_0, i \bx)$ and 
$V:=\iota \R^{n+1}=\R e_0\oplus i\R^n $. Define
\[ [\iota x, \iota y]_V :=\iota x \cdot \iota y = x_0y_0 -\bx\by \] 
and note that $[g.u, g.v]_V=[u,v]_V$ for $u,v \in V$ and 
all elements $g\in   G^c:=\iota \OO_{1,n}(\R)^\uparrow \iota$.

On $\C^{n+1}$ we consider the conjugations 
\[ \sr(z_0, \ldots, z_n) := (\oline{z_0}, \ldots, \oline{z_n}) 
\quad \mbox{ and } \quad 
 \sx(z_0, \ldots, z_n) := (\oline{z_0}, -\oline{z_1}, \ldots, -\oline{z_n}) \] 
with respect to the real subspaces $\R^{n+1}$ and $V$, respectively. 
The conjugations $\sr$ and $\sx$
commute and the holomorphic involution $\sr\sx$ is $-r_0$. The involution 
$\sx$ commutes with $G^c$, but $\sr$ does not, and 
$\sr g \sr = r_0 g r_0 = \tau(g)$ 
is the involution on $G^c$ whose fixed point group is 
$K = G^c_{e_0} \cong \OO_n(\R)$. 
\end{defn}

We also note  that $[x,y]_V = x y = \sum_j x_j y_j$ for $x,y \in \C^{n+1}$ is 
the unique complex bilinear extension of $[\cdot,\cdot]_V$ to 
$V+iV = \C^{n+1}$. This notation underlines the Lorentzian nature of the
situation rather than the euclidean one. We also consider the following sets:
 
\begin{itemize}
\item $V_+:=\{v\in V\: [v,v]>0, v_0>0\}= \iota(\R^{1,n}_+)$, the forward light cone in $V$;
\item  $\wH:= \iota \H^n=\bS^n_\C \cap V_+ = G^c.e_0 \cong G^c/K$, the hyperboloid of one sheet in $V$;
\item $\bS^n_{+,\C} := \{ z \in \bS^n_\C \: \Re z_0 > 0\}$;
\item the tube domain 
$\Tu:= iV+V_+ \cong \SO_{2,n+1}(\R)/\rS (\OO_2(\R)\times \OO_{n+1}(\R ))$; and 
\item $\Xi:=G^c.\bS^n_+\subset \bS^n_{+,\C}$. 
\end{itemize}
The domain $\Xi$ is called the \textit{crown} of \index{crown domain}
$\wH$. Note that $G^c.V =V$, $G^c.\Tu=\Tu$ and ${G^c.\bS^n_{\C}=\bS^n_{\C}}$.
 
\begin{prop}\label{prop:XiTube} The following assertions hold: 
\begin{itemize}
\item[\rm(a)]\ $T_{V_+} \cap \R^{n+1} = \R^{n+1}_+$ and 
$\bS_+^n=\Tu \cap \bS^n$. 
\item[\rm(b)]\ $\Xi =\Tu \cap \bS^n_{+,\C}=\Tu\cap \bS^n_\C$.
\item[\rm(c)]\  We have $\sx \Xi = \sr\Xi =\Xi$ and 
$\Xi^{\sx}= \Xi \cap V = \wH$ and $\Xi^{\sr} = \Xi \cap \R^{n+1}= \bS^n_+.$
\item[\rm (d)]\ $\C_\Xi := \{z\in \C \: (\exists u,v\in\Xi )\  z=[u, v]_V\} 
= \C\setminus (-\infty , -1]$.

\end{itemize}
\end{prop}

\begin{proof} (a) This follows from   $z=u+iv\in \Tu\cap \R^{n+1}$ if and only if $u=re_0$ with
$r>0$ and $iv = (0,\bv)$ with $\bv\in  \R^n$.  

(b) By (i) we have $\Xi =G^c.\bS^n_+= G^c.(\Tu \cap \bS^n)
\subseteq \Tu \cap \bS^n_{\C}$.  
Let $z=u+iv \in \Tu\cap \bS^n_{\C}$. Then $u_0>0$ and, as $G^c$ acts transitively on all level sets $[u,u]=r>0$ in $V_+$, we
may
assume that $u=re_0$ with $r>0$. Thus $z=(r+iv_0,\bv)$ 
with $v_0 \in \R$ and $\bv \in \R^n$. As $z\in \bS^n_\C$, we have $1=z z
= r^2-v_0^2+2irv_0 +\bv^2$.
Hence $v_0=0$ and this implies that  $z\in \bS_+^n\subset \Xi$. Finally, we note that, if 
$z\in \Tu$, then $\Re z_0>0$, hence $\Tu\cap \bS^n_\C = \Tu \cap \bS^n_{+,\C}$.

(c) follows from (a) and (b). 

(d) follows from a lengthy, but elementary, calculation; 
see \cite[Lemma~3.8]{NO18} for details.
\end{proof}

We point out the following two simple consequences  of Proposition \ref{prop:XiTube}
\begin{cor}\label{cor:totallyreal} The subset 
$\Xi$ is an open complex submanifold of $\bS^n_{\C}$ 
on which the group $G^c$ acts by holomorphic maps. 
The fixed point set $\bS^n_+ =\Xi^{\sr}$ and $\wH = \Xi^{\sx}$ 
of the antiholomorphic involutions $\sr$ and $\sx$ are totally real submanifolds 
of~$\Xi$.
\end{cor}

\subsection{The distribution kernel of $(m^2-\Delta)^{-1}$}
\label{subsec:distker}

In this section we use polar-coordinates $x_{e_0}(t,u)=\cos (t)e_0+\sin (t)u$. 
If $\varphi \: \bS^n \to \C$ is $K$-invariant, 
then it is determined by the values on 
$x_{e_0}(t,e_n)$, $0 \leq t \leq \pi,$ and we may simply write 
$\tilde\varphi (t) := \varphi (x_{e_0}(t,e_n))$. Let $\bS^n_*:=\bS^n\setminus \{-e_0\}$
and note that $\bS^n_*$ is $K$-invariant as $K$ fixes $\pm e_0$.
  
\begin{lem}{\rm(\cite[Cor. 9.2.4]{Fa08})} For $\varphi \in C^2(\bS^n_*)^K$ 
and $t\in  (0,\pi)$
\[
(\Delta_{\bS^n} \varphi) (x_{e_0}(t,u))
=\tilde\varphi ^{\prime\prime}(t)+(n-1)\cot (t)\tilde\varphi ^\prime (t)
=\frac{1}{\sin^{n-1}(t)}\dfrac{d}{dt}\left(\sin^{n-1}(t)\dfrac{d}{dt}\right)  
\tilde\varphi (t).\]
In particular, 
$\Delta_{\bS^n} \varphi  =m^2\varphi $ if and only if
\begin{equation}\label{eq:RadPart2}
\tilde\varphi ^{\prime\prime}(t)+(n-1)\cot (t)\tilde\varphi ^\prime (t)
-m^2\tilde\varphi (t)=0.
\end{equation}
\end{lem}

We already discussed the case $n=1$ in Chapter \ref{ch:5} and will therefore assume that
$n>1$.  The substitution $s=\sin^2 (t/2) = \frac{1-\cos t}{2}$, 
$0 \leq t \leq \pi$, and $\xi (s):= \varphi (x_{e_0}(t,e_n))$  transforms (\ref{eq:RadPart2}) 
into the  
\textit{hypergeometric differential equation}:
\index{differential equation!hypergeometric} 
\begin{equation}\label{eq:HypDeEq}
s(1-s)\xi^{\prime\prime}(s) +\left(\frac{n}{2}-ns\right)\xi^\prime (s)-m^2 \xi (s)=0\, .
\end{equation}
As $\xi$ does not have a singularity in $s = 0$, 
this leads to 
\begin{equation}\label{eq:hyp}
\xi(s)= c\cdot \hgf\Big(\frac{n-1}{2} +\lambda , \frac{n-1}{2}-\lambda, \frac{n}{2}; s\Big)
\end{equation}
with 
\begin{equation}
  \label{eq:lambda-df}
c=\xi (0) \qquad \mbox{ and }  \quad 
\lambda := 
\begin{cases}
\sqrt{\left(\frac{n-1}{2}\right)^2-m^2} & \text{ if }\quad m^2<\left(\frac{n-1}{2}\right)^2 \\ 
i \sqrt{m^2-\left(\frac{n-1}{2}\right)^2} & \text{ if }\quad m^2\ge \left(\frac{n-1}{2}\right)^2.  
\end{cases}
\end{equation}
 
We recall here the definition of $\hgf$ and refer to  %\cite[Chap. 9]{L72} and 
\cite[\S 14.2]{WW63} for more details.
For $\alpha \in \C$ and $k\in \N$ let %$(a)_0=1$ and 
\[ (\alpha )_k :=\prod_{j=0}^{k-1}(\alpha +j)
=\Gamma (\alpha +k)/\Gamma (\alpha ).\] 
For $a,b \in \C$ and $c\in \C\setminus -\N_0$, we have
\begin{equation}\label{def:hgf}
\hgf (a, b , c ; z)=\sum_{k=0}^\infty \frac{(a)_k(b)_k}{(c)_k} \frac{z^k}{k!} 
\quad \mbox{ for } \quad  |z|<1.
\end{equation}

For $a, b,c>0$ or $b=\oline a$, $c>0$,  
\eqref{def:hgf} implies that 
$\hgf (a , b ; c ; z)>0$ for $0\le z<1$, and $\hgf (a,b;c;0)=1$.
As $\hgf (a, b; c;z)= \hgf (b, a; c; z)$ and $m>0$, 
we obtain the same function if we replace $\lambda$ by $-\lambda$. 
In particular, 
\[ \hgf \left(\frac{n-1}{2}+\lambda, \frac{n-1}{2}-\lambda ; \frac{n}{2}; x\right)>0\quad \mbox{ for } \quad 0\le x<1.\]  
 
 We now apply this to the kernel $\Psi _m (u,v)$ corresponding to $C^\tau$. 
%We simply write $\Psi$ as $m$ will  be fixed for the moment.  
Because of the $G$-invariance it follows 
 that the function $\psi := \Psi_m(\cdot ,e_0)
\in C^\infty (\bS^n_*)$ is $K$-invariant. It also satisfies
 the differential equation $\Delta_{\bS^n} \psi = m^2\psi$ on $\bS^n_*$ 
(Section~\ref{sec:7.3.5}). Hence there exists a
 constant $\gamma$ such that
 \[\psi (x_{e_0}(t,v))=\gamma \cdot 
\hgf \left(\frac{n-1}{2} +\lambda , \frac{n-1}{2}-\lambda, \frac{n}{2}; \sin^2(t/2)\right) .\]
 
 \begin{thm}\label{th:Psi} For every $m > 0$, 
there exists a constant $\gamma_{m}>0$ 
such that the $G^c$-invariant kernel
 $\Psi_m$, corresponding to $C^\tau$, 
extends to a hermitian kernel on $\Xi\times \Xi$ is given by
\begin{equation}\label{eq:PsiExt}
\Psi_m (x,y)= \gamma_m \cdot  \hgf \left(\frac{n-1}{2} +
\lambda , \frac{n-1}{2} -\lambda ; \frac{n}{2};
\frac{1-\lf{x}{\sx(y)}}{2}\right)
\end{equation}
and defines a positive definite hermitian kernel on $\Xi\times \Xi$.
\end{thm}

\begin{proof} %As $m$ is fixed we do not include it in the notation for the moment. 
First note that $\sin^2(t/2)=\frac{1}{2}\left( 1-\cos (t)\right)  
=\frac{1}{2}\left(1-[x_{e_0}(t, v), \sx (e_0)]_V\right) $ for $v \in e_0^\bot$.
Hence, by the above discussion:
\[\Psi_m(x,e_0)=c \cdot
 \hgf \left(\frac{n-1}{2} +\lambda , \frac{n-1}{2}-\lambda, \frac{n}{2}; 
\frac{1-\lf{x}{\sx(e_0)}}{2}\right).\]

The hypergeometric function $\hgf$ has
an analytic continuation to $\C\setminus [1,\infty)$, see \cite[p. 288]{WW63}. It
follows from  Proposition \ref{prop:XiTube}(d) that the right hand side, and hence 
also $\Psi_m(\cdot,e_0)$, has
an extension to $\Xi\times \Xi$ given by (\ref{eq:PsiExt}). The extension
is unique as $\bS^n_+$ is a totally real submanifold in $\Xi$ 
(Proposition~\ref{prop:XiTube}(c)). It is holomorphic in
the first variable and antiholomorphic in the second variable and
$\overline{\Psi_m(x,y)}=\Psi_m(y,x)$ which follows from (\ref{def:hgf}) and the form of
the parameters. As the kernel $C$ is reflection positive, 
$\Psi_m(u,v)$ is positive definite on $\bS^n_+\times \bS^n_+$. 
Hence Theorem A.1 in \cite{NO14}
implies that $\Psi_m$ is positive definite on $\Xi\times \Xi$.  In particular
$\Psi_m(e_0,e_0)>0$. It follows that $c>0$.
\end{proof}
 
Above we introduced 
the principal series representations $(U^\lambda, \cE^\lambda)_{\lambda \in i \R\cup (0,\frac{n}{2})}$ and
just after 
 Theorem \ref{thm:spherical} we defined the spherical function 
$\varphi_\lambda (g)$ with spectral parameter~$\lambda$. We will
use this now for to the group $G^c$.
The following proposition follows from \cite[p.~1158]{OP13} and \cite[p. 126]{vD09}: 

\begin{prop}\label{prop:psiSpherical}
For 
$\psi_m(x)=\frac{1}{c}\Psi_m(x,e_o)$ and $\lambda = \sqrt{\left(\frac{n-1}{2}\right)^2-m^2}$ as before, we have 
$\psi_m|_{\H^n} =\varphi_\lambda$. In particular, the
spherical function  $\varphi_\lambda$ extends to a holomorphic
function on~$\Xi$ and the unitary representation of $G^c$ 
on the reproducing kernel space $\cH_{\Psi_m} \subeq \cO(\Xi)$ 
with kernel $\Psi_m$ is equivalent to $(U^\lambda,\cE^\lambda)$. 
\end{prop}

The last statement on analytic continuation 
in Proposition~\ref{prop:psiSpherical} is a special case of
the general theorem due to Kr\"otz and Stanton \cite{KS04, KS05}.

In view of Proposition~\ref{prop:psiSpherical}, 
we will from now on use the notation $\varphi_\lambda$ and $\Phi_\lambda$ for the
normalized version of $\psi_m$ and $\Psi_m$, where $\lambda$ is defined by 
\eqref{eq:lambda-df}. 

\begin{lem} Let $\Psi\not= 0$ be a hermitian 
positive definite $G^c$-invariant kernel on $\Xi\times \Xi$.
If $\Psi_{e_0}|_{\wH}=0$,  then $\Psi=0$.
\end{lem}
\begin{proof} As $\wH$ is totally real in $\Xi$ by Corollary~\ref{cor:totallyreal}, 
$\Psi_{e_0}|_{\wH}=0$ implies $\Psi_{e_0}=0$ 
on all of~$\Xi$. By the $G^c$-invariance it follows that $\Psi (z,x)=0$ for all $z\in \Xi$ and
all $x\in\wH$. As $\Psi (z,\cdot)$ is antiholomorphic, it vanishes 
on all of~$\Xi$. 
\end{proof}

\begin{thm}\label{thm:extemHerm} If $\Psi$ is a 
positive definite $G^c$-invariant kernel on $\Xi\times \Xi$ for which the 
canonical representation $(U,\cH_\Psi)$ of $G^c$ is irreducible, then 
there exist $c > 0$ and $\lambda\in i\R\cup \big(0,\frac{n}{2}\big]$ such that
$\Psi =c\Phi_\lambda$ and $(U,\cH_\Psi)$ is equivalent 
to $(U^\lambda ,\cE^\lambda)$.
\end{thm}

\begin{proof} The function $\Psi_{e_0}$ is $K$-invariant. Hence 
Theorem~\ref{thm:spherical} implies that there
exists a $\lambda \in i\R \cup (0,\frac{n-1}{2}]$ 
such that $(L, \cH_{\Psi}) \simeq (U^\lambda,\cH_\lambda)$. 
We can assume that \break $\Psi (e_0,e_0)=~1$.
Then 
$\Psi (e_0, g.e_0)=\ip{\Psi_{e_0}}{\Psi_{g.e_0}} 
= \ip{\Psi_{e_0}}{L_g\Psi_{e_0}}
=\ip{1}{U^\lambda_g1}=\varphi_\lambda(g).$ 
\end{proof}

\subsubsection*{The boundary of the crown and the spherical function $\varphi_\lambda$}

In this subsection we describe the boundary of the crown as a disjoint union of two
orbits. Both are homogeneous spaces that we have already met, 
the de Sitter space $\dS^n$ and  
the forward pointing light like vectors $\Lnp$ which we have 
already introduced in Subsections~\ref{subsec:H-invar} 
and \ref{ssec:Principal}, respectively. For details we refer to \cite{NO18}.

A simple calculation shows that the boundary of $\Xi$ in $\bS^n_\C$ is given by 
  \begin{equation}
    \label{eq:boundxi}
\partial \Xi =\left\{z=u+iv\in V + i V = \C^{n+1}\: 
\smm{[u,u]_V=0, u_0\ge 0,\\ [v,v]_V=-1, [u,v]_V =0}\right\}
  \end{equation} 
(\cite[Lemma~3.10]{NO18}). 
If $u=0$, then \eqref{eq:boundxi} leads to a realization of 
\textit{de Sitter space} \index{deSitter space}
\[\dS^n:=i \{v\in V\: [v, v]_V =-1\} 
= i V \cap \bS^n_\C = G^c.ie_n \cong G^c/H \subeq \partial \Xi \cap iV\]
where $H\simeq \OO_{1,n-1}(\R)$ is the stabilizer of $e_n\in\dS^n$. 

Let $\xi_0:=  e_0+ ie_n$ and write $G^c_{\xi_0} = MN$, 
where $M$ and $N$ are similar to the groups introduced in Section \ref{ssec:Principal}
except one has to replace $v$ by $iv$ and interchange the second and last columns and
second and last row as we now consider $e_n$ as a base point.  
 
\begin{lem} \label{lem:3.7} Suppose that $n\geq 2$ and let $\cO 
:= G^c. (e_0 + i e_n +e_{n-1})$. The boundary of the crown is the union two $G^c$-orbits 
\[ \partial \Xi = \dS^n\dot \cup  \cO.\] 
In particular, $\dS^n$ is the unique open $G^c$-orbit in the boundary.
The projection of $\cO$   onto 
$V$ is $\Lnp$ and the projection onto $iV$ is $\dS^n$. 
\end{lem}

The orbit $G^c.ie_n \cong \dS^n$ 
is called the \textit{Shilov boundary of $\Xi$} in $\bS^n_\C$. \index{Shilov boundary} 
The tangent space of $\dS^n$ at $e_n$ is the $n$-dimensional Minkowski space 
\[ T_{e_n}(\dS^n) = i \R \oplus \R^{n-1}\simeq  \R^{1,n-1}.\]
By \eqref{eq:exp-rel}, we have 
\begin{equation}
  \label{eq:exp1}
\Exp_{e_n}(z )=S(z^2) z  + C(z^2)e_n 
\quad \mbox{ for } \quad z\in T_{e_n}(\dS^n)_\C=\C\oplus \C^{n-1}. 
\end{equation}

We now describe how one can obtain the crown by moving inward from the de Sitter 
space $\dS^n$. 
\begin{thm} For $g\in G$ let 
$V_{+,g.e_n}^\pi = \{v\in g.V_+\: \lf{v}{v}<\pi^2 \}\subset T_{g.e_n}(\dS^n)$.
Then
\[
\Xi = G^c.\Exp_{e_n}(V_{+,e_n}^\pi) = \bigcup_{p \in \dS^n} \Exp_p(V_{+,p}^\pi). 
\]
\end{thm}

\begin{prf} In view of the $G^c$-invariance of $\Xi$ and the equivariance 
of the exponential map of $\bS^n_\C$, it suffices to verify the first equality. 
From \eqref{eq:exp1} we obtain for $v\in \R_+\oplus i\R^{n-1}
\subset iT_{e_n}(\dS^n)$ and 
$\R_+ = (0,\infty)$:
\begin{equation}\label{eq:Expen}
\Exp_{e_n}(v )=S([v,v]_V) v  + C([v,v]_V)e_n 
\end{equation}
and this is contained in $T_{V_+} \cap \bS^n_{\C}=\Xi$ if $[v,v]_V \in (0,\pi^2)$. 
Therefore $\Exp_{e_n}(\Omega_{e_n}) \subeq \Xi$. 
If, conversely, $z \in \Xi = G^c.\bS^n_+$, there exists a 
$t\in (0,\pi)$ such that $z$ is $G^c$-conjugate to   $x = (\sin t, 0, \ldots, 0, \cos t)$. 
But then $te_0 \in V_{+,e_n}^\pi$, and \eqref{eq:Expen} yields 
$x = \Exp_{e_n}(t e_0)$. This proves the claim. 
\end{prf}

In this section we give a different description of the spherical 
function $\varphi_\lambda$ and the kernel $\Phi_\lambda (x,y)$ 
(cf.~Subsection~\ref{subsec:distker}) using the
space $\bL^n_+$. For that we have to assume that $m\ge \frac{n-1}{2}$ 
(which corresponds to the principal series), which we do from now on.

Recall the map $\xi : \bS^{n-1}\to \bL^n_+=G^c.\xi_0, x \mapsto (1,x)$ 
and the action of $G^c$ on $\bS^{n-1}\cong \bL^n_+/\R^\times_+$, given by
\[\xi (g.u)=J(g,u)^{-1} g(\xi (u)). \]

\begin{lem} Let $z\in \Xi$ and $\xi\in \Lnp$. Then $\Re \lf{z}{ \xi}> 0$.
\end{lem}

\begin{proof} Write $z = u +i v\in \Xi \subeq T_{V_+}$ (Proposition~\ref{prop:XiTube}). 
Then $u\in V_+$  implies that 
\[ \Re \lf{ z}{ \xi_0}= \lf{u}{\xi_0}= u_0 - u_n>0
\quad \mbox{ for } \quad \xi_0 = e_0 + i e_n \in V.\] But
then $\Re \lf{z}{g.\xi_0}=\Re \lf{g^{-1}.z}{\xi_0}> 0$ 
for all $g \in G^c$. 
\end{proof}

For $\lambda  \in \C$ we define the analytic kernel 
\begin{equation}\label{eq:kappalambda}
\kl: \Xi \times \Lnp \to \C, \qquad 
\kl (z,\xi) := K_{\lambda, \xi}(z) := \lf{z}{\xi}^{\lambda -\frac{n-1}{2}}\, .
\end{equation}
This kernel is continuous for $\Re \lambda > (n-1)/2$.
Note the similarity with the distribution vector $p_\lambda$ from
\eqref{eq:plambda}. 

\begin{thm} For $\lambda \in i[0,\infty)\cup (0,(n-1)/2)$,  
the assignment 
\[ \cP_\lambda : L^2(\bS^{n-1})\to \cO (\Xi), \quad 
(\cP_\lambda \varphi)(z):=\int_{\bS^{n-1}} K_{\lambda }(z,\xi (u))
\phi(u)\, d\mu_{\bS^{n-1}}(u)\] 
defines a $G^c$-intertwining operator 
$(U^\lambda, L^2(\bS^{n-1}))
\to (L,\cH_{\Phi_\lambda})$ 
with $\cP_\lambda 1 = \varphi_\lambda$. 
\end{thm}

%\textbf{Still missing a part on the boundary values}

\section*{Notes on Chapter 7} 

Most of Sections  \ref{sec:7.1}--\ref{sec:7.3} is from \cite{NO14}, 
with slightly different notation. 
Under some additional assumptions, Theorem~\ref{thm:rea3} 
can be found in \cite[Thm. 8.2.1]{vD09}, 
and in \cite[Sect. 2]{NO14} for the special case $H=\{e\}$. 
 In Section~\ref{sec:7.3} we added some material on the principal series
representations and the $H$-invariant distributions vectors.  The notation in Section \ref{sec:7.3}
has also been adapted to the standard notation from \cite{JOl98, JOl00, Ol00, FOO18} as well as the notation for the
last section in this chapter. The material about the sphere is from \cite{NO18}; 
for its relation to construction of QFTs on de Sitter space, 
we refer to \cite{BJM16}.

The crown is the  maximal $G^c$-invariant 
domain for the holomorphic extension of all spherical functions on the
Riemannian symmetric space $G^c/K$. It is shown in \cite {KS05},  
that in our case the crown is
a Riemannian symmetric space $\SO_{2,n}(\R)_0/(\SO_2(\R) \times \SO_n(\R))$; 
see also \cite{NO18} for a direct argument. 
For more information about the crown see \cite{AG90, KO08, KS04}.

\end{bibunit}

\chapter{Generalized free fields} 
\label{ch:8} 

\begin{bibunit}
%{\bf tentative size: 8pp} 

We now turn to representations 
of the Poincar\'e group corresponding to scalar generalized free fields 
and their euclidean realizations by representations of the euclidean 
motion group. We start in Section~\ref{subsec:8.1} with a brief 
discussion of Lorentz invariant measures on the forward 
light cone $\oline{V_+}$ and turn in Subsection~\ref{subsec:8.2} 
to the corresponding unitary representations. Applying the dilation 
construction to the time translation semigroup leads immediately 
to a euclidean Hilbert space $\cE$ on which we have a unitary 
representation of the euclidean motion group. 
In Subsection~\ref{subsec:8.3} we characterize those representations 
which extend to the conformal group $\OO_{2,d}(\R)$ 
of Minkowski space. Then the euclidean 
realization is a unitary representation of the Lorentz group 
$\OO_{1,d+1}(\R)$, acting as the conformal group on euclidean~$\R^d$.

\section{Lorentz invariant measures on the light cone 
and their relatives} 
\label{subsec:8.1}

Before we turn to unitary representations of the Poincar\'e group, 
it is instructive to have a closer look at Lorentz invariant measures $\mu$ on 
the forward light cone $\oline{V_+}$ 
and their projections to $\R^{d-1}$. We shall 
also see that these measures are directly related to 
rotation invariant measures $\nu$ on euclidean space $\R^d$, and this 
establishes the key link between unitary representations of the Poincar\'e group 
$P(d)$ and the euclidean group $E(d)$. 

\begin{defn} \label{def:mum}
For $m \geq 0$ or $d > 1$, we define a Borel measure 
$\mu_m$ on 
\begin{align*}
 H_m 
&:= \{ p \in \R^d \: [p,p] = p_0^2 - \bp^2 = m^2, p_0 > 0 \} \\ 
&\subeq \oline{V_+} 
= \{ p= (p_0,\bp) \in \R^d \: p_0 \geq 0, [p,p] = p_0^2 - \bp^2\geq  0\}
\end{align*}
by 
\[ \int_{\R^d} f(p)\, d\mu_m(p) = 
\int_{\R^{d-1}} f\Big(\sqrt{m^2+ \bp^2}, \bp\Big)\frac{d\bp}{\sqrt{m^2 + \bp^2}}\]  
(cf.\ \cite[Ch.~IX]{RS75}, \cite[Lemma~9.1.2/3]{vD09}). 
These measures are invariant under the Lorentz group $\OO_{1,d-1}(\R)^\uparrow$ and every 
Lorentz invariant measure $\mu$  on $\oline{V_+}$ 
is of the form 
\begin{equation}
  \label{eq:nu-rho-a}
\mu = c \delta_0  + \int_0^\infty \mu_m\, d\rho(m),  
\end{equation}
where $c \geq 0$ and $\rho$ is a Borel measure on $[0,\infty)$ 
(with $\rho(\{0\}) = 0$ for $d = 1$) 
whose restriction to $\R_+$ is a Radon 
measure (see \cite[Thm.~B.1]{NO15a}).\begin{footnote}{In Quantum Field Theory this is known as the 
Lehmann Spectral Formula for two-point functions; see \cite[Thm.~6.2.4]{GJ81}.} 
\end{footnote}
\end{defn}

\begin{rem} \label{rem:8.1.2} 
(a) For $d = 1$, we have $H_m = \{m\}$ for $m > 0$ and $H_0 = \eset$. 
Therefore $\mu_0$ does not make sense. For $m > 0$, we have 
$\mu_m = \frac{1}{m} \delta_m$, where $\delta_m$ is the Dirac measure in $m$. 

(b) For $d =2$, the measure $\mu_0$ is singular in~$0$, but every 
$\phi \in \cS(\R^2)$ vanishing in $0$ is integrable 
(cf.\ \cite[p.~103]{GJ81}). 
In particular, this measure does not define a distribution, it defines a 
functional on the smaller space of test functions 
$\cS_*(\R^d) := \{  \phi  \in \cS(\R^d) \: \phi(0)=0\}$.

(c) By \cite[Thm.~B.1]{NO15a}, the 
measure $\mu$ in \eqref{eq:nu-rho-a} is tempered if and only if 
the measure $\rho$ is tempered and, in addition,
\begin{equation}
  \label{eq:rho-temp-cond}
   \int_0^1 \frac{1}{m}\,  d\rho(m) < \infty \quad \mbox{ for } \quad d = 1, \qquad
\int_0^1 \ln(m^{-1})\,  d\rho(m) < \infty\quad \mbox{ for } \quad d = 2. 
\end{equation}
\end{rem}

\begin{ex} \label{ex:3.15} (Generalized free fields) 
(a) For the scalar generalized free field of spin zero on $\R^d$, 
 the corresponding one-particle  Hilbert space 
is $\cH := L^2(\R^d, \mu)$, where 
$\mu$ is a Lorentz invariant measure on $\oline{V_+}$ 
(see \eqref{eq:nu-rho-a}). \index{time translation semigroup}
Here the {\it time translation semigroup} $C_t$ acts by the contractions 
\[ (C_tf)(p) = e^{-t p_0} f(p).\] 
The dilation construction from Example~\ref{ex:4.2.5} 
leads to the space $\cE := L^2(\R^{d+1},\zeta)$ 
with 
\begin{equation}
  \label{eq:zeta-def}
 d\zeta(\lambda,p) = \frac{1}{\pi}\frac{p_0}{p_0^2 + \lambda^2}\, d\lambda \, d\mu(p) 
\quad \mbox{ and }\quad 
(U_t f)(\lambda,p) = e^{it\lambda} f(\lambda,p). 
\end{equation}
For $\pr_2(\lambda, p_0,\bp) = (\lambda,\bp)$, the projected measure 
$\nu := (\pr_2)_*\zeta$ on $\R^d$ 
is given, in the special case $\mu = \mu_m$, by the measure 
$\nu_m$ from Example~\ref{ex:dm}: 
\begin{equation}
  \label{eq:num-def2}
d\nu_m(p_0, \bp) 
:= \frac{1}{\pi} \frac{\sqrt{m^2 + \bp^2}}
{m^2 + \bp^2 + p_0^2}\, dp_0 \frac{1}{\sqrt{m^2 + \bp^2}}\, d\bp 
= \frac{1}{\pi} \frac{dp} {m^2 + p^2},
\end{equation}
so that 
\begin{equation}
  \label{eq:nu-def}
d\nu(p_0, \bp) 
= \frac{1}{\pi} \int_0^\infty \frac{1}{m^2 + p^2}\, dp\, d\rho(m)
= \Big(\frac{1}{\pi} \int_0^\infty \frac{d\rho(m)}{m^2 + p^2}\Big)\, dp 
= \Theta(p)\, dp,
\end{equation}
for $\Theta(p) := \frac{1}{\pi} \int_0^\infty \frac{d\rho(m)}{m^2 + p^2}$. 

(b) Since elements of $L^2(\R^d,\nu)$ correspond to functions 
in $L^2(\R^{d+1}, \zeta)$ not depending on the second argument~$p_0$, 
we obtain an isometric embedding 
\begin{equation}
  \label{eq:time-0-iso}
\pr_2^* \: L^2(\R^d,\nu) \to \cE = L^2(\R^{d+1}, \zeta).
\end{equation}

(c) The free scalar field of mass $m$ and spin $s = 0$ on $\R^d$ 
(with $m > 0$ or ${d > 1}$) corresponds to the measure $\mu = \mu_m$ 
(cf.\ \cite[p.~103]{GJ81}). In this case $\pr_2^*$ is surjective, so that 
we can identify $L^2(\R^d,\nu)$ with $\cE$. 
The measure $\nu_m$ is finite if and only if $d = 1$ and $m > 0$. 
It is tempered if and only if $d > 2$ or $m > 0$. 
\end{ex}

\index{measure!tame Borel}
\begin{defn} We call a positive Borel measure $\rho$ on $[0,\infty)$ {\it tame} 
if $\int_0^\infty \frac{d\rho(m)}{1 + m^2} < \infty$. Note that this implies 
in particular that $\rho$ is tempered. 
\end{defn}

\begin{rem} \label{rem:8.1.5}  
In view of 
\cite[Lemma~7.1]{NO15a}, the measure $\rho$ is tame if and only if 
$\Theta(p) < \infty$ for 
every $p \in \R^d$ and this in turn is equivalent to $L^2(\R^d, \nu) \not=\{0\}$. 

If this is the case, then the measure $\nu$ on $\R^d$ is tempered 
if and only if $d > 2$ or 
the conditions \eqref{eq:rho-temp-cond} characterizing the temperedness of 
$\mu$ for $d = 1,2$ are satisfied (\cite[Prop.~7.3]{NO15b}). 
As tameness of $\rho$ implies that $\rho$ is tempered, $\mu$ is  
tempered if $\nu$ has this property. 
\end{rem}

\section{From the Poincar\'e group to the euclidean group} 
\label{subsec:8.2}

We have already seen in Example~\ref{ex:3.15} that Lorentz invariant 
measures on the forward light cone lead by the dilation construction 
to rotation invariant measures on euclidean space. 
We now take a closer look of the implications of this correspondence 
for unitary representations of the Poincar\'e group $P(d)$ and 
the euclidean group $E(d)$. 
In QFT, this corresponds to the one-particle representations 
of scalar generalized free fields.

\begin{ex} \label{ex:3.16} (One particle representation of 
generalized free fields) 
Let $\mu$ be a Lorentz invariant Radon measure as in 
\eqref{eq:nu-rho-a} on the forward light cone 
$\oline{V_+}\subeq \R^d$ with $c = \mu(\{0\}) = 0$. 
Then we have a natural unitary representation 
of the Poincar\'e group $G^c := P(d)^\uparrow = \R^d \rtimes \OO_{1,d-1}(\R)^\uparrow$ on 
\[ \cH := L^2(\oline{V_+},\mu) 
\quad \mbox{ by } \quad (U^\mu(x,g)f)(p) := e^{i xp} f(g^{-1}p).\] 
Analytic continuation of the time-translation group leads to the contraction semigroup 
\[ (C_t f)(p) = (U^\mu(it e_0, \1)f)(p) =e^{-t p_0} f(p), \] 
and the dilation construction from Example~\ref{ex:4.2.5}, applied to this 
contraction semigroup, leads to the Hilbert space 
\[ \cE = L^2(\R \times \oline{V_+},\zeta) 
= L^2(\R^{d+1},\zeta) \quad \mbox{ with } \quad 
d\zeta(\lambda,p) = \frac{1}{\pi}\frac{p_0}{p_0^2 + \lambda^2}\, d\lambda \, d\mu(p)\] 
(cf.\ Example~\ref{ex:3.15}). 

We consider the unitary representation $U$ of the euclidean translation 
group of $\R^d$ on~$\cE$, given by 
\begin{equation}
  \label{eq:transrep}
\big(U(x_0, \bx)f\big)(\lambda,p_0, \bp) 
= e^{-i (x_0 \lambda + \bx \bp)} f(\lambda, p_0, \bp).
\end{equation}
%A test function $\psi \in C^\infty_c(\R^d)$ then acts on $\cE$ by 
%\[ (\pi(\psi) F)(t,p_0, \bp) 
%= \hat\psi(t, \bp) F(t,p_0,\bp).\] 
The constant function $1$ on $\R^{d+1}$ is a distribution vector for $U$ 
if and only if the projected measure $\nu = (\pr_2)_* \zeta$ 
under $\pr_2(x,p_0, \bp) = (x,\bp)$, 
it tempered (cf.~Remark~\ref{rem:8.1.5} for criteria), 
and then the corresponding distribution is 
$D = \hat\nu$ \break (Lemma~\ref{lem:dist-vec}). 

It is remarkable that the 
measure $\nu$ on $\R^d$ is rotation invariant, so that 
dilation with respect to the contraction semigroup $(C_t)_{t \geq 0}$ 
leads directly from the representation $U^{\mu}$  
of the Poincar\'e group on $L^2(\R^d,\mu)$ 
to a representation $U^{\nu}$ of the euclidean motion group $E(d)$ on~$L^2(\R^d,\nu)$ by 
\[ (U^{\nu}(x,g)f)(p) := e^{-i xp} f(g^{-1}p).\] 
For $\mu = \mu_m$, the representation $U^{\mu_m}$ of the Poincar\'e group 
is irreducible because the measure $\mu_m$ lives on a single 
$\OO_{1,d-1}(\R)^\uparrow$-orbit in $\R^d$ (it is $\OO_{1,d-1}(\R)^\uparrow$-ergodic). 
As the measure $\nu_m$ is a proper superposition 
of the invariant measures on spheres of any radius, the corresponding 
representation $U^{\nu_m}$ of $E(d)$ is reducible and a direct integral 
of representations corresponding to the invariant measures on the spheres 
of radius~$r$. Since 
all measures $(\nu_m)_{m > 0}$ are equivalent to Lebesgue measure, 
all representations $(U^{\nu_m})_{m > 0}$ are actually equivalent. 

\begin{prop} \label{prop:repo-cm} 
If $\rho$ is a tame measure on $[0,\infty)$ for which 
the measure $\nu$ is tempered, then 
the rotation invariant 
distribution $\hat{\nu} \in C^{-\infty}(\R^d)$ is reflection positive 
for $(\R^d, \R^d_+, \theta)$ and $\theta(x_0, \bx) = (-x_0, \bx)$, i.e., 
\[ \int_{\R^d} \oline{\hat \psi}\cdot \theta \hat  \psi\, d\nu \geq 0 
\quad \mbox{ for } \quad f \in C^\infty_c(\R^d_+).\] 
\end{prop}

\begin{prf} (see also \cite[Prop.~6.2.5]{GJ81}) 
Writing 
$\nu = \int_0^\infty \nu_m\, d\rho(m)$ with 
$d\nu_m(p)=\frac{1}{\pi} \frac{dp} {m^2 + p^2}$, 
the assertion follows from the reflection positivity of the 
distributions $\hat{\nu_m}$ verified in 
Example~\ref{ex:dm}. 
\end{prf}

We now assume that the measure $\nu$ is tempered (cf.~Remark~\ref{rem:8.1.5}). 
Then the corresponding distribution $D = \hat\nu$ is reflection positive 
by Proposition~\ref{prop:repo-cm}. 
Let 
\[ \cF := \pr_2^*(L^2(\R^d,\nu)) \subeq \cE = L^2(\R^{d+1}, \zeta) \] 
be the image under the isometry $\pr_2^*$ from  \eqref{eq:time-0-iso}.
It coincides with $\lbr U_{C^\infty_c(\R^d)}1\rbr$ 
(see~\ref{eq:transrep}), and reflection positivity 
of $\hat\nu$ implies that the subspace \break 
$\cF_+ := \lbr U_{C^\infty_c(\R^d_+)}1\rbr$ is $\theta$-positive. 

To see how this fits the subspace $\cE_0$ and $\cE_+$ of $\cE$, 
we first note that 
\[ \cF_0 := \cF \cap \cE_0 \] 
consists of those function in $\cE_0$ that are also independent of~$p_0$. 
This is the $L^2$-space of the 
projected measure $\tilde\nu := (\pr_1)_*\nu$ on $\R^{d-1}$ for 
$\pr_1(p_0,\bp) = \bp$. Since $\nu = (\pr_2)_*\zeta$, 
we also have $\tilde\nu := (\pr_1)_*\mu$. 
According to \cite[Thm.~B.1]{NO15a}, $\cF_0 \cong L^2(\R^{d-1}, \tilde\nu)$ is non-zero 
if and only if the measure $\tilde\nu$ is tempered, which is equivalent to the 
additional condition 
\begin{equation}
  \label{eq:tilde-mu-rel}
\int_1^\infty\frac{d\rho(m)}{m} < \infty
\end{equation}
on the growth of $\rho$ at infinity. 
Assume that $\tilde\nu$ is tempered. 
Then $1$ is a distribution vector for the representation 
$U\res_{\R^{d-1}}$ on $\cE$, and the corresponding cyclic subspace 
coincides with $\cF_0 \subeq \cE_0$ 
(Lemma~\ref{lem:dist-vec}). This in turn implies that $\cF_+ \subeq \cE_+$. 
%Proposition~\ref{prop:3.15a} {\bf (avoid this ref and argue directly)}  
%now yields $\cF_+ \subeq \cE_+$. 
%In particular, $\cF_+$ is $\theta$-positive, which provides a second 
%argument for the reflection positivity of $\hat\nu$. 
Further, $\hat\cE \cong L^2(\R^d,\mu)$ contains the subspace 
$\hat\cF_0 \cong \cF_0 \cong L^2(\R^{d-1}, \tilde\nu)$ 
of functions not depending on $p_0$, and the canonical map 
$\cF_0 \to \hat\cE_0$ is unitary. 
Accordingly, the ``time zero-subspace'' 
$\hat\cF_0$ is the same on the euclidean and the Minkowski side. 
\index{time-zero subspace} 

Since $\cF_0$ is $U$-cyclic in $\cF$, the subspace 
$\hat\cF_0$ is $\hat U$-cyclic in $\hat\cE$, showing that $\hat\cF = \hat\cE$. 
Therefore the representation $U^\nu$ of the euclidean group $E(d)$ on 
$\cF$ provides a euclidean realization of the representation 
$(U^\mu, L^2(\R^d,\mu))$ of $P(d)^\uparrow$. 
To see how $\cF$ is generated from $\cF_0$, we now determine 
the corresponding positive definite operator-valued function 
\[ \phi \: \R \to B(\cF_0), \quad \phi(t) = P_0 U(t,0) P_0^*,\] 
where $P_0 \: \cE \to \cF_0$ is the orthogonal projection. This function is determined by 
the relation 
\[ \la \xi, \phi(t) \eta \ra = \la \xi, U_t  \eta\ra  
\quad \mbox{ for } \quad \xi,\eta \in \cF_0.\]  
We have 
\begin{align*}
\la \xi, \phi(t)\eta \ra 
&= \int_{\R^d} e^{-it p_0} \oline{\xi(\bp)} {\eta(\bp)}\, d\nu(\bp) 
= \int_{\R^d} e^{-it p_0} \oline{\xi(\bp)} {\eta(\bp)}
\, \Theta(p_0,\bp)\, dp_0\, d\bp \\
&= \int_{\R^{d-1}} \oline{\xi(\bp)} {\eta(\bp)} 
\int_\R e^{-it p_0} \Theta(p_0,\bp)\, dp_0\, d\bp 
= \int_{\R^{d-1}} \oline{\xi(\bp)} {\eta(\bp)} \Theta_t(\bp)\, d\bp, 
%= \int_{\R^{d-1}} \oline{\xi(\bp)} {\eta(\bp)} 
%\frac{\Theta_t(\bp)}{\Theta_0(\bp)}\, d\tilde\nu(\bp),
\end{align*}
where \eqref{eq:nu-def} yields 
\begin{align*}
\Theta_t(\bp) 
&:= \int_\R e^{-it p_0} \Theta(p_0,\bp)\, dp_0 
= \frac{1}{\pi} \int_\R \int_0^\infty  e^{-it p_0} \frac{m}{m^2 + \bp^2 + p_0^2}\, 
\, d\rho(m)\, dp_0 \\
&= \int_0^\infty  \Big(\frac{1}{\pi} \int_\R e^{-it p_0} 
\frac{m}{(m^2 + \bp^2) + p_0^2}\, \, dp_0 \Big)\, d\rho(m) \\
&= \int_0^\infty \frac{m}{\sqrt{m^2 + \bp^2}} e^{-|t|\sqrt{m^2 + \bp^2}}\, d\rho(m).
%\leq \Theta_0(\bp)
\end{align*}
Here we have used Example~\ref{ex:2.1.9} in the calculation. 
Now $d\tilde\nu(\bp) = \Theta_0(\bp) d\bp$ implies that the operator 
$\phi(t)$ on $\cF_0$ is given by multiplication with the function $\Theta_t/\Theta_0$. 
%These functions are bounded with $0\le \Theta_t/\Theta_0 \le 1$.

For the subspace $\hat\cF_0 \subeq \hat\cE$ and $f,g \in \cF_0$, the relation 
\[ \la \hat f, \hat U_t \hat g\ra 
= \la f, \theta U_t g \ra 
= \la f, U_t  g \ra = \la f,\phi(t)g \ra = \la \hat f, \phi(t)\hat g \ra \] 
implies that $\phi\res_{\R_+}$ is the positive  definite function on $\R_+$ 
corresponding to the cyclic subspace $\hat\cF_0 \subeq \hat\cE$. 
\end{ex}

\begin{ex} \label{ex:5.14}
For the special case where $\rho =\delta_m$ with $m > 0$ or $d > 2$, we have
\[ \Theta_t(\bx) = \frac{m}{\sqrt{m^2 + \bx^2}} e^{-|t|\sqrt{m^2 + \bx^2}},  
\quad \mbox{ and } \quad 
\frac{\Theta_t(\bx)}{\Theta_0(\bx)} =e^{-|t|\sqrt{m^2 + \bx^2}} \] 
is multiplicative for $t \geq 0$. This corresponds to the fact that 
$q(\cE_0) = \hat\cE$ (the Markov case; Proposition~\ref{prop:3.9}), 
which in turn is due to the fact that 
the inclusion \break $L^2(\R^{d-1}, \tilde\nu_m) \into L^2(\R^d,\mu_m)$ is surjective. 

This has the interesting consequence that, if we consider elements of $\hat\cE$ as 
functions 
\[ f \:\R_+ \times \R^{d-1} \to \C \] 
as in the preceding example, we have 
\begin{equation}
  \label{eq:time0sem}
f(t,\bp) = (\hat U_t f)(0,\bp)  =e^{-t\sqrt{m^2 + \bp^2}} f(0,\bp).
\end{equation}
This in turn leads by analytic continuation to 
\begin{equation}
  \label{eq:time0grp}
 f(it,\bp) = (U^c_t f)(0,\bp)  =e^{it\sqrt{m^2 + \bp^2}} f(0,\bp).
\end{equation}
These formulas provide rather conceptual direct arguments for 
formulas like \cite[Prop.~6.2.5]{GJ81}. 
\end{ex}

\index{representation!positive energy}
\begin{rem} A unitary representation $(U, \cH)$ of the Poincar\'e group 
is said to be of {\it positive energy} if the spectrum of the time translation 
group is non-negative. In view of the covariance with respect to the 
Lorentz group $\OO_{1,d-1}(\R)^\uparrow$, 
this is equivalent to the spectral measure of 
$U\res_{\R^d}$ to be supported in the closed forward light cone 
$\oline{V_+}$ because this is the set of all orbits of $\OO_{1,d-1}(\R)^\uparrow$ on which 
the function $p_0$ is non-negative. 

If such a representation is multiplicity free on $\R^d$, then
$\cH \cong  L^2(\oline{V_+},\mu)$ for a measure $\mu$ on $\oline{V_+}$ 
which is quasi-invariant under $\OO_{1,d-1}(\R)^\uparrow$. Since the action of $\OO_{1,d-1}(\R)^\uparrow$ 
on $\oline{V_+}$ has a measurable cross section and every orbit 
carries an invariant measure, the measure $\mu$ can be chosen 
$\OO_{1,d-1}(\R)^\uparrow$-invariant. The representation $U$ is irreducible 
if and only if the measure $\mu$ is ergodic, i.e., 
$\mu = \mu_m$ for some $m \geq 0$ (with $m > 0$ for $d = 1$) or 
$\mu = \delta_0$ (the Dirac measure in $0$). 
\end{rem}

For all the multiplicity free representations 
$(U^\mu, L^2(\oline{V_+}, \mu))$, Example~\ref{ex:3.16} 
provides a euclidean realization in the dilation space
$\cE = L^2(\R_+ \times \oline{V_+},\zeta)$, as far as the representation of the 
subgroup $\R^d \rtimes \OO_{d-1}(\R)$ is concerned. 
The subspace $\cE_0 \subeq \cE$ is invariant under the subgroup 
$G^\tau \cong \R^{d-1} \rtimes \OO_{d-1}(\R)$, which also implies the 
invariance of $\cE_+$ under this group. 

A euclidean realization for the full group is obtained 
in Example~\ref{ex:3.15} for irreducible representations, 
i.e., $\mu = \mu_m$. In the general case we assume that 
$\nu$ is tempered. Then the following theorem 
is the bridge between  the reflection  positive 
representation $U^\nu$ of $E(d)$ on $\cF \cong L^2(\R^d,\nu)$ 
and the representation 
$U^c\cong U^\mu$ of the Poincar\'e group on $\hat\cF \cong L^2(\R^d, \mu)$. 

\begin{thm} \label{thm:6.11} 
If $\nu$ is tempered, then $1 \in \cE^{-\infty}$ is a 
reflection positive distribution vector for the representation $U$ of $\R^d$.
Accordingly, we obtain a reflection positive representation 
of $\R^d$ on the subspace $\cF \subeq \cE$ generated by $U^{-\infty}(C^\infty_c(\R^d))1$. 
The corresponding reflection positive distribution $\hat\nu$ on $\R^d$ is 
rotation invariant, so that $\cF$ 
carries a reflection positive representation of $E(d)$ for which 
$\cF_0$ and $\cF_+$ are invariant under $H := E(d)^\tau \cong 
\R^{d-1} \rtimes \OO_{d-1}(\R)$. 

Moreover, $\hat\cF \cong L^2(\oline{V_+},\mu)$, 
$q \: \cF_+ \to \hat\cF$ is $H$-equivariant and we have the relation 
$\hat{U(t,0)} = U^\mu(it,{\bf 0},\1)$ for the positive energy representation 
$U^\mu$ of the Poincar\'e group $P(d)^\uparrow$ on $\hat\cF$. 
\end{thm}

\begin{prf} We have already seen that $1 \in \cE^{-\infty}$ is equivalent to 
$\nu$ being tempered (Lemma~\ref{lem:dist-vec}). 
To determine the corresponding space 
$\hat \cF$, we have to take a closer look at the corresponding 
reflection positive distribution $D = \hat\nu$ for 
$(\R^d, \R^d_+, \theta)$ (Proposition~\ref{prop:repo-cm}). 
In view of \cite[Prop.~2.12]{NO14}, this follows 
if we can show that $D\res_{\R^d_+}$ coincides with the 
Fourier--Laplace transform 
\[ \cF\cL(\mu)(x) := \int_{\R^d_+} 
e^{-x_0 p_0} e^{i \bx \bp}\, d\mu(p). \] 

First we observe that the temperedness of $\mu$ implies that 
$\cF\cL(\mu)(x)$ exists pointwise and defines an analytic function on $\R^d_+$. 
Here the main point is that, on $\oline{V_+}$ we have 
$p^2 = p_0^2 + \bp^2 \leq 2 p_0^2$ (cf.\ \cite[Ex.~4.12]{NO14}). 
We have  
\begin{align*}
 \cF\cL(\mu)(x) 
&= \int_{\oline{V_+}} e^{-x_0 p_0} e^{i \bx \bp}\, d\mu(p) 
= \int_0^\infty \int_{\R^{d-1}} e^{-x_0 p_0} e^{i \bx \bp}\, d\mu(p_0, \bp) \\
&= \int_0^\infty \int_{\R^{d-1}} \Big(\frac{1}{\pi} \int_\R 
e^{itx_0}\frac{p_0}{p_0^2 + t^2}\, dt\Big)\, e^{i \bx \bp}\, d\mu(p_0, \bp) \\
&= \int_{\R \times \R^d}  e^{i(tx_0 + \bx \bp)}\, d\zeta(t, p_0, \bp) 
= \int_{\R^d}  e^{i(t x_0 + \bx \bp)}\, d\nu(t, \bp)  = \hat\nu(x).
\end{align*}
If $\mu$ is infinite, then the triple integral only exists as 
an iterated integral in the correct order, not in the sense that 
the integrand is Lebesgue integrable. One can deal with this problem 
by integrating against a test function on $\R^d_+$, 
and then the above calculation shows that $\cF\cL(\mu)$ coincides with 
$\hat\nu$ on $\R^d_+$ as a distribution. 
\end{prf}

\section{The conformally invariant case} 
\label{subsec:8.3}

In this section we study the special case where the measure
$\mu$ on $\oline{V_+}$ is semi-invariant under homotheties. This provides a 
bridge to the complementary series representations of 
$\OO_{1,d+1}(\R)^\uparrow$ discussed in Subsection~\ref{subsec:compser} 
 because then the representation of 
$E(d)$ on $L^2(\R^d,\nu)$ extends to the conformal group 
$\OO_{1,d+1}(\R)$ of $\R^d$. 

\begin{lem} \label{lem:7.4.1} 
{\rm(\cite[Lemma~5.17]{NO15b})} 
An $\OO_{1,d-1}(\R)^\uparrow$-invariant measure \break 
$\mu = \int_0^\infty \mu_m\, d\rho(m)$ on $\oline{V_+}$ is semi-invariant 
under homotheties if and only if 
\[ d\rho(m) = m^{s-1}\, dm \quad \mbox{ for some } \quad s \in \R. \]
If this is the case, then 
$\rho$ is tempered if and only if $s > 0$, and 
$\mu$ is tempered if $d > 1$ or $s > 1$. 
For $d > 1$, the measure $\tilde\nu$ on $\R^{d-1}$ is tempered if and 
only if $s > 1$. For $d = 1$, the measure $\mu$ is never finite. 
\end{lem}

{}From now on we write $d\rho_s(m) = m^{s-1}\, dm$ on $[0,\infty)$. 
As the  measure $\mu$ is semi-invariant 
under homotheties, we can expect the corresponding
representation of the Poincar\'e group to extend to the conformal group 
$\SO_{2,d}(\R)$ of Minkowski space. 

\begin{lem} \label{lem:6.13} 
{\rm(\cite[Lemma~5.18]{NO15b}, Proposition~\ref{prop:posdef-powerfct})} 
The measure $\nu = \Theta_s \cdot dp$ corresponding to $\rho_s$ 
is tempered if and only if 
$0 < s < 2$ for $d > 1$ and if $0 < s < 1$ for $d = 1$. In this case 
$\Theta_s$ is a multiple of $\|p\|^{s-2}$ and the Fourier transform 
$\hat\nu$ is a positive multiple of $\|x\|^{-d+2-s}$. 
\end{lem}

The preceding lemma implies in particular that the distribution 
$\|x\|^{-a}$ on $\R^d$ is reflection  positive for $d - 2 < a < d$, 
which has been obtained in \cite[Prop.~6.1]{NO14}, 
\cite[Lemma~2.1]{FL10} and \cite[Lemma~3.1]{FL11} by other means. 
This connection is made more precise in the following theorem:

\begin{thm} \label{thm:6.14} Let $0 < s < 2$, resp., $1 < s < 2$ for $d =1$. 
Then 
  \begin{enumerate}
  \item[\rm(i)] The canonical representation 
of the conformal motion group 
\[ CE(d) := \R^d \rtimes (\OO_d(\R) \times \R^\times_+) \] 
on $\cE := L^2(\R^d, \nu) \cong \cH_D$ for 
$D(x) = \|x\|^{-d+2-s}$ 
extends to a complementary series representation of the orthochronous 
euclidean conformal group $\OO_{1,d+1}(\R)_+$. 
  \item[\rm(ii)] The corresponding representation 
of the conformal Poincar\'e group 
\[ CP(d)^\uparrow := \R^d \rtimes (L^{\up} \times \R^\times_+) \] 
is irreducible and extends to a representation of a covering of 
the relativistic conformal group $\SO_{2,d}(\R)_0$.   
  \end{enumerate}
\end{thm}

\begin{prf} (i) From Lemma~\ref{lem:6.13} 
we know that $\cE := L^2(\R^d,\nu)$ can be identified 
with the Hilbert space $\cH_D$ obtained by completion of 
$C^\infty_c(\R^d)$ with respect to the scalar product 
\[ \la \phi, \psi \ra_s := \int_{\R^d} \int_{\R^d}
\frac{\oline{\phi(x)}\psi(y)}{\|x-y\|^{d-2+s}}\, dx\ dy\] 
(cf.~Definition~\ref{def:twistdist}). 
Now Theorem~\ref{thm:1.5} 
%\cite[Prop.~6.1]{NO14} (and the proof of Lemma~5.5 in loc.\ cit.) 
implies that the representation 
of $E(d)$ on this space extends to an irreducible complementary 
series representation of the conformal group $\OO_{1,d+1}(\R)_+$ 

(ii) The irreducibility of the representation $U^c$ follows from the 
transitivity of the action of $\R^\times_+ \OO_{1,d-1}(\R)^\uparrow$ on the open light cone $V_+$. 
To see that this representation extends to $\SO_{2,d}(\R)_0$, 
we can use the fact that the representation 
$U$ of the conformal group  of $\R^d$ is reflection positive 
with respect to the open subsemigroup 
of strict compressions of the open half space $\R^d_+$ in the conformal 
compactification $\bS^d$. 
As explained in \cite[\S\S 6,10]{JOl00}, see also \cite{HN93}, \cite{JOl98}, 
the reflection positivity and the L\"uscher--Mack 
Theorem now provide an irreducible  representation of the 
simply connected $c$-dual group $G^c$ on~$\hat\cE$. 
\end{prf}

%\section*{Notes on Chapter~\ref{ch:8}} 

\end{bibunit}

\chapter{Reflection positivity and stochastic processes} 
\label{ch:9} 

\begin{bibunit}

In this chapter we describe some recent generalizations of 
classical results by Klein and Landau \cite{Kl78, KL75} 
concerning the interplay between reflection positivity and stochastic 
processes. Here the main step is the passage from the symmetric 
semigroup $(\R,\R_+,-\id_\R)$ to  more general triples 
$(G,S,\tau)$. This leads to the concept of a 
$(G,S,\tau)$-measure space generalizing Klein's Osterwalder--Schrader 
path spaces for $(\R,\R_+,-\id_\R)$. 
A key result is the correspondence between 
$(G,S,\tau)$-measure spaces and the corresponding 
positive semigroup structures on the Hilbert space $\hat\cE$. 

The exposition in this chapter is  minimal 
in the sense that the main results are explained and full definitions 
are given. For more details we refer to  \cite{JNO16a}.

\section{Reflection positive 
group actions on measure spaces} 
\label{sec:1}

We start with the basic concepts related to $(G,S,\tau)$-measure spaces 
which provide a measure theoretic perspective on reflection positive 
representations of symmetric semigroups~$(G,S,\tau)$. 

\index{$(G,\tau)$-measure space}
\begin{defn} \label{def:0.1G} Let $(G,\tau)$ be a symmetric group. 
A {\it $(G,\tau)$-measure space} is a quadruple 
$((Q,\Sigma,\mu), \Sigma_0, U,\theta)$ consisting of the following ingredients: 
\begin{itemize}
\item[\rm(GP1)]\quad\quad a measure space $(Q,\Sigma,\mu)$, 
\item[\rm(GP2)]\quad\quad a sub-$\sigma$-algebra $\Sigma_0$ of $\Sigma$, 
\item[\rm(GP3)]\quad\quad  a measure preserving action $U \: G_\tau \to 
\Aut(\cA)$ on the von Neumann algebra $\cA := L^\infty(Q,\Sigma,\mu)$,  
for which the corresponding unitary representation on 
$L^2(Q,\mu)$ is continuous, and 
\item[\rm(GP4)]\quad\quad $\theta = U_\tau$ satisfies 
$\theta E_0 \theta =E_0$, where 
$E_0 \: L^\infty(Q,\Sigma,\mu) \to L^\infty(Q,\Sigma_0,\mu)$ is 
the conditional expectation. 
\item[\rm(GP5)]\quad\quad $\Sigma$ is generated by the sub-$\sigma$-algebras $\Sigma_g 
:= U_g \Sigma_0$, $g \in G$. 
\end{itemize} \index{$(G,\tau)$-probability space}
If $\mu$ is a probability measure, we speak of a {\it $(G,\tau)$-probability space}. 
If $S = S^\sharp \subeq G$ is a symmetric subsemigroup, then we 
write $\Sigma_\pm$ for the sub-$\sigma$-algebra generated by 
$(\Sigma_s)_{s \in S^{\pm 1}}$, 
and $E_\pm$ for the corresponding conditional expectations.
%\begin{footnote}{This concept generalizes naturally to subsets 
%$G_+ \subeq G$ which are not semigroups. 
%Here we define $\Sigma_+$ as the $\sigma$-algebra generated by 
%$g\Sigma_0$, $g \in G_+$.}\end{footnote}
\end{defn} 

\index{$(G,\tau)$-measure space, reflection positive}
\index{$(G,S,\tau)$-measure space}
\begin{defn} \label{def:9.1.2} 
(a) A $(G,\tau)$-measure space is called 
{\it reflection positive with respect to the symmetric subsemigroup $S$} if 
\[ \la \theta f, f \ra\geq 0 \quad \mbox{  for } \quad f \in 
\cE_+ := L^2(Q,\Sigma_+, \mu).\] 
This is equivalent to $E_+ \theta E_+ \geq 0$ as an operator on 
$L^2(Q,\Sigma,\mu)$ and obviously implies $\theta E_0 =E_0$.
If this condition is satisfied and, in addition, 
$\Sigma_0$ is invariant under the unit group $H(S) := S\cap S^{-1}$, 
then we call it a {\it $(G,S,\tau)$-measure space.}
\begin{footnote}{
Note that $E_+ \theta E_+ \geq 0$ is equivalent to the kernel 
$K^\theta(A,B) := \mu(A \cap \theta(B))$ on $\Sigma_+$ being positive definite, i.e., 
the kernel $K(A,B) := \mu(A\cap B)$ on $\Sigma$ is reflection positive 
with respect to $(\Sigma, \Sigma_+, \theta)$ 
(Definition~\ref{def:2.1.3}).}
\end{footnote}
 
\index{$(G,S,\tau)$-measure space, Markov}
(b) A {\it Markov $(G,S,\tau)$-measure space}  is a 
$(G,S,\tau)$-measure space with the {\it Markov property} $E_+ E_- = E_+ E_0 E_-$ 
(cf.\ Definition~\ref{def:markov-cond}). \index{Markov property}
\end{defn}

Proposition~\ref{prop:3.9} immediately provides a reflection positive 
representation on the corresponding $L^2$-space: 

\begin{prop} \label{prop:markov} 
For a $(G,S,\tau)$-measure space $((Q,\Sigma,\mu), \Sigma_0, U,\theta)$,  
we put $\cE := L^2(Q,\Sigma,\mu)$, $\cE_0 := L^2(Q,\Sigma_0,\mu)$ and 
$\cE_\pm:= L^2(Q,\Sigma_\pm,\mu)$. Then the natural action of $G$ on $\cE$ 
defines a reflection positive representation of $(G,S,\tau)$. 

The Markov property is equivalent to the 
natural map $\cE_0 \to \hat\cE$ being unitary 
and this implies that the 
positive definite function $\vphi \: S \to B(\cE_0), \vphi(s) = E_0 U_s E_0$ 
is multiplicative and the unitary representation $U$ of $G$
on $(\cE,\cE_+,\theta)$ is a euclidean realization of the $*$-representation 
$(\vphi,\cE_0)$ of $(S,\sharp)$. 
\end{prop}

\begin{ex} \label{ex:klein} 
Typical examples arise in QFT as follows. 
Let $\cE$ be a real Hilbert space 
and $X = \cE^*$ be its algebraic dual space, i.e., the space 
of all linear functionals $\cE \to \R$, continuous or not. 
On this set we consider the smallest $\sigma$-algebra for which all 
evaluation functionals $\phi(\xi)(\alpha) := \alpha(\xi)$ 
are measurable. Then there exists a Gaussian measure 
$\mu$ on $X$ 
such that any tuple $(\phi(\xi_1), \ldots, \phi(\xi_n))$ is jointly 
Gaussian with covariance $(\la \xi_i, \xi_j\ra)_{1 \leq i,j \leq n}$ 
(\cite[Ex.~4.3]{JNO16a}, \cite[Thm.~2.3.4]{Sim05}). 
The orthogonal group $\OO(\cE)$ acts in a measure preserving way 
on $X$ by $U\alpha := \alpha \circ U^{-1}$. 

If we start with a reflection positive unitary representation
$(U,\cE,\cE_+,\theta)$ of $(G,S,\tau)$,
for which $\cE_0$ is $U$-cyclic and $\cE_+$ is generated by 
$U_S \cE_0$, then all this structure is reflected 
in $(X,\Sigma,\mu)$. In particular, we obtain a measure 
preserving action of $G_\tau$ on $X$. 
We write $\Sigma_0 \subeq \Sigma$ for the smallest 
$\sigma$-subalgebra for which all functions $(\phi(\xi))_{\xi \in \cE_0}$ are measurable. 
Then $\Sigma_+$ is generated by the translates 
$(U_s\Sigma_0)_{s \in S}$ and (GP1-5) are satisfied. 
%
%Here $\cF := L^2(X,\Sigma,\mu)$ can be identified with the 
%symmetric Fock space of $\cE$ 
%and $\cF_+ := L^2(X,\Sigma_+,\mu)$ with the 
%Fock space of $\cE_+$ and the representation of $(G,S,\tau)$ 
%on $(\cF,\cF_+, \theta)$ is reflection positive 
%(\cite[Prop.~4.6]{JNO16a}). 
\end{ex}

The following concept aims at an axiomatic characterization of the 
corresponding semigroup representations on the spaces $\hat\cE$.
It generalizes the corresponding classical concepts 
for the case $(G,S,\tau) = (\R,\R_+,-\id_\R)$ (\cite{Kl78} for (a) and  
\cite{KL75} for~(b)). 

\begin{defn} \label{def:1.5G} \index{positive semigroup structure}
(a) A {\it positive semigroup structure} for a 
symmetric semigroup $(G,S,\tau)$ 
is a quadruple $(\cH, P, \cA, \Omega)$  consisting of 
\begin{itemize}
\item[\rm(PS1)]\quad\ \  a Hilbert space $\cH$,  
\item[\rm(PS2)]\quad\ \  a strongly continuous $*$-representation 
$(P_s)_{s \in S}$ of $(S,\sharp)$ by contractions on~$\cH$,  
\item[\rm(PS3)]\quad\ \  a commutative von Neumann algebra $\cA$ on $\cH$ normalized 
by the operators $(P_h)_{h \in S\cap S^{-1}}$, and 
\item[\rm(PS4)]\quad\ \ a unit vector $\Omega \in \cH$, such that 
  \begin{enumerate}
  \item[\rm(i)]\quad $P_s\Omega = \Omega$ for every $s \in S$.
  \item[\rm(ii)]\quad $\Omega$ is cyclic for the (not necessarily selfadjoint) 
subalgebra $\cB \subeq B(\cH)$ generated by $\cA$ and $\{ P_s \: s \in S\}$. 
  \item[\rm(iii)]\quad For positive elements $A_1,\ldots, A_n \in \cA$ 
and $s_1,\ldots, s_{n-1} \in S$, we have 
\[ \la \Omega, A_1 P_{s_1} A_2 \cdots P_{s_{n-1}} A_n \Omega \ra \geq 0.\] 
  \end{enumerate}
\end{itemize}

\index{positive semigroup structure!standard}
(b)  A {\it standard positive semigroup structure} for a 
symmetric semigroup $(G,S,\tau)$ consists of a $\sigma$-finite measure 
space $(M, \fS, \nu)$ and 
\begin{itemize}
\item[\rm(SPS1)]\quad\quad a representation 
$(P_s)_{s \in S}$ of $S$ on $L^\infty(M,\nu)$ by positivity preserving operators, i.e., 
$P_s f \geq 0$ for $f \geq 0$.
\item[\rm(SPS2)]\quad\quad $P_s 1 = 1$ for $s \in S$ (the Markov condition). 
\item[\rm(SPS3)]\quad\quad $P$ is involutive with respect to $\nu$, i.e., 
$\int_M P_s(f) h d\nu = \int_M f P_{s^\sharp}(h)\, d\nu$ for 
$s \in S,f,h \geq 0.$
\item[\rm(SPS4)]\quad\quad $P$ is strongly continuous in measure, i.e., 
for each $f \in L^1(M ,\nu) \cap L^\infty(M,\nu)$ and every $\delta > 0$, 
$s_0 \in S$, we have 
$\lim_{s \to s_0} \nu(\{|P_s f - P_{s_0}f| \geq \delta\}) = 0$. 
\end{itemize}
\end{defn}

The main difference between these two concepts is that 
(b) concerns the situation where $\cH$ is an $L^2$-space, but 
it also leaves some additional freedom because the measure 
$\nu$ is not required to be finite 
so that the constant function $1$ need not be~$L^2$. 

The following proposition shows that the requirement that 
$\Omega$ is cyclic for $\cA$ describes those positive semigroup 
structures which are standard. 

\begin{prop}  \label{prop:2.19}
Let $(M,\fS,\nu)$ be a probability space and $(P_s)_{s \in S}$  be a positivity 
preserving continuous $*$-representation of $(S, \sharp)$ 
by contractions on $L^2(M,\nu )$, i.e., 
\[  P_s 1 = 1 \quad \mbox{ and } \quad 
  P_s f \geq 0  \quad \mbox{ for } \quad f\geq 0, s\in S.\]   
Then $(L^2(M), Q, L^\infty(M), 1)$ is a standard 
positive semigroup structure for which 
$1$ is a cyclic vector for $L^\infty(M)$. 

Conversely, let $(\cH,P, \cA, \Omega)$ be a positive semigroup structure for 
which $\Omega$ is a cyclic vector for~$\cA$. Then there exists a probability 
space $M$ and a positivity preserving semigroup $(Q_s)_{s \in S}$ 
on $L^2(M)$ such that 
$(\cH,P,\cA,\Omega) \cong (L^2(M), Q, L^\infty(M), 1)$ as positive 
semigroup structures. 
\end{prop}

The following theorem characterizes the positive semigroup structures arising 
in the Markov context as those for which $\Omega$ is a cyclic vector for $\cA$, 
which is considerably stronger than condition (PS4)(b). 

\begin{thm} \label{thm:kl1}
  Let $((Q,\Sigma,\mu), \Sigma_0, U,\theta)$ be a $(G,S,\tau)$-probability space and let
$(\hat\cE, \hat U, \cA, \Omega)$  be its associated positive semigroup structure. Then \break % the tuple \break  
$((Q,\Sigma,\mu), \Sigma_0, U,\theta)$ is Markov if and only if $\Omega$ is 
$\cA$-cyclic in~$\hat\cE$. 
\end{thm}

\begin{prf} The Markov property is equivalent to 
$q(\cE_0) = \hat\cE$ (Proposition~\ref{prop:markov}). Since 
$\cA \cdot 1$ is dense in $\cE_0$, 
this is equivalent to $\Omega = q(1)$ being 
$\cA$-cyclic in~$\hat\cE$. 
\end{prf}

\begin{ex} \label{ex:osci} (The real oscillator semigroup)  
We consider the Hilbert space $\cH = L^2(\R^d)$, with respect to Lebesgue measure. 

(a) On $\cH$ we have a unitary representation by the group $\GL_d(\R)$ by 
\[ (T_h f)(x) := |\det(h)|^{-d/2} f(h^{-1}x)\quad \mbox{ for} \quad h \in \GL_d(\R), x \in \R^d,\]
and we also have two representations of the additive abelian semigroup 
$\Sym_d(\R)_+$ (the convex cone of positive semidefinite matrices): 
\begin{enumerate}
\item[\rm(1)] Each $A \in \Sym_d(\R)_+$ defines a multiplication operator 
$(M_A f)(x) := e^{-\la Ax,x\ra} f(x)$ which is positivity preserving on $L^\infty(\R^n)$ 
but does not preserve $1$; it preserves the Dirac measure $\delta_0$ in  the origin. 
\item[\rm(2)] Each $A \in \Sym_d(\R)_+$ specifies a uniquely determined 
(possibly degenerate) Gaussian measure $\mu_A$ on $\R^d$ whose Fourier transform is 
given by $\hat\mu_A(x) = e^{-\la Ax,x\ra/2}$. Then 
the convolution operator $C_A f := f * \mu_A$ is positivity preserving 
and leaves Lebesgue measure on $\R^d$ invariant. For $A = \1$, 
we thus obtain the heat semigroup as $(\mu_{t\1})_{t \geq 0}$. 
\end{enumerate}

Any composition of these $3$ types of operators $T_h, M_A$ and $C_A$ 
is positivity preserving on $L^\infty(\R^d)$, and they generate a 
$*$-representation of the 
Olshanski semigroup $S := H \exp(C)$ in the symmetric Lie group 
$G := \Sp_{2d}(\R)$, where
$H = \GL_d(\R)$, $C = \Sym_d(\R)_+ \times \Sym_d(\R)_+ \subeq 
\Sym_d(\R) \oplus \Sym_d(\R) = \fq$, and 
\[ \tau\pmat{A & B \\ C & - A^\top} = \pmat{A & -B \\ -C & - A^\top}  
\quad \mbox{ for } \quad 
\pmat{A & B \\ C & - A^\top} \in \sp_{2d}(\R) \ \mbox{ with } \ 
B^\top = B, C^\top = C  \] 
(cf.\ Examples~\ref{ex:semigroups}). 
The real Olshanski semigroup $S$ is the fixed point set of an antiholomorphic involutive 
automorphism of the so-called oscillator semigroup $S_\C = G^c \exp(iW)$ 
which is a 
complex Olshanskii semigroup (\cite{How88, Hi89}). 
The elements in the interior of $S$ act on $L^2(\R^d)$ by kernel operators 
with positive Gaussian kernels and the elements of $S_\C$ correspond to 
complex-valued Gaussian kernels. The semigroup $S$ contains many 
interesting symmetric one-parameter semigroups such as the Mehler semigroup  
$e^{-t H_{\rm osc}}$ generated by the oscillator Hamiltonian
\begin{equation}
  \label{eq:osc-ham}
H_{\rm osc}= -\sum_{j = 1}^n \partial_j^2 + \frac{1}{4} \sum_{j = 1}^n x_j^2 - \frac{d}{2} \1
\end{equation}
which fixes the Gaussian $e^{-\|x\|^2/4}$.

(b) The subsemigroup $S := \Sym_d(\R)_+ \rtimes \GL_d(\R) \subeq \Sp_{2d}(\R)$ 
also is a symmetric subsemigroup of $(G,\tau)$ with 
$G = \Sym_d(\R) \rtimes\GL_d(\R)$ and $\tau(A,g) = (-A,g)$. 
Here the commutative von Neumann algebra $\cA = L^\infty(\R^d)$ 
is invariant under 
conjugation with the operators $T_h$, so that 
$(A,h) \mapsto C_A T_h$ defines a $*$-representation of $(S,\sharp)$ 
that leads to a standard positive semigroup structure on $L^2(\R^d)$. 
\end{ex}

\section{Stochastic processes indexed by Lie groups} 

We now introduce stochastic processes where the more
common index set $\R$ is replaced by a Lie group~$G$. The
forward direction is then given by a subsemigroup $S$ of~$G$. 
So called stationary stochastic processes correspond naturally 
to measure preserving $G$-actions on spaces $G^Q$ 
of all maps $Q \to G$. 
 
\begin{defn} \label{def:g-process} 
Let $(Q,\Sigma,\mu)$ be a probability space.  \index{stochastic process}
A {\it stochastic process} indexed by a group $G$ is a family $(X_g)_{g \in G}$ of measurable 
functions $X_g \: Q \to (B,\fB)$, where $(B,\fB)$ is a measurable space, called the {\it state 
space} of the process.  \index{stochastic process!state space of}
It is called {\it full} if, up to sets of measure zero, 
$\Sigma$ is the smallest $\sigma$-algebra  for which all functions 
$(X_g)_{g \in G}$ are measurable.  
\index{stochastic process!full}
\index{stochastic process!distribution of} 

For such a process, we obtain a measurable map 
\[ \Phi \: Q \to B^G, \quad \Phi(q)=(X_g(q))_{g \in G} \] 
with respect to the product $\sigma$-algebra $\fB^G$. 
Then $\nu := \Phi_*\mu$ is a measure on $B^G$, called the 
{\it distribution of the process $(X_g)_{g \in G}$}. 
This measure is uniquely determined by the measures $\nu_\bg$ on $G^n$, 
obtained for any finite tuple $\bg := (g_1, \ldots, g_n) \in G^n$ 
as the image of $\mu$ under the map 
$X_\bg = (X_{g_1}, \ldots, X_{g_n}) \: Q \to B^n$ 
(cf.~\cite[\S 1.3]{Hid80}).  If $\bg=(g)$ for some $g\in G$, then we
write $\nu_g$ for $\nu_{\bg}$.

\index{stochastic process!stationary} 
\index{stochastic process!$\tau$-invariant} 

The process $(X_g)_{g \in G}$ is called {\it stationary} if the corresponding 
distribution on $B^G$ is invariant under the translations 
\[ (U_g \nu )_h := \nu_{g^{-1}h} \quad \mbox{ for } \quad g,h \in G.\] 
If $\tau \in \Aut(G)$ is an automorphism, then we call the process {\it $\tau$-invariant}
if its distribution is invariant under 
$(\tau\nu )_h := \nu_{\tau^{-1}(h)}$ for $h \in G.$ 
\end{defn}

The connection with $(G,S,\tau)$-measure spaces is now easily described: 

\begin{ex} \label{ex:2.9} Let $(G,\tau)$ be a symmetric Lie group and 
$(X_g)_{g \in G}$ be a stationary, $\tau$-invariant, full stochastic process on 
$(Q,\Sigma_Q,\mu_Q)$. Then its distribution $(B^G, \fB^G, \nu)$ satisfies the conditions 
(GP1,2,4,5) of a $(G,\tau)$-probability space with respect to the canonical 
actions of $G$ and 
$\tau$ on $B^G$, where $\Sigma_0$ is the $\sigma$-algebra generated by 
$(X_h)_{h \in G^\tau}$, i.e., the smallest subalgebra for which these functions  are measurable. 
In this context (GP3) is equivalent to the continuity of the unitary representation of 
$G$ on $L^2(B^G,\fB^G,\nu)$. 
\end{ex}

\section{Associated positive semigroup structures and 
reconstruction}

The main result of this section is the 
Reconstruction Theorem.  
It  asserts that, if $G = S \cup S^{-1}$, positive semigroup 
structures come from $(G,S,\tau)$-measure spaces. 
For a subsemigroup $S \subeq G$, we consider the 
left invariant {\it partial order $\prec_S$} on $G$ defined by 
$g \prec_S h$ if $g^{-1}h \in S$, i.e., $h \in gS$. 
 \index{partial order $\prec_S$}

\begin{lem} \label{lem:1.6G}  
Let   $((Q,\Sigma,\mu), \Sigma_0, U,\theta)$ be a  $(G,S,\tau)$-measure space. 
Consider the von Neumann algebra $\cA := L^\infty(Q,\Sigma_0,\mu)$ on 
\[ \cE := L^2(Q,\Sigma,\mu) \supeq \cE_+ := L^2(Q,\Sigma_+, \mu)\supeq 
\cE_0 := L^2(Q,\Sigma_0,\mu),\] 
and the canonical map $q \: \cE_+\to \hat\cE$. Then the following assertions hold: 
\begin{enumerate} 
\item[\rm(a)]\ For 
$f \in \cA$, let $M_f$ denote the corresponding multiplication operator on $\cE$. 
Then there exists a bounded operator $\hat M_f \in B(\hat\cE)$ with 
$q \circ M_f\res_{\cE_+}  = \hat M_f \circ q$ and 
$\|\hat M_f \| = \|f\|_\infty$. 
\item[\rm(b)]\ $U(f) := \hat M_f$ is a faithful weakly continuous representation 
of the commutative von Neumann algebra $\cA$ on $\hat\cE$. 
\item[\rm(c)]\ In the Markov case we identify $\hat\cE$ with 
$\cE_0$ and $q$ with $E_0$ {\rm(Proposition~\ref{prop:markov})}. 
For \break 
$g_1 \prec_S g_2 \prec_S \cdots \prec_S g_n$ in $G$, 
non-negative functions $f_1, \ldots, f_n \in \cA$ and $f_{g_j} := U_{g_j} f_j$, we 
have 
\[ \int_Q f_{g_1} \cdots f_{g_n}\,  d\mu 
= \int_Q \hat M_{f_1} \hat U_{g_1^{-1} g_2} \cdots \hat M_{f_{n-1}} 
\hat U_{g_{n-1}^{-1} g_n} \hat M_{f_n} 1\, d\mu. \] 
\end{enumerate}
If, in addition, $\mu$ is finite, 
then $\Omega := \mu(Q)^{-1/2} q(1)$ satisfies: 
\begin{enumerate} 
\item[\rm(d)] For $g_1 \prec_S g_2 \prec_S \cdots \prec_S g_n$ in $G$, 
$f_1, \ldots, f_n \in \cA$ and $f_{g_j} := U_{g_j} f_j$, we have 
\[ \int_Q f_{g_1} \cdots f_{g_n}\,  d\mu 
= \Big\la  \hat M_{f_1} \hat U_{g_1^{-1} g_2} \cdots \hat M_{f_{n-1}} 
\hat U_{g_{n-1}^{-1} g_n}   \hat M_{f_n} \Omega, \Omega  \Big\ra. \] 
\item[\rm(e)] $\Omega$ is a separating vector for $\cA$ 
and $\hat U_s \Omega = \Omega$ for every $s \in S$. 
\item[\rm(f)] $\Omega$ is cyclic for the algebra $\cB$ generated by $\cA$ and 
$(\hat U_s)_{s \in S}$. 
\end{enumerate}
\end{lem}

\begin{defn} The preceding lemma shows that, if 
$((Q,\Sigma,\mu), \Sigma_0, U,\theta)$ is a finite 
$(G,S,\tau)$-measure space, 
then $(\hat\cE, \hat U, \cA, q(1))$  is a positive semigroup 
structure for $\cA = \{ \hat M_f \: f \in L^\infty(Q,\Sigma_0,\mu)\}$. 
We call it the  {\it associated positive semigroup structure.} 
\end{defn}\index{positive semigroup structure!associated}

We now turn to our version of Klein's Reconstruction Theorem. 
Note that every discrete group is in particular a $0$-dimensional Lie group, 
so that the following theorem applies in particular to discrete groups. 

\index{Theorem!Reconstruction, for positive semigroup structures}
\begin{thm} \label{thm:1.8G} {\rm(Reconstruction Theorem)}  
Let $(G,S,\tau)$ be a symmetric semigroup satisfying $G= S \cup S^{-1}$. Then 
the following assertions hold: 
\begin{itemize}
\item[\rm(a)]\ Every positive semigroup structure for $(G,S,\tau)$ 
is associated to some $(G,S,\tau)$-probability space $((Q,\Sigma,\mu), \Sigma_0, U,\theta)$. 
\item[\rm(b)]\ Every standard positive semigroup structure 
for $(G,S,\tau)$ is associated to some $(G,S,\tau)$-measure 
space $((Q,\Sigma,\mu), \Sigma_0, U,\theta)$ which is unique up to 
$G$-equivariant isomorphism of measure spaces. 
\end{itemize}
\end{thm}

\begin{rem} \label{rem:mark-meas}  
Without going into details of the proof, it is instructive to take a closer 
look at the construction of the measure space $(Q,\Sigma,\mu)$ 
in the proof of Theorem~\ref{thm:1.8G} in \cite{JNO16a}. 
Here $G$ acts unitarily on the Hilbert space $\cH \cong L^2(M,\fS, \nu)$. 
For simplicity, we assume that $(M,\fS, \nu)$ 
(cf.\ Proposition~\ref{prop:2.19}) is a polish space, 
i.e., $M$ carries a topology for which it is 
completely metrizable and separable and $\fS$ is the $\sigma$-algebra of Borel sets. 
Then \cite[Cor.~35.4]{Ba96} 
implies the existence of a Borel measure $\mu$ on 
the measurable space $(M^G, \fS^G)$ with the projections onto finite 
products satisfying 
\[ \int_{M^G} f_1(\omega(g_1))  \cdots f_n(\omega(g_n)) \, d\mu(\omega) 
= \int_Q M_{f_1} P_{g_1^{-1} g_2} 
\cdots M_{f_{n-1}} P_{g_{n-1}^{-1} g_n}  M_{f_n} 1\, d\nu\] 
for $0 \leq f_1, \ldots, f_n \in L^\infty(M,\fS,\nu)$ 
and $g_1 \prec_S \cdots \prec_S g_n$. 
We thus obtain a realization of our $(G,S,\tau)$-measure space 
on $(M^G, \fS^G, \mu)$, where the measure preserving 
$G$-action on $M^G$ is given by $(g.\omega)(h) := \omega(g^{-1}h)$.  
\end{rem}

\begin{defn} (\cite{Ba78}) 
(a) Let $(Q,\Sigma)$ and $(Q',\Sigma')$ be  measurable spaces. Then a function 
$K \: Q \times \Sigma' \to [0,\infty]$ 
is called a {\it kernel} if  \index{kernel!on measurable spaces}
\begin{itemize}
\item[\rm(K1)]\quad for every $A' \in \Sigma'$, the function $K^{A'}(\omega) := K(\omega, A')$ is 
$\Sigma$-measurable, and 
\item[\rm(K2)]\quad for every $\omega \in Q$, the function $K_\omega(A') := K(\omega, A')$ 
is a (positive) measure. 
\end{itemize} \index{kernel!Markov}
A kernel is called a {\it Markov kernel} if the measures $K_\omega$ are probability measures. 

(b) A kernel $K \: Q \times \Sigma' \to [0,\infty]$ associates to a measure 
$\mu$ on $(Q,\Sigma)$ the measure 
\[ (\mu K)(A') 
:=   \int \mu(d\omega) K(\omega,A') 
=   \int_Q K(\omega,A')\, d\mu(\omega).\] 
%To every measurable function $f' \: Q' \to [0,\infty]$, it associates the function 
%\[ (Kf')(\omega) 
%:=  \int_{Q'} K(\omega, d\omega') f'(\omega').\] 

(c) If $(Q_j, \Sigma_j)_{j=1,2,3}$ are measurable spaces, then 
composition of kernels $K_1$ on $Q_1 \times \Sigma_2$ and 
$K_2$ on $Q_2 \times \Sigma_3$ is defined by 
$(K_1 K_2)(\omega_1, A_3) = \int K_1(\omega_1, d\omega_2) K_2(\omega_2, A_3).$ 
If $S$ is a semigroup, then a family 
$(P_s)_{s \in S}$ of Markov kernels on the 
measurable space $(Q,\Sigma)$ is called a {\it semigroup 
of (Markov) kernels} if $P_s P_t = P_{st}$ for $st \in S$. 
\end{defn} \index{semigroup!of Markov kernels}

\begin{rem} \label{rem:mark} 
(a) Let $(P_t)_{t \geq 0}$ be a Markov semigroup on the polish space $(Q,\Sigma)$. 
Then we obtain for $0 \leq t_1 < \ldots < t_n$ and 
$\bt = (t_1, \ldots, t_n)$ a Markov kernel $P_\bt$ on 
$Q \times \Sigma^n$ by 
\[ P_\bt(x_0,B) = \int_{Q^{n}} \chi_B(x_1, \ldots, x_n) 
P_{t_1}(x_0, dx_1)P_{t_2- t_1}(x_1, dx_2) \cdots P_{t_n- t_{n-1}}(x_{n-1}, dx_n)\] 
(\cite[Satz~64.2]{Ba78}). 
Fixing $x_0$, we thus obtain a  a projective family of measures, 
which leads to a probability measure 
$P_{x_0}$ on the $\sigma$-algebra $\Sigma^{\R_+}$ on $Q^{\R_+}$ 
whose restrictions to cylinder sets are given by the $P_\bt(x_0,\cdot)$. 
We thus obtain a Markov kernel $P(x,\cdot) := P_x(\cdot)$ on 
$Q \times \Sigma^{\R_+}$. 
For any $t \geq 0$, we then have 
\[ P_t(x, B) 
= \int_{Q^{\R_+}} \chi_B(\omega(t))\, P(x, d\omega)
= P(x, \{ \omega(t) \in B\}),\] 
which leads to 
\begin{equation}
  \label{eq:ptf}
 (P_t f)(x) 
= \int_Q P_t(x, dy) f(y) 
= \int_{Q^{\R_+}} P(x, d\omega)f(\omega(t))\, .
\end{equation}
This is an abstract version of the 
Feynman--Kac--Nelson formula 
\index{Feynman--Kac--Nelson formula} 
that expresses the value of $P_t f$ in $x \in Q$  as an integral over all 
paths $[0,t] \to Q$ starting in $x$ with respect to the probability measure~$P_x$.

(b) For any measure $\nu$ on $Q$, we thus obtain a measure 
$P^\nu := \nu P$ on $(Q^{\R_+},\Sigma^{\R_+})$. 
If $\nu$ is  a probability measure, 
then $P^\nu$ likewise is, and we obtain a 
stochastic process $(X_t)_{t\geq 0}$ with state 
space $(Q,\Sigma)$ and initial distribution~$\nu$ 
 (\cite[Satz~62.3]{Ba78}). 
According to \cite[Satz~65.3]{Ba78}, the so obtained 
stochastic process has the Markov property.  

For $t > 0$, we have the relation 
\[ \int_{Q^\R_+} f(\omega(t))\, d P^\nu(\omega) 
= \int_Q \nu(dx) P_t(x,dy)f(y) \quad \mbox{ for } \quad t \in \R, \] 
and, for $t < s$,  
\begin{align*}
 \int_Q \int_Q f_1(x) \nu(dx) P_{s-t}(x, dy) f_2(y) 
%&= \int_Q f_1(x) (P_{s-t} f_2)(x)\, \nu(dx)\\
&=  \int_{Q^{\R_+}} f_1(\omega(t))f_2(\omega(s))\, d P^\nu(\omega). 
\end{align*}

(c) In the special case where $Q = G$ is a topological group and 
$P_t f = f * \mu_t$ for probability measures $\mu_t$ on $G$, 
we have 
\[ (P_t f)(x) = \int_G f(xy)\, d\mu_t(y) = \int_G P_t(x,dy) f(y) \quad \mbox{ for } \quad 
P_t(x,A) = \mu_t(x^{-1}A).\] 

\index{path group}
Let $P(G) := G^\R$ be the {\it path group} of all maps $\R \to G$ 
and let $P_*(G)$ be the subgroup of pinned paths
$P_*(G)=\{\omega \in P(G)\: \omega (0)=e\}.$
We have the relations 
\[ \int_{P(G)} f_1(\omega(0)) f_2(\omega(t))\, dP^\nu(\omega) = 
 \int_G\int_G f_1(g_1) f_2(g_1 g_2)\, d\nu(g_1) d\mu_t(g_2) \] 
for $t > 0$, and 
\[ \int_{P_*(G)} f(\omega(t))\, dP^\nu(\omega) 
= \int_G  \int_G f(xg)\, d\nu(x) d\mu_t(g) 
= \int_G  f(g)\, d(\nu * \mu_t)(g).\] 
This leads for $f \in L^2(G,\nu)$ and $t\geq 0$ to 
\[ (P_t f)(x) 
= (f * \mu_t)(x) = \int_G f(xg)\, d\mu_t(g) 
= \int_{G^{\R_+}} f(x\omega(t))\, dP(\omega).\] 
This is a group version of the Feynman--Kac--Nelson formula  
\eqref{eq:ptf}. 
\end{rem}

We now assume that $G$ is a second countable locally compact group 
and that $(\mu_t)_{t \geq 0}$ is a convolution semigroup of 
probability measures on $G$ which is  \index{convolution semigroup}
{\it strongly continuous} in the sense that 
$\lim_{t \to 0} \mu_t = \delta_e = \mu_0$ weakly on the space $C_b(G)$ of bounded 
continuous functions on~$G$.
We further assume that $\nu$ is a measure on $G$ satisfying 
$\nu * \mu_t = \nu$ for every $t > 0$, and, in addition, that the operators 
$P_t f := f * \mu_t$ on $L^2(G,\nu)$ are symmetric. If 
$\nu$ is a right Haar measure, 
then the symmetry of the operators $P_t$ is equivalent to 
$\mu_t^* = \mu_t$. Then we obtain for 
$t_1 \leq \cdots \leq t_n$ and 
$\bt := (t_1, \ldots, t_n)$  on $G^n$ a consistent family of measures 
\[ P^\mu_\bt :=  (\psi_n)_* (\nu \otimes \mu_{t_2-t_1}\otimes \cdots \otimes 
\mu_{t_n-t_{n-1}}),\]  
where 
$\psi_n(g_1, \ldots, g_n) = (g_1, g_1 g_2, \cdots, g_1 \cdots g_n).$ 
This in turn leads to a unique measure $P^\nu$ on the 
two-sided path space $G^\R$ with \index{two-sided path space}
$(\ev_\bt)_* P^\nu = P^\nu_\bt$ for $t_1 < \ldots < t_n$. 

{}From the Klein--Landau Reconstruction Theorem we obtain the following specialization. 

\begin{thm} \label{thm:4.11b} 
Suppose that $G$ is a second countable locally compact group. 
Let $P^\nu$ be the measure on $G^\R$ corresponding to the 
convolution semigroup $(\mu_t)_{t \geq 0}$ of symmetric probability measures 
on $G$ and the measure $\nu$ on $G$ for which the operators $P_t f = f * \mu_t$ 
define a positive semigroup structure on $L^2(G,\nu)$. Then the translation action 
$(U_t \omega)(s) := \omega(s-t)$ on $G^\R$ is measure preserving 
and $P^\nu$ is invariant under $(\theta\omega)(t) := \omega(-t)$. 
We thus obtain a reflection positive one-parameter group of Markov type 
on 
\[ \cE := L^2(G^\R, \fB^\R,\mu)\quad \mbox{  with respect to } \quad 
\cE_+ := L^2(G^\R, \fB^{\R_+},\mu),\] for which 
$\cE_0 := \ev_0^*(L^2(G,\nu))\cong L^2(G,\nu)$ 
and $\hat \cE \cong L^2(G,\nu)$ with $q(F) = E_0 F$ for $F \in \cE_+$. We further have 
\[ E_0 U_t E_0 = P_t \quad \mbox{ for } \quad P_t f = f * \mu_t,\] 
so that the $U$-cyclic subrepresentation of $\cE$ generated by $\cE_0$ 
is a unitary dilation of the 
one-parameter semigroup $(P_t)_{t \geq 0}$ of hermitian contractions 
on $L^2(G,\nu)$. 
\end{thm}

\begin{ex} (a) For $G = \R^d$, the heat semigroup is given on 
$L^2(\R^d)$ by 
\[ e^{t\Delta}f = f * \gamma_t  \quad \mbox{ where } \quad 
d\gamma_t(x) = \frac{1}{(2\pi t)^{d/2}} e^{-\frac{\|x\|^2}{2t}}\, 
dx.\] 
We call the corresponding measure on $G^\R$ the {\it Lebesgue--Wiener 
measure} (cf.\ Theorem~\ref{thm:4.11b}).  \index{measure!Lebesgue--Wiener}

(b) If $G$ is any finite dimensional Lie group and $X_1, \ldots, X_n$ is a basis of the Lie algebra,  \index{Laplacian, on $\R^n$}
then we obtain a left invariant {\it Laplacian} by 
$\Delta := \sum_{j =1}^n L_{X_j}^2,$ 
where $L_{X_j}$ denotes the left invariant vector field with 
$L_{X_j}(e) = X_j$. Then there also exists a semigroup $(\mu_t)_{t \geq 0}$ 
of probability measures on $G$ such that  
$e^{t\Delta} f = f * \mu_t$ for \break {$t \geq 0$}
(\cite[Sect.~8]{Nel69}).  
Accordingly, we obtain a {\it Haar--Wiener measure} on the path space~$G^\R$. 
\end{ex} \index{measure!Haar--Wiener}

\section*{Notes on Chapter~\ref{ch:9}} 

The material in this section is condensed from \cite{JNO16a} to which 
we refer for more details and background. This paper draws heavily 
from the work of Klein and Landau \cite{KL75,Kl78}. 
In particular, the Markov $(G,S,\tau)$-measure spaces 
generalize the path spaces studied by Klein and Landau in \cite{KL75}. 
For $(G,S,\tau)= (\R,\R_+,-\id_\R)$, the work of Klein and Landau 
was largely motivated by 
Nelson's work on the Feynman--Kac Formula in \cite{Nel64}. 

In A.~Klein's papers \cite{Kl77, Kl78} 
concerning $(G,S,\tau) = (\R,\R_+, -\id)$, the reflection positivity 
condition from Definition~\ref{def:9.1.2} is called 
Osterwalder--Schrader positivity. 
Theorem~\ref{thm:kl1} is adapted from \cite[Thm.~3.1]{Kl78}. 

Stochastic processes index by Lie groups also appear in \cite{AHH86}.

\end{bibunit}

\appendix 

\chapter{Background material}

\begin{bibunit}

In this appendix we collect precise statements of 
some basic facts on positive definite kernels and 
positive definite functions on groups and semigroups.

\section{Positive definite kernels} 
\label{sec:a}

%\subsection{Form-valued positive definite kernels} 
%\label{app:1.1}

Let $X$ be a set. Classically, reproducing kernels arise from Hilbert spaces
$\cH$ which are subspaces of the space $\C^X$ of complex-valued 
functions on $X$, for which the evaluations $f \mapsto f(x)$ are continuous, 
hence representable by elements $K_x \in \cH$ by 
\[ f(x) = \la K_x, f\ra \quad \mbox{ for }\quad f \in \cH, x \in X. \] 
Then 
\[ K \: X \times X \to \C, \quad 
K(x,y) := K_y(x) = \la K_x, K_y\ra \] 
is called the {\it reproducing kernel} of~$\cH$. 
\index{kernel!reproducing}
\index{kernel!positive definite}
As the kernel $K$ determines $\cH$ uniquely, we  write 
$\cH_K\subeq \C^X$ for the Hilbert space determined by $K$ 
and $\cH_K^0 \subeq \cH_K$ for the subspace 
spanned by the functions $(K_x)_{x \in X}$. 
A kernel function $K \: X \times X \to \C$ is the reproducing kernel 
of some Hilbert space if and only if it is {\it positive definite} in the sense 
that, for any finite collection $x_1, \ldots, x_n \in X$, the matrix 
$(K(x_j,x_k))_{1 \leq j,k\leq n}$ is positive semidefinite 
(cf.\ \cite{Ar50}, \cite[Ch.~1]{Ne00}). 
There is a natural generalization to Hilbert spaces $\cH$ of functions 
with values in a Hilbert space $\cV$, i.e., $\cH \subeq \cV^X$. Then 
$K_x(f) = f(x)$ is a 
linear operator $K_x \: \cH \to \cV$ and we obtain a kernel 
$K(x,y) := K_x K_y^* \in B(\cV)$ with values in the bounded operators on $\cV$. 
However, there are also situations where one would like to deal with kernels 
whose values are unbounded operators, so that one has to generalize this 
context further. The notion of a positive 
definite kernel with values in the space $\Bil(V)$ of bilinear 
complex-valued forms on a real linear space~$V$ 
provides a natural context to deal with all relevant cases. 

\begin{defn} \label{def:a.1.1} Let $X$ be a set and $V$ be a real vector space. 
We write $\Bil(V) = \Bil(V,\C)$ for the space of complex-valued 
bilinear forms on~$V$. We call a map 
$K \: X \times X \to \Bil(V)$ 
\index{kernel!positive definite!operator-valued}
{\it a positive definite kernel} if the associated scalar-valued kernel 
\[  K^\flat \: (X \times V) \times (X \times V) \to \C, \quad 
K^\flat((x,v),(y,w)) := K(x,y)(v,w) \] 
is positive definite.

The corresponding reproducing kernel Hilbert space 
$\cH_{K^\flat} \subeq \C^{X \times V}$ is generated by the elements 
$K^\flat_{x,v}$, $x \in X, v \in V$, with the inner product 
\[   \la K^\flat_{x,v}, K^\flat_{y,w}\ra = K(x,y)(v,w) =: K^\flat_{y,w}(x,v),\] 
so that, for all $f \in \cH_{K^\flat}$, we have 
\begin{equation}
  \label{eq:eval1}
f(x,v) = \la K^\flat_{x,v}, f \ra.
\end{equation}
We identify $\cH_{K^\flat}$ with a subspace of the space $(V^*)^X$ 
of functions on $X$ with values in the space $V^*$ of complex-valued 
linear functionals on $V$ by identifying 
$f \in \cH_{K^\flat}$ with the function $f^* \: X \to V^*, f^*(x):= f(x,\cdot)$. 
We call 
\[ \cH_K := \{ f^* \: f \in \cH_{K^\flat} \} \subeq (V^*)^X \] 
\index{Hilbert space!reproducing kernel} 
the {\it (vector-valued) reproducing kernel space associated to $K$}. The elements 
\[ K_{x,v} := (K^\flat_{x,v})^*
\quad \mbox{ with } \quad K_{x,v}(y) = K(y,x)(\cdot, v)
\quad \mbox { for } \quad x,y \in  X, v,w \in V, \] 
then form a dense subspace of $\cH_K$ with 
\begin{equation}
  \label{eq:app}
\la K_{x,v}, K_{y,w} \ra = K(x,y)(v,w)
\end{equation}
and 
\begin{equation}
  \label{eq:eval-vstar}
\la K_{x,v}, f \ra = f(x)(v) \quad \mbox{ for } \quad 
f \in \cH_K, x \in X, v \in V.
\end{equation}
\end{defn}

\begin{rem} \label{rem:a.1.2} 
Equation~\eqref{eq:app} shows that 
positive definiteness of $K$ implies the existence of a Hilbert space 
$\cH$ and a map $\gamma \: X \to \Hom(V,\cH), \gamma(x)(v) := K_{x,v}$ such that 
\[ K(x,y)(v,w)= \la \gamma(x)(v), \gamma(y)(w) \ra.\] 
If, conversely, such a factorization exists, then the positive definiteness 
follows from 
\[ \sum_{j,k = 1}^n \oline{c_j} c_kK(x_j, x_k)(v_j, v_k) 
=  \sum_{j,k = 1}^n \oline{c_j} c_k \la \gamma(x_j)(v_j), \gamma(x_k)(v_k) \ra
=  \Big\|\sum_{k = 1}^n c_k \gamma(x_k)(v_k) \Big\|^2 \geq 0.\] 
\end{rem}

\begin{ex} If $V$ is a complex Hilbert space, then we write 
$\Sesq(V) \subeq \Bil(V)$ for the linear subspace of 
\index{sesquilinear maps on $V$, $\Sesq(V)$}
{\it sesquilinear maps}, i.e., maps which are anti-linear in the first 
and complex linear in the second argument. 
If $X$ is a set and 
$K \: X \times X \to B(V)$ is an operator-valued kernel, then $K$ is 
positive definite if and only if the corresponding kernel 
\[ \tilde K \: (X \times V) \times (X \times V) \to \C, 
\quad \tilde K((x,v),(y,w)) := \la v, K(x,y)w\ra \] 
is positive definite (Definition~\ref{def:a.1.1}).
%, and this means that the kernel 
%\[ K' \: X \times X \to \Sesq(V) \subeq \Bil(V), \quad 
%K'(x,y)(v,w) := \la v, K(x,y) w \ra \] 
%is positive definite. 
Then, for each $f \in \cH_{\tilde K}$, the linear functionals 
$f^*(x) \: V \to \C$ are continuous, hence can be identified 
with elements of $V$. Accordingly, we consider $\cH_K$ as a space 
of $V$-valued functions (see \cite[Ch.~I]{Ne00} for more details). 
\end{ex}

\begin{ex} \label{ex:a.2} 
Let $\cA$ be a $C^*$-algebra. 
% and 
%$\cA_h := \{ A \in \cA \:  A^* = A\}$ be the real linear space 
%of hermitian elements in $\cA$. 
A linear functional $\omega \in \cA^*$ is called {\it positive} if \index{positive functional}
the kernel $K_\omega(A,B) :=\omega(A^*B)$  on $\cA \times \cA$ 
is positive definite. Then the corresponding 
Hilbert space $\cH_\omega := \cH_{K_\omega}$ can be realized in the space 
$\cA^\sharp$ of anti-linear functionals on $\cA$. It can be obtained from the 
GNS representation $(\pi_\omega, \cH_\omega, \Omega)$ 
(\cite[Cor.~2.3.17]{BR02}) by 
\[ \Gamma \: \cH_\omega \to \cA^\sharp, \quad 
\Gamma(\xi)(A) := \la \pi(A)\Omega, \xi \ra  \]
because 
$\la \pi(A)\Omega, \pi(B)\Omega \ra = \omega(A^*B) = K_\omega(A,B).$
Note that $\cA$ has a natural representation on $\cA^\sharp$ by 
$(A.\beta)(B) := \beta(A^*B)$ and that $\Gamma$ is equivariant with respect to this 
representation.\begin{footnote}{This realization of the Hilbert space 
$\cH_\omega$ has the advantage that we can view its elements as elements of the 
space $\cA^\sharp$ (see \cite{Ne00} for many applications of this perspective). 
Usually, $\cH_\omega$ is obtained as the Hilbert completion of 
a quotient of $\cA$ by a left ideal which leads to a much less concrete space.   
}\end{footnote}
\end{ex}

If $X = G$ is a group and the kernel $K$ is invariant 
under right translations, then it is of the form 
$K(g,h) = \phi(gh^{-1})$ for a function 
$\phi \: G \to \Bil(V)$. 

\index{positive definite!function!on group}
\index{positive definite!function!on involutive semigroup}
\begin{defn} \label{def:a.3} Let $G$ be a group and 
let $V$ be a real vector space.  
A function $\phi \: G \to \Bil(V)$ is said to be {\it positive definite} 
if the $\Bil(V)$-valued kernel $K(g,h) := \phi(gh^{-1})$ is positive definite. 

\index{semigroup!involutive}
Suppose, more generally, that $(S,*)$ is an {\it involutive semigroup}, 
i.e., a semigroup~$S$, endowed with an involutive map $s \mapsto s^*$
satisfying $(st)^* = t^*s^*$ for $s,t \in S$. A~function 
$\phi \:  S \to \Bil(V)$ is called {\it positive definite} 
if the kernel $K(s,t) := \phi(st^*)$ is positive definite. 
\end{defn}

The following proposition generalizes the GNS construction 
to form-valued positive definite functions on groups (\cite[Prop.~A.4]{NO15b}). 

\begin{prop}{\rm(GNS-construction for groups)} \label{prop:gns} Let $V$ be a real vector space.
  \begin{enumerate}
  \item[\rm(a)] Let $\phi \: G \to \Bil(V)$ be a positive definite function. Then 
$(U^\phi_g f)(h) := f(hg)$ defines a unitary representation of $G$ on the 
reproducing kernel Hilbert space $\cH_\phi := \cH_K \subeq (V^*)^G$ with 
kernel $K(g,h) = \phi(gh^{-1})$ and the range of the map 
\[ j \: V \to \cH_\phi, \quad j(v)(g)(w):=  \phi(g)(w,v), \qquad 
j(v) = K_{e,v}, \] 
is a cyclic subspace, i.e., $U_G^\phi j(V)$ spans a dense subspace of $\cH$. 
We then have 
\begin{equation}
  \label{eq:gns-a}
\phi(g)(v,w) = \la j(v), U^\phi_g j(w) \ra 
\quad \mbox{ for } \quad g \in G, v,w, \in V.
\end{equation}
  \item[\rm(b)]
If, conversely, $(U, \cH)$ is a unitary representation of $G$ and 
$j \: V \to \cH$ a linear map, then 
\[ \phi \: G \to \Bil(V), \quad \phi(g)(v,w) := \la j(v), U_g j(w) \ra \] 
is a $\Bil(V)$-valued positive definite function. 
If, in addition, $j(V)$ is cyclic, then 
$(U, \cH)$ is unitarily equivalent to $(U^\phi, \cH_\phi)$. 
  \end{enumerate}
\end{prop}

\begin{prf} (a) For the kernel 
$K(g,h) := \phi(gh^{-1})$ and $v \in V$, the right invariance of the kernel 
$K$ on $G$ implies on $\cH_\phi$ the existence of well-defined 
unitary operators $U_g$ with 
\[ U_g K_{h,v} = K_{hg^{-1}, v} \quad \mbox{ for } \quad 
g,h \in G, v \in V.\] 
In fact, \eqref{eq:app} shows that 
\[ \la K_{h_1 g^{-1}, v_1}, K_{h_2 g^{-1}, v_2} \ra 
= K(h_1g^{-1}, h_2 g^{-1})(v_1, v_2) 
= K(h_1, h_2)(v_1, v_2) = \la K_{h_1, v_1}, K_{h_2, v_2} \ra.\] 
For $f \in \cH_\phi$, we then have 
\[ (U_g f)(h)(v) 
= \la K_{h,v}, U_g f \ra 
= \la U_{g^{-1}}K_{h,v}, f \ra 
= \la K_{hg,v}, f \ra = f(hg)(v),\] 
i.e., $(U_g f)(h) = f(hg)$. 
Further, $j(v) = K_{e,v}$ satisfies 
$U_g j(v) = K_{g^{-1},v}$, which shows that $U_G j(V)$ is total in $\cH_\phi$. 
Finally we note that 
\[ \la j(v), U_g j(w) \ra = \la K_{e,v}, K_{g^{-1},w} \ra 
= K(g)(v,w) = \phi(g)(v,w).\]

(b) The positive definiteness of $\phi$ follows 
with Remark~\ref{rem:a.1.2} easily from the relation 
$\phi(gh^{-1})(v,w) = \la U_g^{-1}v, U_h^{-1} w\ra$. Since $j(V)$ is cyclic, 
the map 
$\Gamma(\xi)(g)(v) := \la U_g^{-1} j(v), \xi \ra$ 
defines an injection $\cH \into (V^*)^G$ whose range is the subspace 
$\cH_\phi$ and which is equivariant with respect to the right translation 
representation~$U^\phi$. 
\end{prf}

\begin{rem} (a) If $\phi \: G \to \Bil(V)$ is a positive definite function, 
then \eqref{eq:gns-a} shows that,if $\tilde V := \oline{j(V)}$, 
which is the real Hilbert space defined by completing $V$ with respect to the 
positive semidefinite form $\phi(e)$, then 
\begin{equation}
  \label{eq:tilde-phi}
\tilde \phi(g)(v,w) = \la v, U_g w \ra 
\end{equation}
defines a positive definite function 
\[ \tilde \phi \: G \to \Bil(\tilde V) \quad \mbox{ with } \quad 
\tilde \phi(g)(j(v),j(w)) = \phi(g)(v,w) \quad \mbox{ for } \quad v,w \in V.\] 
Therefore it often suffices to consider $\Bil(V)$-valued positive definite 
functions for real Hilbert space $V$ for which 
$\phi(e)$ is a positive definite hermitian form on $V$ 
whose real part is the scalar product on $V$. 
In terms of \eqref{eq:gns-a}, this means that 
$j \: V \to \cH$ is an isometric embedding of the real Hilbert 
space~$V$.

(b) If $\cV$ is a real Hilbert space 
and $j$ is continuous, then the adjoint operator 
$j^* \: \cH \to \cV$ is well-defined and we obtain from \eqref{eq:tilde-phi} 
the $B(\cV)$-valued positive definite function 
$\phi(g) := j^* U_g j$ which can be used to realize $\cH$ in $\cV^G$.  
\end{rem} 

\begin{ex} \label{ex:vv-gns} (Vector-valued GNS construction for semigroups) 
\break (\cite[Sect.~3.1]{Ne00}) Let $(U, \cH)$ be a representation of the 
unital involutive semigroup $(S,*,e)$, $\cV$ be a Hilbert space and  
$j \: \cV \to \cH$ be a linear map for which 
$U_S j(\cV)$ is total in $\cH$. 
Then $\phi(s) := j^* U_s j$ is a $B(\cV)$-valued positive definite function 
on~$S$  with $\phi(e) = j^*j$ (which is $\1$ if and only if $j$ is isometric) 
because we have the factorization 
\[\phi(st^*) = j^* U_{st^*}j = (j^* U_s)(j U_t)^*.\] 
The map 
\[ \Phi \: \cH \to \cV^S, \quad \Phi(v)(s) = j^* U_s v \] 
is an $S$-equivariant realization of $\cH$ as the reproducing kernel space
$\cH_\phi \subeq \cV^S$, on which $S$ acts by right translation, i.e., 
$(U^\phi_s f)(t) = f(ts)$. 

Conversely, let $S$ be a unital involutive semigroup 
and $\phi \: S \to B(\cV)$ be a positive definite function. 
Write $\cH_\phi \subeq \cV^S$ for the corresponding reproducing kernel space 
with kernel $K(s,t) = \phi(st^*)$ and 
$\cH_\phi^0$ for the dense subspace spanned by 
$K_{s,v} = \ev_s^*v, s \in S, v \in \cV$. 
Then $(U^\phi_s f)(t) := f(ts)$ defines a 
$*$-representation of $S$ on $\cH_\phi^0$. 
\index{representation!exponentially bounded}
We say that $\phi$ is {\it exponentially bounded} if 
all  operators $U^\phi_s$ are bounded, so that we actually 
obtain a representation of $S$ by bounded operators on $\cH_\phi$ 
(cf.\ \cite[\S II.4]{Ne00}). 
Then 
$\ev_e \circ U^\phi_s = \ev_s$ leads to 
\begin{equation}
  \label{eq:factori}
\phi(s) = \ev_s \ev_e^* =\ev_e U^\phi_s \ev_e^* 
\quad \mbox{ and } \quad 
\phi v = \ev_e^*v = K_{e,v}.
\end{equation}

If $S = G$ is a group with $s^* = s^{-1}$, then $\phi$ is always exponentially bounded and the representation $(U^\phi, \cH_\phi)$ is unitary. 
\end{ex}

\begin{lem} \label{lem:mult} Let $(S,*,e)$ be a unital involutive semigroup 
and $\phi \: S \to B(\cV)$ be a positive definite function  
with $\phi(e) = \1$. We write $(U^\phi, \cH_\phi)$ for the representation 
on the corresponding reproducing kernel Hilbert space 
$\cH_\phi \subeq \cV^S$ by $(U^\phi(s)f)(t) := f(ts)$. 
Then the inclusion 
$\iota \: \cV \to \cH_\phi,  \iota(v)(s) := \phi(s)v$ 
 is surjective if and only if 
$\phi$ is multiplicative, i.e., a representation. 
\end{lem}

\begin{prf} If $\phi$ is multiplicative, then 
$(U^\phi_s \iota(v))(t) = \phi(ts)v = \phi(t)\phi(s)v \in \iota(\cV)$. 
Therefore the $S$-cyclic subspace $\iota(\cV)$ is invariant, which 
implies that $\iota$ is surjective. 

Suppose, conversely, that $\iota$ is surjective. 
Then each $f \in \cH_\phi$ satisfies 
$f(s) = \phi(s)f(e)$. For $v \in \cV$ and $t,s \in S$, this leads to 
\[ \phi(st)v = (U^\phi_t\iota(v))(s) 
= \phi(s) \cdot (U^\phi_t\iota(v))(e) 
= \phi(s) \iota(v)(t) = \phi(s) \phi(t)v.\] 
Therefore $\phi$ is multiplicative. 
\end{prf}

\section{Integral representations} 

For a realization of unitary representations associated to 
positive definite functions in $L^2$-spaces, integral representations are of crucial importance. 
The following result is a straight-forward generalization of Bochner's Theorem 
for locally compact abelian groups. Here we 
write $\Sesq^+(V) \subeq \Sesq(V)$ for the 
convex cone of positive semidefinite forms if $V$ is a complex linear space.  

\index{Theorem!Bochner}
\begin{thm} \label{thm:bochner} Let $G$ be a locally compact abelian group. 
Then a function \break $\phi \: G \to \Sesq(V)$ 
for which all functions $\phi^{v,w} := \phi(\cdot)(v,w), v,w \in  V$,
 are continuous is positive definite if and only if 
there exists a (uniquely determined) finite $\Sesq^+(V)$-valued Borel 
measure $\mu$ on the locally compact group $\hat G$ 
such that $\hat\mu(g) := \int_{\hat G} \chi(g)\, d\mu(\chi) = \phi(g)$ holds 
for every $g \in G$ pointwise on $ V \times  V$. 
\end{thm}

\begin{prf} If $\phi = \hat\mu$ 
holds for a finite $\Sesq^+(V)$-valued Borel 
measure $\mu$ on the locally compact group $\hat G$, then the kernel 
$\phi(gh^{-1})(\xi,\eta) 
= \int_{\hat G} \chi(g)\oline{\chi(h)}\, d\mu^{\xi,\eta}(\chi)$ 
on $G \times V$ is positive definite because 
\begin{align*}
\sum_{j,k=1}^n  \phi(g_j g_k^{-1})(\xi_j,\xi_k) 
&= \sum_{j,k=1}^n 
\int_{\hat G} \chi(g_j)\oline{\chi(g_k)}\, d\mu^{\xi_j,\xi_k}(\chi)\\
&=  \sum_{j,k=1}^n   \int_{\hat G}\, d\mu^{
\oline{\chi(g_j)}\xi_j,\oline{\chi(g_k)}\xi_k}(\chi) 
=  \int_{\hat G} d\mu^{\xi,\xi} \geq 0 
\end{align*}
holds for $\xi := \sum_{j = 1}^n \oline{\chi(g_j)}\xi_j$ 
and $\mu^{\xi,\eta}(\cdot) = \mu(\cdot)(\xi,\eta)$. 

Suppose, conversely, that $\phi$ is positive definite. 
Then Bochner's Theorem for scalar-valued positive definite 
functions yields for every $v \in  V$ a finite positive measure 
$\mu^v$ on $\hat G$ such that 
\[ \phi^{v,v}(g) =  \hat{\mu^v}(g) 
= \int_{\hat G} \chi(g)\, d\mu^v(\chi).\]
By polarization, we obtain for $v,w \in V$ complex measures 
$\mu^{v,w} := \frac{1}{4} \sum_{k = 0}^3 i^{-k} \mu^{v+ i^k w}$  
on $\hat G$ with $\phi^{v,w} = \hat{\mu^{v,w}}$. Then the 
collection $(\mu^{v,w})_{v,w \in  V}$ of complex measures on $\hat G$ defines a 
$\Sesq^+(V)$-valued measure by 
$\mu(\cdot)(v,w) := \mu^{v,w}$ for $v,w \in  V,$ 
and this measure satisfies $\hat\mu = \phi$. 
\end{prf} 

\begin{rem} Suppose that $E$ is the spectral measure on the character group $\hat G$ for which 
the continuous unitary representation $(U, \cH)$ is represented by 
$U_g = \int_{\hat G} \chi(g)\, dE(\chi)$. Then, for $\xi \in \cH$, the positive definite 
function $U^\xi(g) := \la \xi, U_g \xi\ra$ is the Fourier transform of the measure 
$E^{\xi,\xi} = \la \xi, E(\cdot)\xi\ra$. 
This establishes a close link between spectral measures 
and the representing measures in the preceding theorem.
\end{rem}

\index{Theorem!Laplace transforms and positive definite kernels}
The following theorem follows from  \cite[Thm.~B.3]{NO15b}: 
\begin{thm}\label{thm:I.7} {\rm(Laplace transforms and positive definite kernels)} 
Let $E$ be a finite dimensional real vector space and $\cD \subeq E$ 
be a non-empty open convex subset. Let $V$ be a Hilbert space and 
$\phi \: \cD \to B(V)$ be such that 
\begin{itemize} %enumerate}
\item[\rm(L1)]\ \ \ \ \  the kernel $K(x,y) 
=\phi\big(\frac{x+y}{2}\big)$ is positive definite. 
\item[\rm(L2)]\ \ \ \ \  $\phi$ is weak operator continuous on every line segment in $\cD$, 
i.e., all functions \break 
$t \mapsto \la v,\phi(x+th)v\ra$, $v \in V$, are continuous 
on $\{ t \in \R \: x + th \in \cD\}$. 
\end{itemize} %enumerate}
Then the following assertions hold: 
\begin{enumerate}
\item[\rm(i)] There exists a unique $\Herm^+(V)$-valued Borel measure 
$\mu$ on the dual space $E^*$ such that 
\[ \phi(x) = \cL(\mu)(x) := \int_{E^*} e^{-\lambda(x)}\, d\mu(\lambda) \quad \mbox{ for } \quad 
x \in \cD.\] 
\item[\rm(ii)] Let $T_\cD = \cD + i E \subeq E_\C$ be the tube domain over $\cD$. 
Then the map 
\[ {\cal F} \: L^2(E^*, \mu;V) \to \cO(T_\cD,V),\quad 
\la \xi, \cF(f)(z)\ra  := \la e_{-{\oline z}/2}\xi, f \ra \] 
is unitary  onto the reproducing kernel space ${\cal H}_\phi := \cH_{K}$
corresponding to the kernel associated to $\phi$. 
It intertwines the unitary representation 
\begin{align*}
 (U_x f)(\alpha) &:= e^{i\alpha(x)}f(\alpha) \quad \mbox{ on } \quad 
L^2(E^*,\mu) \quad \mbox{ and } \quad \\
(\tilde U_x f)(z) &:= f(z - 2 i x)\quad \mbox{ on } \quad \cH_\phi. 
\end{align*}
\item[\rm(iii)] $\phi$ extends to a  unique 
holomorphic function $\hat\phi$ on the tube domain $T_\cD$ 
which is positive definite in the sense that the
 kernel $\hat\phi\big(\frac{z + \oline w}{2}\big)$ is positive definite.
\end{enumerate}
\end{thm}

\begin{cor} A continuous function $\phi \: \cD \to \C$ on an open 
convex subset of a finite dimensional real vector space $E$ 
is positive definite if and only if there exists a positive 
measure $\mu$ on $E^*$ such that $\phi = \cL(\mu)\res_{\cD}$. 
\end{cor}

The preceding theorem generalizes in an obvious way to $\Sesq(V)$-valued functions, where 
the corresponding measure $\mu$ has values in the cone $\Sesq^+(V)$. One can use the 
same arguments as in the proof of Bochner's Theorem 
(Theorem~\ref{thm:bochner}).

The following lemma sharpens the ``technical lemma'' in \cite[App.~A]{KL82}. 
We recall the notation 
$\cS_\beta  = \{ z \in \C \: 0 <  \Im z< \beta\}$ 
for horizontal strips in $\C$. 

\begin{lem} \label{lem:a.2} 
Let $U_t = e^{itH}$ be a unitary one-parameter group on 
$\cH$, $E$ the spectral measure of $H$, $\xi \in \cH$, 
$E^\xi := \la \xi, E(\cdot)\xi\ra$, $\beta > 0$ and 
$\phi(t) := \la \xi, U_t \xi\ra  = \int_\R e^{it\lambda}\, dE^\xi(\lambda).$ 
Then the following are equivalent: 
\begin{enumerate}
\item[\rm(i)] There exists a continuous function $\psi$ on $\oline{\cS_\beta}$, 
holomorphic on $\cS_\beta$, such that $\psi\res_{\R} = \phi$.
\item[\rm(ii)] $\cL(E^\xi)(\beta) = \int_\R e^{-\beta \lambda}\, dE^\xi(\lambda) < \infty$. 
\item[\rm(iii)] $\xi \in \cD(e^{-\beta H/2}).$
\end{enumerate}
\end{lem}

\begin{prf} That (i) implies (ii) follows from \cite[p.~311]{Ri66}. 
If, conversely, (ii) is satisfied, then 
$\psi(z) := \cL(E^\xi)(-iz)$ is defined on $\oline{\cS_\beta}$, holomorphic on $\cS_\beta$ 
and $\psi\res_\R = \phi$. 
Finally, the equivalence of (ii) and (iii) follows from the 
definition of the unbounded operator 
$e^{-\beta H/2}$ in terms of the spectral measure~$E$. 
\end{prf}

\begin{lem} \label{lem:b.4} {\rm(Criterion for the existence of $\cL(\mu)(x)$)} 
Let $\cV$ be a Hilbert space and $\mu$ be a finite 
$\Herm^+(\cV)$-valued Borel measure on $\R$, so that we can consider 
its Laplace transform $\cL(\mu)$, taking values in $\Herm(\cV)$, whenever 
the integral 
\[ \tr\big(\cL(\mu)(x)S\big) 
= \int_\R e^{-\lambda x}\, d\mu^S(\lambda) \quad \mbox { for } \quad 
d\mu^S(\lambda) = \tr(S d\mu(\lambda)),\] 
exists for every positive trace class operator $S$ on $\cV$.
This is equivalent to the finiteness of the integrals 
$\cL(\mu^v)(x)$ for every $v \in \cV$, where 
$d\mu^v(\lambda) = \la v,d\mu(\lambda) v\ra$. 
\end{lem}

\begin{prf} For $x \in \R$, the existence of $\cL(\mu)(x)$ implies the finiteness of the 
integrals $\cL(\mu^v)(x)$ for $v \in \cV$. Suppose, conversely, that all these 
integrals are finite. Then we obtain by polarization a hermitian form 
$\beta(v,w) := \int_\R e^{-\lambda x}\, \la v, d\mu(\lambda)w \ra$ 
on $\cV$. We claim that $\beta$ is continuous. 
As $\cV$ is in particular a Fr\'echet space, it suffices to show that, for every 
$w \in \cV$, the linear functional $\lambda(v) := \beta(w,v)$ is continuous 
(\cite[Thm.~2.17]{Ru73}). 

The linear functionals $f_n(v) := \int_{-n}^n e^{-\lambda x}\, \la w,d\mu(\lambda)v \ra$ 
are continuous because $\mu$ is a bounded measure and the functions 
$e_x(\lambda) := e^{\lambda(x)}$ 
are bounded on bounded intervals. By the Monotone Convergence 
Theorem, combined with the Polarization Identity,
 $f_n \to f$ holds pointwise on $\cV$, and this implies the continuity 
of $f$ (\cite[Thm.~2.8]{Ru73}). 

For a positive trace class operators $S = \sum_n \la v_n, \cdot \ra v_n$ with 
$\tr S = \sum_n \|v_n\|^2 < \infty$, we now obtain 
\[ \cL(\mu^S)(x) 
= \sum_n \cL(\mu^{v_n})(x)
= \sum_n \beta(v_n,v_n) 
\leq \|\beta\|  \sum_n \|v_n\|^2 < \infty.\qedhere\] 
\end{prf}

\printindex

\end{bibunit}

%%%%%%%%%%%%%%%%%%%%%%%%%%%%%%%%%%%%%%%%%%%%%%%%%%%%%%%%%%%%%%%%%%%%%%

\end{document}